\documentclass{article}

\usepackage{arxiv}

\usepackage[utf8]{inputenc} 
\usepackage[T1]{fontenc}    
\usepackage{hyperref}       
\usepackage{url}            
\usepackage{booktabs}       
\usepackage{amsfonts}       
\usepackage{microtype}      
\usepackage{color}
\usepackage{lipsum}
\usepackage{amsthm}
\usepackage{chemarrow}
\usepackage{indentfirst}
\usepackage{graphicx}
\usepackage{epstopdf}
\usepackage{fourier}
\usepackage{bm}
\usepackage{bbm} 
\usepackage{mathtools}
\usepackage{latexsym,enumerate}
\usepackage{soul}
\usepackage{amsfonts}
\usepackage{amsmath}
\usepackage{amssymb}
\usepackage{float}
\usepackage{geometry}
\usepackage{graphicx}
\usepackage{epstopdf}
\usepackage{bm}
\usepackage{stackrel}
\usepackage{array}
\usepackage{multicol}
\usepackage{ascii}

\newcommand{\BEA}{\begin{eqnarray}}
\newcommand{\EEA}{\end{eqnarray}}
\newcommand{\BR}{\mathbb{R}}

\newtheorem{lem}{Lemma}[section]
\newtheorem{theo}{Theorem}[section]
\newtheorem{prop}{Proposition}[section]

\newtheorem{coro}{Corollary}[section]
\newtheorem{assu}{Assumption}[section]
\newcommand{\comment}[1]{}

\newtheorem{remark}{Remark}

\newtheorem{definition}{Definition}[section]

\title{Radial Basis Approximation of Tensor Fields on Manifolds: From Operator Estimation to Manifold Learning}

\author{
John Harlim \\
Department of Mathematics, Department of Meteorology and Atmospheric
Science, \\
Institute for Computational and Data Sciences \\
The Pennsylvania State University, University Park, PA 16802, USA\\
\texttt{jharlim@psu.edu} \\
\And
Shixiao Willing Jiang \\
Institute of Mathematical Sciences \\
ShanghaiTech University, Shanghai 201210, China\\
\texttt{jiangshx@shanghaitech.edu.cn} 
\And
John Wilson Peoples \\
Department of Mathematics \\
The Pennsylvania State University, University Park, PA 16802, USA \\
\texttt{jwp5828@psu.edu}\\ 
}

\begin{document}

\maketitle

\begin{abstract}
In this paper, we study the Radial Basis Function (RBF) approximation to differential operators on smooth tensor fields defined on closed Riemannian submanifolds of Euclidean space, identified by randomly sampled point cloud data. {The formulation in this paper leverages a fundamental fact that the covariant derivative on a submanifold is the projection of the directional derivative in the ambient Euclidean space onto the tangent space of the submanifold. To differentiate a test function (or vector field) on the submanifold with respect to the Euclidean metric, the RBF interpolation is applied to extend the function (or vector field) in the ambient Euclidean space. When the manifolds are unknown, we develop an improved second-order local SVD technique for estimating local tangent spaces on the manifold. When the classical pointwise non-symmetric RBF formulation is used to solve Laplacian eigenvalue problems, we found that while accurate estimation of the leading spectra can be obtained with large enough data, such an approximation often produces irrelevant complex-valued spectra (or pollution) as the true spectra are real-valued and positive. To avoid such an issue,} we introduce a symmetric RBF discrete approximation of the Laplacians induced by a weak formulation on appropriate Hilbert spaces. Unlike the non-symmetric approximation, this formulation guarantees non-negative real-valued spectra and the orthogonality of the eigenvectors. Theoretically, we establish the convergence of the eigenpairs of both the Laplace-Beltrami operator and Bochner Laplacian {for the symmetric formulation} in the limit of large data with convergence rates. Numerically, we provide supporting examples for approximations of the Laplace-Beltrami operator and various vector Laplacians, including the Bochner, Hodge, and Lichnerowicz Laplacians.
\end{abstract}


\section{Introduction}

Estimation of differential operators is an important computational task in applied mathematics and engineering science. While this estimation problem has been studied since the time of Euler in the mid-18th century, numerical differential equations emerges as an important sub-field of computational mathematics in the 1940s when modern computers were starting to be developed for solving differential equations. Among many available numerical methods, finite-difference, finite-volume, and finite-element methods are considered the most reliable algorithms to produce accurate solutions whenever one can control the distribution of points (or nodes) and meshes. Specification of the nodes, however, requires some knowledge of the domain and is subjected to the curse of dimension.

On the other hand, the Radial Basis Function (RBF) method \cite{b2003radial} has been considered as a promising alternative that can produce very accurate approximations \cite{narcowich1991norms,wu1993local} even with randomly distributed nodes \cite{kanagawa2018gaussian} on high-dimensional domains. Beyond the mesh-free approximation of differential operators, the deep connection of the RBF to kernel technique has also been documented in nonparametric statistical literature \cite{christmann2008support,kanagawa2018gaussian} for machine learning applications. While kernel methods play a significant role in supervised machine learning \cite{christmann2008support,cucker2002mathematical}, in unsupervised learning, a kernel approach usually corresponds to the construction of a graph whose spectral properties \cite{chung1997spectral,von2008consistency} can be used for clustering \cite{ng2002spectral}, dimensionality reduction \cite{belkin2003laplacian}, and manifold learning \cite{coifman2006diffusion}, among others. In the context of manifold learning, given a set of data that lie on a $d-$dimensional compact Riemannian sub-manifold of Euclidean domain $\mathbb{R}^n$, the objective is to represent observable with eigensolutions of an approximate Laplacian operator. For this purpose, spectral convergence type results, concerning the convergence of the graph Laplacian matrix induced by exponentially decaying kernels to the Laplacian operator on functions, are well-documented for closed manifolds \cite{belkin2007convergence,burago2014graph,trillos2020error,calder2019improved,dunson2021spectral}  and for manifolds with boundary \cite{peoples2021spectral}. Beyond Laplacian on functions, graph-based approximation of connection Laplacian on vector fields \cite{singer2017spectral} and a Galerkin-based approximation to Hodge Laplacian on 1-forms \cite{berry2020spectral} have also been considered separately. Since these manifold learning methods fundamentally approximate Laplacian operators that act on smooth tensor fields defined on manifolds, it is natural to ask whether this learning problem can be solved using the RBF method. Another motivating point is that the available graph-based approaches can only approximate a limited type of differential operators whereas RBF can approximate arbitrary differential operators, including general $k$-Laplacians.

Indeed, RBF has been proposed to solve PDEs on 2D surfaces \cite{Fuselier2009Stability,fuselier2012scattered,piret2012orthogonal,fuselier2013high}. In these papers, they showed that RBF solutions converge, especially when the point clouds are appropriately placed which requires some parameterization of the manifolds. When the surface parameterization is unknown, there are several approaches to characterize the manifolds, such as the
closest point method \cite{ruuth2008simple}, the orthogonal gradient \cite{piret2012orthogonal}, the moving least squares \cite{liang2013solving}, and the local SVD method \cite{donoho2003hessian,zhang2004principal,tyagi2013tangent}. The first two methods require laying additional grid points on the ambient space, which can be numerically expensive when the ambient dimension is high. The moving least squares locally fit multivariate quadratic functions of the local coordinates approximated by PCA (or local SVD) on the metric tensor. While fast theoretical convergence rate when the data point is well sampled (see the detailed discussion in Section~2.2 of \cite{liang2013solving}), based on the presented numerical results in the paper, it is unclear whether the same convergence rate can be achieved when the data are randomly sampled. The local SVD method, which is also used in determining the local coordinates in the moving least-squares method, can readily approximate the local tangent space for randomly sampled training data points. This information alone readily allows one to approximate differential operators on the manifolds. 

From the perspective of approximation theory, the radial basis type kernel is universal in the sense that the induced Reproducing Kernel Hilbert Space (RKHS) is dense in the space of continuous and bounded function on a compact domain, under the standard uniform norm \cite{steinwart2001influence,sriperumbudur2011universality}. While this property is very appealing, previous works on RBF suggest that the non-symmetric pointwise approximation to the Laplace-Beltrami operator can be numerically problematic \cite{Fuselier2009Stability,piret2012orthogonal, fuselier2013high}. In particular, they numerically reported that when the number of training data is small the eigenvalues of the non-symmetric RBF Laplacian matrix that approximates the negative-definite Laplace-Beltrami operator are not only complex-valued, but they can also be on the positive half-plane. These papers also empirically reported that this issue, which is related to spectral pollution \cite{llobet1990spectral,davies2004spectral,lewin2010spectral}, can be overcome with more data points. The work in this paper is somewhat motivated by many open questions from these empirical results.


\subsection{Contribution of This Paper and a Summary of Our Findings}

One of the objectives of this paper is to assess the potential of the RBF method in solving the manifold learning problem, involving approximating the Laplacians acting on functions and vector fields of smooth manifolds, where the underlying manifold is identified by a set of randomly sample point cloud data. { The work in this paper extends the fundamental fact that the covariant derivative on a submanifold is the projection of the directional derivative in the ambient Euclidean space onto the tangent space of the submanifold to various differential operators, including Laplacians on functions and vector fields. Since the formulation involves the ambient dimension, $n$, the resulting approximation will be computationally feasible for problems where the ambient dimension is much smaller than the data size, $n \ll N$. While such dimensionality scaling requires an extensive computational memory for manifold learning problems with $n=O(N)$, it can still be useful for a broad class of applications involving {eigenvalue problems and PDEs where the ambient dimension is moderately high, $n=10-100$, but much smaller than $N$. See eigenvalue problem examples in Sections \ref{sec:gentorus} and \ref{4dtorus}. Also see} e.g., the companion paper \cite{yan2023spectral} that uses the symmetric approximation discussed below to solve elliptic PDEs on unknown manifolds.}

First, we study the non-symmetric Laplacian RBF matrix, which is a pointwise approximation to the Laplacian operator, that is used in \cite{Fuselier2009Stability,piret2012orthogonal} to approximate the Laplace-Beltrami operator. Through a numerical study in Section~\ref{sec6}, we found that {non-symmetric RBF (NRBF) formulation} can produce a very accurate estimation of the leading spectral properties whenever the number of sample points used to approximate the RBF matrix is large enough and the local tangent space is sufficiently accurately approximated. In this paper:
\begin{itemize}
\item[1a)] We propose an improved local SVD algorithm for approximating the local tangent spaces of the manifolds, where the improvement is essentially due to the additional steps designed to correct errors induced by the curvatures (see Section~\ref{sec3.2}). We provide a theoretical error bound (see Theorem~\ref{main P Theorem}). We numerically find that this error (from the local tangent space approximation) dominates the errors induced by the non-symmetric RBF approximation of the leading eigensolutions (see Figures~\ref{fig_gentor_4} and \ref{fig_erreigflattori_1p5}).
On the upside, since this local SVD method is numerically not expensive even with very large data, applying NRBF with the accurately estimated local tangent space will give a very accurate estimation of the leading spectra with possibly expensive computational costs. Namely, solving eigenvalue problems of non-symmetric, dense discrete NRBF matrices, which can be very large, especially for the approximation of vector Laplacians. On the downside, since this pointwise estimation relies on the accuracy of the local tangent space approximation, such a high accuracy will not be attainable when the data is corrupted by noise.
\item[1b)] Through numerical verification in Sections~\ref{4dtorus} and \ref{sphereexample}, we detect another issue with the non-symmetric formulation. That is, when the training data size is not large enough, the non-symmetric RBF Laplacian matrix is subjected to spectral pollution \cite{llobet1990spectral,davies2004spectral,lewin2010spectral}. Specifically, the resulting matrix possesses eigenvalues that are  irrelevant to the true eigenvalues. If we increase the training data, while the leading eigenvalues (that are closer to zero) are accurately estimated, the irrelevant estimates will not disappear. Their occurrence will be on higher modes. This issue can be problematic in manifold learning applications since the spectra of the underlying Laplacian to be estimated are unknown. As we have pointed out in 1a), large data size may not be numerically feasible for the non-symmetric formulation, especially in solving eigenvalue problems corresponding to Laplacians on vector fields.
\end{itemize}
To overcome the limitation of the non-symmetric formulation, we consider a symmetric discrete formulation induced by a weak approximation of the Laplacians on appropriate Hilbert spaces. Several advantages of this symmetric approximation are that the estimated eigenvalues are guaranteed to be non-negative real-valued, and the corresponding estimates for the eigenvectors (or eigenvector fields) are real-valued and orthogonal. {Here, the Laplace-Beltrami operator is defined to be semi-positive definite.} The price we are paying to guarantee estimators with these nice properties is that the approximation is less accurate compared to the non-symmetric formulation provided that the latter works. Particularly, the error of the symmetric RBF is dominated by the {Monte-Carlo rate. Our findings are based on:}
\begin{itemize}
\item[2a)] A spectral convergence study with error bounds for the estimations of eigenvalues and eigenvectors (or eigenvector fields). See Theorems~\ref{eigvalconv} and \ref{conveigvec} for the approximation of eigenvalues and eigenfunctions of Laplace-Beltrami operator, respectively. See Theorems~\ref{eigvalconv Bochner} and \ref{conveigvec Bochner} for the approximation of eigenvalues and eigenvector fields of Bochner Laplacian, respectively.
\item[2b)] Numerical inspections on the estimation of Laplace-Beltrami operator. We show the empirical convergence as a function of training data size.
{Based on our numerical comparison on a 2D torus embedded in $\mathbb{R}^{21}$, we found that the symmetric RBF produces less accurate estimates compared to the Diffusion Maps (DM) algorithm \cite{coifman2006diffusion} in the estimation of leading eigenvalues, but more accurate in the estimation of non-leading eigenvalues.}
\item[2c)] Numerical inspections on the estimation of Bochner, Hodge, and Lichnerowicz Laplacians. We show the empirical convergence as a function of training data size. 
\end{itemize}


\subsection{Organization of This Paper}

\noindent {\bf Section~\ref{sec2}:} We provide a detailed formulation for discrete approximation of differential operators on smooth manifolds, where we will focus on the RBF technique as an interpolant. Overall, the nature of the approximation is ``exterior'' in the sense that the tangential derivatives will be represented as a projection (or restriction) of appropriate ambient derivatives onto the local tangent space. To clarify this formulation, we overview the notion of the projection matrix that allows the exterior representation that will be realized through appropriate discrete tensors. We give a concrete discrete formulation for gradient, the Laplace-Beltrami operator, covariant derivative, and the Bochner Laplacian. We discuss the symmetric and non-symmetric formulations of Laplacians. We list the RBF approximation in Table~\ref{tab:rbfd}, where we also include the estimation of the Hodge and Lichnerowicz Laplacian (which detailed derivations are reported in Appendix~\ref{app:A}).

\noindent {\bf Section~\ref{sec3}:} We present a novel algorithm to improve the characterization of the manifold from randomly sampled point cloud data, which subsequently allows us to improve the approximation of the projection matrix, which is explicitly not available when the manifold is unknown. The new local SVD method, which accounts for curvature errors, will improve the accuracy of RBF in the pointwise estimation of arbitrary differential operators.

{\noindent{\bf Section~\ref{section4}:} We deduce the spectral convergence of the symmetric estimation of the Laplace-Beltrami operator and the Bochner Laplacian. Error bounds for both the eigenvalues and eigenfunctions estimations will be given in terms of the number of training data, the smoothness, and the dimension and co-dimension of the manifolds. These results rely on a probabilistic type convergence result of the RBF interpolation error, {extending the deterministic error bound of the interpolation using the RKHS induced by the Mat\'ern kernel \cite{fuselier2012scattered}}, which is reported in Appendix~\ref{app:B}. To keep the section short, we only present the proof for the spectral convergence of the Laplace-Beltrami operator. We document the proofs of the intermediate bounds needed for this proof in Appendices ~\ref{app:B}, \ref{app:C1} and \ref{app:C2}. Since the proof of the eigenvector estimation is more technical, we also present it in Appendix~\ref{app:C3}. Since the proofs of the Bochner Laplacian approximation follow the same arguments as those for the Laplace-Beltrami operator, we document them in Appendix \ref{app:D}.}

\noindent{\bf Section~\ref{sec6}:} We present numerical examples to inspect the non-symmetric and symmetric RBF spectral approximations. The first two examples focus on the approximations of Laplace-Beltrami on functions. {In the first example (the two-dimensional general torus), we verify the effectiveness of the approximation when the low-dimensional manifold is embedded in a high-dimensional ambient space (high co-dimension). In this example, we also compare the results from a graph-based approach, diffusion maps \cite{coifman2006diffusion}. To ensure the diffusion maps result is representative, we report additional numerical results over an extensive parameter tuning in Appendix~\ref{app:E}. In the second example (four- and five-dimensional flat torus), our aim is to verify the effectiveness of the approximation when the intrinsic dimension of the manifold is higher than that in the first example. In the third example, we verify the accuracy of the spectral estimation of the Bochner, Hodge, and Lichnerowicz Laplacians on the sphere.}

\noindent{\bf Section~\ref{sec7}:} We close this paper with a short summary and discussion of remaining and emerging open problems.

\noindent { For reader's convenience, we provide a list of notations in Table~\ref{tab:notat}.}

\begin{table}[tbp]
\caption{List of notations. The notations below are used throughout the entire manuscript, and defined in Sections 2-4.}
\vspace{2pt}
\label{tab:notat}
\par
\begin{center}
\scalebox{0.95}[0.95]{
\begin{tabular}{p{30pt} p{170pt}|p{30pt} p{170pt}}
\hline
\textbf{Symbol} & \textbf{ Definition} & \textbf{Symbol} & \textbf{ Definition} \\
\hline
\textbf{Sec. } & \textbf{2} & & \\
\hline
$M$ & \text{a manifold} & $\mathbb{R}$ & \text{Euclidean space} \\
$d$ & \text{intrinsic dimension of $M$} & $n$ &\text{ambient space dimension}\\
$\theta^i$ & \text{intrinsic coordinate} $i=1,\ldots,d$ & $X^i$ & \text{ambient coordinate} $i=1,\ldots,n$\\
$\mathbf{P}$ & $P(x)$, \text{tangential project matrix} \text{ } & $\mathcal{P}$ & \text{tangential projection tensor in} (\ref{eqn:Pmatr}) \\
$P_{ij}$ & \text{entries of projection matrix} $i,j=1,\ldots,n$ & $\mathbf{p}_i$ & \text{$i$th column of projection matrix $\mathbf{P}$} $\mathbf{p}_i=(P_{1i}(x),\ldots,P_{ni}(x))$\\
$\iota$ & the local parameterization of the manifold & $D\iota$ & the pushforward or the $n\times d$ Jacobian matrix \\
$N$ & \text{number of data points} & $X$ & $X=\{x_i\}_{i=1}^N=\{x_1,\ldots,x_N\}$  \\
$R_N$ & restriction operator, $R_N f = \mathbf{f}$ \newline $=(f(x_1),\ldots,f(x_N))^\top $ & $|_M$, $|_X$& $|_M$ is the restriction on manifold $M$ \newline $|_X$ is the restriction on data set $X$ \\
$f$ & \text{a function} & $\mathbf{f}$ & \text{a function on data points}, \newline $\mathbf{f} = (f(x_1),\ldots,f(x_N))^\top$ \\
$I_{\phi_s}\mathbf{f}$ & \text{interpolation of $\mathbf{f}$ using RBF $\phi_s$ } & $s$ &\text{shape parameter}\\
$\boldsymbol{\Phi}$ & \text{interpolation matrix} & $\mathbf{c}$ & \text{interpolation coefficient}, $(c_1,\ldots,c_N)^\top$ \\
$\textrm{grad}_g$ & \text{gradient w.r.t.} {Riemannian metric,} $\textrm{grad}_g f = g^{ij}\frac{\partial f}{\partial \theta^i} \frac{\partial}{\partial \theta^j}$ & $\overline{\textrm{grad}}_{\mathbb{R}^n}$ & \text{gradient in Euclidean space}, $\overline{\textrm{grad}}_{\mathbb{R}^n} f = \delta^{ij}\frac{\partial f}{\partial X^i} \frac{\partial}{\partial X^j}$ \\
$\mathcal{G}_i$ & $\mathcal{G}_i = \mathbf{p}_i \cdot \overline{\textrm{grad}}_{\mathbb{R}^n}$   & $\mathbf{G}_i$ & $N\times N$ \text{matrix that approximates} $\mathcal{G}_i$\\
$\mathbf{T},\boldsymbol{\tau}_i$ &  $d$ \text{orthonormal tangent vectors}, \newline $\mathbf{T}=[\boldsymbol{\tau}_1,\ldots,\boldsymbol{\tau}_d]$ & $\mathbf{N},\mathbf{n}_i$ & \text{vectors that are orthogonal to} {tangent space}, $\mathbf{N}=[\mathbf{n}_1,\ldots,\mathbf{n}_{n-d}]$\\
$\nabla$ & Levi-Civita connection on $M$ & $\bar{\nabla}$ & Euclidean connection on $\mathbb{R}^n$ \\
$\Delta_M$ & \text{Laplace-Beltrami} & $\Delta_B$ &\text{Bochner Laplacian}\\
$q(x)$ & \text{sampling density} & $\mathbf{Q}$ & \text{diagonal matrix with entries} $q(x_i)$\\
$\mathcal{H}_i$ & $\mathcal{H}_i U(x)=\mathbf{P}\mathrm{diag}(\mathcal{G}_i,\ldots,\mathcal{G}_i)U(x)$ &
$\mathbf{H}_i$ & $Nn\times Nn$ \text{matrix that estimates} $\mathcal{H}_i$ \\
\hline
\textbf{Sec. } & \textbf{3} &  &  \\
\hline
$\boldsymbol{\rho}$ & \text{geodesic normal coordinate,} $\boldsymbol{\rho}=(\rho_1,\ldots,\rho_d)$ & $\rho$ & \text{geodesic distance}, $\rho=|\boldsymbol{\rho}|$\\
$\tilde{\mathbf{P}}$ & \text{1st-order SVD estimate of $\mathbf{P}$} & $\hat{\mathbf{P}}$ &\text{2nd-order SVD estimate of $\mathbf{P}$}\\
$K$ & $K$-\text{nearest neighbors} & $K_{max}$& \text{maximum principal curvature}\\
$D$ & $K>D:=\frac{1}{2}d(d+1)$, \text{the minimum} \text{number for the $K$-nearest neighbors} & $\mathbf{D}$ & $\mathbf{D}=[\mathbf{D}_1,\ldots,\mathbf{D}_K]$, \text{$\mathbf{D}_i=y_i-x$ is} \text{the $i$th column of $\mathbf{D}\in \mathbb{R}^{n\times K}$}\\
$\boldsymbol{\Psi}$ & \text{matrix with orthonormal tangent} \text{column vectors} $\boldsymbol{\Psi}=[\boldsymbol{{\psi}}_1,\ldots,\boldsymbol{{\psi}}_d]$ & $\boldsymbol{\hat{\Psi}}$ & $\boldsymbol{\hat{\Psi}}=[\boldsymbol{\hat{\psi}}_1,\ldots,\boldsymbol{\hat{\psi}}_d]\in \mathbb{R}^{n\times d}$ \text{is an} \text{estimator of} $\boldsymbol{\Psi}=[\boldsymbol{{\psi}}_1,\ldots,\boldsymbol{{\psi}}_d]$\\
\hline
\textbf{Sec. } & \textbf{4} & &\text{}\\
\hline
$\langle f,h\rangle_{L^2(M)}$ & \quad inner product in $L^2(M)$ defined as $\langle f,h\rangle_{L^2(M)} = \int_M fh \ d\mathrm{Vol}$ & $\langle\mathbf{f},\mathbf{h} \rangle_{L^2(\mu_N)}$ & \quad \ appropriate inner product defined as $\langle\mathbf{f},\mathbf{h} \rangle_{L^2(\mu_N)} = \frac{1}{N}\mathbf{f}^\top \mathbf{h}$\\
$\lambda_i$ & \text{eigenvalues}  &  $\hat{\lambda}_i$ &\text{approximate eigenvalues}\\
$m$ & the geometric multiplicity of eigenvalues & & \\
 \hline
\end{tabular}
}
\end{center}
\end{table}

\section{Basic Formulation for Estimating Differential Operators on Manifolds}\label{sec2}

In this section, we first review the approximations of the gradient, the Laplace-Beltrami operator, and covariant derivative and then formulate the detailed discrete approximations of the connection of a vector field and the Bochner Laplacian, on smooth closed manifolds. Both the Laplace-Beltrami operator and the Bochner Laplacian have two natural discrete estimators: symmetric and non-symmetric formulations. Since each formulation has its own practical and theoretical advantages, we describe in detail both approximations. Following similar derivations, we also report the discrete approximation to other differential operators that are relevant to vector fields, e.g. the Hodge and Lichnerowicz Laplacian (see Appendix~\ref{app:A} for the detailed derivations).

In this paper, we consider estimating differential operators acting on functions (and vector fields) defined on a $d-$dimensional closed manifold $M$ embedded in ambient space $\mathbb{R}^n$, where $d\leq n$. Each operator is estimated using an ambient space formulation followed by a projection onto the local tangent space of the manifold using a projection matrix $\mathbf{P}$. Before describing the discrete approximation of the operators, we introduce $\mathbf{P}$, discuss some of its basic properties, and quickly review the radial basis function (RBF) interpolation which is a convenient method to approximate functions and tensor fields from the point cloud training data $X = \{x_i\}_{i=1}^N$.

In the following, we periodically use the notation $\textup{diag}(a_1 , \dots , a_j)$ to denote a $j \times j$ diagonal matrix with the listed entries along the diagonal. We also define $f \mapsto R_N f = \mathbf{f} = ( f(x_1), f(x_2), \ldots, f(x_N) )^\top$ for all $f:M \to \mathbb{R}$, which is a restriction operator to the function values on training data set $X = \{x_1,\ldots,x_N\}$. In the remainder of this paper, we use boldface to denote discrete objects (vectors, matrices) and script font to denote operators on continuous objects.

\begin{definition}\label{defn2.1} For any point $x\in M$, the local parameterization $\iota :O\subseteq \mathbb{R}^{d} \longrightarrow M\subseteq \mathbb{R}^{n}$, is defined through the following map, $\boldsymbol{\Theta }_{\iota ^{-1}(x)} \boldsymbol{\longmapsto }\mathbf{X}%
_{x}$. Here, $O$ denotes a domain that contains the point $\iota ^{-1}(x)$, which we denoted as $\boldsymbol{\Theta }_{\iota ^{-1}(x)}$ in the canonical coordinates $\left(\frac{\partial}{\partial  \theta ^{1}}\Big\vert_{\iota ^{-1}(x)},\ldots,\frac{\partial}{\partial  \theta ^{d}}\Big|_{\iota ^{-1}(x)}\right)$ and $\mathbf{X}_{x}$ is the embedded point represented in the ambient coordinates $\left(\frac{\partial}{\partial  X ^{1}}\Big\vert_x,\ldots,\frac{\partial}{\partial  X ^{n}}\Big\vert_x\right)$. The pushforward $D \iota(x): T_{\iota^{-1}(x)}O \to T_xM$ is an $n\times d$ matrix given by
$
D\iota(x) =\left[\frac{\partial \mathbf{X}_x}{\partial
\theta ^{1}},\ldots ,\frac{\partial \mathbf{X}_x}{\partial \theta ^{d}}\right],
$
a matrix whose columns form a basis of $T_{x}M$ in $\mathbb{R}%
^{n}$.
\end{definition}

\begin{definition}\label{def:P} The projection matrix $\mathbf{P} =P(x) \in \mathbb{R}^{n\times n}$ on $x\in M$ is defined with the matrix-valued function $P: M \to \textup{Proj}(n, \mathbb{R}) \subset \mathbb{R}^{n\times n}$,
$$
P=[P_{st}]_{s,t=1}^n:= \left[ \frac{\partial X^{s}}{\partial \theta ^{i}}%
g^{ij}\frac{\partial X^{t}}{\partial \theta ^{j}}\right]_{s,t=1}^{n},
$$
where $g^{ij}$ denotes the matrix entries of the inverse of the Riemannian metric tensor $g_{ij}$.
Since $[g^{ij}]_{d\times d}=(D\iota ^{\top }D\iota)^{-1}$, the
projection matrix can be equivalently defined as%
\begin{equation}
P:= D\iota (D\iota ^{\top }D\iota \mathbf{)}^{-1}D\iota ^{\top
}.  \label{eqn:PDi}
\end{equation}%
\end{definition}

Before proving some basic properties of the projection matrix $\mathbf{P}$, we must fix a set of orthonormal vectors that span the tangent space $%
T_{x}M $ for each $x\in M$. In particular, for any $x\in M$, let $\{\boldsymbol{\tau }%
_{i}\in \mathbb{R}^{n\times 1}\}_{i=1}^{d}$ be the $d$ orthonormal vectors
that span $T_{x}M$ and let $\{\mathbf{n}_{i}\in \mathbb{R}^{n\times
1}\}_{i=1}^{n-d}$ be the $n-d$ orthonormal vectors that are orthogonal to $%
T_{x}M$. Here, we suppress the dependence of $\boldsymbol{\tau}_i$ and $\mathbf{n}_{i}$ on $x$ to simplify the notation. Further, let $\mathbf{T}=[\boldsymbol{\tau}_{1}$ $\ldots $ $%
\boldsymbol{\tau}_{d}]\in \mathbb{R}^{n\times d}$ and $\mathbf{N}=[\mathbf{n}%
_{1}$ $\ldots $ $\mathbf{n}_{n-d}]\in \mathbb{R}^{n\times (n-d)}$. Since $\begin{pmatrix}
\mathbf{T}$ $\mathbf{N} \end{pmatrix}\in \mathbb{R}^{n\times n}$ is an orthonormal
matrix, one has the following relation,
\begin{equation*}
\mathbf{I}=\begin{pmatrix}\mathbf{T} & \mathbf{N}\end{pmatrix}\begin{pmatrix}
\mathbf{T}^{\top } \\
\mathbf{N}^{\top }
\end{pmatrix} =\mathbf{TT}^{\top }+\mathbf{NN}^{\top }=\sum_{i=1}^{d}\boldsymbol{\tau}%
_{i}\boldsymbol{\tau}_{i}^{\top }+\sum_{i=1}^{n-d}\mathbf{n}_{i}\mathbf{n}%
_{i}^{\top }.
\end{equation*}%
We have the following proposition summarizing the basic properties of $\mathbf{P}$.

\begin{prop}
\label{propP} For any $x\in M$, let $\mathbf{P} = P(x)\in \mathbb{R}^{n\times n}$ be the projection matrix defined in Definition \ref{def:P} and let $\mathbf{T}=[\boldsymbol{\tau}_{1},\ldots, \boldsymbol{\tau}_{d}]\in
\mathbb{R}^{n\times d}$ be any $d$ orthonormal vectors
that span $T_{x}M$. Then \newline
\noindent (1) $\mathbf{P}$ is symmetric; \newline
\noindent (2) $\mathbf{P}^{2}=\mathbf{P}$; \newline
\noindent (3) $\mathbf{P} = P(x) = D\iota(x) (D\iota(x) ^{\top }D\iota(x) \mathbf{)}%
^{-1}D\iota(x)^{\top }=\mathbf{TT}^{\top }$. \newline
\noindent (4) $\sum_{i=1}^{n}|\mathbf{p}_{i}|^{2}=d$, where $\mathbf{p}%
_{i}=\left( P_{1i}(x),\ldots ,P_{ni}(x)\right) ^{\top }$ is the $i$th column of $%
\mathbf{P}$.
\end{prop}

\begin{proof}
Properties (1) and (2) are obvious from the definition of $\mathbf{P}$\ in (%
\ref{eqn:PDi}).\ Property (3) can be easily obtained by observing that both sides of the equation are orthogonal projections, and $%
\mathrm{span}\left\{ \frac{\partial \mathbf{X}_x}{\partial \theta ^{1}},\ldots
,\frac{\partial \mathbf{X}_x}{\partial \theta ^{d}}\right\} =\mathrm{span}\{%
\boldsymbol{\tau}_{1}$ $\ldots $ $\boldsymbol{\tau}_{d}\}$. To see (4), for each point $x\in M$, write $\mathbf{P}$ as
\begin{equation*}
\mathbf{P}=\mathbf{TT}^{\top }=\left[
\begin{array}{ccc}
P_{11}(x) & \cdots & P_{1n}(x) \\
\vdots & \ddots & \vdots \\
P_{n1}(x) & \cdots & P_{nn}(x)%
\end{array}%
\right] =\left[
\begin{array}{c}
\mathbf{p}_{1}^{\top } \\
\vdots \\
\mathbf{p}_{n}^{\top }%
\end{array}%
\right] ,
\end{equation*}%
where $\mathbf{p}_{i}\in \mathbb{R}^{n\times 1}$ for $i=1,\ldots ,n$ is the $%
i$th column of $\mathbf{P}$\textbf{. } It remains to observe the following chain of equalities:
\begin{equation*}
\sum_{i=1}^{n}\left\vert \mathbf{p}_{i}\right\vert
^{2}=\mathrm{tr}\left( \mathbf{P^{\top }P}\right)
=\mathrm{tr}\left( \mathbf{P}\right) =\mathrm{tr}\left( \mathbf{TT}^{\top
}\right) =\mathrm{tr}\left( \mathbf{T}^{\top }\mathbf{T}\right) =\mathrm{tr}%
\left( \mathbf{I}_{d\times d}\right) =d.
\end{equation*}
\end{proof}
Notice that while the specification of $\mathbf{T}$ is not unique which can be different
by a rotation, the projection matrix $\mathbf{P}$ is uniquely determined at each point $x\in M$.

Let $f:M\rightarrow \mathbb{R}$ be an arbitrary some smooth function. Given
function values $\mathbf{f}:=(f(x_1),\ldots, f(x_N))^\top$\ at $X = \{x
_{j}\}_{j=1}^{N}$, the radial basis function (RBF) interpolant of $f$ at $x$ takes
the form%
\begin{equation}
I_{\phi_s}\mathbf{f}(x):=\sum_{k=1}^{N}c_{k}\phi_s \left( \left\Vert x-x%
_{k}\right\Vert \right) ,  \label{eqn:intp}
\end{equation}%
where $\phi_s$ denotes the kernel function with shape parameter $s$ and $\Vert \cdot \Vert$ denotes the standard Euclidean distance. {We should point out that by being a kernel, it is positive definite as in the standard nomenclature (see e.g., Theorem 4.16 in \cite{christmann2008support}).} Here, one can interpret $I_{\phi_s}:\mathbb{R}^{N} \to C^{\alpha}(\mathbb{R}^n)$, where $\alpha$ denotes the smoothness of $\phi_s$. In practice, common choices of kernel include the Gaussian $\phi_s(r) = \exp (-\left( s r\right) ^{2})$ \cite{fasshauer2012stable}, inverse quadratic function $\phi_s(r) = (1+\left( s r\right) ^{2})^{-1}$,  or Mat\'{e}rn class kernel  \cite{fuselier2013high}. In our numerical examples, we have tested these kernels and they do not make too much difference when the shape parameters are tuned properly. {However, we will develop the theoretical analysis with the Mat\'{e}rn kernel  in Section~\ref{sec4} as it induces a reproducing kernel Hilbert space (RKHS) with Sobolev-like norm.}

The expansion coefficients $\{c_{k}\}_{k=1}^{N}$ in \eqref{eqn:intp} can be determined by a collocation method, which enforces the interpolation condition $I_{\phi_s} \mathbf{f}(x_j) = f(x_j)$ for all $j=1,\ldots, N$, or  the following linear system with the interpolation matrix $\mathbf{\Phi}$:%
\begin{equation}
\underbrace{\left[
\begin{array}{cccc}
\phi_s \left( \Vert x_{1}-x_{1}\Vert \right) & \phi_s
\left( \Vert x_{1}-x_{2}\Vert \right) & \cdots &
\phi_s \left( \Vert x_{1}-x_{N}\Vert \right) \\
\phi_s \left( \Vert x_{2}-x_{1}\Vert \right) & \phi_s
\left( \Vert x_{2}-x_{2}\Vert \right) & \cdots &
\phi_s \left( \Vert x_{2}-x_{N}\Vert \right) \\
\vdots & \vdots & \ddots & \vdots \\
\phi_s \left( \Vert x_{N}-x_{1}\Vert \right) & \phi_s
\left( \Vert x_{N}-x_{2}\Vert \right) & \cdots &
\phi_s \left( \Vert x_{N}-x_{N}\Vert \right)%
\end{array}%
\right]}_{\boldsymbol{\Phi}} \underbrace{\left[
\begin{array}{c}
c_{1} \\
c_{2} \\
\vdots \\
c_{N}%
\end{array}%
\right]}_{\mathbf{c}} =\underbrace{\left[
\begin{array}{c}
f(x_{1}) \\
f(x_{2}) \\
\vdots \\
f(x_{N})%
\end{array}%
\right]}_{\mathbf{f}}.  \label{eqn:phicf}
\end{equation}%
In general, better accuracy is obtained for flat kernels (small $s$)
[see e.g., Chap. 16--17 of \cite{fasshauer2007meshfree}], however, the corresponding system in
(\ref{eqn:phicf}) becomes increasingly ill-conditioned \cite%
{fornberg2004stable,fornberg2011stable,fasshauer2012stable}. In this article, we will
not focus on the shape parameter issues. Numerically, we will empirically choose the shape
parameter and will use pseudo-inverse to solve the linear system in (\ref{eqn:phicf}) when it is effectively singular.

With these backgrounds, we are now ready to discuss the RBF approximation to various differential operators.

{
\subsection{Gradient of a Function}
We first review the RBF projection method proposed since \cite{kansa1990multiquadrics}'s pioneering work and following works in \cite{fuselier2013high,Natasha2015Solving,shankar2015radial,lehto2017radial} for approximating gradients of functions on
manifolds. The projection method represents the manifold differential operators as
tangential gradients, which are formulated as the projection of the
appropriate derivatives in the ambient space. Precisely, the manifold gradient on a smooth function $f:M\rightarrow
\mathbb{R}$ evaluated at $x\in M$ in the Cartesian coordinates is given as,
\[
\mathrm{grad}_{g}f({x}):=\mathbf{P}\overline{\mathrm{grad}}_{%
\mathbb{R}^{n}}f({x})=( \sum_{i=1}^{d}\boldsymbol{\tau}_{i}\boldsymbol{\tau}%
_{i}^{\top } ) \overline{\mathrm{grad}}_{\mathbb{R}^{n}}f({x}),
\]%
where the
subscript $g$ is to associate the differential operator to the Riemannian
metric $g$ induced by $M$ and $\overline{\mathrm{grad}}_{\mathbb{R}^{n}}=[\partial _{X^{1}},\cdots
,\partial _{X^{n}}]^{\top }$ is the usual Euclidean gradient operator in $\mathbb{R}^{n}$. Let $\mathbf{e}_{\ell
},\ell =1,...,n$ be the standard orthonormal vectors in $X^{\ell }$
direction in $\mathbb{R}^{n}$, we can rewrite above expression in component
form as
\BEA
\mathrm{grad}_{g}f({x}):=\left[
\begin{array}{c}
\left( \mathbf{e}_{1}\cdot \mathbf{P}\right) \overline{\mathrm{grad}}_{%
\mathbb{R}^{n}}f({x}) \\
\vdots  \\
\left( \mathbf{e}_{n}\cdot \mathbf{P}\right) \overline{\mathrm{grad}}_{%
\mathbb{R}^{n}}f({x})%
\end{array}%
\right] =\left[
\begin{array}{c}
\mathbf{p}_{1}\cdot \overline{\mathrm{grad}}_{\mathbb{R}^{n}}f({x})
\\
\vdots  \\
\mathbf{p}_{n}\cdot \overline{\mathrm{grad}}_{\mathbb{R}^{n}}f({x})%
\end{array}%
\right] :=\left[
\begin{array}{c}
\mathcal{G}_{1}f({x}) \\
\vdots  \\
\mathcal{G}_{n}f({x})%
\end{array}%
\right] \label{sec2.1:gradg}
\EEA%
where $\mathbf{p}_{\ell }$ is the $\ell -$th column of the projection matrix $%
\mathbf{P}$.

One can now consider estimating the gradient operator from the available training data set $X=\{x_1,\ldots, x_N\}$.
In such a case, one considers the RBF interpolant $ I_{\phi_s} \mathbf{f} \in C^\alpha(\mathbb{R}^n)$ in \eqref{eqn:intp} which interpolates
$f$ on the available training data set $X=\{x_1,\ldots, x_N\}$ by solving \eqref{eqn:phicf}.
Using the interpolant (\ref{eqn:intp}) and denoting $r_k:=\Vert x-x_k \Vert$, one can evaluate the
tangential derivative in the $X^{i}$ direction at each node $\{{x}_{j} \in X\}_{j=1}^{N}$ as,%
\begin{eqnarray}
\mathcal{G}_{i}I_{\phi_s }\mathbf{f}({x})|_{{x}={x}%
_{j}} &=&\mathbf{p}_{i}^{\top }\cdot \overline{\mathrm{grad}}_{\mathbb{R}^{n}}I_{\phi_s }\mathbf{f}(%
{x})|_{{x}={x}_{j}}=\mathbf{p}_{i}^{\top
}\cdot \overline{\mathrm{grad}}_{\mathbb{R}^{n}}\sum_{k=1}^{N}c_{k}\phi_s (r_{k}(%
{x}))|_{{x}={x}_{j}}  \notag \\
&=&\sum_{k=1}^{N}c_{k}\mathbf{p}_{i}^{\top }\cdot \overline{\mathrm{grad}}_{\mathbb{R}%
^{n}}\phi_s (r_{k}({x}))|_{{x}={x}%
_{j}} \notag \\
&=&\sum_{k=1}^{N}c_{k}\mathbf{p}_{i}^{\top }\cdot \left( {x}-%
{x}_{k}\right) \frac{\phi_s ^{\prime }(r_{k}({x}))}{%
r_{k}({x})}|_{{x}={x}_{j}}:=
\sum_{k=1}^{N}J_{jk}^{[i]}c_{k}.  \notag 
\end{eqnarray}%
Let $\mathbf{J}_{i}=%
[ J_{jk}^{[i]}] _{j,k=1}^{N}$, and also let $\mathbf{c=}\ (c_{1},\ldots ,c_{N})^\top$ and $\mathbf{f}=f({x}%
)|_{{X}}=(f({x}_{1}),\ldots ,f({x}%
_{N})) $ as defined in \eqref{eqn:phicf}. Above equation 
can be written in matrix form for $\mathcal{G}%
_{i}I_{\phi_s }f$ at all nodes ${X}$,
\begin{equation}
\left( \mathcal{G}_{i}I_{\phi_s }\mathbf{f}\right) |_{{X}}=\mathbf{J}_{i}%
\mathbf{c}=\mathbf{J}_{i}\boldsymbol{\Phi }^{-1}\mathbf{f}:= \mathbf{G}_{i}%
\mathbf{f},  \label{eqn:DGP}
\end{equation}%
for $i=1,\ldots ,n$. Thereafter, we define

\begin{definition}\label{def:G} Let $\mathcal{G}_i I_{\phi_s}: \mathbb{R}^N \to C^\alpha(\mathbb{R}^n)$ and $R_N:C(\mathbb{R}^n) \to \mathbb{R}^N$ be the restriction operator defined as $R_N f = \mathbf{f} = (f(x_1),\ldots, f(x_N))^\top$  for any $f\in C(M)$ and $ I_{\phi_s} \mathbf{f} \in C^{\alpha}(\mathbb{R}^n)$ be the RBF interpolant as defined in \eqref{eqn:intp}. We define a linear map $\mathbf{G}_i: \mathbb{R}^N \to \mathbb{R}^N$
\BEA
\mathbf{G}_i \mathbf{f} =  R_N \mathcal{G}_i I_{\phi_s} \mathbf{f}.\label{sec2.1:eq1}
\EEA
as a discrete estimator of the differential operator $\mathcal{G}_i$, restricted on the training data $X$. On the right-hand-side, we understood $R_N \mathcal{G}_i I_{\phi_s} \mathbf{f} = \mathcal{G}_i I_{\phi_s} \mathbf{f}|_X \in \mathbb{R}^N $.
\end{definition}

To conclude, we define a linear map $\mathbf{G}: \mathbb{R}^N \to \mathbb{R}^{nN}$ with,
\BEA
\mathbf{G} \mathbf{f} = \left( \mathbf{G}_1 \mathbf{f},
\mathbf{G}_2 \mathbf{f},
\ldots,
\mathbf{G}_n \mathbf{f}
\right)^\top, \label{sec2.1:eq2}
\EEA
as a discrete estimator to the gradient restricted on the training data set $X$, in the sense that
\BEA
 \mathbf{G}  \mathbf{f} = \mathbf{G} R_N f =  R_N \left( \textup{grad}_g I_{\phi_s} \mathbf{f} \right) = (\textup{grad}_g I_{\phi_s} \mathbf{f})|_X,\label{sec21:discrete_grad}
\EEA
where the first and third equalities follow from the definition of $R_N$, and the second equality follows from \eqref{sec2.1:gradg}, \eqref{sec2.1:eq1}, \eqref{sec2.1:eq2}.

}

\subsection{The Laplace-Beltrami Operator} \label{Laplace-Beltrami definition}
The Laplace-Beltrami operator on a smooth function $f \in C^\infty(M)$ is defined as $\Delta_M f = -\textup{div}_g ( \textup{grad}_g f )$ which is semi-positive definite. Using the previous ambient space formulation for gradient, one can equivalently write
\BEA
\Delta_M f(x) &:=& -\textup{div}_g ( \textup{grad}_g f (x))
=-( \mathbf{P}\overline{\mathrm{grad}}_{\mathbb{R}^{n}}) \cdot
\left( \mathbf{P}\overline{\mathrm{grad}}_{\mathbb{R}^{n}}\right) f({x}) \notag \\
&=& -(\mathcal{G}_1 \mathcal{G}_1  + \mathcal{G}_2 \mathcal{G}_2 + \dots + \mathcal{G}_n \mathcal{G}_n) f(x), \notag
\EEA
for any $x\in M$. This identity yields a pointwise estimate of $\Delta_M$ by composing the discrete estimators for $\textup{div}_g$  and $\textup{grad}_g$. Particularly, a {\bf non-symmetric estimator} of the Laplace-Beltrami operator is a map $\mathbb{R}^N \to \mathbb{R}^N$ given by
\BEA
\mathbf{f}  \mapsto -\left(\mathbf{G}_1 \mathbf{G}_1  + \dots + \mathbf{G}_n \mathbf{G}_n\right) \mathbf{f}.  \label{eqn:Nrbf}
\EEA
\begin{remark}
The above discrete version of Laplace-Beltrami is not new. In particular, it has been well studied in the deterministic setting. See \cite{dziuk2007surface,wardetzky2007discrete,dziuk2013finite} for a detailed review of using the ambient space for estimation of Laplace-Beltrami, as well as \cite{fuselier2013high}. {In Section~\ref{sec6}, we will numerically study the spectral convergence of this non-symmetric discretization in the setting where the data are sampled randomly from an unknown manifold. We will remark on the advantages and disadvantages of such a non-symmetric formulation for manifold learning tasks.}
\end{remark}

In the weak form, for $M$ without boundary, the Laplace-Beltrami operator can be written
\BEA
\int_M  h\Delta_M f d\textup{Vol} = \int_M \langle \textup{grad}_g h, \textup{grad}_g f \rangle d\textup{Vol}, \quad \forall f,h \in C^\infty(M), \notag
\EEA
where $\langle\cdot,\cdot\rangle$  denotes the Riemannian inner product of vector fields. Using the estimators from previous subsections, it is natural  to estimate the Laplace-Beltrami operator in this setting. Based on the weak formulation, we can estimate the gradient with the matrix $\mathbf{G}: \mathbb{R}^N \to \mathbb{R}^{Nn}$, then compose with the matrix adjoint of $\mathbf{G}$, where the domain and range of $\mathbf{G}$ are equipped with appropriate inner products approximating the corresponding inner products defined on the manifold. It turns out that the adjoint of $\mathbf{G}$ following this procedure is just the standard matrix transpose. In particular, we have that $\mathbf{G}^\top \mathbf{G} : \mathbb{R}^N \to \mathbb{R}^N$ given by
\BEA
\mathbf{G}^\top \mathbf{G}\, \mathbf{f}  := \left(\mathbf{G}_1^\top \mathbf{G}_1  + \mathbf{G}^\top_2 \mathbf{G}_2+ \ldots + \mathbf{G}_n^\top \mathbf{G}_n\right) \mathbf{f}  \notag
\EEA
is a {\bf symmetric estimator} of the Laplace-Beltrami operator.
\begin{remark}
\label{non unif remark}
The above symmetric formulation makes use of the discrete approximation of continuous inner products, and only holds when the data is sampled uniformly. For data sampled from a non-uniform density $q$, however, we can perform the standard technique of correcting for non-uniform data by dividing by the sampling density. For example, the symmetric approximation to the eigenvalue problem corresponds to solving $(\lambda,\mathbf{f})$ that satisfies
\BEA
\mathbf{f}^\top\mathbf{G}^\top \mathbf{Q}^{-1}\mathbf{G}\, \mathbf{f}  := \mathbf{f}^\top\left(\mathbf{G}_1^\top \mathbf{Q}^{-1}\mathbf{G}_1  + \mathbf{G}^\top_2 \mathbf{Q}^{-1} \mathbf{G}_2+ \ldots + \mathbf{G}_n^\top\mathbf{Q}^{-1} \mathbf{G}_n\right) \mathbf{f}   = \lambda \mathbf{f}^\top\mathbf{Q}^{-1}\mathbf{f}, \label{eqn:Srbf}
\EEA
where $\mathbf{Q} \in \mathbb{R}^{N\times N}$ is a diagonal matrix with diagonal entries of sampling density $\{q(x_i)\}_{i=1,\ldots,N}$.
This weighted Monte-Carlo provides an estimate for the $L^2(M)$ inner product in the weak formulation. When $q$ is unknown, one can approximate $q$ using standard density estimation techniques, such as Kernel Density Estimation methods {\cite{parzen1962estimation}. In our numerical simulations, we use the MATLAB built-in function {\rm \texttt{mvksdensity.m}}.} 
\end{remark}

{
\subsection{Covariant Derivative}

The basic idea here follows from the tangential connection on a submanifold
of $\mathbb{R}^{n}$\ in Example 4.9 of \cite{lee2018introduction}. For
smooth vector fields $u,y\in \mathfrak{X}\left( M\right) $, the tangential
connection can be defined as%
\begin{equation}
\nabla _{u}y=\mathcal{P}\left( \bar{\nabla}_{U}Y|_{M}\right) ,
\label{eqn:nabuy}
\end{equation}%
where $\nabla $ is the Levi-Civita connection on $M$, $\bar{\nabla}:%
\mathfrak{X}\left( \mathbb{R}^{n}\right) \times \mathfrak{X}\left( \mathbb{R}%
^{n}\right) \rightarrow \mathfrak{X}\left( \mathbb{R}^{n}\right) $ is the
Euclidean connection on $\mathbb{R}^{n}$\ mapping $(U,Y)$\ to $\bar{\nabla}%
_{U}Y$ (Example 4.8 of \cite{lee2018introduction}), $\mathcal{P}:T\mathbb{R}%
^{n}\rightarrow TM$ is the orthogonal projection onto $TM$, and $U$ and $Y$
are smooth extensions of $u$ and $y$ to an open subset in $\mathbb{R}^{n}$
satisfying $U|_{M}=u$ and $Y|_{M}=y$, respectively. Such extensions exist by
Exercise A.23 and the identity result does not depend on the chosen extension
by Proposition 4.5 in \cite{lee2018introduction}. The identity (\ref%
{eqn:nabuy}) holds true based on the properties and uniqueness of
Levi-Civita connection (see \cite{lee2018introduction}). More geometric intuition and detailed results can
also be found in \cite{do1992riemannian,morita2001geometry,crane2010trivial,2015Discrete} and references therein. The key observation
for identity (\ref{eqn:nabuy}) is that the covariant derivative can be
written in terms of the tangential projection and Euclidean derivative.
In the remainder of this section, we review the covariant derivative from a geometric viewpoint and then formulate
the tangential projection identity (\ref{eqn:nabuy}) from a computational viewpoint.



Let $u\in \mathfrak{X}(M)$ be a vector field on $M$, and let $U\in\mathfrak{X}(\mathbb{R}^{n})$ be an extension of $u$ to an open set $O \subseteq M$. Then $U$ is related to $u$ via the local parameterization as follows:
$
U\left( x\right) =D\iota \left( x\right) u\left( x\right), 
$
where $u(x)=(u^1(x),\ldots,u^d(x))$ is the coordinate representation of the vector field $u$ w.r.t. the basis $\left\{\frac{\partial}{\partial \theta^r}\right\}$.
Using $[D\iota(x)]_{sr} = \frac{\partial X^s}{\partial \theta^r}$, we have the following equation relating the components of the vector fields:
\BEA
U^s = u^r \frac{\partial X^s}{\partial \theta^r}. \label{sec2.2:eq1}
\EEA
Using this identity, we can derive
\begin{prop}\label{prop2p2} Let $U =  U^i \frac{\partial}{\partial X^i} \in \mathfrak{X}(\mathbb{R}^n)$ be a vector field such that $U|_M=u\in \mathfrak{X}(M)$. Using the notation in Definition~\ref{defn2.1},  we have
\BEA
\label{ambient cristoffel}
 \sum_r g^{ij} \frac{\partial X^r}{\partial \theta^j} \frac{\partial U^r }{\partial \theta^k} =  \frac{\partial u^i}{\partial \theta^k } +  u^p\Gamma^i_{pk}.
\EEA
\end{prop}
The proof is relegated in appendix \ref{app:A}. As in the previous section, the projection matrix $\mathbf{P}$ in Definition \ref{def:P} has been used for approximating operators acting on functions as matrix-vector multiplication. In the following, we first introduce a tangential projection tensor in order to derive  identities for operators acting on tensor fields with extensions. Geometrically, the tangential projection tensor projects  a vector field of $\mathfrak{X}(\mathbb{R%
}^{n})$ onto a vector field of $\mathfrak{X}(M)$.

\begin{definition}
The tangential
projection tensor $\mathcal{P}:\mathfrak{X}(%
\mathbb{R}^{n})\rightarrow \mathfrak{X}(M)$ is defined as
\begin{equation}
\mathcal{P}=\delta _{sr}\frac{\partial X^{s}}{\partial \theta ^{i}}g^{ij}%
\frac{\partial X^{t}}{\partial \theta ^{j}}\mathrm{d}X^{r}\otimes \frac{%
\partial }{\partial X^{t}},  \label{eqn:Pmatr}
\end{equation}%
where $\delta _{sr}$ is the Kronecker delta function.
In particular, for a vector field $%
Y=Y^{k}\frac{\partial }{\partial X^{k}}\in \mathfrak{X}(\mathbb{R}^{n})$,\
one has
\begin{eqnarray*}
\mathcal{P}\left( Y|_{M}\right) &=&\delta _{sr}Y^{k}\frac{\partial X^{s}}{%
\partial \theta ^{i}}g^{ij}\frac{\partial X^{t}}{\partial \theta ^{j}}%
\mathrm{d}X^{r}\left( \frac{\partial }{\partial X^{k}}\right) \otimes \frac{%
\partial }{\partial X^{t}}=\delta _{sr}Y^{r}\frac{\partial X^{s}}{\partial
\theta ^{i}}g^{ij}\frac{\partial X^{t}}{\partial \theta ^{j}}\frac{\partial
}{\partial X^{t}} \\
&=&\delta _{sr}Y^{r}\frac{\partial X^{s}}{\partial \theta ^{i}}g^{ij}\frac{%
\partial }{\partial \theta ^{j}}\in \mathfrak{X}(M).
\end{eqnarray*}%
For convenience, we simplify our notation as $\mathcal{P}Y :=\mathcal{P}\left( Y|_{M}\right) $ in the rest of the paper since we only concern about the points restricted on manifold $M$.
Obviously, for any vector field $v\in
\mathfrak{X}(M)$, we have $\mathcal{P}v=v\in \mathfrak{X}(M)$.
\end{definition}

Using Definition in (\ref{eqn:Pmatr}) and Proposition in (\ref{ambient cristoffel}), we can examine the identity $\nabla _{u}y= \mathcal{P}\bar{\nabla}_{U}Y$ via a direct calculation (see appendix \ref{app:A}).
}

\subsection{Gradient of a Vector Field}

{
There are several frameworks for the discretization of vector Laplacians on manifolds such as Discrete exterior calculus (DEC)
\cite{hirani2003discrete}, finite element exterior calculus (FEC) \cite{arnold2006finite,arnold2010finite,gillette2017finite}, Generalized Moving Least Squares (GMLS) \cite{gross2020meshfree} and spectral exterior calculus (SEC) \cite{berrysoftware,berry2020spectral}.
All these methods provide pointwise consistent discrete estimates for vector field operators such as curl, gradient, and Hodge Laplacian.
Both DEC and FEC make strong use of a simplicial complex in their formulation which helps to achieve high-order accuracy in PDE problems but is not realistic for many data science applications. GMLS is a mesh-free method that is applied to solving vector PDEs on manifolds in \cite{gross2020meshfree}. SEC can be used for processing raw data as a mesh-free tool which is appropriate for manifold learning applications.
Here, we will only focus on Bochner Laplacian and derive the tangential projection identities for various vector field differential operators in terms of the tangential projection and Euclidean derivative acting on vector fields with extensions. The formulation for the Hodge and Lichnerowicz Laplacians is similar (see Appendix \ref{app:A}).
}

The gradient of a vector field $u\in \mathfrak{X}(M)$ is defined by  ${\mathrm{grad}}_{g}u{=\sharp \nabla u}$, where $\sharp$ is the  standard musical isomorphism notation to raise the index, in this case from a $(1,1)$ tensor to a $(2,0)$ tensor. In local intrinsic coordinates, one can calculate
\begin{equation*}
{\mathrm{grad}}_{g}u= g^{kj} \left( \frac{\partial u^i}{\partial \theta^k} + u^p\Gamma^i_{pk}  \right) \frac{\partial }{\partial \theta ^{i}}%
\otimes \frac{\partial }{\partial \theta ^{j}}. 
\end{equation*}%
Using \eqref{ambient cristoffel}, we can rewrite the above as
\begin{eqnarray}
{\mathrm{grad}}_{g}u &=&g^{km}\left( \delta _{rs}g^{ij}\frac{\partial X^{r}}{%
\partial \theta ^{j}}\frac{\partial U^{s}}{\partial \theta ^{k}}\right)
\frac{\partial }{\partial \theta ^{i}}\otimes \frac{\partial }{\partial
\theta ^{m}},  \notag \\
&=&g^{km}\delta _{rs}g^{ij}\frac{\partial X^{r}}{\partial \theta ^{j}}\left(
\frac{\partial U^{s}}{\partial X^{a}}\frac{\partial X^{a}}{\partial \theta
^{k}}\right) \left( \frac{\partial X^{b}}{\partial \theta ^{i}}\frac{%
\partial }{\partial X^{b}}\right) \otimes \left( \frac{\partial }{\partial
X^{c}}\frac{\partial X^{c}}{\partial \theta ^{m}}\right),  \notag \\
&=&\delta ^{ea}\left( \delta _{ep}\frac{\partial X^{c}}{\partial \theta ^{k}}%
g^{km}\frac{\partial X^{p}}{\partial \theta ^{m}}\right) \left( \delta _{rs}%
\frac{\partial X^{b}}{\partial \theta ^{i}}g^{ij}\frac{\partial X^{r}}{%
\partial \theta ^{j}}\right) \frac{\partial U^{s}}{\partial X^{a}}\frac{%
\partial }{\partial X^{b}}\otimes \frac{\partial }{\partial X^{c}},  \notag \\
&=&\mathcal{P}_{1}\mathcal{P}_{2}\left( \delta ^{ea}\frac{\partial U^{r}}{%
\partial X^{a}}\frac{\partial }{\partial X^{r}}\otimes \frac{\partial }{%
\partial X^{e}}\right) \equiv \mathcal{P}_{1}\mathcal{P}_{2}\overline{{%
\mathrm{grad}}}_{\mathbb{R}^{n}}U,  \label{eqn:grug}
\end{eqnarray}%
where $U=U^{s}\frac{\partial }{\partial X^{s}}=u^{p}\frac{\partial X^{s}}{%
\partial \theta ^{p}}\frac{\partial }{\partial X^{s}}$, and $\mathcal{P}_{1}%
\mathcal{=}\delta _{ts}\frac{\partial X^{t}}{\partial \theta ^{k}}g^{km}%
\frac{\partial X^{b}}{\partial \theta ^{m}}\mathrm{d}X^{s}\otimes \frac{%
\partial }{\partial X^{b}}$ acting on the first tensor component $\frac{\partial }{\partial X^r}$, and $\mathcal{P%
}_{2}=\delta _{rq}\frac{\partial X^{c}}{\partial \theta ^{i}}g^{ij}\frac{%
\partial X^{r}}{\partial \theta ^{j}}\mathrm{d}X^{q}\otimes \frac{\partial }{%
\partial X^{c}}$ acting on the second component $\frac{\partial }{\partial X^e}$. Evaluating at each $x\in M$ and using the notation in \eqref{eqn:PDi}, Eq. (\ref%
{eqn:grug}) can be written in a matrix form as,%
\begin{equation*}
\begin{small}
{\mathrm{grad}}_{g}u(x)=\mathbf{P}\left( \overline{{\mathrm{grad}}}_{\mathbb{R}%
^{n}}U(x)\right) \mathbf{P}=\left[
\begin{array}{ccc}
\mathcal{G}_{1}U^{1}(x) & \cdots & \mathcal{G}_{1}U^{n}(x) \\
\vdots & \ddots & \vdots \\
\mathcal{G}_{n}U^{1}(x) & \cdots & \mathcal{G}_{n}U^{n}(x)%
\end{array}%
\right] \left[
\begin{array}{ccc}
P_{11}(x) & \cdots & P_{1n}(x) \\
\vdots & \ddots & \vdots \\
P_{n1}(x) & \cdots & P_{nn}(x)%
\end{array}%
\right] . 
\end{small}
\end{equation*}%
Interpreting $U(x) = (U^1(x),\ldots, U^n(x))^\top \in \mathbb{R}^{n\times 1}$ as a vector, one sees immediately by taking the transpose of the above formula that
\BEA
\textup{grad}_g u(x) = \left( \mathcal{H}_1U(x), \dots , \mathcal{H}_nU(x)  \right),\label{sec2.4:eq1}
\EEA
where each component can be rewritten as,
$$
\mathcal{H}_i U(x) : = \mathbf{P} \textup{diag}\left( \mathcal{G}_i , \dots , \mathcal{G}_i \right) U(x),
$$
owing to the symmetry of $\mathbf{P}$.

To write the discrete approximation on the training data set $X= \{x_1,\ldots, x_n\}$, we define $\mathbf{U}^i = (U^i(x_1),\ldots,U^i(x_N))^\top \in \mathbb{R}^{N\times 1}$ and concatenate these $\mathbf{U}^i$ to form $\mathbf{U}=((\mathbf{U}^1)^\top ,\ldots, (\mathbf{U}^n)^\top )^\top \in \mathbb{R}^{nN\times 1}$. Consider now the map $\mathbf{H}_i : \mathbb{R}^{nN} \to \mathbb{R}^{nN}$ defined by
\BEA
\mathbf{H}_i \mathbf{U}  := R_N \mathcal{H}_iI_{\phi_s} \mathbf{U},  \label{sec2.4:eq2}
\EEA
where the interpolation is defined on each $\mathbf{U}^i\in \mathbb{R}^N$ such that $I_{\phi_s} \mathbf{U}^i \in C^{\alpha}(\mathbb{R}^n)$ and the restriction $R_N$ is applied on each row. Relating to $\mathbf{G}_i$ in Definition~\ref{def:G}, one can write
\begin{equation*}
\mathbf{H}_{i}\mathbf{U} =
\mathbf{P}^{\otimes }\left[
\begin{array}{ccc}
\mathbf{G}_{i} &  &  \\
& \ddots &  \\
&  & \mathbf{G}_{i}%
\end{array}%
\right] _{Nn\times Nn}\left[
\begin{array}{c}
\mathbf{U}^{1} \\
\vdots \\
\mathbf{U}^{n}%
\end{array}%
\right] _{Nn\times 1},
\end{equation*}%
where tensor projection matrix $\mathbf{P}^{\otimes }\in \mathbb{R}^{Nn\times Nn}$ is
given by
\begin{small}
\begin{equation*}
\mathbf{P}^{\otimes }=\sum_{k=1}^{N}\left[
\begin{array}{ccc}
P_{11}(x_{k}) & \cdots & P_{1n}\left({x}_{k}\right)
\\
\vdots & \ddots & \vdots \\
P_{n1}\left(x_{k}\right) & \cdots & P_{nn}\left( x_{k}\right)%
\end{array}%
\right] _{n\times n}\otimes \left[ \delta _{kk}\right] _{N\times N}=\left[
\begin{array}{ccc}
\mathrm{diag}(\mathbf{p}_{11}) & \cdots & \mathrm{diag}(\mathbf{p}_{1n}) \\
\vdots & \ddots & \vdots \\
\mathrm{diag}(\mathbf{p}_{n1}) & \cdots & \mathrm{diag}(\mathbf{p}_{nn})%
\end{array}%
\right] _{Nn\times Nn},
\end{equation*}%
\end{small}
where $\otimes $ is the Kronecker product between two matrices, $\delta
_{kk} $ has value 1 for the entry in $k$th row and $k$th column and has 0
values elsewhere, and $\mathbf{p}_{ij}=\left( P_{ij}({x}%
_{1}),\ldots ,P_{ij}({x}_{N})\right) \in \mathbb{R}^{N\times 1}$.
Finally, consider the linear map $\mathbf{H}: \mathbb{R}^{nN} \to \mathbb{R}^{nN \times n}$ given by
\BEA
\mathbf{H} \mathbf{U} := [\mathbf{H}_1 \mathbf{U}  , \dots , \mathbf{H}_n \mathbf{U}] = R_N \textup{grad}_g I_{\phi_s} \mathbf{U}, \label{definition_H}
\EEA
as an estimator of the gradient of any vector field $U$ restricted on the training data set $X$, where on the right-hand-side, the restriction is done to each function entry-wise which results in an $N \times 1$ column vector. In the last equality, we have used the identity in \eqref{sec2.4:eq2} and the representation in \eqref{sec2.4:eq1} for the gradient of the interpolated vector field $I_{\phi_s}\mathbf{U}$ whose components are functions in $C^{\alpha}(\mathbb{R}^n)$.

\subsection{Divergence of a (2,0) Tensor Field}
Let $v$ be a $(2,0)$ tensor field of $v=v^{jk}%
\frac{\partial }{\partial \theta ^{j}}\otimes \frac{\partial }{\partial
\theta ^{k}}$ and $V$ be the corresponding extension in ambient space. The divergence of $v$ is defined as
\BEA
\mathrm{div}_{1}^{1}\left( v\right)
=C_{1}^{1}(\nabla v) \label{div_1^1},
\EEA
where $C^1_1$ denotes the contraction operator. Following a similar layout as before, we obtain an ambient space formulation of the divergence of a $(2,0)$ tensor field (see Appendix~\ref{app:A} for the detailed derivations). Interpreting $V(x)$ as an $n\times n$  matrix with columns $\bar{V}_i(x)=(V_{i1},\ldots,V_{in})^\top$ , the divergence of a $%
(2,0)$ tensor field evaluated at any $x\in M$ can be written as
\begin{equation}
\mathrm{div}_{1}^{1}\left( v(x)\right) =\mathbf{P}\mathrm{tr}_{1}^{1}\left(
\mathbf{P}\bar{\nabla}_{\mathbb{R}^{n}}\left( V(x)\right) \right) =   \sum_i \mathbf{P} \textup{diag}(\mathcal{G}_i , \dots , \mathcal{G}_i) \bar{V}_i(x) = \sum_i \mathcal{H}_i \bar{V}_i(x). \label{eqn:divvec}
\end{equation}
Using the same procedure as before, employing $\mathrm{div}_{1}^{1}$ on the RBF interpolant (2,0) tensor field $I_{\phi_s}\mathbf{\bar{V}}_i$, where $\mathbf{\bar{V}}_i = \bar{V}_i|_X \in \mathbb{R}^{nN\times 1}$ denotes the restriction of $\bar{V}_i$ on the training data, we arrive at an estimator of the divergence $\textup{div}_1^1$ of a $(2,0)$ tensor field. Namely, replacing each $\mathcal{H}_i$ with the discrete version $\mathbf{H}_i$ as defined in \eqref{sec2.4:eq2}, we obtain a map $\mathbb{R}^{nN \times n} \to \mathbb{R}^{nN}$ given by
\BEA
[\mathbf{\bar{V}}_{1},\ldots ,\mathbf{\bar{V}}_{n}]\mapsto \mathbf{H}_{1}\mathbf{\bar{V}}%
_{1}+\cdots +\mathbf{H}_{n}\mathbf{\bar{V}}_{n}.\notag
\EEA

\subsection{Bochner Laplacian} \label{Bochner definition}
The Bochner Laplacian $\Delta_B : \mathfrak{X}(M) \to \mathfrak{X}(M)$ is defined by
\BEA
\Delta_B u = -\textup{div}^1_1 \left( \textup{grad}_g u \right), \notag
\EEA
where $\textup{div}^1_1$ is in fact the formal adjoint of $\textup{grad}_g$ acting on vector fields.
With an extension to Euclidean space, the Bochner Laplacian can be formulated as:
\BEA
\bar{\Delta}_B{U} = -( \mathcal{H}_1 \mathcal{H}_1 + \dots + \mathcal{H}_n \mathcal{H}_n) {U},  \notag
\EEA
where \eqref{sec2.4:eq1} and \eqref{eqn:divvec} have been used.
A natural way to estimate $\Delta_B$ is to compose the discrete estimators for $\textup{div}_1^1$ and $\textup{grad}_g$. In particular, a {\bf non-symmetric estimator} of the Bochner Laplacian is a map $\mathbb{R}^{nN} \to \mathbb{R}^{nN}$ given by
\BEA
\mathbf{U} \mapsto -( \mathbf{H}_1 \mathbf{H}_1 + \dots + \mathbf{H}_n \mathbf{H}_n) \mathbf{U},  \notag
\EEA
where $\mathbf{U} = U|_X \in \mathbb{R}^{nN\times 1}$ and $\mathbf{H}_i$ are as defined in \eqref{sec2.4:eq2}.

The second formulation relies on the fact that  $\textup{div}^1_1$ is indeed the formal adjoint of $\textup{grad}_g$ acting on a vector field
such that,
\BEA
\int_M \langle \Delta_B u, v \rangle_x d\textup{Vol}(x) = \int_M \langle \textup{grad}_g u, \textup{grad}_g v \rangle_x d\textup{Vol}(x),\quad\quad \forall u,v\in \mathfrak{X}(M), \notag \label{bochner:weak}
\EEA
where the inner product on the left is the Riemannian inner product of vector fields, and the inner product on the right is the Riemannian inner product of $(2,0)$ tensor fields. Similar to the symmetric discrete estimator of the Laplace-Beltrami operator, we take advantages of the ambient space formulations in previous two subsections to approximate the inner product  with appropriate normalized inner products in Euclidean space.

First, we notice that the transpose of the map $\mathbf{H}:
\mathbf{U}  \mapsto \begin{bmatrix}
\mathbf{H}_1 \mathbf{U}, \mathbf{H}_2 \mathbf{U}, \dots , \mathbf{H}_n \mathbf{U}
\end{bmatrix}
$
is given by the standard transpose. Due to the possibility, however, that $\mathbf{H}^\top$ may produce a vector corresponding to a vector field with components normal to the manifold which are nonzero, there is a need to compose such an estimator with the projection matrix. With this consideration, we have $\mathbf{P}^{\otimes}\mathbf{H}^\top \mathbf{H}\mathbf{P}^{\otimes}: \mathbb{R}^{nN} \to \mathbb{R}^{nN}$ given by
\BEA
\mathbf{P}^{\otimes} \mathbf{H}^\top \mathbf{H} \mathbf{P}^{\otimes} \mathbf{U}  = \mathbf{P}^{\otimes} \mathbf{H}_1^\top \mathbf{H}_1 \mathbf{P}^{\otimes} \mathbf{U}  + \mathbf{P}^{\otimes}\mathbf{H}^\top_2 \mathbf{H}_2 \mathbf{P}^{\otimes} \mathbf{U}  + \dots + \mathbf{P}^{\otimes} \mathbf{H}_n^\top \mathbf{H}_n \mathbf{P}^{\otimes} \mathbf{U}  \notag
\EEA
as a {\bf symmetric estimator} of the Bochner Laplacian on vector fields.
\begin{remark}
Again, we note that this symmetric formulation makes use of approximating continuous inner products, and hence obviously holds only for uniform data. For data sampled from a non-uniform density $q$, we perform the same trick mentioned in Remark \ref{non unif remark}.
\end{remark}

We conclude this section with a list of RBF discrete formulation in Table~\ref{tab:rbfd}. {One can see the detailed derivation for the non-symmetric approximations of Hodge and Lichnerowicz Laplacians in Appendix~\ref{app:A}. We neglect the derivations for the symmetric approximations of the Hodge and Lichnerowicz Laplacians as they are analogous to that for the Bochner Laplacian but involve more terms. }

\begin{table}[tbp]
\caption{RBF formulation for functions and vector fields from Riemannian
geometry.
Here, {\textit{non-symmetric}} and {\textit{symmetric}} correspond to the non-symmetric and symmetric approximations
to the differential operator.
The asterisk $^*$ is the formal adjoint of the differential operator.
$\mathcal{S}_i$ and $\mathbf{S}_i$ are defined around (\ref{eqn:Vcol}) in Appendix \ref{app:A}. }
\label{tab:rbfd}\renewcommand\arraystretch{1.5}
\par
\begin{center}
\scalebox{0.75}[0.75]{
\begin{tabular}{c c c}
\hline\hline
\text{Object} & \text{Continuous operator} & \text{Discrete matrix} \\ \hline
\text{gradient} & $\mathrm{grad}_{g}:C^{\infty }(M)\rightarrow \mathfrak{X}%
(M) $ & $\mathrm{grad}_{g}:\mathbb{R}^{N}\rightarrow \mathbb{R}^{N\times n}$
\\
\text{functions} & {\quad\quad\quad} $f\mapsto \text{ }\left[ \mathcal{G}%
_{1}f,\ldots ,\mathcal{G}_{n}f\right]$ & \quad\quad\quad\ $\mathbf{f}\mapsto
(\mathbf{G}_{1}\mathbf{f,\ldots ,G}_{n}\mathbf{f})$ \\ \hline
\text{divergence} & $\mathrm{div}_{g}:\mathfrak{X}(M)\rightarrow C^{\infty
}(M)$ & $\mathrm{div}_{g}:\mathbb{R}^{N\times n}\rightarrow \mathbb{R}%
^{N\times 1}$ \\
\text{vector fields} & {\quad\quad\ } $U\mapsto \mathcal{G}_{1}U^{1}+\cdots +%
\mathcal{G}_{n}U^{n}$ & {\quad\quad\ } $\mathbf{U}\mapsto \mathbf{G}_{1}%
\mathbf{U}^{1}+\cdots +\mathbf{G}_{n}\mathbf{U}^{n}$ \\ \hline
\text{Laplace-Beltrami } & $\Delta _{g}:C^{\infty }(M)\rightarrow C^{\infty
}(M)$ & $\Delta _{g}:\mathbb{R}^{N\times 1}\rightarrow \mathbb{R}^{N\times
1} $ \\
\textit{non-symmetric} & {\quad\quad} ${f}\mapsto -\left( \mathcal{G}_{1}\mathcal{G}%
_{1}+\cdots +\mathcal{G}_{n}\mathcal{G}_{n}\right) f$ & {\quad\quad} $%
\mathbf{f}\mapsto -\left( \mathbf{G}_{1}\mathbf{G}_{1}+\cdots +\mathbf{G}_{n}%
\mathbf{G}_{n}\right) \mathbf{f}$ \\
\textit{symmetric} & {\quad\quad} ${f}\mapsto \left( \mathcal{G}_{1}^{*}\mathcal{G}%
_{1}+\cdots +\mathcal{G}_{n}^{*}\mathcal{G}_{n}\right) f$ & {\quad\quad} $%
\mathbf{f}\mapsto \left( \mathbf{G}_{1}^\top\mathbf{G}_{1}+\cdots +\mathbf{G}_{n}^\top%
\mathbf{G}_{n}\right) \mathbf{f}$ \\
\hline
\text{gradient } & ${\mathrm{grad}}_{g}:\mathfrak{X}(M)\rightarrow \mathfrak{%
X}(M)\times \mathfrak{X}(M)$ & ${\mathrm{grad}}_{g}:\mathbb{R}^{Nn\times
1}\rightarrow \mathbb{R}^{Nn\times n}$ \\
\text{vector fields} & {\quad\quad\quad} $U\mapsto \left[ \mathcal{H}%
_{1}U,\ldots ,\mathcal{H}_{n}U\right]$ & {\quad\quad\quad} $\mathbf{U}%
\mapsto \left[ \mathbf{H}_{1}\mathbf{U},\ldots ,\mathbf{H}_{n}\mathbf{U}%
\right]$ \\ \hline
\text{divergence} & $\mathrm{div}_{1}^{1}:\mathfrak{X}(M)\times \mathfrak{X}%
(M)\rightarrow \mathfrak{X}(M)$ & $\mathrm{div}_{1}^{1}:\mathbb{R}^{Nn\times
n}\rightarrow \mathbb{R}^{Nn\times 1}$ \\
\text{(2,0) tensor fields} & {\quad\quad\ } $V\mapsto \mathcal{H}_{1}\bar{V}%
_{1}+\cdots +\mathcal{H}_{n}\bar{V}_{n}$ & {\quad\quad} $[\mathbf{\bar{V}}%
_{1},\ldots ,\mathbf{\bar{V}}_{n}]\mapsto \mathbf{H}_{1}\mathbf{\bar{V}}%
_{1}+\cdots +\mathbf{H}_{n}\mathbf{\bar{V}}_{n}$ \\
& {\quad\quad\ } $\bar{V}_{i}\text{ is the }i\text{th row of }V$ & {%
\quad\quad} $\mathbf{\bar{V}}_i = [\mathbf{V}_{i1}, \ldots,\mathbf{V}_{in}]
\in \mathbb{R}^{Nn\times 1}$ \\ \hline
\text{Bochner Laplacian } & $\Delta _{B}:\mathfrak{X}(M)\rightarrow
\mathfrak{X}(M)$ & $\Delta _{B}:\mathbb{R}^{Nn\times 1}\rightarrow \mathbb{R}%
^{Nn\times 1}$ \\
\textit{non-symmetric} & $U\mapsto -(\mathcal{H}_{1}\mathcal{H}_{1}+\cdots +%
\mathcal{H}_{n}\mathcal{H}_{n})U$ & $\mathbf{U}\mapsto -( \mathbf{H}_{1}\mathbf{H%
}_{1}+\cdots +\mathbf{H}_{n}\mathbf{H}_{n})\mathbf{U}$ \\
\textit{symmetric} & $U\mapsto (\mathcal{H}_{1}^{*}\mathcal{H}_{1}+\cdots +%
\mathcal{H}_{n}^{*}\mathcal{H}_{n})U$ & $\mathbf{U}\mapsto (\mathbf{P}^{\otimes} \mathbf{H}_{1}^\top\mathbf{H%
}_{1}\mathbf{P}^{\otimes}+\cdots +\mathbf{P}^{\otimes}\mathbf{H}_{n}^\top\mathbf{H}_{n}\mathbf{P}^{\otimes})\mathbf{U}$ \\ \hline
\text{Hodge Laplacian } & $\Delta _{H}:\mathfrak{X}(M)\rightarrow \mathfrak{X%
}(M)$ & $\Delta _{H}:\mathbb{R}^{Nn\times 1}\rightarrow \mathbb{R}^{Nn\times
1}$ \\
\text{vector fields} & $U\mapsto -\left[
\begin{array}{c}
\mathcal{H}_{1} \\
\vdots \\
\mathcal{H}_{n}%
\end{array}%
\right] \cdot \mathrm{Ant}\left[
\begin{array}{c}
\mathcal{H}_{1}U \\
\vdots \\
\mathcal{H}_{n}U%
\end{array}%
\right]$ & $\mathbf{U}\mapsto -\left[
\begin{array}{c}
\mathbf{H}_{1} \\
\vdots \\
\mathbf{H}_{n}%
\end{array}%
\right] \cdot \mathrm{Ant}\left[
\begin{array}{c}
\mathbf{H}_{1}\mathbf{U} \\
\vdots \\
\mathbf{H}_{n}\mathbf{U}%
\end{array}%
\right]$ \\
\shortstack{\text{Ant} is the anti-\\
symmetric part} & {\quad\ } $-\left[
\begin{array}{c}
\mathcal{G}_{1} \\
\vdots \\
\mathcal{G}_{n}%
\end{array}%
\right] \left( \sum\limits_{k=1}^{n}\mathcal{G}_{k}U^{k}\right) $ & {\quad\ }
$-\left[
\begin{array}{c}
\mathbf{G}_{1} \\
\vdots \\
\mathbf{G}_{n}%
\end{array}%
\right] \left( \sum\limits_{k=1}^{n}\mathbf{G}_{k}\mathbf{U}^{k}\right)$ \\
{\textit{non-symmetric}} & $=- \sum_{i=1}^{n}\mathcal{H}_{i}(\mathcal{H}_{i}-\mathcal{S}_{i}%
) U$ & $=- \sum_{i=1}^{n}\mathbf{H}_{i}(\mathbf{H}_{i}-\mathbf{S}_{i}%
\mathbf{)} \mathbf{U}$\\
& $-[\mathcal{G}_{j}\mathcal{G}_{k}]_{j,k=1}^{n} U$ & $- [\mathbf{G}_{j}\mathbf{G}_{k}]_{j,k=1}^{n} \mathbf{U}$ \\
{\textit{symmetric}} & $U \mapsto \frac{1}{2}\sum_{i=1}^{n}(\mathcal{H}_{i}-\mathcal{S}_{i}\mathbf{%
)}^{\ast }(\mathcal{H}_{i}-\mathcal{S}_{i}\mathbf{)}U$ & $\mathbf{U} \mapsto \frac{1}{2}\sum_{i=1}^{n}\mathbf{P}^{\otimes }(\mathbf{H%
}_{i}-\mathbf{S}_{i}\mathbf{)}^{\top }(\mathbf{H}_{i}-\mathbf{S}_{i}\mathbf{%
)P^{\otimes }U}$ \\
& $+[\mathcal{G}_{j}^{\ast }\mathcal{G}_{k}]_{j,k=1}^{n}U$ & $+\mathbf{P}^{\otimes }[\mathbf{G}_{j}^{\top }\mathbf{G}_{k}]_{j,k=1}^{n}%
\mathbf{P^{\otimes }U}$ \\
\hline
\text{Lichnerowicz Lap. } & $\Delta _{L}:\mathfrak{X}(M)\rightarrow
\mathfrak{X}(M)$ & $\Delta _{L}:\mathbb{R}^{Nn\times 1}\rightarrow \mathbb{R}%
^{Nn\times 1}$ \\
\shortstack{\text{Sym} is the \\
symmetric part } & $U\mapsto -\left[
\begin{array}{c}
\mathcal{H}_{1} \\
\vdots \\
\mathcal{H}_{n}%
\end{array}%
\right] \cdot \mathrm{Sym}\left[
\begin{array}{c}
\mathcal{H}_{1}U \\
\vdots \\
\mathcal{H}_{n}U%
\end{array}%
\right]$ & $\mathbf{U}\mapsto -\left[
\begin{array}{c}
\mathbf{H}_{1} \\
\vdots \\
\mathbf{H}_{n}%
\end{array}%
\right] \cdot \mathrm{Sym}\left[
\begin{array}{c}
\mathbf{H}_{1}\mathbf{U} \\
\vdots \\
\mathbf{H}_{n}\mathbf{U}%
\end{array}%
\right]$ \\
{\textit{non-symmetric}} & $=-\sum_{i=1}^{n}\mathcal{H}_{i}(\mathcal{H}_{i}+\mathcal{S}_{i})U$ &
$=-\sum_{i=1}^{n}\mathbf{H}_{i}(\mathbf{H}_{i}+\mathbf{S}_{i})\mathbf{U}$\\
{\textit{symmetric}} & $U\mapsto \frac{1}{2}\sum_{i=1}^{n}(\mathcal{H}_{i}+\mathcal{S}_{i}\mathbf{)}%
^{\ast }(\mathcal{H}_{i}+\mathcal{S}_{i}\mathbf{)}U$ & $\mathbf{U}\mapsto \frac{1}{2}\sum_{i=1}^{n}\mathbf{P}^{\otimes }(\mathbf{H}_{i}+\mathbf{S}_{i})^{\top }(\mathbf{H}_{i}+\mathbf{S}_{i})\mathbf{P}^{\otimes }\mathbf{U}$ \\
\hline
\text{covariant derivative} & $\nabla :\mathfrak{X}(M)\times \mathfrak{X}%
(M)\rightarrow \mathfrak{X}(M)$ & $\nabla :\mathbb{R}^{Nn\times 1}\times
\mathbb{R}^{Nn\times 1}\rightarrow \mathbb{R}^{Nn\times 1}$ \\
& $\left( U,Y\right) \mapsto \mathcal{P}\bar{\nabla}_{U}Y$ & $\left( \mathbf{U%
},\mathbf{Y}\right) \mapsto \mathbf{P}^{\otimes }\bar{\nabla}_{%
\mathbf{U}}\mathbf{Y}$ \\ \hline\hline
\end{tabular}%
}
\end{center}
\end{table}

\subsection{Numerical Verification for Operator Approximation}

\begin{figure*}[tbp]
{\scriptsize \centering
\begin{tabular}{ccc}
{\footnotesize (a) Truth of Bochner Laplacian} & {\footnotesize (b) Truth of
Lich. Laplacian} & {\footnotesize (c) Truth of Covariant Deriv.} \\
\includegraphics[width=1.8
in, height=1.4 in]{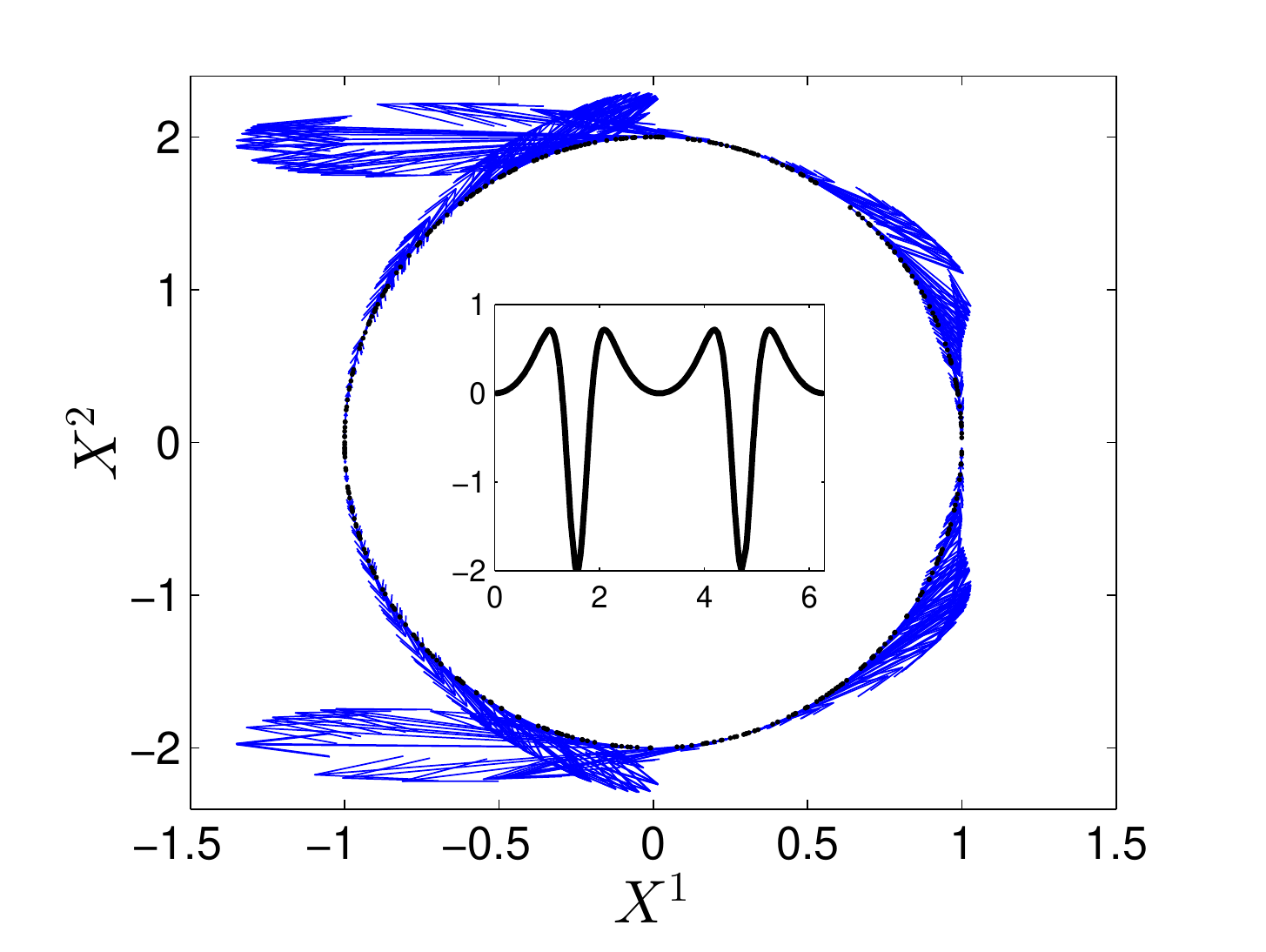}
&
\includegraphics[width=1.8
in, height=1.4 in]{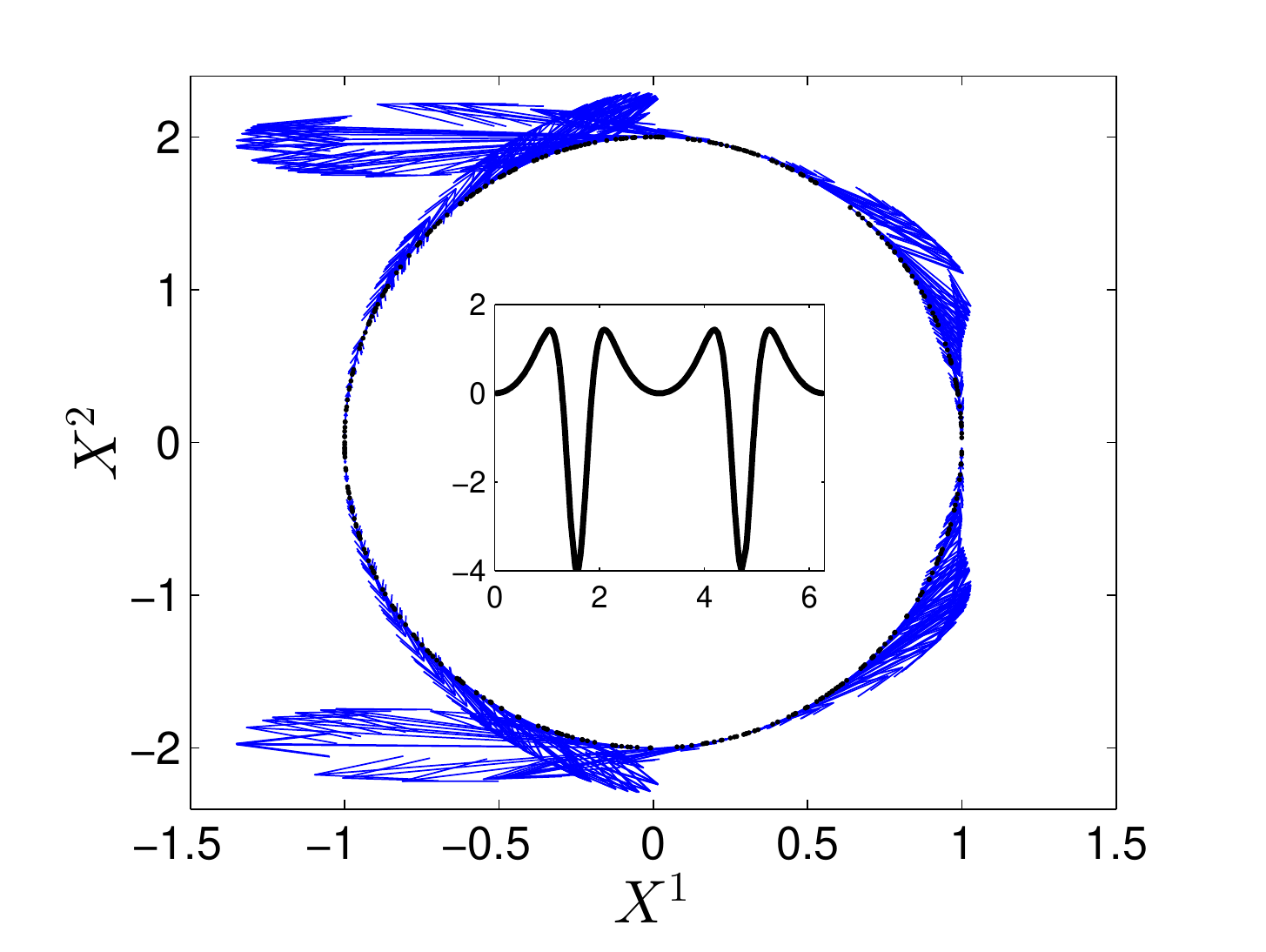}
&
\includegraphics[width=1.8
in, height=1.4 in]{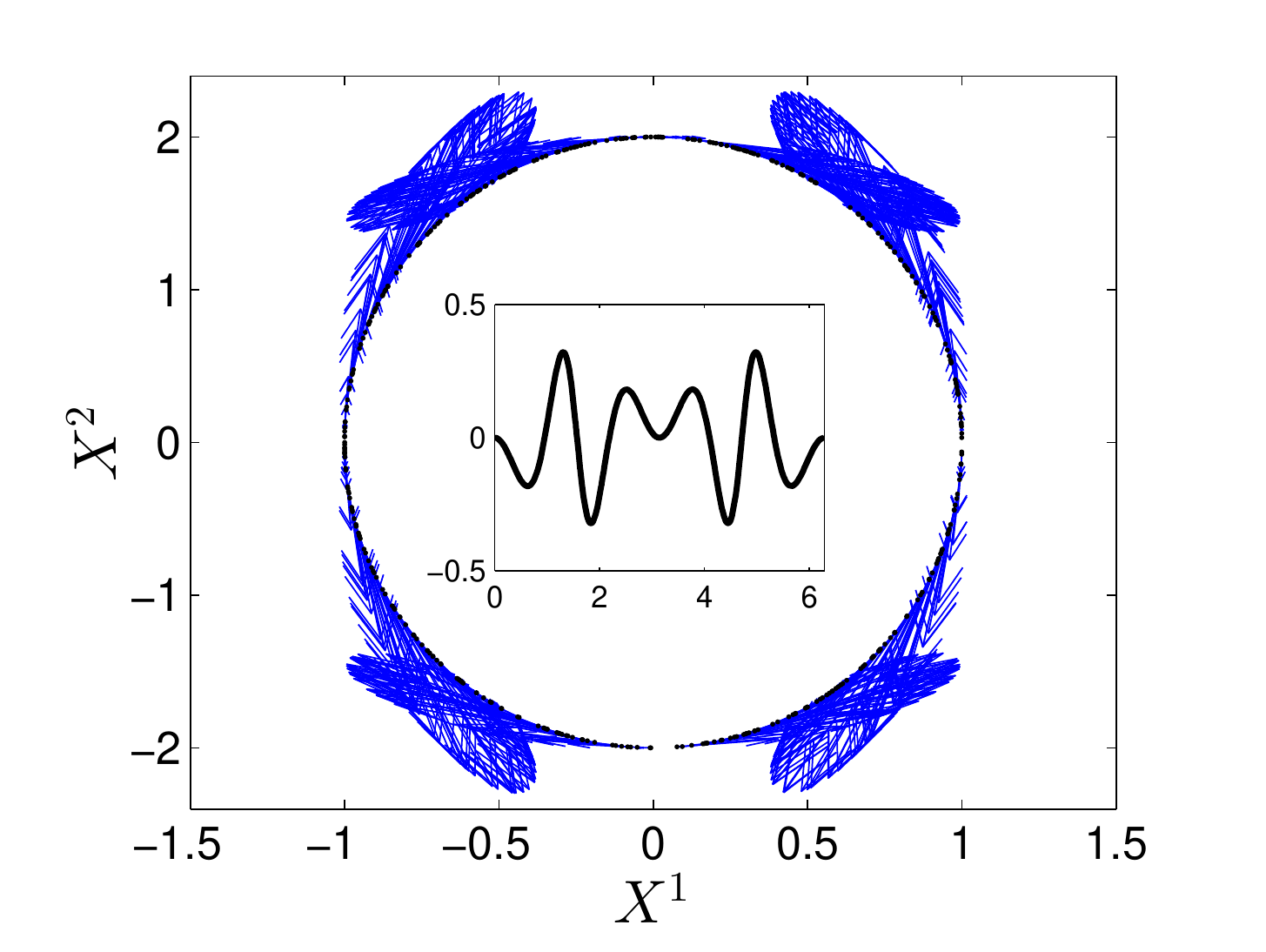}
\\
{\footnotesize (d) Error of Boch. Laplacian} & {\footnotesize (e) Error of Lich.
Laplacian} & {\footnotesize (f) Error of Covariant Deriv. } \\
\includegraphics[width=1.8
in, height=1.4 in]{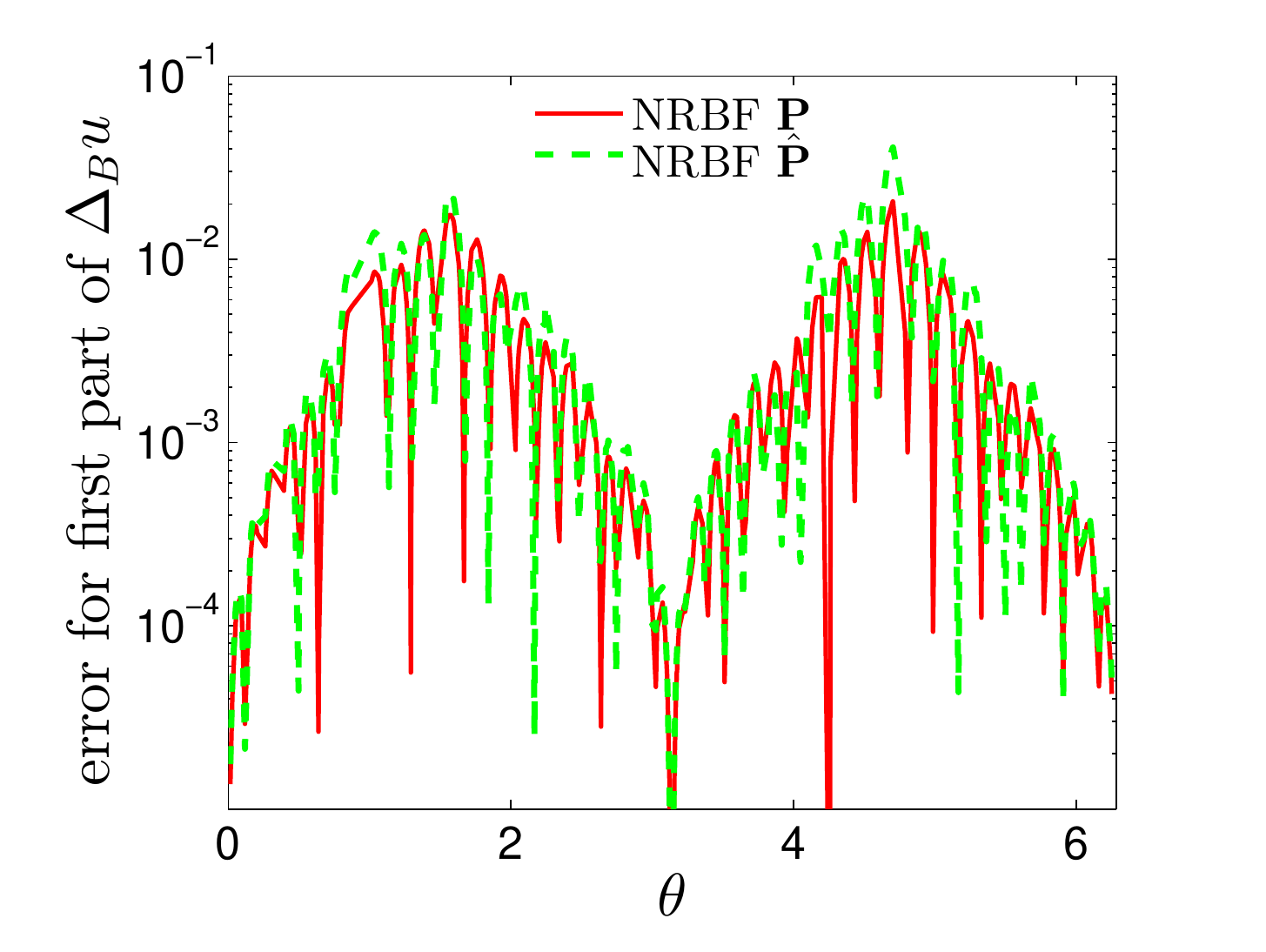} &
\includegraphics[width=1.8
in, height=1.4 in]{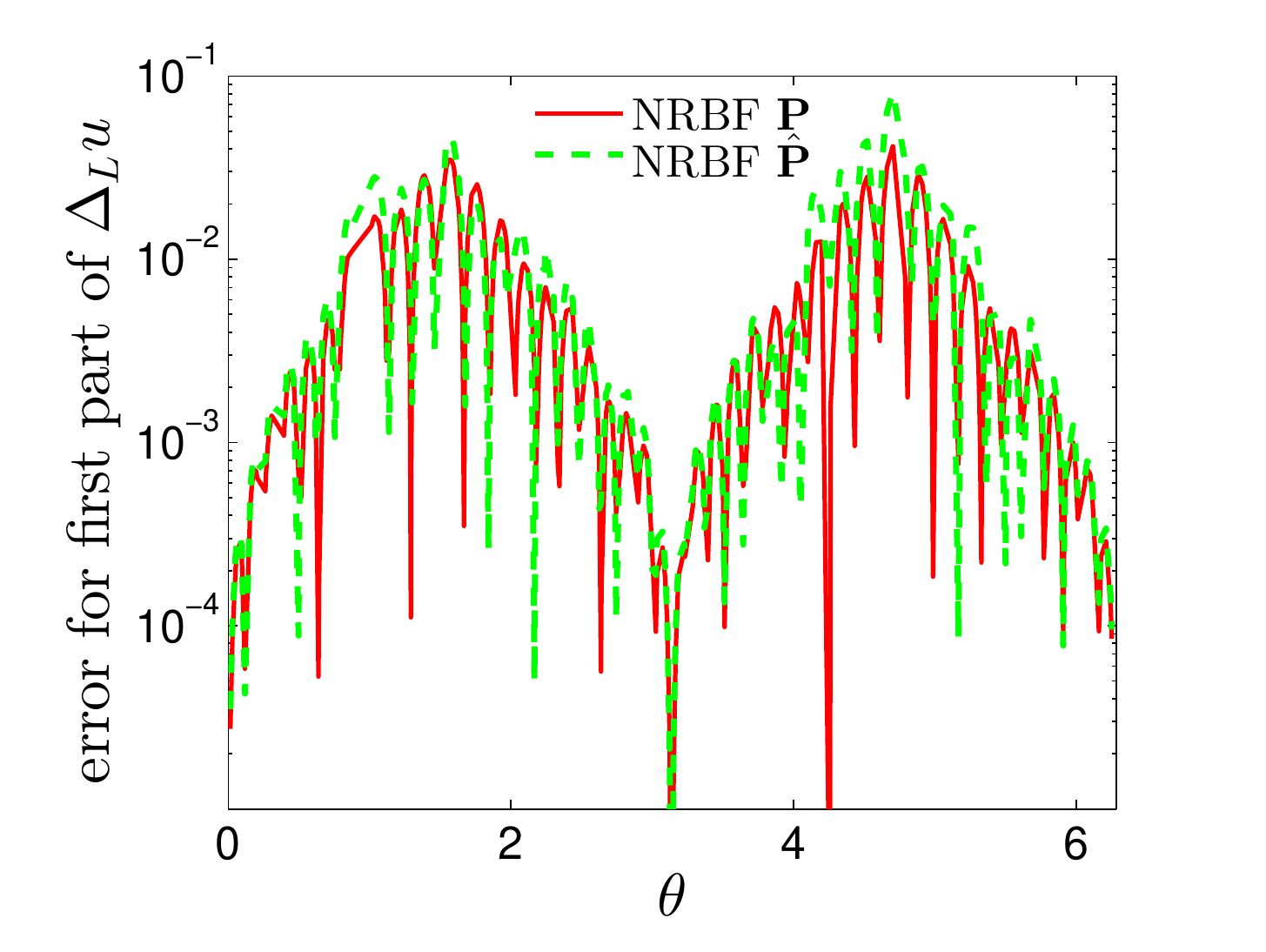} &
\includegraphics[width=1.8
in, height=1.4 in]{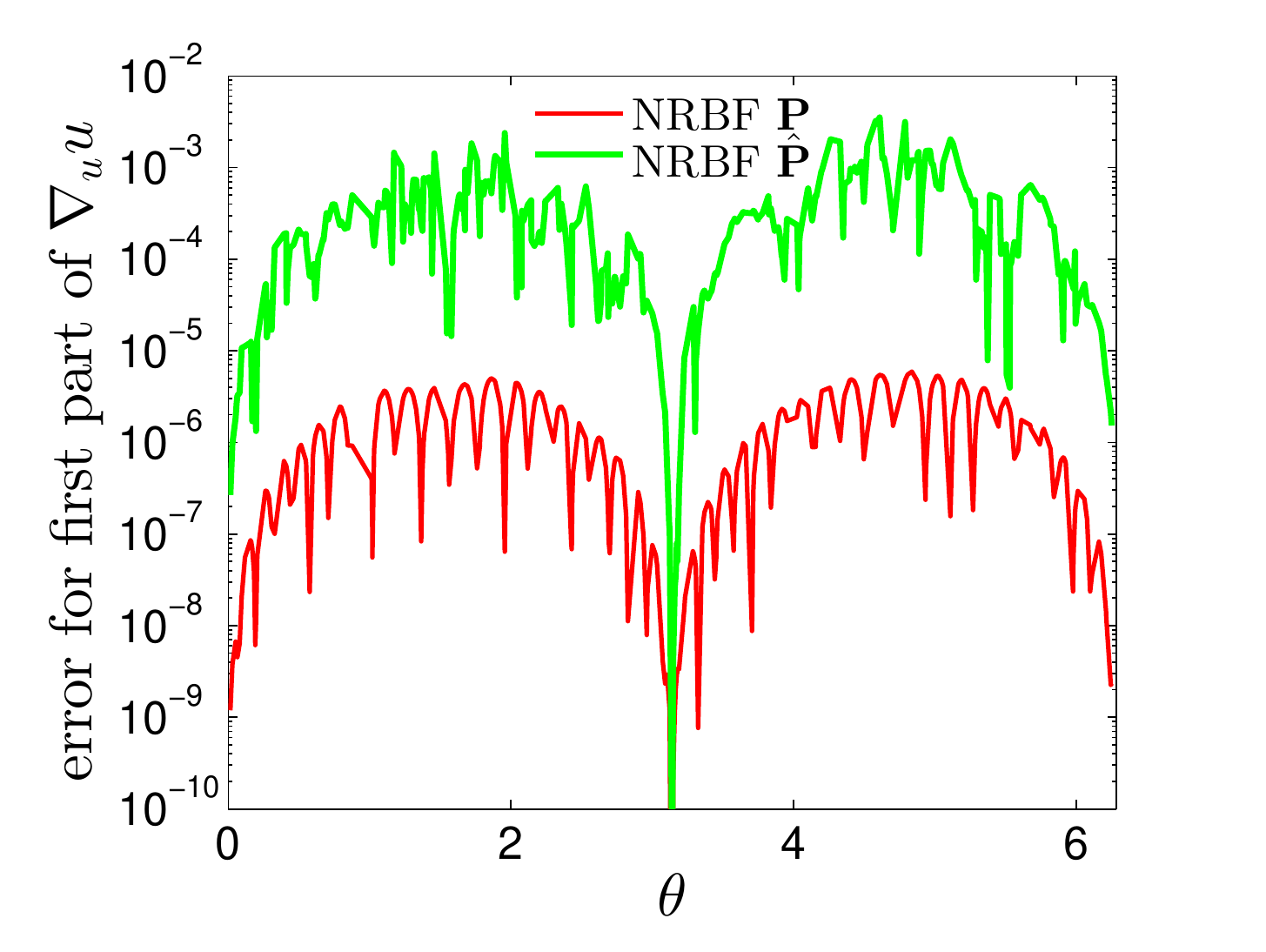}%
\end{tabular}
}
\caption{{\bf 1D ellipse in $\mathbb{R}^2$.} The upper panels
display the truth of (a) Bochner Laplacian, (b) Lichnerowicz Laplacian, and
(c) covariant derivative of a vector field. The insets of upper panels
display the first components of these operator approximations. The bottom
panels display the errors of NRBF approximations using analytic $\mathbf{P}$
(red curve) and approximated $\hat{\mathbf{P}}$ (green curve) for (d)
Bochner Laplacian, (e) Lichnerowicz Laplacian, and (f) covariant derivative.
The Gaussian kernel with shape parameter $s =1.5$ was used. The $N=400$ data
points are randomly distributed. }
\label{fig1_rbfopercheck}
\end{figure*}

We now show the non-symmetric RBF (NRBF) estimates for vector Laplacians and covariant
derivative. The manifold is a one-dimensional full ellipse,
\begin{equation}
{x}=(x^{1},x^{2})=(\cos \theta ,a\sin \theta ),  \label{eqn:xellips}
\end{equation}%
defined with the Riemannian metric $g=\sin ^{2}\theta +a^{2}\cos ^{2}\theta $
for $0\leq \theta < 2\pi ,$where $a=2>1$. The $N=400$ data points are
randomly distributed on the ellipse. The Gaussian kernel with the shape
parameter $s=1.5$ was used.

We first approximate the vector Laplacians. We take a vector field $u=u^{1}%
\frac{\partial }{\partial \theta }$ with $u^{1}\left( x\right) \equiv
u^{1}\left( x\left( \theta \right) \right) =\sin \theta $. The Bochner
Laplacian acting on $u$ can be calculated as $\Delta _{B}u=g^{-1}u_{,11}^{1}%
\frac{\partial }{\partial \theta }$, where $u_{,11}^{1}=\frac{\partial
^{2}u^{1}}{\partial \theta ^{2}}+\frac{\partial u^{1}}{\partial \theta }%
\Gamma _{11}^{1}+u^{1}\frac{\partial \Gamma _{11}^{1}}{\partial \theta }$
with $\Gamma _{11}^{1}=\frac{1}{2}g^{-1}\frac{\partial g}{\partial \theta }$%
. Numerically, Fig. \ref{fig1_rbfopercheck}(a) shows the true Bochner
Laplacian $\Delta _{B}u=\bar{\Delta}_{B}U$ pointwisely, which is a $2\times 1$ vector
lying in the tangent space of each given point. The inset of Fig.
\ref{fig1_rbfopercheck}(a) displays the first vector component of the
Bochner Laplacian as a function of the intrinsic coordinate $\theta $.
Figure \ref{fig1_rbfopercheck}(d) displays the error for the first
vector component of Bochner Laplacian as a function of $\theta $. Here, we show the errors of using the analytic
$\mathbf{P}$\ and an approximated $\mathbf{\hat{P}}$. Here (and in the remainder of this paper), we used the notation $\mathbf{\hat{P}}$ to denote the approximated projection matrix obtained from a second-order method discussed in Section~\ref{sec3}. One can
clearly see that the errors for NRBF using both analytic $\mathbf{P}$\ and
approximated $\mathbf{\hat{P}}$\ are small about $0.01$. Since Hodge
Laplacian is identical to Bochner Laplacian on a 1D manifold, the results
for Hodge Laplacian are almost the same as\ those for Bochner [not shown
here]. The Lichnerowicz Laplacian is the double of Bochner Laplacian, $%
\Delta _{L}u=2g^{-1}u_{,11}^{1}\frac{\partial }{\partial \theta }$, as shown
in Fig. \ref{fig1_rbfopercheck}(b). The error for the first vector
component of Lichnerowicz Laplacian is also nearly doubled as shown in Fig.
\ref{fig1_rbfopercheck}(e).

We next approximate the covariant derivative. The covariant derivative can
be calculated as $\nabla _{u}u=u^{1}(\frac{\partial u^{1}}{\partial \theta }%
+u^{1}\Gamma _{11}^{1})\frac{\partial }{\partial \theta }$ as shown in Fig.
\ref{fig1_rbfopercheck}(c). Figure \ref{fig1_rbfopercheck}%
(f) displays the error of NRBF approximation\ for the first vector component
of covariant derivative $\nabla _{u}u$. One can see from Fig. \ref%
{fig1_rbfopercheck}(f)\ that the error for analytic $\mathbf{P}$ (red) is
very small about $10^{-6}$ and the error for approximated $\mathbf{\hat{P}}$
(green)\ is about $10^{-3}$.

\section{Estimation of the Projection Matrix}\label{sec3}

When the manifold $M$ is unknown and identified only by a point cloud data $X=\{x_{1},\ldots,x_N\}$, where $x_i\in M$, we do not immediately have access to the matrix-valued function $P$. In this section, we first give a quick overview of the existing first-order local SVD method for estimating $\mathbf{P}=P(x)$ on each $x\in M$. Subsequently, we present a novel second-order method (which is a local SVD method that corrects the estimation error induced by the curvature) under the assumption that the data set $X$ lies on a $C^{3}$ $d$-dimensional Riemannian manifold $M$ embedded in $\mathbb{R}^{n}$.


Let $x, y \in X \subset M$ such that $\left\vert
y-x\right\vert = {O(\rho)}$. Define $\gamma $\ to be a geodesic, connecting $x$
and $y$. The curve is parametrized by the arc-length,
\begin{equation*}
\rho=\int_{0}^{\rho}\left\vert \gamma ^{\prime }(t)\right\vert dt,
\end{equation*}%
where $\gamma (0)=x,\gamma (\rho)=y.$ Taking\ derivative with respect to $\rho$,
we obtain constant velocity, $1=\left\vert \gamma ^{\prime }(t)\right\vert $
for all $0\leq t\leq \rho$. Let $\boldsymbol{\rho}=\left( \rho_{1},\ldots ,\rho_{d}\right) $ be
the geodesic normal coordinate of $y$ defined by an exponential map $\exp
_{x}:T_{x}M\rightarrow M$. Then $\boldsymbol{\rho}$ satisfies
\begin{equation*}
\rho\gamma ^{\prime }(0)=\boldsymbol{\rho}=\exp _{x}^{-1}(y),
\end{equation*}%
where
\begin{equation*}
\rho^{2}=\rho^{2}\left\vert \gamma ^{\prime }(0)\right\vert ^{2}=\left\vert \boldsymbol{\rho}%
\right\vert ^{2}=\sum_{i=1}^{d}\rho_{i}^{2}.
\end{equation*}%
For any point $x,y \in M$, let $\iota $ be the local parametrization of manifold such that $\iota
\left( \boldsymbol{\rho}\right)=y$ and $\iota(\mathbf{0})= x$. Consider the Taylor expansion of $%
\iota \left( \boldsymbol{\rho}\right) $ centered at $\mathbf{0},$%
\begin{equation}
\iota \left( \boldsymbol{\rho}\right) =\iota (\mathbf{0})+\sum_{i=1}^{d}\rho_{i}\frac{%
\partial \iota (\mathbf{0})}{\partial \rho_{i}}+\frac{1}{2}%
\sum_{i,j=1}^{d}\rho_{i}\rho_{j}\frac{\partial ^{2}\iota (\mathbf{0})}{\partial
\rho_{i}\partial \rho_{j}}+O(\rho^{3}).  \label{eqn:iots}
\end{equation}%
Since the Riemannian metric tensor at the based point $\mathbf{0}$ is an identity matrix, $\bm{\tau}_{i}=\frac{\partial \iota (\mathbf{0})}{\partial \rho_{i}}\Big/\left|\frac{\partial\iota (\mathbf{0})}{\partial \rho_{i}}\right| = \frac{\partial \iota (\mathbf{0})}{\partial \rho_{i}} $ are $d$ orthonormal tangent vectors that span $T_{x}M$.

\subsection{First-order Local SVD Method}

The classical local SVD method \cite{donoho2003hessian,zhang2004principal,tyagi2013tangent} uses the difference vector $y-x = \iota(\boldsymbol{\rho})-\iota(\mathbf{0})$ to estimate $\mathbf{T} = (\bm{\tau}_1,\ldots,\bm{\tau}_d)$ (up to an orthogonal rotation) and subsequently use it to approximate $\mathbf{P}=\mathbf{T}\mathbf{T}^\top$. { The same technique has also been proposed to estimate the intrinsic dimension of the manifold given noisy data \cite{little2009estimation}.}
Numerically, the first-order local SVD proceeds as follows:
\begin{enumerate}
\item For each $x\in X$, let $\{y_1,\ldots, y_K\} \subset X$ be the $K$-nearest neighbor (one can also use a radius neighbor) of $x$. Construct the distance matrix $\mathbf{D}:=[\mathbf{D}_1,\ldots, \mathbf{D}_K] \in \mathbb{R}^{n\times K}$, where $K>d$ and $\mathbf{D}_i: = y_i-x$.
\item Take a singular value decomposition of $\mathbf{D} = \mathbf{U}\mathbf{\Sigma}\mathbf{V}^\top$. Then the leading $d-$columns of $\mathbf{U}$ consists of $\mathbf{\tilde{T}}$ which approximates a span of column vectors of $\mathbf{T}$, which forms a basis of $T_xM$.
\item Approximate $\mathbf{P}$ with $\mathbf{\tilde{P}} = \mathbf{\tilde{T}}\mathbf{\tilde{T}}^\top$.
\end{enumerate}
Based on the Taylor's expansion in \eqref{eqn:iots}, {intuitively,} such an approximation can only provide an estimate with accuracy $\Vert \mathbf{\tilde{P}-P}\Vert _{F}=O(\rho)$, which is an order-one scheme, { where the constant in the big-oh notation, $O(\rho)$, depends on the base point $x$ through the curvature, number of nearest neighbors $K$, intrinsic dimension $d$, and extrinsic dimension $n$, as we discuss next. Here $\Vert \cdot \Vert_F$ denotes the Frobenius matrix norm.} For uniformly sampled data, we state the following definition and probabilistic type convergence result (Theorem~2 of \cite{tyagi2013tangent}) for this local SVD method, which will be useful in our convergence study.

\begin{definition}
For each point $x\in M$, where $M$ is a $d$-dimensional smooth manifold embedded in $\BR^n$, where $d+1\leq n$.
We define $N_\epsilon(x) = M \cap B_x(\sqrt{\epsilon})$, where $B_x(\sqrt{\epsilon})$ denotes the Euclidean ball (in $\BR^n$) centered at $x$ with radius $\sqrt{\epsilon}$. If $M$ has a positive injectivity radius  $\sqrt{\epsilon}>0$ at $x\in M$, then there is a diffeomorphism between $N_\epsilon(x)$ and $T_xM$. In such a case, there exists a local one-to-one map $T_xM \ni\boldsymbol{\rho}\mapsto y = \exp_x(\boldsymbol{\rho}):=(\boldsymbol{\rho}, f_{1}(\boldsymbol{\rho}),\ldots f_{n-d}(\boldsymbol{\rho}) ) \in M \subset \BR^n$, for $y\in M$ neighboring to $x$, with smooth functions $f_\ell:T_xM \to\BR$ for $\ell = 1,\ldots,n-d$. We also denote the maximum principal curvature at $x$ as $K_{max}$.
\end{definition}

\begin{theo}
\label{local svd result}
Suppose that $\{y_i \in M \}_{i=1}^K$ are the $K$-nearest neighbor data points of $x$ such that their orthogonal projections, $\boldsymbol{\rho}^{(i)} \sim U[-\sqrt{\epsilon},\sqrt{\epsilon}]^d \subset T_x M$ at $x$ are i.i.d. Let $\mathbf{W} \in \BR^{(n-d)\times (n-d)}$ be a matrix with components given as,
\[
W_{ij} = \mathbb{E}_{\boldsymbol{\rho}\sim U[-\sqrt{\epsilon},\sqrt{\epsilon}]^d} [f_{q,i}(\boldsymbol{\rho})f_{q,j}(\boldsymbol{\rho})],
\]
where $f_{q,\ell}$ denotes a quadratic form of $f_\ell$ for $\ell=1,\ldots,n-d$, which is the second-order Taylor expansion about the base point $\mathbf{0}$, involving the curvature at $x$ of $M$. Then, for any $\tau \in (0,1)$, in high probability,
\[
\Vert \mathbf{\tilde{P}-P}\Vert _{F} \leq \sqrt{2}\tau,
\]
if $\epsilon = O(n^{-1}d^{-2}|K_{max}|^{-2})$ and  $K = O(\tau^{-2} d^2\log n)$. Here $\|\cdot\|_F$ denotes the standard Frobenius matrix norm.
\end{theo}

\begin{remark}\label{choice of K}
We should point out that the result above implies that the local SVD has a Monte-Carlo error rate, $\tau = O(K^{-1/2})$ and the choice of local neighbor of radius $\sqrt{\epsilon} = O(\rho)$ should be inversely proportional to the maximum principal curvature and the dimension of the manifold and ambient space. Thus, to expect an error $\tau = O(\sqrt{\epsilon}) {= O(\rho)}$, by balancing $n^{-1}d^{-2}|K_{max}|^{-2} \sim K^{-1}d^2\log n$, this result suggests that one should choose $K \sim d^4 n \log n |K_{max}|^2.$ {The numerical result in the torus suggests of rate $\sqrt{\epsilon}=O(N^{-1/2})$ (see Figure~\ref{Fig_torustan}). For general $d$-dimensional manifolds, the error rate is expected to be $\sqrt{\epsilon} \sim N^{-1/d}$ due to the fact that $\rho \propto N^{-1/d}$.}
\end{remark}

\subsection{Second-order Local SVD Method}\label{sec3.2}

In this section, we devise an improved scheme to achieve the tangent space approximation with accuracy up to order of $O(\rho^{2})$,
by accounting for the Hessian components in \eqref{eqn:iots}. The algorithm proceeds as follows:
\begin{enumerate}
\item Perform the first-order local SVD algorithm and attain the $d$ approximated tangent vectors $\mathbf{\tilde{T}} = [\bm{\tilde{t}}_1,\ldots, \bm{\tilde{t}}_d] \in \mathbb{R}^{n\times d}$.
\item For each neighbor $\{y_i\}_{i=1,\ldots,K}$ of $x$, compute $\boldsymbol{\tilde{\rho}}^{(i)} = (\tilde{\rho}^{(i)}_1,\ldots,\tilde{\rho}^{(i)}_d)$, where
\[
\tilde{\rho}^{(i)}_j = \mathbf{D}_i^\top \bm{\tilde{t}}_j, \quad i= 1,\ldots, K, j=1,\ldots, d,
\]
where $\mathbf{D}_i: = y_i-x$ is the $i$th column of $\mathbf{D}\in \mathbb{R}^{n\times K}.$

\item Approximate the Hessian $\mathbf{Y}_{p}=\frac{\partial ^{2}\iota (%
\mathbf{0})}{\partial \rho_{i}\partial \rho_{j}} \in \mathbb{R}^n$ {up to a difference of a vector in $T_x M$} using the following ordinary least
squared regression, with $p= 1,\ldots ,D=d(d+1)/2$ denoting the upper triangular components ($p\mapsto (i,j)$ such that  $i\leq  j$) of symmetric Hessian matrix. Notice that for each $y_\ell\in \{y_1,\ldots, y_K\}$ neighbor of $x$, the equation (\ref%
{eqn:iots}) can be written as
\begin{equation}
\sum_{i,j=1}^{d}\rho_{i}^{(\ell)}\rho_{j}^{(\ell)}\frac{\partial ^{2}\iota (\mathbf{0})}{\partial
\rho_{i}\partial \rho_{j}}=2\left( \iota ( {\boldsymbol{\rho}^{(\ell)}}) -\iota (\mathbf{0}%
)\right) - 2\sum_{i=1}^{d}\rho_{i}^{(\ell)}\bm{\tau}_i + O(\rho^3), \quad \quad \ell=1,\ldots K,  \label{eqn:sisj}
\end{equation}%
{ where $\boldsymbol{\rho}^{(\ell)} := (\rho^{(\ell)}_1,\ldots, {\rho}^{(\ell)}_d)$ denotes the geodesic coordinate that satisfies $\iota \left( \boldsymbol{\rho}^{(\ell)}\right)=y_\ell$. }
In compact form, we can rewrite \eqref{eqn:sisj} as a linear system,
\BEA
\mathbf{A} \mathbf{Y} = 2\mathbf{D}^\top - 2\mathbf{R}+ O(\rho^3),\label{eqn:Yti}
\EEA
where $\mathbf{R}^\top = (\mathbf{r}_1,\ldots,\mathbf{r}_K)\in\mathbb{R}^{n\times K}$ denotes the order-$\rho$ residual term {in the tangential directions,}
\BEA
\mathbf{r}_{j} = \sum_{i=1}^{d}\rho_{i}^{(j)}\bm{\tau}_i, \quad j=1,\ldots,K. \label{def_r}
\EEA

Here, $\mathbf{Y} \in \mathbb{R}^{D\times n}$ is a matrix whose $p$th row is  $\mathbf{Y}_{p}$
and
\BEA
\mathbf{A} = \begin{pmatrix} (\rho_1^{(1)})^2 & (\rho_2^{(1)})^2 & \ldots & (\rho_d^{(1)})^2 & 2(\rho_1^{(1)}\rho_2^{(1)}) & \ldots & 2(\rho_{d-1}^{(1)}\rho_d^{(1)})  \\ \vdots & \vdots  & & \vdots & \vdots  && \vdots\\ (\rho_1^{({K})})^2 & (\rho_2^{({K})})^2 & \ldots & (\rho_d^{({K})})^2 & 2(\rho_1^{({K})}\rho_2^{({K})}) & \ldots & 2(\rho_{d-1}^{({K})}\rho_d^{({K})})\end{pmatrix} \in \mathbb{R}^{K\times D}.\label{Atilde}
\EEA
With the choice of $K$ in Remark~\ref{choice of K}, $K>D$, we approximate $\mathbf{Y}$ by solving an over-determined linear problem
\BEA
\mathbf{\tilde{A}} \mathbf{Y} = 2\mathbf{D}^\top,\label{eqn:our2ys}
\EEA
where $\mathbf{\tilde{A}}$ is defined as in \eqref{Atilde} except that  $\rho_i^{(j)}$ in the matrix entries is replaced by $\tilde{\rho}_i^{(j)}$.
The regression solution is given by $\tilde{\mathbf{Y}} =2(\tilde{\mathbf{A}}^{\top }\tilde{\mathbf{A}})^{-1}\tilde{\mathbf{A}}^{\top }\mathbf{D}^\top.$ Here $\tilde{\mathbf{Y}}_{ij} = \tilde{Y}_i^{(j)}$, $i = 1,\ldots, D,j=1,\ldots, n$. Here, each row of $\tilde{\mathbf{Y}}$ is denoted as $\tilde{\mathbf{Y}}_p = (\tilde{Y}_p^{(1)},\ldots,\tilde{Y}_p^{(n)})\in \mathbb{R}^{1\times n}$, which is an estimator of $\mathbf{Y}_p$.

\item Apply SVD to
\BEA
2\tilde{\mathbf{R}}^\top :=  2\mathbf{D} - (\tilde{\mathbf{A}}\tilde{\mathbf{Y}})^\top \in
\mathbb{R}^{n\times K}.\label{eq:Rtilde}
\EEA
Let the leading $d$ left singular vectors be denoted as $\bm{\hat{\Psi}}=[\hat{\boldsymbol{\psi}}_{1},\ldots ,\hat{\boldsymbol{\psi}}_{d}] \in \mathbb{R}^{n\times d}$, which is an estimator of $\bm{\Psi}=[\bm{\psi}_{1},\ldots, \bm{\psi}_{d}]$, where $\bm{\psi}_{j}$ are the leading $d$ left singular vectors of $\mathbf{R}$ as defined in \eqref{eqn:Yti}. We define $\mathbf{\hat{P}=\bm{\hat{\Psi}}\bm{\hat{\Psi}}}^{\top }$ as the estimator for $\mathbf{P} = \bm{\Psi}\bm{\Psi}^\top$, where the last equality is valid due to Proposition~\ref{propP}(3).
\end{enumerate}

Figure \ref{Fig_torustan} shows the manifold learning results on a torus with
randomly distributed data. One can see that error of the first-order local SVD method is $O(N^{-1/2})$ whereas second-order method is $O(N^{-1})$.

\begin{figure}[tbp]
{\scriptsize \centering
\begin{tabular}{cc}
{\normalsize (a) max of Frob. norm error} & {\normalsize (b) mean of Frob.
norm error} \\
\includegraphics[width=3
in, height=2 in]{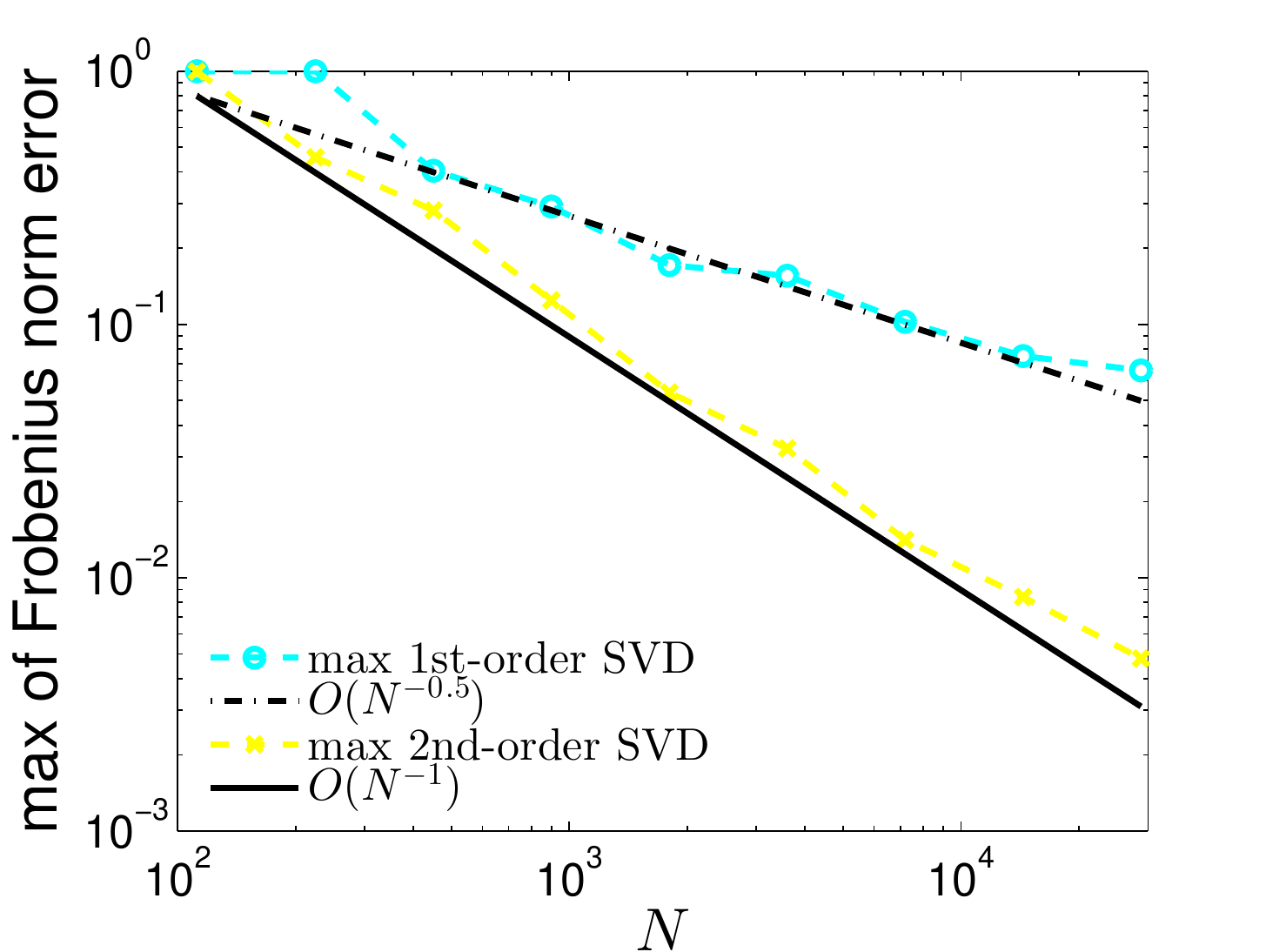}
&
\includegraphics[width=3
in, height=2 in]{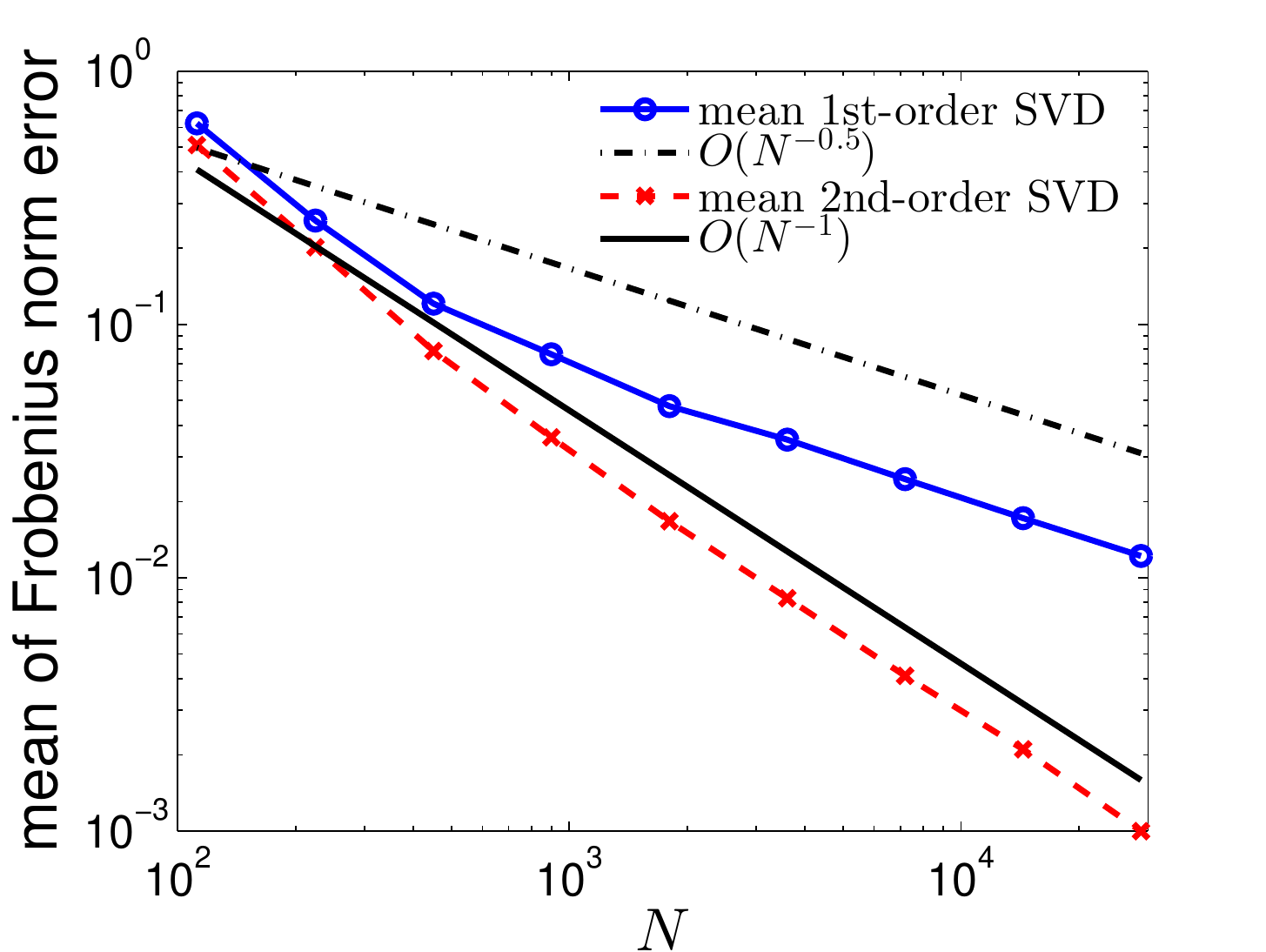}%
\end{tabular}
}
\caption{{\bf 2D torus in $\mathbb{R}^3$.} Comparison of
convergence rates between 1st-order SVD and our 2nd-order SVD for
approximating the tangential projection matrix $\mathbf{P}$.
Panels (a) and (b) show the maximum and the mean of Frobenius norm errors, respectively. The
error of $||\mathbf{\tilde{P}-P}||_{F}$\ is $O(N^{-1/2})$ for
the 1st-order SVD whereas the error of $||\mathbf{\hat{P}-P}||_{F}$ is $O(N^{-1})$ for our 2nd-order SVD.
{The $K=40$ nearest neighbors are fixed for all the simulations.  The data points are uniformly distributed in intrinsic coordinates $[0,2\pi) \times [0,2\pi)$ and are then mapped onto the torus in Euclidean space.} }
\label{Fig_torustan}
\end{figure}

Theoretically, we can deduce the following error bound.
\begin{theo}
\label{main P Theorem} Let the assumptions in Theorem~\ref{local svd result} be valid, particularly $\rho \sim \epsilon^{1/2}$. Suppose that the matrix $\mathbf{D}\in \mathbb{R}^{n\times K}$ is defined as in Step 1 of the algorithm with a fixed $K$ is chosen as in Remark~\ref{choice of K} in addition to $K > D= \frac{1}{2}d(d+1)$. Assume that ${ \|\mathbf{A}\|_2 / \kappa_2(\mathbf{A})} =K \omega(\epsilon^{3/2})$ as $\epsilon\to 0$, {where $\kappa_2(\mathbf{A})$ denotes the condition number of matrix $\mathbf{A}$ based on spectral matrix norm, $\|\cdot\|_2$}, and the eigenvalues $\{\lambda_i\}_{i=1,\ldots,n}$ of $\mathbf{R}^\top {\mathbf{B}^\top\mathbf{B}}\mathbf{R}$, where $ \mathbf{R}^\top \in \mathbb{R}^{n \times K}$ as defined in \eqref{def_r}, are simple with spectral gap $g_i:= \min_{j\neq i}|\lambda_i-\lambda_j| > c\epsilon$ for some $c>0$ and all $i=1,\ldots,n$. Here, $\mathbf{B}:= \mathbf{I}_{K}-\mathbf{\tilde{A}(\tilde{A}}^{\top }\mathbf{\tilde{A})}^{-1}\mathbf{\tilde{A}}^{\top } \in \mathbb{R}^{K\times K}$. Let $\mathbf{\hat{P}=\bm{\hat{\Psi}}\bm{\hat{\Psi}}}^{\top }$ be the second-order estimator of $\mathbf{P}$, where columns of $\hat{\bm{\Psi}}$ are the leading $d$ left singular vectors of $\tilde{\mathbf{R}} $ as defined in \eqref{eq:Rtilde}.
Then, with high probability,
$$
\| \mathbf{\hat{P}}- \mathbf{P} \|_F = O(\epsilon),
$$
as $\epsilon\to0$.
\end{theo}

{ The assumption of $\|\mathbf{A}\|_2 / \kappa_2(\mathbf{A}) =K \omega(\epsilon^{3/2})$ as $\epsilon\to 0$ is to ensure that the perturbed matrix, $\tilde{\mathbf{A}}$, is still full rank. As we will show, this condition arises from the fact that the minimum relative size of the perturbation for the perturbed matrix to be not full rank is $\frac{1}{\kappa_2(\mathbf{A})}$, i.e.,
\[
\min\left\{ \frac{\|\mathbf{\tilde{A}}-\mathbf{A}\|_2}{\|\mathbf{A}\|_2} \,:\, \tilde{\mathbf{A}} \textup{ is not full rank}\right\} = \frac{1}{\kappa_2(\mathbf{A})}.
\]


The simple eigenvalues and spectral gap conditions in the theorem above are two technical assumptions needed for applying the classical perturbation theory of eigenvectors estimation (see e.g. Theorem~5.4 of \cite{demmel1997applied}), which allow one to bound the angle between eigenvectors of unperturbed and perturbed matrices by the ratio of the perturbation error and the spectral gap as we shall see in the proof below.
One can employ the result in Theorem~\ref{main P Theorem} whenever $\mathbf{B}\mathbf{R}$
can be approximated by $\tilde{\mathbf{R}}$ sufficiently well (with an error smaller than the spectral gap of the corresponding eigenvalue). One can see from Fig. \ref{Fig_torustan}(b)  that the asymptotics break down when $N$ decreases to around $500$. Moreover, when $N$ is around $500$, the max norm errors are already large up to at least 0.3 for both 1st-order and 2nd-order methods as seen from Fig. \ref{Fig_torustan}(a).

While we have no access to $\mathbf{R}^\top\mathbf{B}^\top\mathbf{BR}$, since it can be accurately estimated by $\mathbf{\tilde{R}} ^\top \mathbf{\tilde{R}}$ in the sense of $\| \mathbf{\tilde{R}} ^\top \mathbf{\tilde{R}} - \mathbf{R}^\top \mathbf{B}^\top \mathbf{B} \mathbf{R}\|_2 = O(\epsilon^2)$ as we shall see in the following proof, let us investigate the eigenvalues of the approximate matrix $\tilde{\mathbf{R}}^\top\tilde{\mathbf{R}}$ (or equivalently the singular values of $2\tilde{\mathbf{R}}^\top$). Figures \ref{fig_svd}(a) and (b) show the first two singular values of $2\tilde{\mathbf{R}}^\top$ and their difference for the torus and sphere examples, respectively. One can see the spectral gap $|\tilde{\lambda}_1-\tilde{\lambda}_2|$ on the same scale of $\tilde{\lambda}_1$ and $\tilde{\lambda}_2$ almost surely, regardless whether the geometry is highly symmetric (e.g. the sphere) or not (e.g. torus) for a fixed $N$. This empirical verification
suggests that the two assumptions (simple eigenvalues and spectral gap condition) are not unreasonable. In fact, in Fig. \ref{fig_svd}(c), we found that the corresponding pointwise Frobenius norm errors, $\|\mathbf{\hat{P}}-\mathbf{P}\|_F$, for the torus and sphere examples using the 2nd order SVD method are small, especially for the highly symmetric sphere. While these examples suggest that the eigenvalues are most likely simple for randomly sample data of any fixed $N$, we believe that one can use other results from perturbation theory that may require different assumptions if the eigenvalues are non-simple. See, for instance, Chapter 2, Section 6.2 of \cite{kato2013perturbation}.


}

\begin{figure*}[tbp]
{\scriptsize \centering
\begin{tabular}{ccc}
{\small (a) Torus, singular values of $2\tilde{\mathbf{R}}^\top$} & {\small (b) Sphere, singular values of $2\tilde{\mathbf{R}}^\top$} &
{\small (c) Pointwise Frobenius norm error} \\
\includegraphics[width=1.9
in, height=1.4 in]{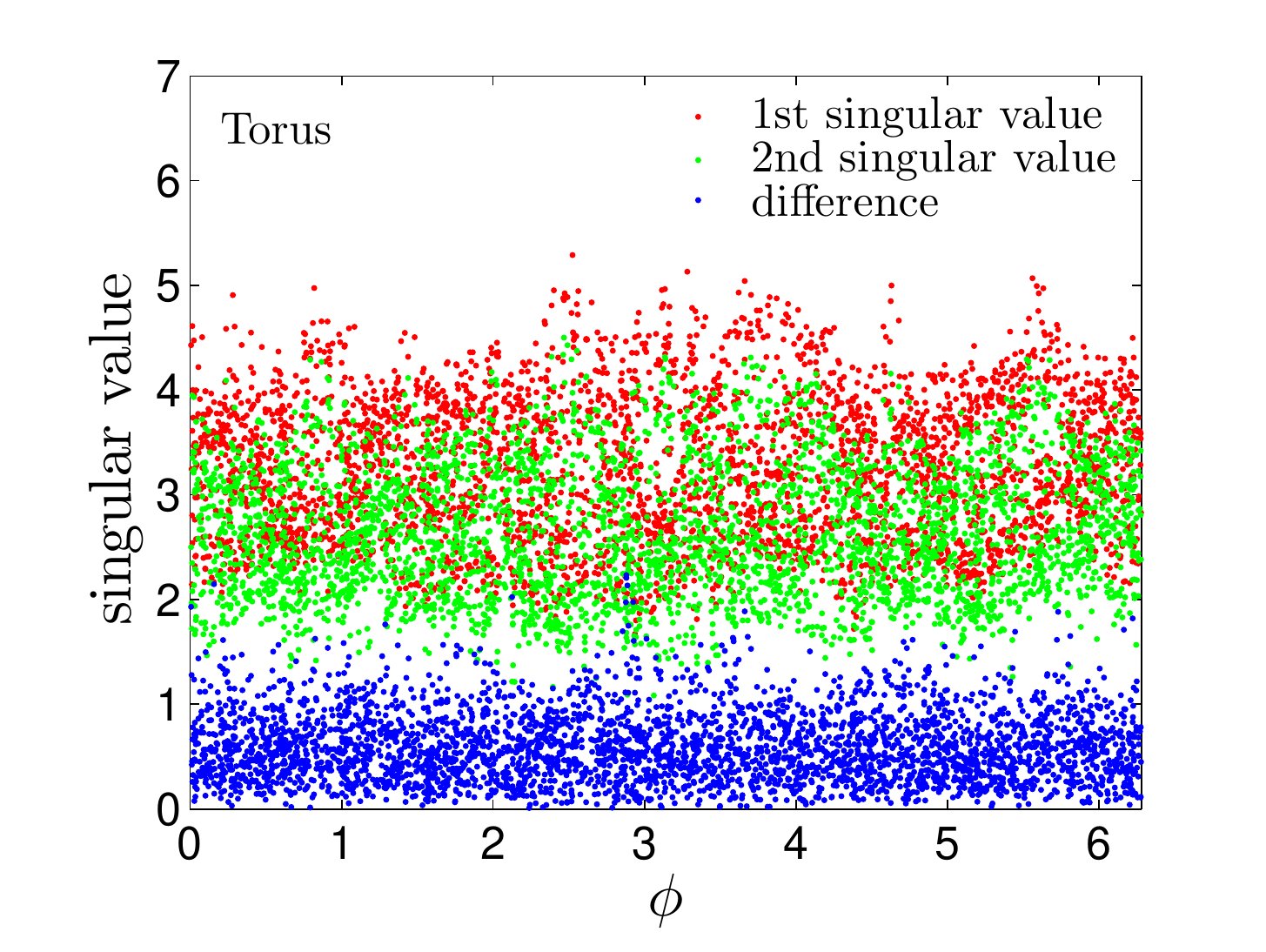}
&
\includegraphics[width=1.9
in, height=1.4 in]{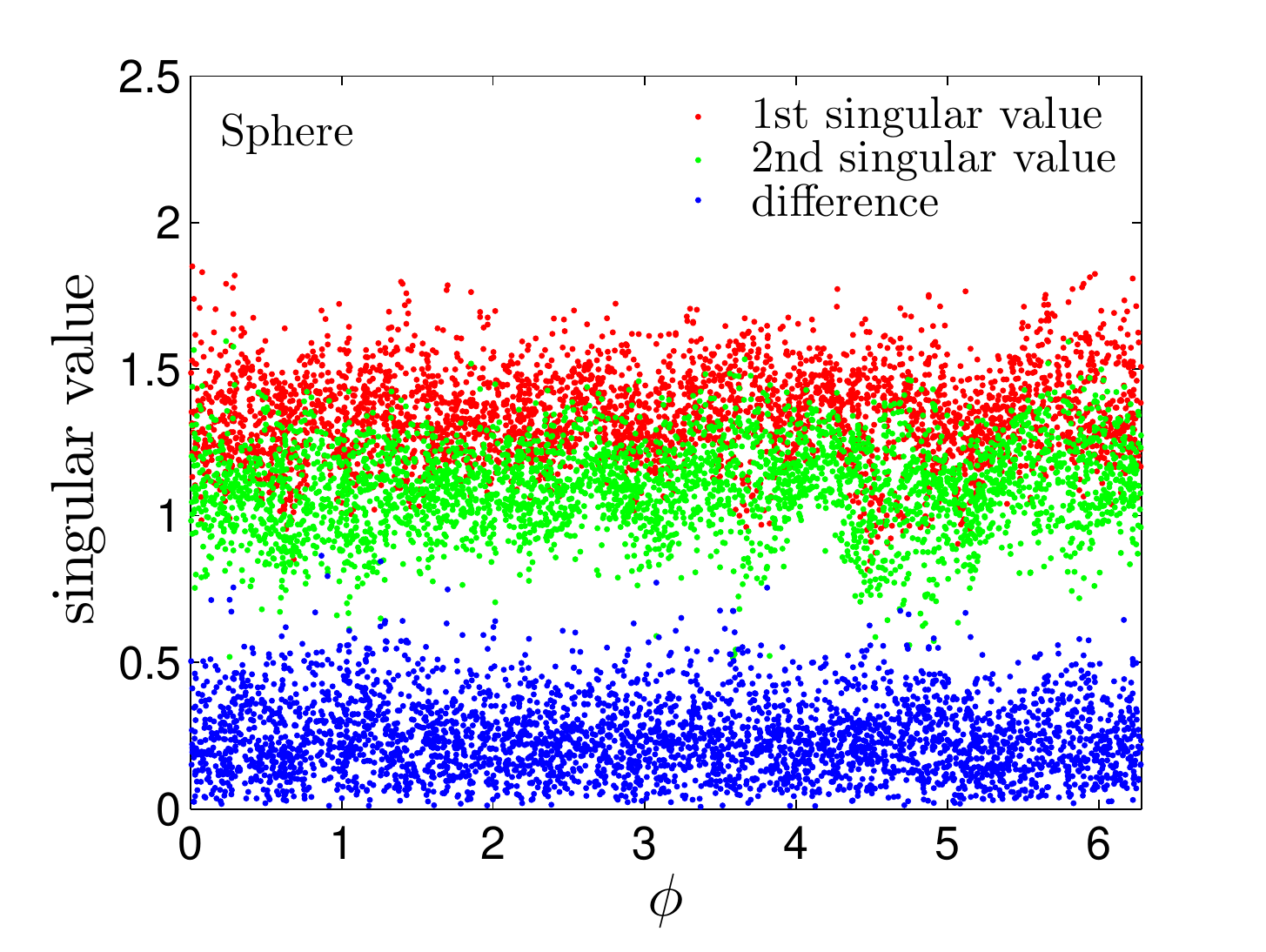}
&
\includegraphics[width=1.9
in, height=1.4 in]{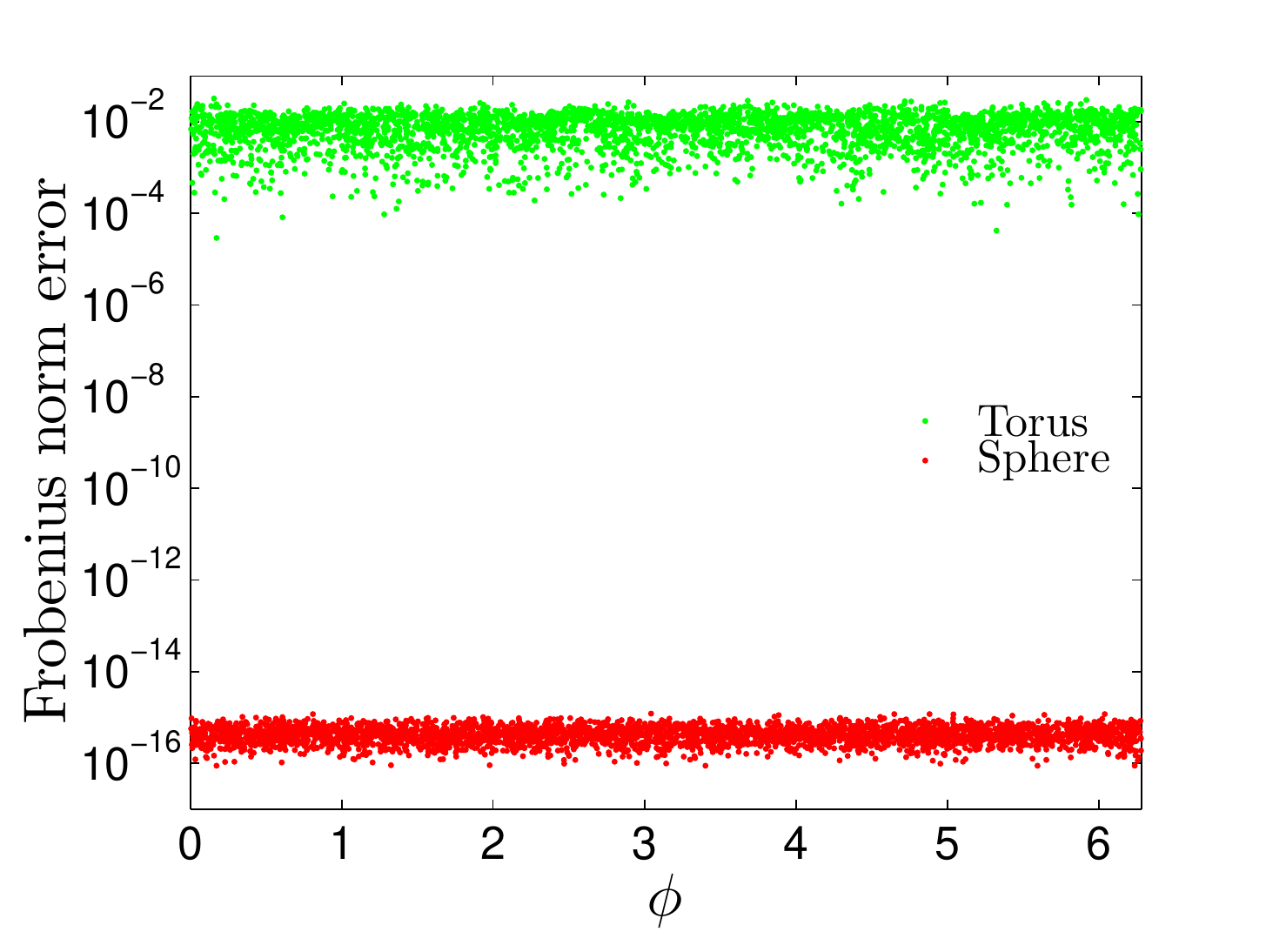}%
\end{tabular}
}
\caption{{\bf 2D torus and 2D sphere in $\mathbb{R}^{3}$.} Comparison of
the 1st  singular value $\tilde{\lambda}_1$ of $2\tilde{\mathbf{R}}^\top$ (red dots), the 2nd singular value $\tilde{\lambda}_2$ of
$2\tilde{\mathbf{R}}^\top$ (green dots),
and their difference $|\tilde{\lambda}_1 - \tilde{\lambda}_2|$ (blue dots) for (a) 2D torus and (b) 2D unit sphere examples.
(c) Pointwise Frobenius norm errors, $\|\hat{\mathbf{P}}-\mathbf{P}\|_F$ , for the 2D torus (green dots) and 2D sphere (red dots) examples. The horizontal axis corresponds
to one of the intrinsic coordinate $\phi \in [0, 2\pi)$.
Randomly distributed $N=3600$ data
points are used on the manifolds. }
\label{fig_svd}
\end{figure*}

We are now ready to prove Theorem \ref{main P Theorem}.

\begin{proof}
{Recall that $\mathbf{T}=(\boldsymbol{\tau}_1,\ldots,\boldsymbol{\tau}_d) \in \mathbb{R}^{n\times d}$ is a set of orthonormal tangent vectors  such that $\mathbf{P} = \mathbf{T}\mathbf{T}^\top$. Also recall that
$\mathbf{\tilde{T}}=(\boldsymbol{\tilde{t}}_1,\ldots,\boldsymbol{\tilde{t}}_d) \in \mathbb{R}^{n\times d}$ is the $d$
approximated tangent vectors such that $\mathbf{\tilde{P}} = \mathbf{\tilde{T}}\mathbf{\tilde{T}}^\top$. Then there exists some orthogonal matrix $\mathbf{O} \in \mathbb{R}^{d \times d}$ such that we have $\mathbf{\tilde{T}} = \mathbf{T}\mathbf{O}+O(\epsilon^{1/2})$ based on the first-order approximation result $\Vert \mathbf{\tilde{P}} - \mathbf{P}\Vert_F = O(\epsilon^{1/2})$.

This claim can be verified as follows. First, since columns of $\mathbf{\tilde{T}}$ are eigenvectors of $\mathbf{\tilde{P}}$ corresponding to eigenvalue one, it is clear that $\mathbf{P}\bm{\tilde{t}}_i = \mathbf{\tilde{P}}\bm{\tilde{t}}_i + (\mathbf{P}-\mathbf{\tilde{P}})\bm{\tilde{t}}_i = \bm{\tilde{t}}_i + O(\epsilon^{1/2})$. This also means that $(\mathbf{{P}}\bm{\tilde{t}}_i)^\top(\mathbf{{P}}\bm{\tilde{t}}_j) = \delta_{i,j} + O(\epsilon^{1/2})$ for $i \neq j$. Since columns of $\mathbf{P}\mathbf{\tilde{T}}$ are in $\textup{span}\{\boldsymbol{\tau}_1,\ldots,\boldsymbol{\tau}_d\}$, then it is clear that there exists an orthogonal matrix $\mathbf{O}$ such that $\mathbf{P}\mathbf{\tilde{T}} = \mathbf{T}\mathbf{O} +  O(\epsilon^{1/2})$.
Thus, $\tilde{\mathbf{T}} = \mathbf{\tilde{P}}\tilde{\mathbf{T}} = \mathbf{P}\mathbf{\tilde{T}}  + O(\epsilon^{1/2}) = \mathbf{T}\mathbf{O} +  O(\epsilon^{1/2})$, where the second equality follows from the fact that $\mathbf{\tilde{P}} - \mathbf{P} = O(\epsilon^{1/2})$ for each entry.

Let $\mathbf{TO}=(\boldsymbol{t}_1,\ldots,\boldsymbol{t}_d) \in \mathbb{R}^{n\times d}$, and we have $\bm{\tilde{t}}_j - \bm{t}_j =O(\epsilon^{1/2})$ for $j=1,\ldots,d$. We can write
\[
\rho_{j}^{(i)} = \mathbf{D}_i^\top\bm{t}_j + O(\epsilon),
\]
using the fact that components of $\mathbf{D}_i$ are  $O(\epsilon^{1/2})$. Based on Step 2 of the Algorithm above and Theorem~\ref{local svd result}, we have in high probability, for any $i=1,\ldots, {K}$,
\[
|\tilde{\rho}_{j}^{(i)}-\rho_{j}^{(i)}| = | \mathbf{D}_i^\top( \bm{\tilde{t}}_j-\bm{t}_j)| + O(\epsilon) \leq \|\mathbf{D}_i\|  \|\bm{\tilde{t}}_j-\bm{t}_j\| + O(\epsilon) \leq  c\epsilon^{1/2}\epsilon^{1/2}  + O(\epsilon)\leq 2c\epsilon.
\]
for some $c>0$, as $\epsilon\to 0$, using the fact that $\|\mathbf{D}_i\| = O(\epsilon^{1/2})$.
}
This implies that,
\[
|\tilde{\rho}_{i}^{(k)}\tilde{\rho}_{j}^{(k)} -\rho_{i}^{(k)}\rho_{j}^{(k)}| \leq  |\tilde{\rho}_{i}^{(k)}||\tilde{\rho}_{j}^{(k)}-\rho_{j}^{(k)}| + |\rho_{j}^{(k)} || \tilde{\rho}_{i}^{(k)} - \rho_{i}^{(k)} | = O(\epsilon^{3/2}),
\]
where we used again $|\rho_j| = O(\epsilon^{1/2})$. This means, $\|\mathbf{A} - \tilde{\mathbf{A}}\|_2  \leq \sqrt{KD}\|\mathbf{A} - \tilde{\mathbf{A}}\|_{max} =O(K\epsilon^{3/2})$.


We note that the components of each row of matrix $\mathbf{A}$ forms a homogeneous polynomial of degree-2, so they are linearly independent. Since the data points are uniformly sampled from $[-\sqrt{\epsilon},\sqrt{\epsilon}]^d$, for large enough samples, $K > D$, these sample points will not lie on a subspace of dimension strictly less than $d$. Thus, $\mathbf{A}$ is not degenerate and $ \textup{rank} (\mathbf{A}) = D$ almost surely.

Furthermore, there exists a constant $C>0$ such that,
\[
\frac{\|\mathbf{A} - \mathbf{\tilde{A}}\|_2}{\|\mathbf{A}\|_2} \leq  \frac{CK\epsilon^{3/2}}{\|\mathbf{A}\|_2} < {\frac{1}{\kappa_{2}(\mathbf{A})}},
\]
where we used the assumption $\|\mathbf{A}\|_2/\kappa_2(\mathbf{A}) = K\omega(\epsilon^{3/2})$, as $\epsilon \to 0$ for the last inequality, {to ensure that the perturbed matrix $\mathbf{\tilde{A}}$ is still full rank.

From \eqref{eqn:Yti}, one can deduce,
\begin{equation}
\underbrace{\mathbf{Y}}_{O(1)}=\underbrace{\mathbf{(A}^{\top }\mathbf{A)}%
^{-1}\mathbf{A}^{\top }}_{O(\frac{1}{\rho ^{2}})}\underbrace{( 2\mathbf{%
D}^{\top }-2\mathbf{R}) }_{O(\rho ^{2})}+O(\rho ),  \label{eqn:YAD}
\end{equation}%
where the leading order terms on both sides are $O(1)$ since both terms $\mathbf{A}$ and $2\mathbf{D}^{\top }-2\mathbf{R}$
are $O(\rho ^{2})$. Since $\tilde{\mathbf{A}}$ is full rank, we can solve the regression problem in \eqref{eqn:our2ys} as,
\begin{equation}
\mathbf{\tilde{Y}}=2\mathbf{(\tilde{A}}^{\top }\mathbf{\tilde{A})}^{-1}%
\mathbf{\tilde{A}}^{\top }\mathbf{D}^{\top }.  \label{eqn:yadtid}
\end{equation}
Using (\ref{eqn:YAD}) and (\ref{eqn:yadtid}), one has
\begin{eqnarray*}
2\tilde{\mathbf{R}}&=& 2\mathbf{D}^{\top }-\mathbf{\tilde{A}\tilde{Y}}
 =[2\mathbf{D}^{\top }-%
\mathbf{AY]+[AY}-\mathbf{\tilde{A}\tilde{Y}]} \\
&=&[2\mathbf{R}+O(\rho ^{3})]+[\mathbf{A(A}^{\top }\mathbf{A)}^{-1}\mathbf{A}%
^{\top }( 2\mathbf{D}^{\top }-2\mathbf{R}) +O(\rho ^{3})-\mathbf{%
\tilde{A}(\tilde{A}}^{\top }\mathbf{\tilde{A})}^{-1}\mathbf{\tilde{A}}^{\top
}2\mathbf{D}^{\top }] \\
&=&2\mathbf{R+(A(A}^{\top }\mathbf{A)}^{-1}\mathbf{A}^{\top }-\mathbf{\tilde{%
A}(\tilde{A}}^{\top }\mathbf{\tilde{A})}^{-1}\mathbf{\tilde{A}}^{\top
})( 2\mathbf{D}^{\top }-2\mathbf{R}) -\mathbf{\tilde{A}(\tilde{A}}%
^{\top }\mathbf{\tilde{A})}^{-1}\mathbf{\tilde{A}}^{\top }2\mathbf{R}+O(\rho
^{3}) \\
&=&\underbrace{\mathbf{(A(A}^{\top }\mathbf{A)}^{-1}\mathbf{A}^{\top }-%
\mathbf{\tilde{A}(\tilde{A}}^{\top }\mathbf{\tilde{A})}^{-1}\mathbf{\tilde{A}%
}^{\top })}_{O(\rho )}\underbrace{( 2\mathbf{D}^{\top }-2\mathbf{R}%
) }_{O(\rho ^{2})} \\
&&+\underbrace{(\mathbf{I}_{K}-\mathbf{\tilde{A}(\tilde{A}}^{\top }\mathbf{%
\tilde{A})}^{-1}\mathbf{\tilde{A}}^{\top })2\mathbf{R}}_{O(\rho )}+O(\rho
^{3}) \\
&=&(\mathbf{I}_{K}-\mathbf{\tilde{A}(\tilde{A}}^{\top }\mathbf{\tilde{A})}%
^{-1}\mathbf{\tilde{A}}^{\top })2\mathbf{R}+O(\rho ^{3}).
\end{eqnarray*}%
The key point here is to notice that the remaining term $2\mathbf{B}\mathbf{R}$, where $\mathbf{B}:= \mathbf{I}_{K}-%
\mathbf{\tilde{A}(\tilde{A}}^{\top }\mathbf{\tilde{A})}^{-1}\mathbf{\tilde{A}%
}^{\top }$, is in the tangent space and all the normal direction
terms up to order of $O(\rho ^{2})$ are cancelled out. Then, one can identify $\text{span}\{\boldsymbol{\tau}_1, \ldots,\boldsymbol{\tau}_d\}$ (or span of the leading $d$-eigenvectors of $\mathbf{R}^\top  \mathbf{B}^\top\mathbf{B} \mathbf{R}$ corresponding to nontrivial spectra) by computing the leading $d$-eigenvectors of $\mathbf{\tilde{R}}^\top \mathbf{\tilde{R}}$.
Moreover,
\[
\| \mathbf{\tilde{R}} ^\top \mathbf{\tilde{R}}  - \mathbf{R}^\top \mathbf{B}^\top \mathbf{B} \mathbf{R}\|_\infty \leq \| \mathbf{\tilde{R}}^\top \|_\infty \| \mathbf{\tilde{R}}  -  \mathbf{B} \mathbf{R} \|_\infty + \|\mathbf{\tilde{R}}^\top   - \mathbf{R}^\top  \mathbf{B}^\top \|_\infty \| \mathbf{B} \mathbf{R} \|_\infty =  O(\rho^4).
\]
Similarly, $\| \mathbf{\tilde{R}} ^\top \mathbf{\tilde{R}}  - \mathbf{R}^\top \mathbf{B}^\top \mathbf{B} \mathbf{R}\|_1 =O(\rho^4)$. Therefore, $\| \mathbf{\tilde{R}} ^\top \mathbf{\tilde{R}}  - \mathbf{R}^\top \mathbf{B}^\top \mathbf{B} \mathbf{R}\|_2 = O(\rho^4)=O(\epsilon^2)$. }

Let $\{\bm{\psi}_i\}_{i=1,\ldots, n}$ and $\{\bm{\tilde{\psi}}_i\}_{i=1,\ldots, n}$ be the unit eigenvectors of $\mathbf{R}^\top{\mathbf{B}^\top\mathbf{B}}\mathbf{R}$ and $\tilde{\mathbf{R}}^\top\tilde{\mathbf{R}}$, respectively. By Theorem~5.4 of \cite{demmel1997applied} and the assumption on the spectral gap of $\mathbf{R}^\top{\mathbf{B}^\top\mathbf{B}}\mathbf{R}$, $g_i> c\epsilon$ for some $c>0$, the acute angle between $\bm{\psi}_i$ and $\bm{\tilde{\psi}}_i$ satisfies,
\[
\sin(2\theta) \leq 2\frac{\|\mathbf{R}^\top{\mathbf{B}^\top\mathbf{B}}\mathbf{R} - \tilde{\mathbf{R}}^\top\tilde{\mathbf{R}}\|_2}{g_i} = O(\epsilon),
\]
where the constant in the big-oh notation above depends on $n$.
Then,
\BEA
\| \bm{\psi}_i -\bm{\tilde{\psi}}_i \|^2 &=& \| \bm{\psi}_i\|^2+\|\bm{\tilde{\psi}}_i\|^2 - 2 \bm{\psi}_i^\top \bm{\tilde{\psi}}_i = 2(1-\cos(\theta)) = 4\sin^2 \left(\frac{\theta}{2}\right) =O(\epsilon^2). \label{eigvecerror}
\EEA
By Proposition~\ref{propP}(3), it is clear that $\mathbf{{P}}:= \bm{{\Psi}} \bm{{\Psi}}^\top$. Based on the step~4 of the Algorithm,
$\mathbf{\hat{P}}:= \bm{\hat{\Psi}} \bm{\hat{\Psi}}^\top$ and from the error bound in \eqref{eigvecerror}, the proof is complete.

\end{proof}
\begin{remark} \label{rem5}The result above is consistent with the intuition that the scheme is of order $\rho^2 = O(\epsilon)$. The numerical result in the torus suggests a rate of $\epsilon=O(N^{-1})$. For general $d$-dimensional manifolds, the error rate for the second-order method is expected to be $\epsilon\sim N^{-2/d}$, due to the fact that $\rho \propto N^{-1/d}$.
\end{remark}

\section{Spectral Convergence Results}\label{section4}
In this section, we { state} spectral convergence results for { the symmetric estimator} $\mathbf{G}^\top\mathbf{G}$ to the Laplace-Beltrami operator $\Delta_M$, {as well as analogous results for the symmetric estimator of the Bochner Laplacian $\mathbf{P}^{\otimes} \mathbf{H}^\top \mathbf{H} \mathbf{P}^{\otimes}$ to the Bochner Laplacian $\Delta_B$ on vector fields. For definitions of the operators and estimators of this section, please see Sections \ref{Laplace-Beltrami definition} and \ref{Bochner definition}.
To keep the section short, we only present the proof for the spectral convergence of the Laplace-Beltrami operator. We document the proofs of the intermediate bounds needed for this proof in Appendices ~\ref{app:B}, \ref{app:C1} and \ref{app:C2}. We should point out that the symmetry of discrete estimators $\mathbf{G}^\top\mathbf{G}$ and $\mathbf{P}^{\otimes} \mathbf{H}^\top \mathbf{H} \mathbf{P}^{\otimes}$ allows the convergence of the eigenvectors to be proved as well. Since the proof of eigenvector convergence is more technical, we present it in Appendix~\ref{app:C3}.
The same techniques used to prove convergence results for the Laplace-Beltrami operator are used to show the results for the Bochner Laplacian acting on vector fields, thanks to the similarity in definitions for the Bochner Laplacian and the Laplace-Beltrami operator and the similarity in our discrete estimators. Since these proofs follow the same arguments as those for the Laplace-Beltrami operator, we document them in Appendix \ref{app:D}.
While we suspect the proof technique can also be used to show the spectral convergence for the Hodge and Lichnerowicz Laplacians, the exact calculation can be more involved since the weak forms of these two Laplacians have more terms compared to that of the Bochner Laplacian.}

\subsection{Spectral Convergence for Laplace-Beltrami}\label{sec4}

Given a data set of points $ X = \{x_1 , \dots x_N \} \subset M$ and a function $f:M \to \mathbb{R}$, recall that the interpolation $I_{\phi_s} f$ only depends on $f(x_1), \dots, f(x_N)$. Hence, $I_{\phi_s}$ can be viewed, after restricting to the manifold, as a map either from $\{ f: M \to \mathbb{R} \} \to C^{\alpha - \frac{(n-d)}{2}}(M)$, or as a map
$$
I_{\phi_s} : \mathbb{R}^N \to C^{\alpha - \frac{(n-d)}{2}}(M).
$$
For details regarding the loss of regularity which occurs when restricting to the $d$-dimensional submanifold $M\subset \BR^n$, see the beginning of Section 2.3 in \cite{fuselier2012scattered}. For the Laplace-Beltrami operator, our focus is on continuous estimators, so we presently regard $I_{\phi_s}$ as a map
$$
I_{\phi_s}: \{ f:M \to \mathbb{R} \} \to C^{\alpha - \frac{(n-d)}{2}}(M).
$$
{  Based on Theorem 10 in \cite{fuselier2012scattered}, one can deduce the following interpolation error.
\begin{lem}
\label{function interpolation}
Let $\phi_s$ be a kernel whose RKHS norm equivalent to Sobolev space of order $\alpha > n/2$. Then there is sufficiently large $N = |X|$ such that with probability higher than $1 - \frac{1}{N} $, for all $f \in H^{\alpha - \frac{(n-d)}{2}}(M)$, we have
$$
\| I_{\phi_s} f - f\|_{L^2(M)} = O\left( N^{\frac{-2\alpha + (n-d)}{2d}} \right).
$$
\end{lem}

\begin{proof}
See Appendix~\ref{sec:B2}.
\end{proof}

To ensure that derivatives of interpolations of smooth functions are bounded, we must have an interpolator that is sufficiently regular. While many results in this paper hold whenever the RKHS induced by $I_{\phi_s}$ is norm equivalent to a Sobolev space of order $\alpha > n/2$, we require slightly higher regularity to prove spectral convergence.} In particular, we assume the following.
{
\begin{assu}[Sufficiently Regular Interpolator]\label{weaklyunstable}
Assume that $I_{\phi_s}$ has an RKHS norm equivalent to $H^\alpha(\mathbb{R}^n)$ with $\alpha \geq n/2 + 3$.
\end{assu}
This assumption allows us to conclude that whenever $f \in C^\infty(M)$, we have the following equation: $\| I_{\phi_s} f\|_{W^{2, \infty}(M)} = O(1)$, where the constants depend on $M, \|f\|_{W^{2, \infty}(M)}, $ and $\|f\|_{H^{\alpha - (n-d)/2}(M)}$. This assumption will be needed for the uniform boundedness of random variables in the concentration inequality in the following lemmas. This statement is made reported concisely as Lemma~\ref{bounded derivatives} in Appendix~\ref{app:B}. Before we prove the spectral convergence result, let us state the following two concentration bounds that will be needed in the proof.

\begin{lem}
\label{weak G convergence continuous}
Let $f,h \in C^\infty(M)$. Then with probability higher than $1 - \frac{2}{N},$
$$
\left| \langle \Delta_M f , h \rangle_{L^2(M)} -\langle \textup{grad}_g I_{\phi_s} f , \textup{grad}_g I_{\phi_s} h \rangle_{L^2(\mathfrak{X}(M))}  \right| \leq C N^{\frac{-2\alpha + (n-d)}{2d}},
$$
as $N\to\infty$, where the constant $C$ depends on $ \|\Delta_M f\|_{L^2(M)} + \| \Delta_M I_{\phi_s} h \|_{L^2(M)}$.
\end{lem}

\begin{proof}
See Appendix~\ref{app:C1}. In the course of the proof, we note that the prove inherits the interpolation error rate in Lemma~\ref{function interpolation}.
\end{proof}

For the discretization, we need to define an appropriate inner product such that it is consistent with the inner product of $L^2(M)$ as the number of data points $N$ approaches infinity. In particular, we have the following definition.
\begin{definition}
    \label{discrete L2}
Given two vectors $\mathbf{f}, \mathbf{h} \in \mathbb{R}^N$, we define
\BEA
\langle \mathbf{f} , \mathbf{h} \rangle_{L^2(\mu_N)} : = \frac{1}{N} \mathbf{f}^\top \mathbf{h}. \notag
\EEA
Similarly, we denote by $\| \cdot \|_{L^2(\mu_N)}$ the norm induced by the above inner product.
\end{definition}
We remark that when $\mathbf{f}$ and $\mathbf{h}$ are restrictions of functions $f$ and $h$, respectively, then the above can be evaluated as $\langle \mathbf{f} , \mathbf{h} \rangle_{L^2(\mu_N)} =  \frac{1}{N} \sum_{i=1}^N f(x_i) h(x_i)$.

By the sufficiently regular interpolator in Assumption~\ref{weaklyunstable}, it is clear that the concentration bound in the lemma above converges as $N \to \infty$. The next concentration bound is as follows:

\begin{lem}
\label{weak G convergence empirical}
Let $f,h \in C^\infty(M)$. Let the sufficiently regular interpolator Assumption~\ref{weaklyunstable} be valid. Then with probability higher than $1-\frac{2}{N}$,
$$
\left| \langle \mathbf{G}^\top \mathbf{G} \mathbf{f}, \mathbf{h} \rangle_{L^2(\mu_N)} - \langle \textup{grad}_g I_{\phi_s} f, \textup{grad}_g I_{\phi_s} h \rangle_{L^2(\mathfrak{X}(M))} \right| \leq C\frac{\sqrt{\log N}}{\sqrt{N}}
$$
for some constant $C>0$, as $N\to \infty$.
\end{lem}

\begin{proof}
See Appendix~\ref{app:C2}.
\end{proof}
}
The main results for the Laplace-Beltrami operator are as follows. First, we have the following eigenvalue convergence result.
\begin{theo}
\label{eigvalconv}
(convergence of eigenvalues: symmetric formulation) Let $\lambda_i$ denote the $i$-th eigenvalue of $\Delta_M$, enumerated $\lambda_1 \leq \lambda_2 \leq \dots $, and { fix some} $ i \in \mathbb{N}$. Assume that $\mathbf{G}$ is as defined in \eqref{sec2.1:eq1} with interpolation operator that satisfies the hypothesis in Assumption~\ref{weaklyunstable}. Then there exists an eigenvalue $\hat{\lambda}_i$ of $\mathbf{G}^\top\mathbf{G}$ such that
\BEA
\left|\lambda_i - \hat{\lambda}_i \right| \leq O\left(N^{-\frac{1}{2}}\right) + {O\left( N^{\frac{-2\alpha + (n-d)}{2d}} \right)},\label{spectralerrorrate}
\EEA
with probability greater than $1 - {\frac{12}{N}}$.
\end{theo}

\begin{proof}
Enumerate the eigenvalues of $\mathbf{G}^\top\mathbf{G}$ and label them $\hat{\lambda}_1 \leq \hat{\lambda}_2 \leq \dots \leq \hat{\lambda}_N$. Let $\mathcal{S}_i' \subseteq C^\infty(M)$ denote an $i$-dimensional subspace of smooth functions on which the quantity $\textup{max}_{f \in \mathcal{S}_i} \frac{\langle \mathbf{G}^\top\mathbf{G} R_Nf , R_Nf \rangle_{L^2(\mu_N)}}{\|R_N f \|_{L^2(\mu_N)}}$ achieves its minimum. Let $\tilde{f} \in \mathcal{S}_i'$ be the function on which the maximum $\textup{max}_{f \in \mathcal{S}_i'} \langle \Delta_M f , f \rangle_{L^2(M)}$ occurs. WLOG, assume that $\|\tilde{f}\|^2_{L^2(M)} = 1.$ Assume that $N$ is sufficiently large so that by Hoeffding's inequality $\left| \|R_N \tilde{f} \|^2_{L^2(\mu_N)} -  1 \right| \leq \frac{\textup{Const}}{\sqrt{N}} \leq 1/2$, with probability $1 - \frac{2}{N}$, so that $\|R_N \tilde{f} \|^2_{L^2(\mu_N)}$ is bounded away from zero. Hence, we can Taylor expand $ \frac{\langle \mathbf{G}^\top\mathbf{G} R_N \tilde{f} , R_N \tilde{f} \rangle_{L^2(\mu_N)}}{\|R_N \tilde{f} \|^2_{L^2(\mu_N)}}$ to obtain
$$
 \frac{\langle \mathbf{G}^\top\mathbf{G} R_N \tilde{f} , R_N \tilde{f} \rangle_{L^2(\mu_N)}}{ \|R_N \tilde{f}\|^2_{L^2(\mu_N)}} = \langle \mathbf{G}^\top\mathbf{G} R_N \tilde{f} , R_N \tilde{f} \rangle_{L^2(\mu_N)} -  \frac{ \textup{Const} \langle \mathbf{G}^\top\mathbf{G} R_N \tilde{f} , R_N \tilde{f} \rangle_{L^2(\mu_N)} }{\sqrt{N}}.
$$
{ By Lemmas~\ref{weak G convergence continuous} and \ref{weak G convergence empirical}, with probability higher than $1 - \frac{4}{N}$, we have that
$$
\left| \langle \mathbf{G}^\top\mathbf{G} R_N \tilde{f} , R_N \tilde{f} \rangle_{L^2(\mu_N)} - \langle \Delta_M \tilde{f} , \tilde{f} \rangle_{L^2(M)} \right| = O \left( N^{-\frac{1}{2}} \right) + {O\left( N^{\frac{-2\alpha + (n-d)}{2d}} \right)}.
$$
Combining the two bounds above, we obtain that with probability higher than $1 - {\frac{6}{N}}$,
$$
\langle \Delta_M \tilde{f} , \tilde{f} \rangle_{L^2(M)} \leq   \frac{\langle \mathbf{G}^\top\mathbf{G} R_N\tilde{f} , R_N \tilde{f} \rangle_{L^2(\mu_N)}}{\|R_N \tilde{f} \|^2_{L^2(\mu_N)}}  + O\left(  N^{-\frac{1}{2}} \right) + {O\left( N^{\frac{-2\alpha + (n-d)}{2d}} \right)}.
$$}
Since $\tilde{f}$ is the function on which $\langle \Delta_M f, f \rangle_{L^2(M)}$ achieves its maximum over all $f \in \mathcal{S}_i'$, and since certainly
\[
\frac{\langle \mathbf{G}^\top\mathbf{G} R_N \tilde{f} , R_N \tilde{f} \rangle_{L^2(\mu_N)}}{\|R_N \tilde{f} \|^2_{L^2(\mu_N)}} \leq \textup{max}_{f \in \mathcal{S}_i'} \frac{\langle \mathbf{G}^\top\mathbf{G} R_Nf , R_Nf \rangle_{L^2(\mu_N)}}{\|R_N f \|^2_{L^2(\mu_N)}},
\]
we have the following:
$$
\textup{max}_{f \in \mathcal{S}_i'} \langle \Delta_M f , f \rangle_{L^2(M)} \leq  \textup{max}_{f \in \mathcal{S}_i'} \frac{\langle \mathbf{G}^\top\mathbf{G} R_Nf , R_Nf \rangle_{L^2(\mu_N)}}{\|R_N f \|^2_{L^2(\mu_N)}}  +  O\left( N^{-\frac{1}{2}} \right) + {O\left( N^{\frac{-2\alpha + (n-d)}{2d}} \right)}.
$$
But we assumed that $\mathcal{S}_i'$ is the exact subspace on which $\textup{max}_{f \in \mathcal{S}_i} \frac{\langle \mathbf{G}^\top\mathbf{G} R_Nf , R_Nf \rangle_{L^2(\mu_N)}}{\|R_N f \|^2_{L^2(\mu_N)}}$ achieves its minimum. Hence,
$$
\textup{max}_{f \in \mathcal{S}_i'} \langle \Delta_M f , f \rangle_{L^2(M)} \leq \hat{\lambda}_i +  O\left( N^{-\frac{1}{2}}\right) + {O\left( N^{\frac{-2\alpha + (n-d)}{2d}} \right)}.
$$
But the left-hand-side certainly bounds from above by the minimum of $\textup{max}_{f \in \mathcal{S}_i} \langle \Delta_M f , f \rangle_{L^2(M)}$ over all $i$-dimensional smooth subspaces $\mathcal{S}_i$. Hence,
$$
\lambda_i \leq \hat{\lambda}_i +  O\left( N^{-\frac{1}{2}}\right) + {O\left( N^{\frac{-2\alpha + (n-d)}{2d}} \right)}.
$$
The same argument yields that $\hat{\lambda}_i \leq \lambda_i +  O\left( N^{-\frac{1}{2}}\right) + {O\left( N^{\frac{-2\alpha + (n-d)}{2d}} \right)}$, with probability higher than $1 - \frac{6}{N}$. This completes the proof.
\end{proof}

The first error term in \eqref{spectralerrorrate} can be seen as coming from discretizing a continuous operator, while the second error term in \eqref{spectralerrorrate} comes from the fact that continuous estimators in our setting differ from the true Laplace-Beltrami by pre-composing with interpolation. {The convergence holds true for eigenvectors,} though in this case, the constants involved depend heavily on the multiplicity of the eigenvalues.

{ The convergence of eigenvector result is stated as follows.}
\begin{theo}\label{conveigvec}
Let $\epsilon_{\lambda_\ell}:=|\lambda_\ell-\hat{\lambda}_\ell|$ denote the error in approximating the $\ell$-th \underline{distinct} eigenvalue, $\lambda_\ell$, as defined in Theorem~\ref{eigvalconv}. Let Assumption~\ref{weaklyunstable} be valid.
For any $\ell$, there is a constant $c_\ell$ such that whenever $\epsilon_{\lambda_{\ell-1}}, \epsilon_{\lambda_{\ell+1}}  < c_\ell$, then with probability higher than $1 - {\left( \frac{2m^2 + 5m + 24}{N}\right)}$, {where $m$ is the geometric multiplicity of eigenvalue $\lambda_\ell$,} we have the following situation:
for any normalized eigenvector $\mathbf{u}$ of $\mathbf{G}^\top\mathbf{G}$ with eigenvalue $\hat{\lambda}_\ell$, there is a normalized eigenfunction $f$ of $\Delta_M$ with eigenvalue $\lambda_\ell$ such that
\begin{equation*}
    \|R_N f - \mathbf{u} \|_{L^2(\mu_N)} =  O\left(N^{-\frac{1}{4}}\right) + {O\left( N^{\frac{-2\alpha + (n-d)}{4d}} \right)}.
\end{equation*}
\end{theo}
{ As the proof is more technical, we present it in Appendix \ref{app:C}}.  It is important to note that the results of this section use analytic $\mathbf{P}$. This allows us to conclude that, depending on the smoothness of the kernel, any error slower than the Monte-Carlo convergence rate observed numerically is introduced through the approximation of $\mathbf{P}$, as discussed in the previous section.

\subsection{Spectral Convergence for the Bochner Laplacian}\label{sec5}
The Bochner Laplacian on vector fields is defined in such a way that makes the theoretical discussion in this setting almost identical to that of the Laplace-Beltrami operator. { Hence, we relegate the proofs of the results below to Appendix~\ref{app:D}. We emphasize that the results below will also rely on the Assumption \ref{weaklyunstable} which allows us to have a stable interpolator of smooth vector fields. In particular, as a corollary to Lemma~\ref{bounded derivatives vector field}, the Assumption \ref{weaklyunstable} gives $\|I_{\phi_s} u\|^2_{W^{2, \infty}(\mathfrak{X}(M))} = O(1)$ for any vector field $u \in \mathfrak{X}(M)$ whose components are $C^\infty(M)$ functions. Details regarding the previous statement are found in Appendix \ref{app:B}.}

The main results for the Bochner Laplacian are as follows. First, we have the following eigenvalue convergence result.
\begin{theo}
\label{eigvalconv Bochner}
(convergence of eigenvalues: symmetric formulation) Let $\lambda_i$ denote the $i$-th eigenvalue of $\Delta_B$, enumerated $\lambda_1 \leq \lambda_2 \leq \dots $. Let Assumption \ref{weaklyunstable} be valid. For {some fixed} $i \in \mathbb{N}$, there exists an eigenvalue $\hat{\lambda}_i$ of $\mathbf{P}^\otimes \mathbf{H}^\top\mathbf{H} \mathbf{P}^\otimes$ such that
$$
\left|\lambda_i - \hat{\lambda}_i \right| \leq O\left( N^{- \frac{1}{2}} \right) + {O\left( N^{\frac{-2\alpha + (n-d)}{2d}} \right)},
$$
with probability greater than $1 - {\left( \frac{4n+4}{N} \right)}$.
\end{theo}
The same rate holds true for convergence of eigenvectors, just as in the Laplace-Beltrami case. In fact, the proof of convergence of eigenvectors follows the same argument in Section~\ref{app:C3}. {Before stating the main result, we have the following definition.
\begin{definition}
    Given two vectors $\mathbf{U},\mathbf{V} \in \mathbb{R}^{nN}$ representing restrictions of vector fields, define discrete inner product $L^2(\mu_{N,n})$ in the following way:
$$
\langle \mathbf{U}, \mathbf{V} \rangle_{L^2(\mu_{N,n})} = \frac{1}{N} \sum_{j=1}^{n} \mathbf{U}^j \cdot \mathbf{V}^j,
$$
where $\mathbf{U}^j \in \mathbb{R}^N$,  $\mathbf{U} = ((\mathbf{U}^1)^{\top} , \dots , (\mathbf{U}^{n})^{\top})^{\top}$, and similarly for $\mathbf{V}^j$ and $\mathbf{V}$.
\end{definition}
With this definition, we can formally state the convergence of eigenvectors result for the Bochner Laplacian.}

\begin{theo}\label{conveigvec Bochner}
Let $\epsilon_{\lambda_\ell}:=|\lambda_\ell - \hat{\lambda}_\ell|$ denote the error in approximating the $\ell$-th \underline{distinct} eigenvalue,  following the notation in Theorem~\ref{eigvalconv Bochner}. Let the Assumption \ref{weaklyunstable} be valid. For any $\ell$, assume that there is a constant $c_\ell$ such that if $\epsilon_{\lambda_{\ell-1}}, \epsilon_{\lambda_{\ell+1}}  < c_\ell$, then with probability higher than $1 - {\left( \frac{2m^2 + 2m + 3nm + 8n + 8}{N}\right)} $, we have the following situation: for any normalized eigenvector $\mathbf{U}$ of $\mathbf{P}^\otimes\mathbf{H}^\top\mathbf{H} \mathbf{P}^\otimes$ with eigenvalue $\hat{\lambda}_\ell$, there is a normalized eigenvector field $v$ of $\Delta_B$ with eigenvalue $\lambda_\ell$ such that
\begin{equation*}
    \|R_N v - \mathbf{U} \|_{L^2(\mu_{N,n})} = {O\left(N^{-\frac{1}{4}} \right)} + {O\left( N^{\frac{-2\alpha + (n-d)}{4d}} \right)},
\end{equation*}
where $m$ is the geometric multiplicity of eigenvalue $\lambda_\ell$.
\end{theo}
{ Proofs of the results for the Bochner Laplacian can be found in Appendix \ref{app:D}.}
\section{Numerical Study of Eigenvalue Problems}
\label{sec6}
In this section, we first discuss two examples of eigenvalue problems of functions defined on simple manifolds: one being a 2D generalized torus embedded in $\mathbb{R}^{21}$ and the other 4D flat torus embedded in $\mathbb{R}^{16}$. In these two examples, we will compare the results between the Non-symmetric and Symmetric RBFs, which we refer to as NRBF and SRBF respectively, using analytic $\mathbf{P}$ and the approximated $\mathbf{\hat{P}}$. {In the first example, we further compare with diffusion maps (DM) algorithm, which is an important manifold learning algorithm that estimates eigen-solutions of the Laplace-Beltrami operator.} When the manifold is unknown, as is often the case in practical applications, one does not have access to analytic $\mathbf{P}$. Hence, it is most reasonable to compare DM and RBF with $\mathbf{\hat{P}}$. { While DM algorithm can be implemented with a sparse approximation via the $K$-Nearest Neighbors (KNN) algorithm, we would like to verify how SRBF $\mathbf{\hat{P}}$, which is a dense approximation, performs compared to DM for various degree of sparseness (including not using KNN as reported in Appendix~\ref{app:E}).} Next, we discuss an example of eigenvalue problems of vector fields defined on a sphere. Numerically, we compare the results between the RBF method and the analytic truth. Since the size of our vector Laplacian approximation is $Nn\times Nn$, the current RBF methods are only numerically feasible for data sets with small ambient dimension $n$.


\subsection{Numerical Setups}

In the following, we introduce our numerical setups for NRBF, SRBF, and DM
methods  of finding approximate  solutions to eigenvalue problems  associated to Laplace-Beltrami or vector Laplacians on
manifolds.

\noindent \textbf{Parameter specification for RBF:} For the implementation of RBF methods, there are two groups of kernels to be used. One group includes infinitely smooth RBFs, such as Gaussian (GA), multi-quadric (MQ), inverse multiquadric (IMQ), inverse quadratic (IQ), and Bessel (BE) \cite{fornberg2011stable,Natasha2015Solving,Fornberg2011JCP}. The other group includes piecewise smooth RBFs, such as Polyharmonic spline (PHS), Wendland (WE), and Mat\'{e}rn (MA) \cite%
{Natasha2015Solving,Fuselier2009Stability,1995Piecewise,Wendland2005Scat}. In this work, we only apply GA and IQ kernels and test their numerical performances. To compute the interpolant matrix $\mathbf{\Phi }$ in \eqref{eqn:phicf}, all points are connected and we did not use KNN truncations. The shape parameter $s$ is manually tuned but fixed for
different $N$ when we examine the convergence of eigenmodes.

Despite  not needing a structured mesh, many RBF techniques impose strict requirements for uniformity of the underlying data points. For uniformly distributed grid points, it often occurs that the operator approximation error decreases rapidly with the number
of data $N$ until the calculation breaks down due to the increasing
ill-conditioning of the interpolant matrix $\boldsymbol{\Phi }$ defined in \eqref{eqn:phicf} \cite%
{Tarwater1985Parameter,Schaback1995Error,Natasha2015Solving}. In this numerical section, we consider data points randomly distributed on manifolds, which means that two neighboring points can be very close to each other. In this case, the interpolant matrix $\mathbf{\Phi }$ involved in most of the global RBF techniques tends to be ill-conditioned or even singular for sufficiently large $N$. In fact, one can show that with inverse quadratic kernel, the condition number of the matrix $\boldsymbol{\Phi }$ grows exponentially as a function of $N$.
To resolve such an ill-conditioning issue, we apply the pseudo inversion instead of the direct inversion in approximating the interpolant matrix $\mathbf{\Phi }$. In our implementation, we will take the tolerance parameter of pseudo-inverse around $10^{-9}\sim 10^{-4}$.



If the parametrization or the level set representation of the manifold is
known, we can apply the analytic tangential projection matrix $\mathbf{P}$ for constructing the RBF Laplacian matrix. If the parametrization is unknown, that is, only the point cloud is given, we can first learn $\mathbf{%
\hat{P}}$ using the 2nd-order SVD method and then construct  the RBF
Laplacian.  Notice that we can also construct the Laplacian matrix using $%
\mathbf{\tilde{P}}$ estimated from the 1st-order SVD (not shown in this
work). We found that the results of eigenvalues and eigenvectors using $\mathbf{\tilde{P}}$ are not as
good as those using $\mathbf{\hat{P}}$ from our 2nd-order SVD. We also
notice that the estimation of $\mathbf{P}$ and the construction of the Laplacian matrix can be performed separately using two sets of points. For example, one can use $10,000$ points to approximate $\mathbf{P}$ but use only $2500$ points to construct the Laplacian matrix. This allows one to leverage large amounts of data in the estimation of $\mathbf{P}$, while too much data may not be computationally feasible with graph Laplacian-based approximators such as the diffusion maps algorithm.

For SRBF, the estimated sampling density is needed if the distribution of
the data set is unknown. Note that for NRBF, the sampling density is not needed for constructing Laplacian. In our numerical experiments, {we apply the MATLAB built-in kernel density estimations (KDEs)
function \texttt{mvksdensity.m} 
for approximating the sampling density. We also apply Silverman's rule
of thumb \cite{silverman2018density} for tuning the bandwidth parameter in the KDEs.}

\noindent \textbf{Eigenvalue problem solver for RBF: } For NRBF, we apply the non-symmetric estimator in \eqref{eqn:Nrbf} for solving the eigenvalue problem. The NRBF eigenvectors might be complex-valued and are only linearly independent (i.e., they are not necessarily orthonormal). For SRBF, we apply the symmetric estimator in \eqref{eqn:Srbf} for solving the generalized eigenvalue problem. When the sampling density $q$ is unknown, we employ the symmetric formulation with the estimated sampling density $\tilde{q}(x)$ obtained from the KDE.

Since we used pseudo inversion to resolve the ill-conditioning issue of $%
\mathbf{\Phi }$, the resulting RBF Laplacian matrix will be of low rank, $L$, and will have many zero eigenvalues, depending on the choice of tolerance in the pseudo-inverse algorithm. Two issues naturally arise in this situation. First, it becomes difficult to compute the eigenspace
corresponding to the zero eigenvalue(s), especially for the eigenvector-field problem. At this moment, we have not developed appropriate schemes to detect the existence of the nullspace and estimate the harmonic functions in this nullspace if it exists. Second, finding even the leading nonzero eigenvalues (that are close to zero) can be numerically expensive. Based on the rank of the RBF $\mathbf{\Phi }$, for symmetric approximation, one can use the ordered real-valued eigenvalues to attain the nontrivial $L$ eigenvalues in descending (ascending) order when $L$ is small (large). For the non-symmetric approximation, one can also employ a similar idea by sorting the magnitude of the eigenvalues (since the eigenvalues may be complex-valued). This naive method, however, can be very expensive when the number of data points $N$ is large and when the rank of the matrix $L$ is neither $O(10)$ nor close to $N$.

\noindent \textbf{Comparison of eigenvectors for repeating eigenvalues:}
When an eigenvalue is non-simple, one needs to be careful in quantifying the errors of eigenvectors for these repeated
eigenvalues since the set of true orthonormal eigenvectors is only unique  up to a rotation matrix. To quantify the errors of eigenvectors,
we apply the following Ordinary Least Square (OLS) method. Let $\mathbf{F} =\left(
\mathbf{f} _{1},\ldots ,\mathbf{f} _{m}\right) $ be the true eigenfunctions located at $X$
corresponding to one repeated eigenvalue $\lambda_i$ with multiplicity $m$%
, and  let $\mathbf{\tilde{U}}=\left( \mathbf{\tilde{u}}_{1},\ldots ,\mathbf{\tilde{u}}_{m}\right) $ be their DM or RBF approximations. Assume that the
linear regression model is written as $\mathbf{F} =\mathbf{\tilde{U}}\beta +\varepsilon $%
, where $\varepsilon $ is an $N\times m$\ matrix representing the errors and
$\beta $ is a $m\times m$\ matrix representing the regression coefficients.
The coefficients matrix $\beta $ can be approximated using OLS by $\hat{\beta%
}=(\mathbf{\tilde{U}}^{\top}\mathbf{\tilde{U}})^{-1}(\mathbf{\tilde{U}}^{\top}\mathbf{F} )$. The
rotated DM or RBF eigenvectors can be written as a linear combination, $\mathbf{\hat{V}}=$ $\mathbf{\tilde{U}}\hat{\beta}$, where these new $\mathbf{\hat{V}}=\left( \mathbf{\hat{v}}_{1},\ldots ,\mathbf{\hat{v}}_{m}\right) $ are in ${\mathrm{Span}}\{%
\mathbf{\tilde{u}}_{1},\ldots ,\mathbf{\tilde{u}}_{m}\}$. Finally, we can quantify
the pointwise errors of eigenvectors between $\mathbf{f} _{j}$ and $\mathbf{\hat{%
v}}_{j}$ for each $j=1,\ldots ,m$. For eigenvector fields, we can follow
 a similar outline to quantify the errors of vector fields.
Incidentally, we mention that there are many ways to measure  eigenvector errors since the approximation of the rotational coefficient matrix $%
\beta $ is not unique. Here, we only provide  a practical metric that we will use in our numerical examples below.

\subsection{2D General Torus\label{sec:gentorus}}

In this section, we investigate the eigenmodes of the Laplace-Beltrami operator
on a general torus. The parameterization of the general torus is given by
\begin{equation}
{x}=\left[
\begin{array}{c}
x^{1} \\
x^{2} \\
\vdots \\
x^{n-2} \\
x^{n-1} \\
x^{n}%
\end{array}%
\right] =\left[
\begin{array}{c}
(a+\cos \theta )\cos \phi \\
(a+\cos \theta )\sin \phi \\
\vdots \\
\frac{2}{n-1}(a+\cos \theta )\cos \frac{n-1}{2}\phi \\
\frac{2}{n-1}(a+\cos \theta )\sin \frac{n-1}{2}\phi \\
\sqrt{\sum_{i=1}^{(n-1)/2}\frac{1}{i^{2}}}\sin \theta%
\end{array}%
\right] ,  \label{eqn:gentor}
\end{equation}%
where the two intrinsic coordinates $0\leq \theta ,\phi \leq 2\pi $ and the
radius $a=2>1$. The Riemannian metric is
\begin{equation}
g=\left[
\begin{array}{cc}
b_{\frac{n-1}{2}} & 0 \\
0 & \frac{n-1}{2}(a+\cos \theta )^{2}%
\end{array}%
\right],  \label{eqn:metgentors}
\end{equation}%
where $b_{\frac{n-1}{2}} :=  {\sum_{i=1}^{(n-1)/2}\frac{1}{i^{2}}}$.
We  solve the following eigenvalue problem for Laplace-Beltrami
operator:%
\begin{equation}
\Delta _{g}\psi _{k}=\frac{-1}{(a+\cos \theta )}\left[ \frac{\partial }{%
\partial \theta }\left( \left( a+\cos \theta \right) \frac{1}{b_{\frac{n-1}{2}}}\frac{%
\partial \psi _{k}}{\partial \theta }\right) +\frac{\partial }{\partial \phi
}\left( \frac{2}{n-1}\frac{1}{(a+\cos \theta )}\frac{\partial \psi _{k}}{%
\partial \phi }\right) \right] =\lambda _{k}\psi _{k},  \label{eqn:psigento}
\end{equation}%
where $\lambda _{k}$ and $\psi _{k}$ are the eigenvalues and eigenfunctions,
respectively. After separation of variables (that is, we set $\psi _{k}=\Phi
_{k}\left( \phi \right) \Theta _{k}\left( \theta \right) $ and substitute $%
\psi _{k}$ back into (\ref{eqn:psigento}) to deduce the equations for $\Phi
_{k}$ and $\Theta _{k}$ ), we obtain:%
\begin{eqnarray*}
\frac{{\mathrm{d}}^{2}\Phi _{k}}{{\mathrm{d}}\phi ^{2}}+m_{k}^{2}\Phi _{k}
&=&0, \\
\frac{{\mathrm{d}}}{{\mathrm{d}}\theta }\left( (a+\cos \theta )\frac{{%
\mathrm{d}}\Theta _{k}}{{\mathrm{d}}\theta }\right) -\frac{b_{\frac{n-1}{2}}m_{k}^{2}}{%
\frac{n-1}{2}\left( a+\cos \theta \right) }\Theta _{k} &=&-b_{\frac{n-1}{2}}\left( a+\cos \theta
\right) \lambda _{k}\Theta _{k}.
\end{eqnarray*}%
The eigenvalues to the first equation are $m_{k}^{2}=k^{2}$ with $%
k=0,1,2,\ldots $ and the associated eigenvectors are $\Phi _{k}=\left\{
1,\sin k\phi ,\cos k\phi \right\} $. The second eigenvalue problem can be
written in the Sturm--Liouville form and then numerically solved on a fine
uniform grid with $N_{\theta }$ points $\{\theta _{j}=\frac{2\pi j}{%
N_{\theta }}\}_{j=0}^{N_{\theta }-1}$ \cite{prycel1993SL}. The eigenvalues $%
\lambda _{k}$ associated with the eigenfunctions $\psi _{k}$ obtained above
are referred to as the true semi-analytic solutions to the eigenvalue
problem (\ref{eqn:psigento}).

\begin{figure*}[tbp]
{\scriptsize \centering
\begin{tabular}{c}
\includegraphics[height=1.15
in]{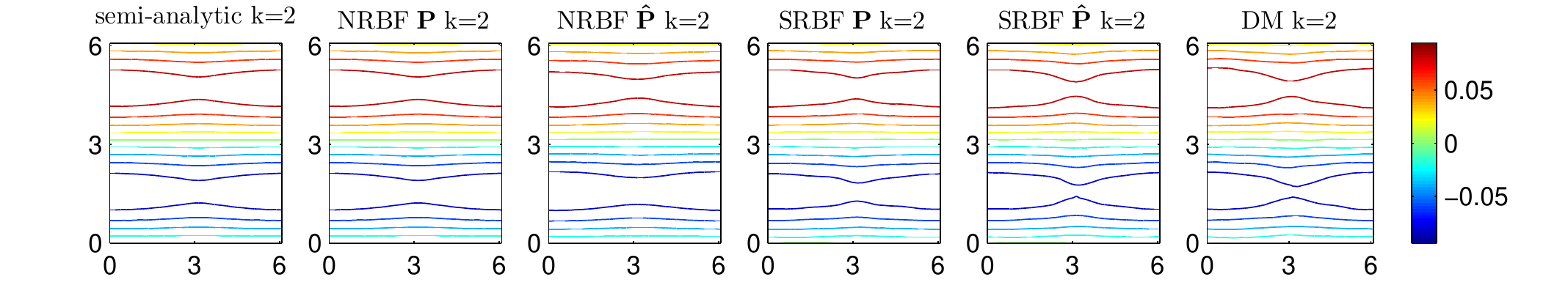}
\\
\includegraphics[height=1.15
in]{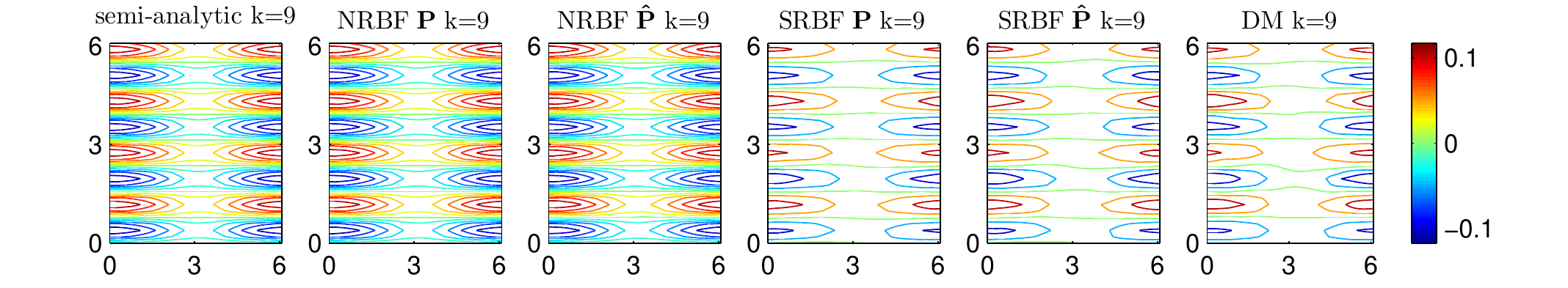}
\\
\includegraphics[height=1.15
in]{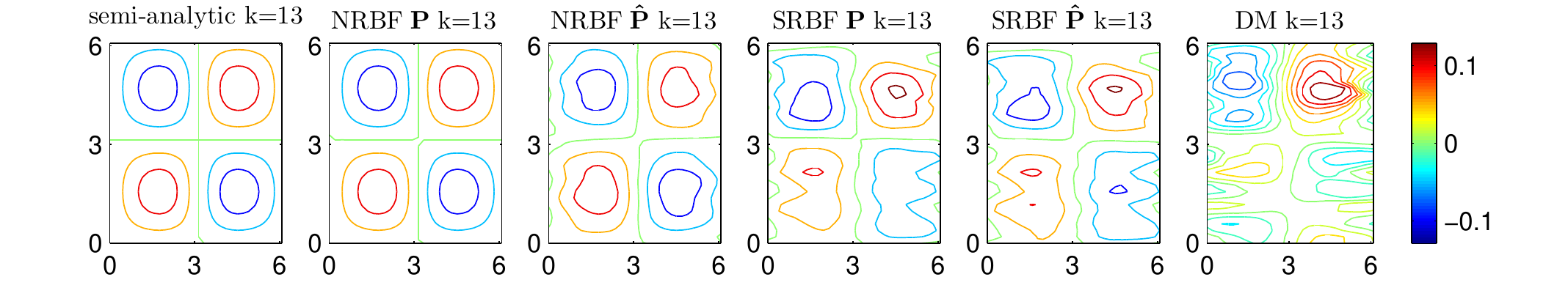}%
\end{tabular}
}
\caption{{\bf 2D general torus in $\mathbb{R}^{21}$.} Comparison of
eigenfunctions of Laplace -Beltrami for $k=2, 9, 13$ among NRBF using $%
\mathbf{P}$ and $\mathbf{\hat{P}}$, SRBF using $\mathbf{P}$ and $\mathbf{%
\hat{P}}$, and DM using $K=100$. For NRBF, IQ kernel with $s=0.5$ is used, and for SRBF,
IQ kernel with $s=0.1$ is used. For DM with $K=100$, the auto-tuned algorithm for the bandwidth
$\epsilon$ is used. The horizontal and vertical axes correspond
to $\protect\theta$ and $\protect\phi$, respectively. $N=2500$ randomly distributed data points on the manifold are used for computing the
eigenvalue problem. The eigenvectors are then generalized onto the $%
32\times 32$ well-sampled grid points $\{\protect\theta_i,\protect\phi_j\}=\{%
\frac{2\protect\pi i}{32},\frac{2\protect\pi j}{32}\}_{i,j=0}^{31}$ for  plotting.}
\label{fig_gentor_1}
\end{figure*}

In our numerical implementation, data points are randomly sampled from the general torus with uniform distribution in intrinsic coordinates. {For this example, we also show results based on the diffusion maps algorithm with a Gaussian kernel, $K_\epsilon(x,y) = \exp(-\frac{\|x-y\|^2}{4\epsilon})$, where $\epsilon$ denotes the bandwidth parameter to be specified. The sampling density is estimated using the KDE estimator proposed by \cite{lq:65}. For an efficient implementation, we use the $K$-nearest neighbors algorithm to avoid computing the graph affinity between pair of points that are sufficiently far away. In Appendix~\ref{app:E}, additional numerical results for other choices of $K$ (including $K=N$ or no KNN being used) are reported. It is worth noting that the additional results in Appendix~\ref{app:E} suggest that the accuracy of the estimation of eigenvectors does not improve for any values of $K$ that we tested.
To further check the numerical optimality of the choice of bandwidth parameter $\epsilon$, we also empirically check whether an improved estimate can be attained by varying $\epsilon$ around the auto-tuned value. We found that the auto-tuned method is more effective when $K$ is relatively small, which motivates the use of small $K$ in verifying the numerical convergence. Figure~\ref{fig_gentor_5} shows the sensitivity of the estimates as $\epsilon$ is varied for fixed $K=100$ and $N=2500$. For the convergence result, we choose $K \sim \sqrt{N}$ following the theoretical guideline in \cite{calder2019improved} that guarantees a convergence rate of $\epsilon = O\left( (\frac{K}{N})^{2/d}\right)$. Once $K$ is fixed, the parameter $\epsilon$ is selected based on the auto-tuned algorithm introduced in \cite{coifman2008TuningEpsilon}. Specifically, we set $K=60,100,144,200$ for $N = 1024,2500,5000,10000$, respectively.}

To apply NRBF, we use IQ kernel with $s=0.5$. To apply SRBF, we use IQ kernel with $s=0.1$. Figure \ref{fig_gentor_1} shows the comparison of eigenfunctions for modes $k=2,9,13$ among the semi-analytic
truth, NRBF with $\mathbf{P}$ and $\mathbf{\hat{P}}$, SRBF with $\mathbf{P}$ and $\mathbf{\hat{P}}$, and DM. One can see from the first row of Fig. \ref{fig_gentor_1} that when $k=2$ is very small, all the methods can provide excellent approximations of eigenfunctions. For larger $k$, such as $9$ and $13$, NRBF methods with $\mathbf{P}$ or $\mathbf{\hat{P}}$ provide more accurate approximations compared to SRBF with $\mathbf{P}$ and $\mathbf{\hat{P}}$ and DM. In fact, the eigenvectors obtained from NRBF with $\mathbf{P}$ are accurate and very smooth as seen from the second column of Fig. \ref{fig_gentor_1}. { On the other hand, SRBF with $\mathbf{\hat{P}}$ does not produce eigenvectors that are qualitatively much better than those of DM. }

\begin{figure*}[tbp]
{\scriptsize \centering
\begin{tabular}{ccc}
{\small (a) Eigenvalues} & {\small (b) Error of Eigenvalues} &
{\small (c) Error of Eigenvectors } \\
\includegraphics[width=1.9
in, height=1.4 in]{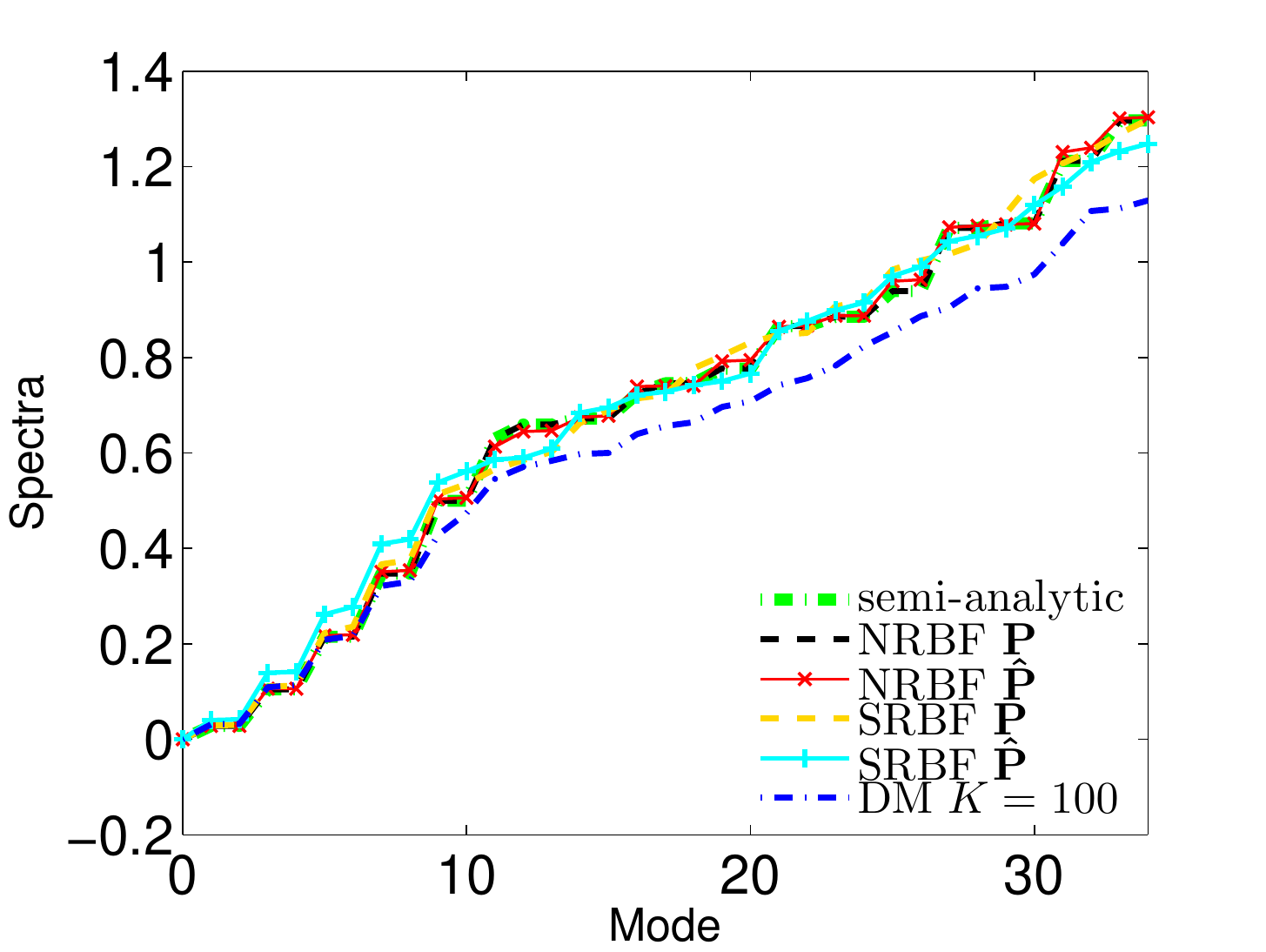}
&
\includegraphics[width=1.9
in, height=1.4 in]{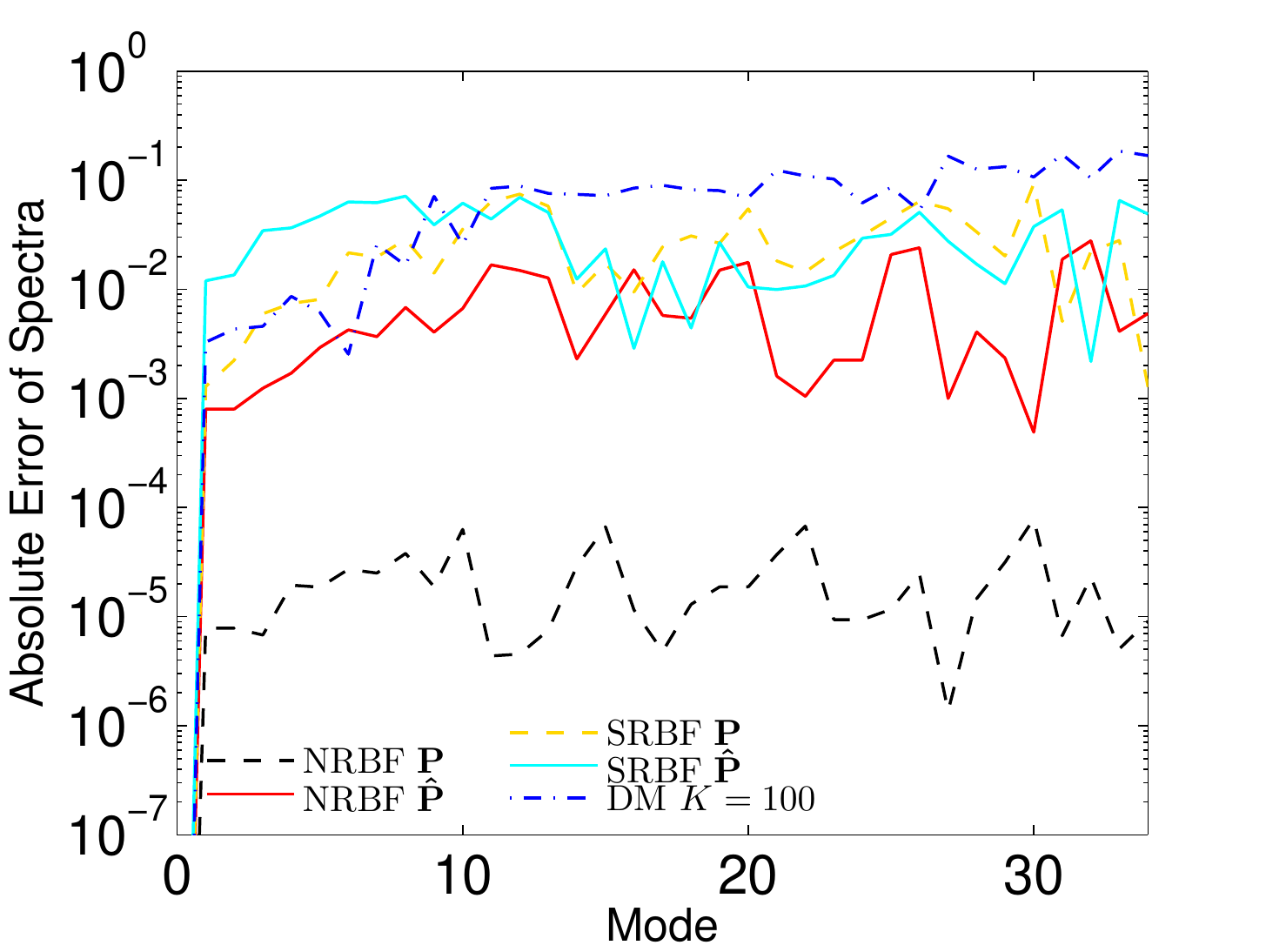}
&
\includegraphics[width=1.9
in, height=1.4 in]{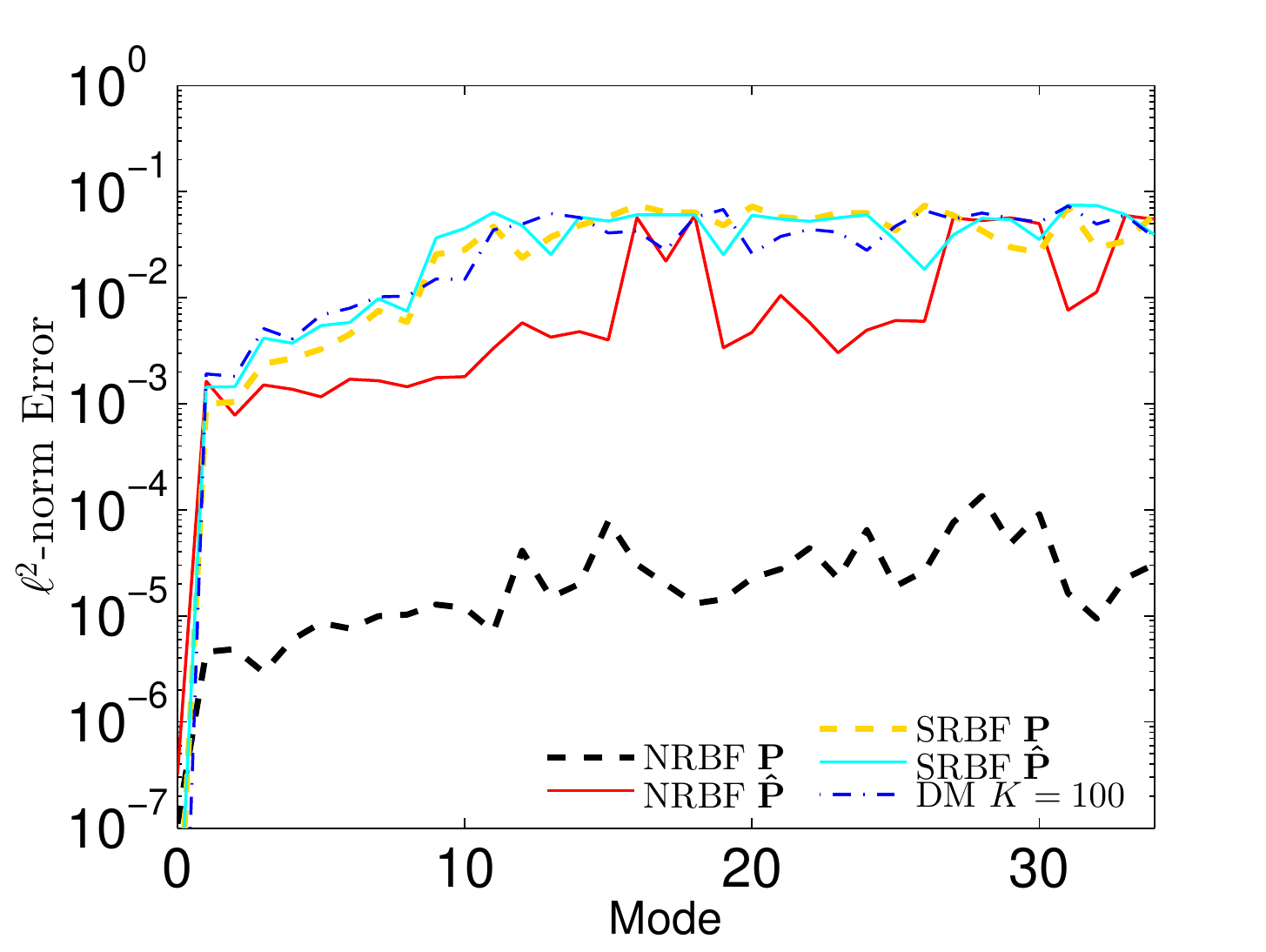}%
\end{tabular}
}
\caption{{\bf 2D general torus in $\mathbb{R}^{21}$.} Comparison of
(a) eigenvalues, (b) error of eigenvalues, and (c) error of eigenvectors,
among NRBF, SRBF, and DM. For NRBF, IQ kernel with $s=0.5$ is used, and for
SRBF, IQ kernel with $s=0.1$ is used. For SRBF, the KDE estimated $\tilde{q}$ is used. Randomly distributed $N=2500$ data
points on the manifold are used for solving the eigenvalue problem. }
\label{fig_gentor_2}
\end{figure*}

Figure \ref{fig_gentor_2} further quantifies the errors of eigenvalues and
eigenfunctions for all the methods. One can see that NRBF with $\mathbf{P}$ performs much better than all  other methods on this 2D manifold example. When the manifold is unknown, one can diagnose the manifold learning capabilities of the symmetric and non-symmetric RBF using $\mathbf{\hat{P}}$ compared to DM. One can see that NRBF with $\mathbf{\hat{P}}$\ (red curve) performs better than the other two methods. One can also see that DM (blue curve) performs slightly better than SRBF with $\mathbf{\hat{P}}$ (cyan curve) in estimating the leading eigenvalues but somewhat worst in estimating eigenvalues corresponding to the higher modes (Fig. \ref{fig_gentor_2}(b)). In terms of the estimation of eigenvectors, they are comparable  (blue and cyan curves Fig. \ref{fig_gentor_2}(c)). Additionally, SRBF with $\mathbf{P}$ (yellow curve) and SRBF with $\mathbf{\hat{P}}$ (cyan curve) produce comparable accuracies in terms of the eigenvector estimation. This result, where no advantage is observed using the analytic $\mathbf{P}$ over the approximated $\mathbf{\hat{P}}$, is consistent with the theory which suggests that the error bound is dominated by the Monte-Carlo error rate provided smooth enough kernels are used.

\begin{figure*}[htbp]
{\scriptsize \centering
\begin{tabular}{cc}
{\small (a) Error of Spectra vs. parameter} & {\small (b) Error of Eigenfunctions vs. parameter} \\
\includegraphics[width=3.0
in, height=2.2 in]{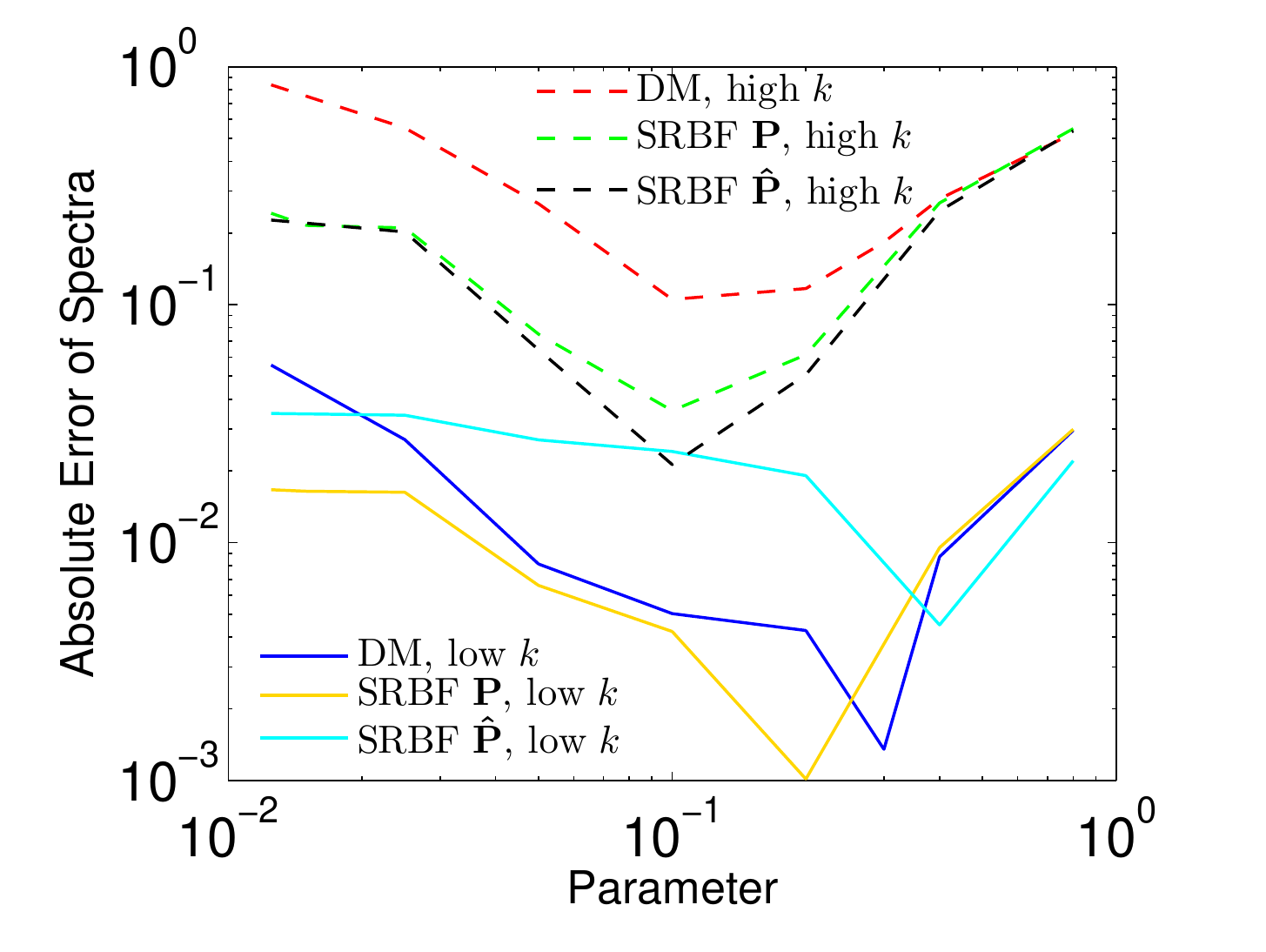}
&
\includegraphics[width=3.0
in, height=2.2 in]{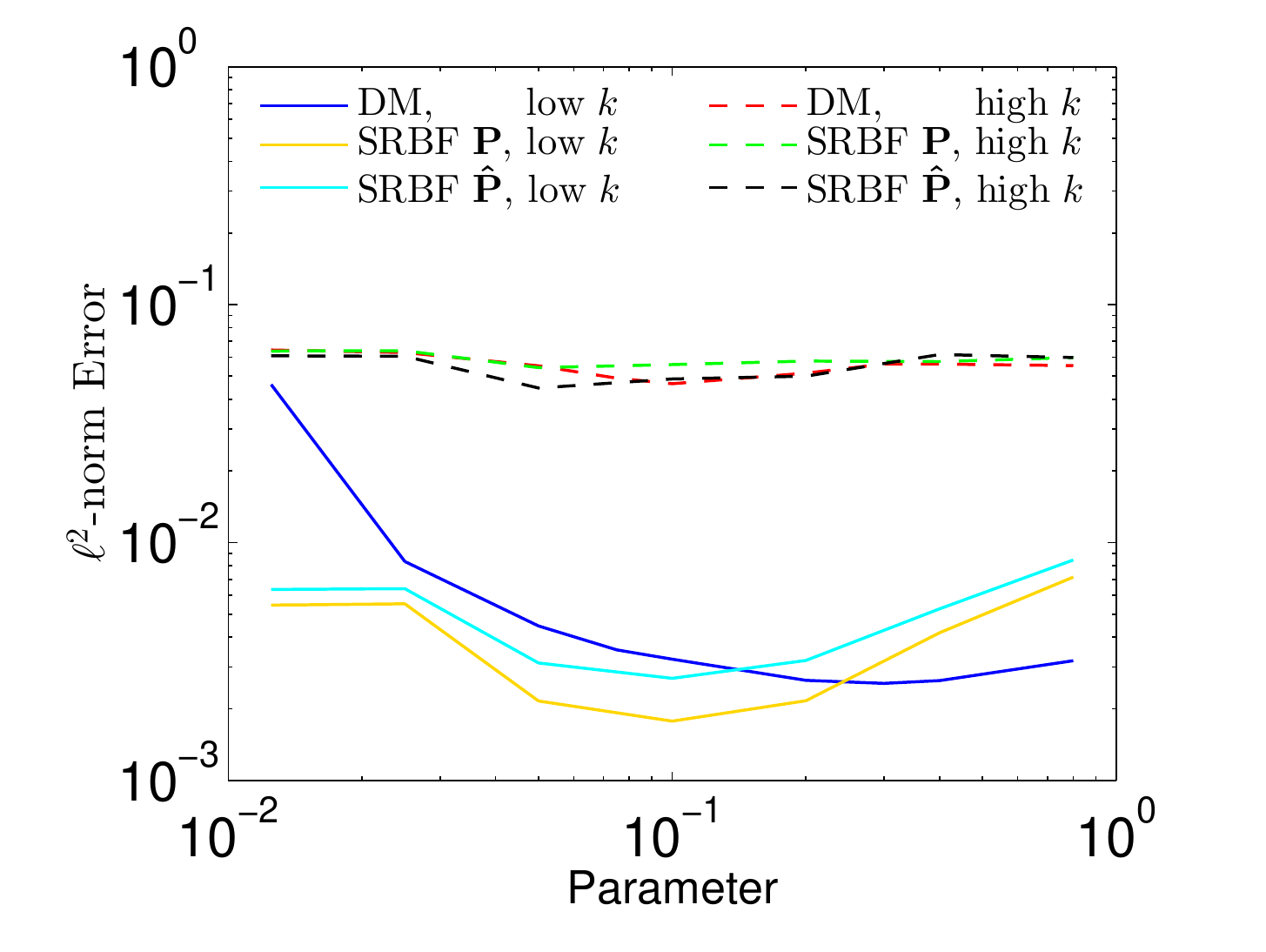}
\end{tabular}
}
\caption{{\bf 2D general torus in $\mathbb{R}^{21}$.}
Sensitivity of (a) eigenvalues and (b) eigenfunctions with respect to the parameter (bandwidth for DM and shape parameter for SRBF using true $\mathbf{P}$ and SRBF using estimated $\mathbf{\hat{P}}$).
The blue, yellow and cyan curves denote the average error of eigenvalues or eigenvectors over modes 2-5 (low $k$), while the red, green, and black curves denote the average error of eigenvalues or eigenvectors over modes 21-30 (high $k$).
For SRBF, IQ kernel is used and shape parameter $s$ ranges from $10^{-2}$ to $10^0$.
For DM, $K=100$ nearest neighbors are used and bandwidth  $\epsilon$ ranges from $10^{-2}$ to $10^0$.
In this experiment, we fixed the $N=2500$ data points which are randomly distributed on the general torus with uniform distribution in the intrinsic coordinates.  }
\label{fig_gentor_5}
\end{figure*}
In the previous two figures, we showed the SRBF estimates corresponding to a specific choice of $s=0.1$. Now let us check the robustness of the method with respect to other choices of the shape parameter. In Fig.~\ref{fig_gentor_5}, we show the errors in the estimation of eigenvalues and eigenvectors. Specifically, we report the average errors of low modes (between modes 2-5) for DM (blue), SRBF with $\mathbf{P}$ (yellow), and SRBF with $\mathbf{\hat{P}}$ (cyan). For these low modes, notice that {with the optimal shape parameter, $s=0.4$, the eigenvalue estimates from SRBF with $\mathbf{\hat{P}}$ are slightly less accurate than the diffusion maps.} For this shape parameter value, the SRBF with $\mathbf{\hat{P}}$ produces an even more accurate estimation of the leading eigenvalues. However, the accuracy of the SRBF eigenvectors decreases slightly under this parameter value. We also report the average errors of high modes (between modes 21-30) for DM (red), SRBF with $\mathbf{P}$ (green), and SRBF with $\mathbf{\hat{P}}$ (black). For these high modes, both SRBFs are uniformly more accurate than DM in the estimation of eigenvalues, but the accuracies of the estimation of eigenvectors are comparable.  {More comparisons can be found in Appendix \ref{app:E}. Overall, SRBF and DM show comparable results for the eigenvalue problem. For both SRBF and DM, the eigenvalue estimates are sensitive to the choice of the parameter while the eigenvector estimates are not.  Based on numerical experiments with a wide range of parameters, we empirically found that DM has a slight advantage in the estimation of leading spectrum while the SRBF has a slight advantage in the estimation of non-leading spectrum.}

\begin{figure*}[tbp]
{\scriptsize \centering
\begin{tabular}{cc}
{\small (a) Conv. of NRBF wrt $N$} & {\small (b) Conv. of NRBF wrt $\mathbf{\hat{P}}$} \\
\includegraphics[width=3.0
in, height=2.2 in]{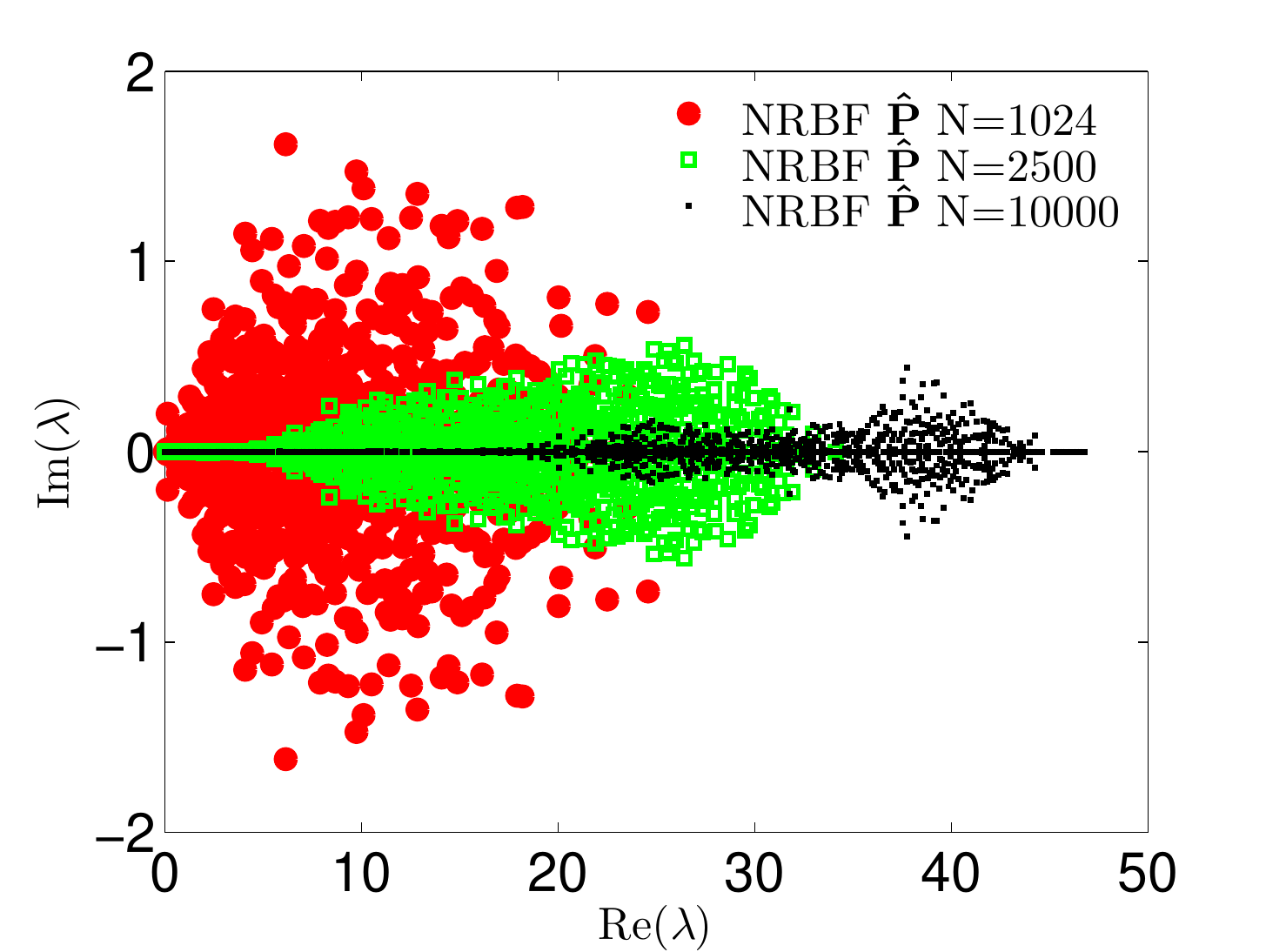}
&
\includegraphics[width=3.0
in, height=2.2 in]{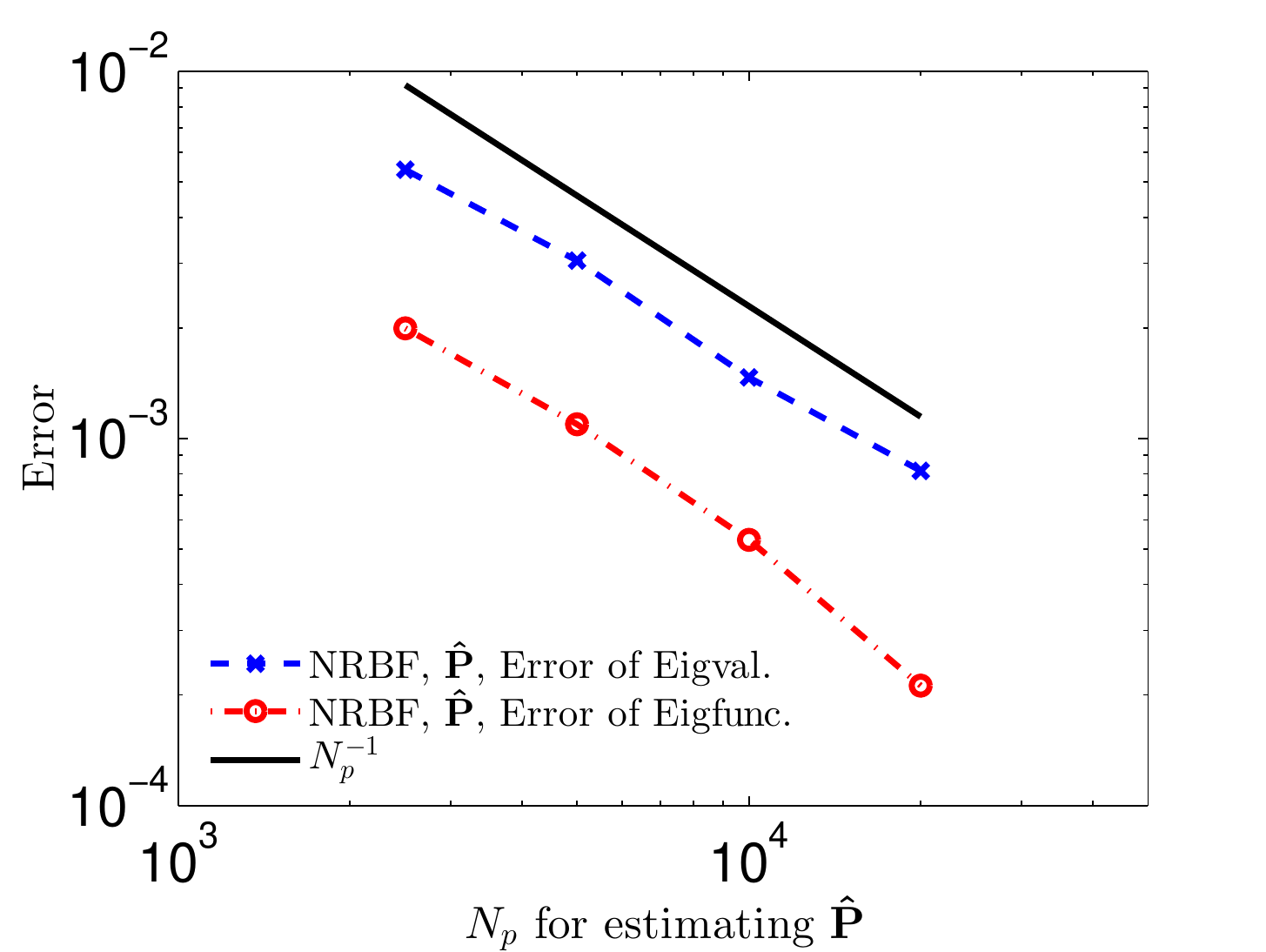}\\
{\small (c) Conv. of Spectra wrt $\mathbf{\hat{P}}$} & {\small (d) Conv. of Eigenfunc. wrt $\mathbf{\hat{P}}$} \\
\includegraphics[width=3.0
in, height=2.2 in]{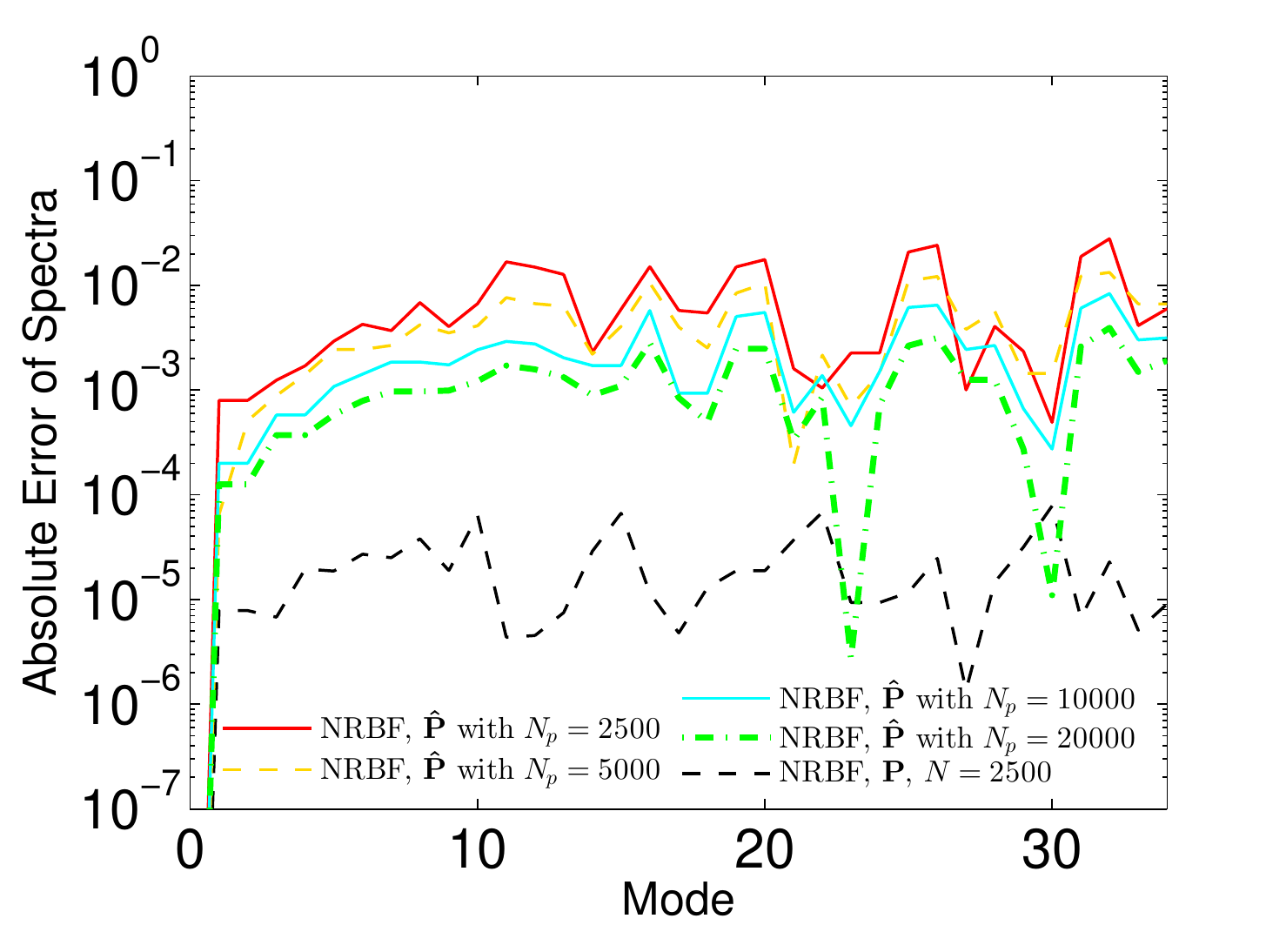}
&
\includegraphics[width=3.0
in, height=2.2 in]{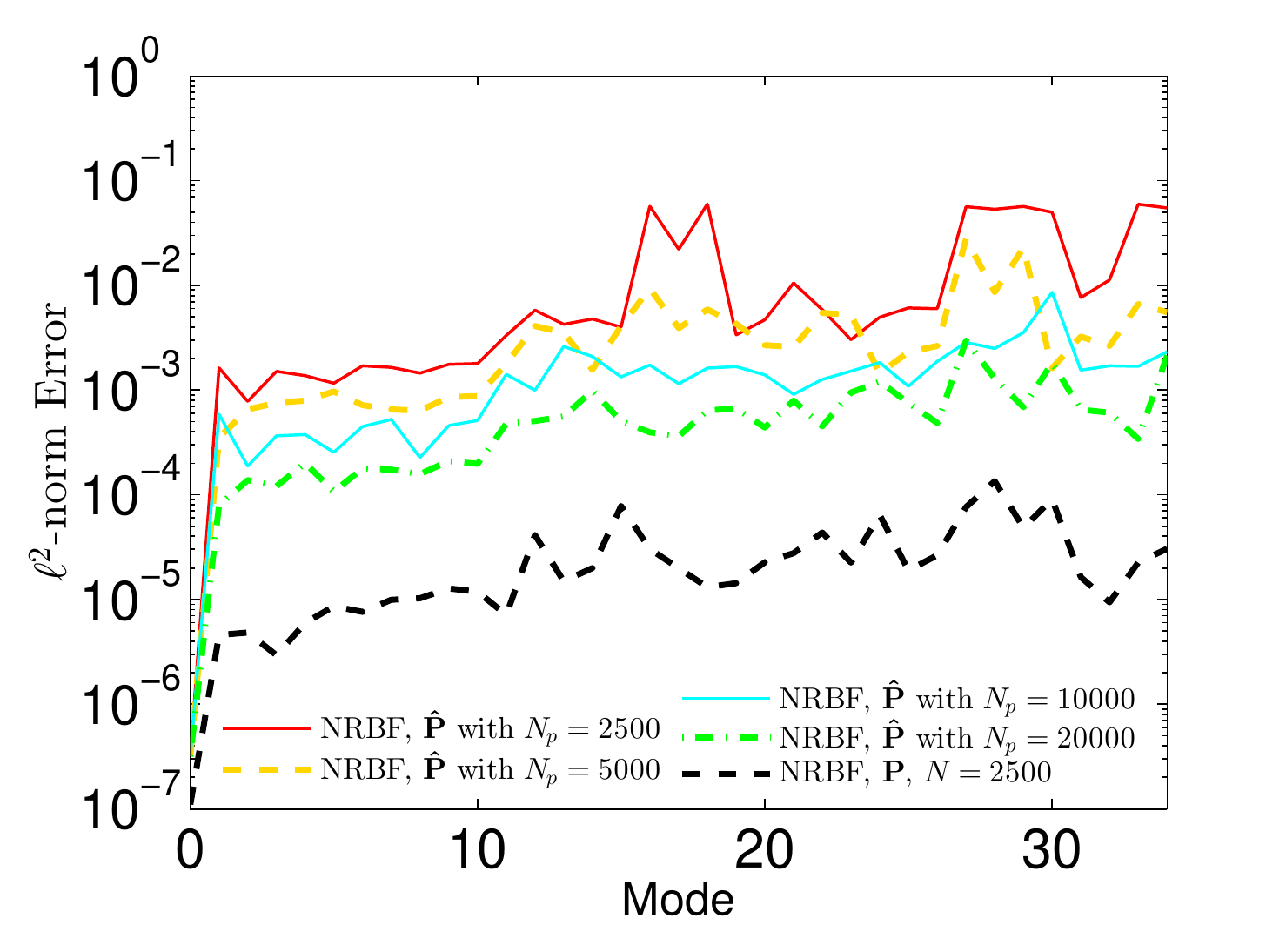}
\end{tabular}
}
\caption{{\bf 2D general torus in $\mathbb{R}^{21}$.}
Convergence of eigenvalues and eigenvectors for the NRBF method. Here $N$ denotes the number of points for solving the eigenvalue problem, while $N_p$ denotes the number of points for estimating $\mathbf{\hat{P}}$.
(a) Convergence of eigenvalues with respect to $N$ for NRBF using $\mathbf{\hat{P}}$. In panel (a), $N$ data points are used for both approximating $\mathbf{\hat{P}}$ and evaluating NRBF matrices.
In panels (c) and (d), shown is the convergence of NRBF
eigenvalues and NRBF eigenfunctions with respect to $\mathbf{\hat{P}}$ for varying values of $N_p$, but with fixed $N=2500$ data points used for solving the NRBF eigenvalue problem.
(b) For each $N_p$, plotted
are the averages of errors of eigenvalues or eigenfunctions for the leading
12 modes (2nd-13rd modes). The convergence rate of eigenvalues and
eigenfunctions are both $N_p^{-1}$.
IQ kernel with
$s=0.5$ was fixed for all cases. The data points
are randomly distributed on the general torus according to a uniform distribution in the intrinsic coordinates.
}
\label{fig_gentor_4}
\end{figure*}

In Fig. \ref{fig_gentor_4}, we examine the convergence of eigenvalues and
eigenvectors for NRBF with $\mathbf{\hat{P}}$ in the case of
unknown manifold. For NRBF with $\mathbf{\hat{P}}$, since the eigenvalues
are complex value, Fig. \ref{fig_gentor_4}(a) displays all the eigenvalues on a complex plane.
Here are four observations from our numerical results:
\begin{enumerate}
\item  When $N$ increases, more eigenvalues with
large magnitudes flow to the ``tail'' as a cluster packet.
\item  The magnitude of imaginary parts decays as $N$ increases.
\item  For the leading modes with small magnitudes, NRBF eigenvalues
converge fast to the real axis and converge to the true spectra at the same
time. 
\item It appears that all of the eigenvalues lie in the right half plane with positive real parts for this 2D manifold as long as $N$ is large enough ($N>1000$). Notice that this result is consistent with the previous result reported in \cite{fuselier2013high}. In that paper, the authors considered the negative definite Laplace-Beltrami operator and numerically observed that all eigenvalues are in the left half plane with negative real parts for many complicated 2D manifolds for large enough data. 
\end{enumerate}
In Fig. \ref{fig_gentor_4}(b)-(d), we would like to verify that the convergence rate of the NRBF is dominated by the error rate in the estimation of $\mathbf{P}$. In all numerical experiments in these panels, we solve eigenvalue problems of discrete approximation with a fixed 2500 data points as in previous examples. Here, we verify the error rate in terms of the number of points used to construct $\hat{\mathbf{P}}$, which we denote as $N_p$. For the 2D manifolds, we found that the convergence rate (panel (b)) for the leading 12 modes decay with the rate $N_p^{-1}$, which is consistent with the theoretical rate deduced in Theorem~\ref{main P Theorem} and the discussion in Remark~\ref{rem5}. This rate is faster than the Monte-Carlo rate even for randomly distributed data. In panels (c)-(d), we report the detailed errors in the eigenvalue and eigenvector estimation for each mode. This result suggests that if $N$ is large enough (as we point out in bullet point 4 above), one can attain accurate estimation by improving the accuracy of the estimation of $\hat{\mathbf{P}}$ by increasing the sample size, $N_p$.



\begin{figure*}[tbp]
{\scriptsize \centering
\begin{tabular}{cc}
{\small (a) Conv. of Eigenvalues} & {\small (b) Conv. of Eigenfunctions} \\
\includegraphics[width=3.0
in, height=2.2 in]{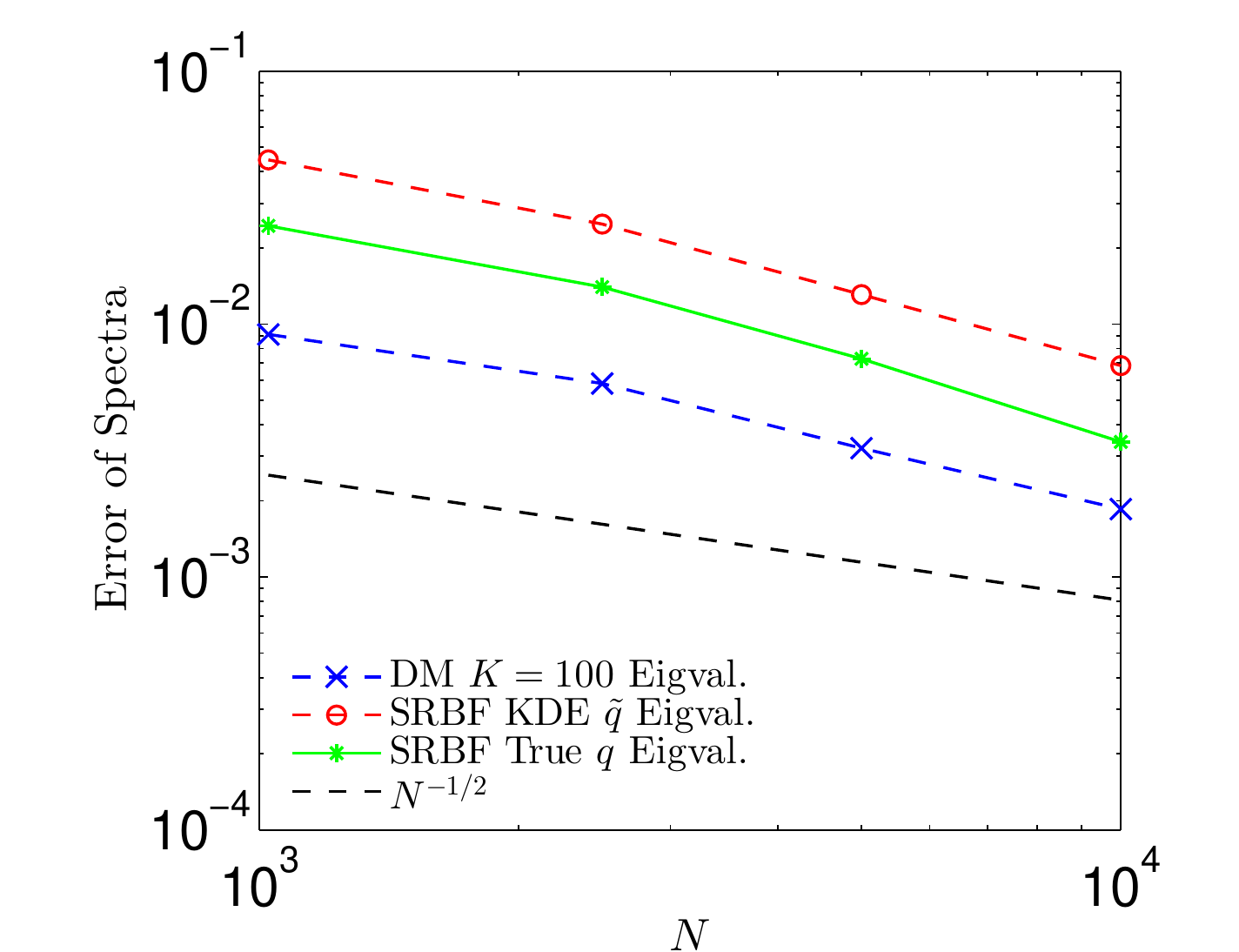}
&
\includegraphics[width=3.0
in, height=2.2 in]{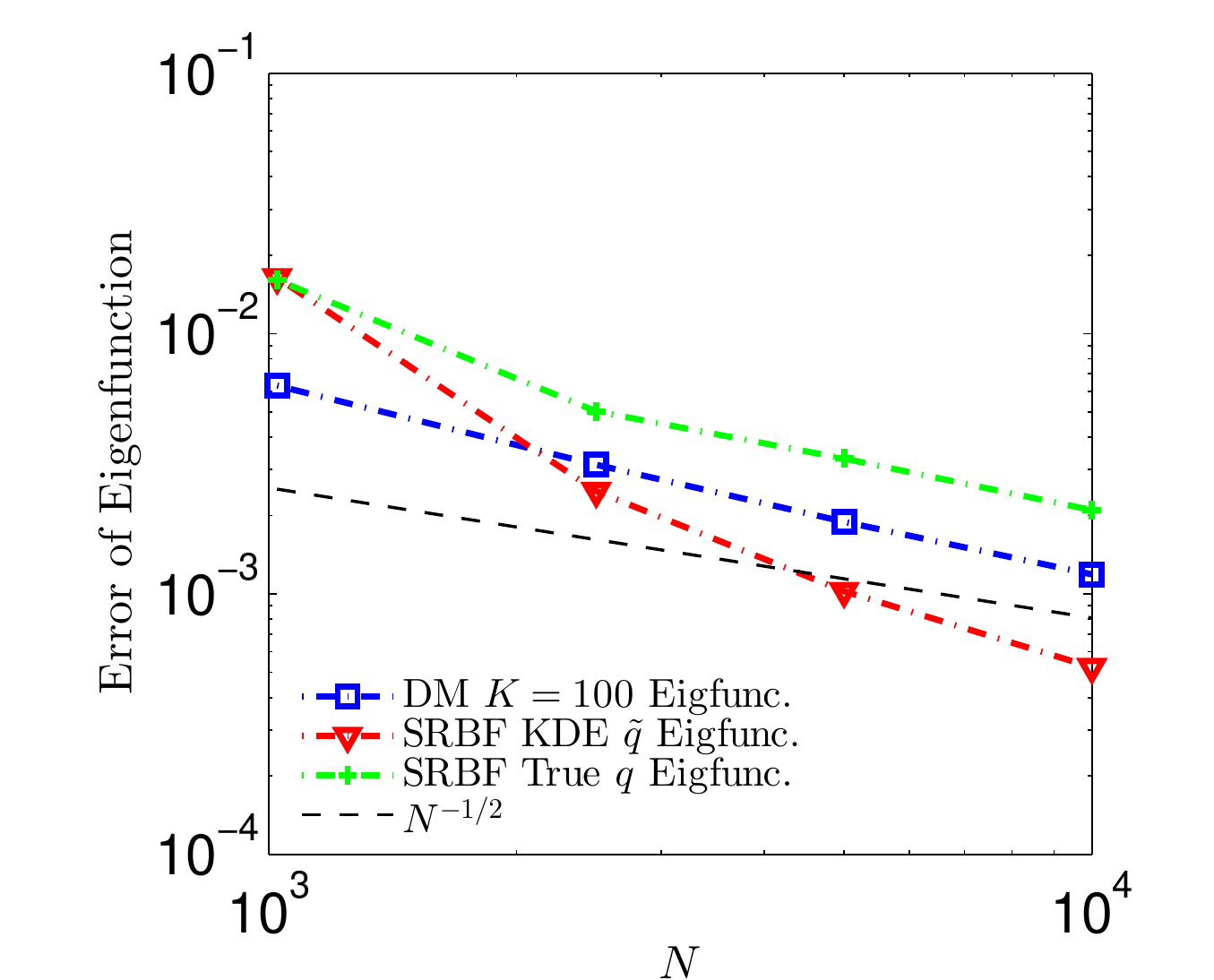}
\\
&
\end{tabular}
\newline
\begin{tabular}{ccc}
{\small (c1) DM Eigenvalues} & {\small (d1) SRBF $\mathbf{\hat{P}}$, KDE $%
\tilde{q}$, Eigenvalues} & {\small (e1) SRBF $\mathbf{\hat{P}}$, True ${q}$,
Eigenvalues } \\
\includegraphics[width=1.92
in, height=1.5 in]{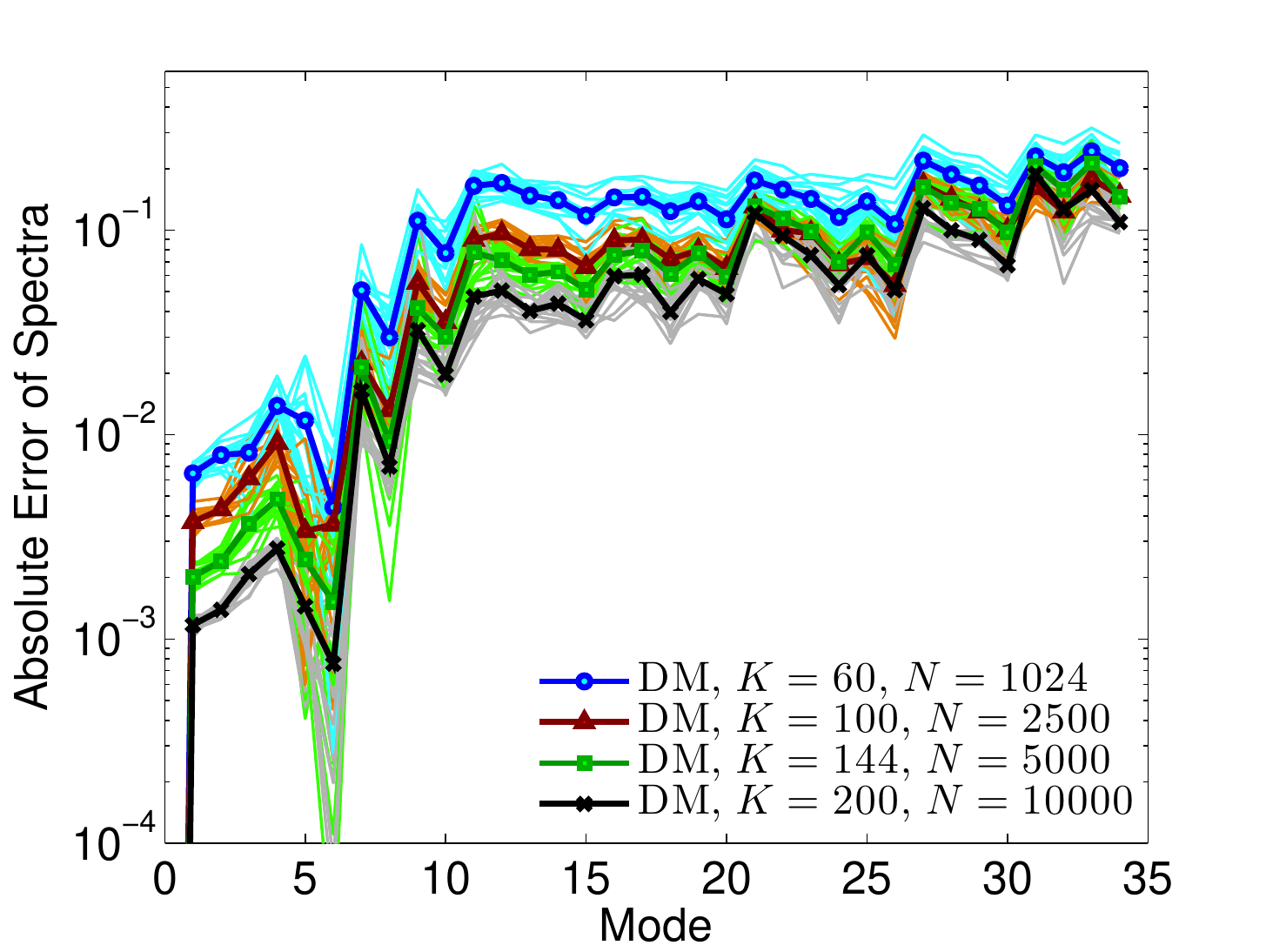}
&
\includegraphics[width=1.92
in, height=1.5 in]{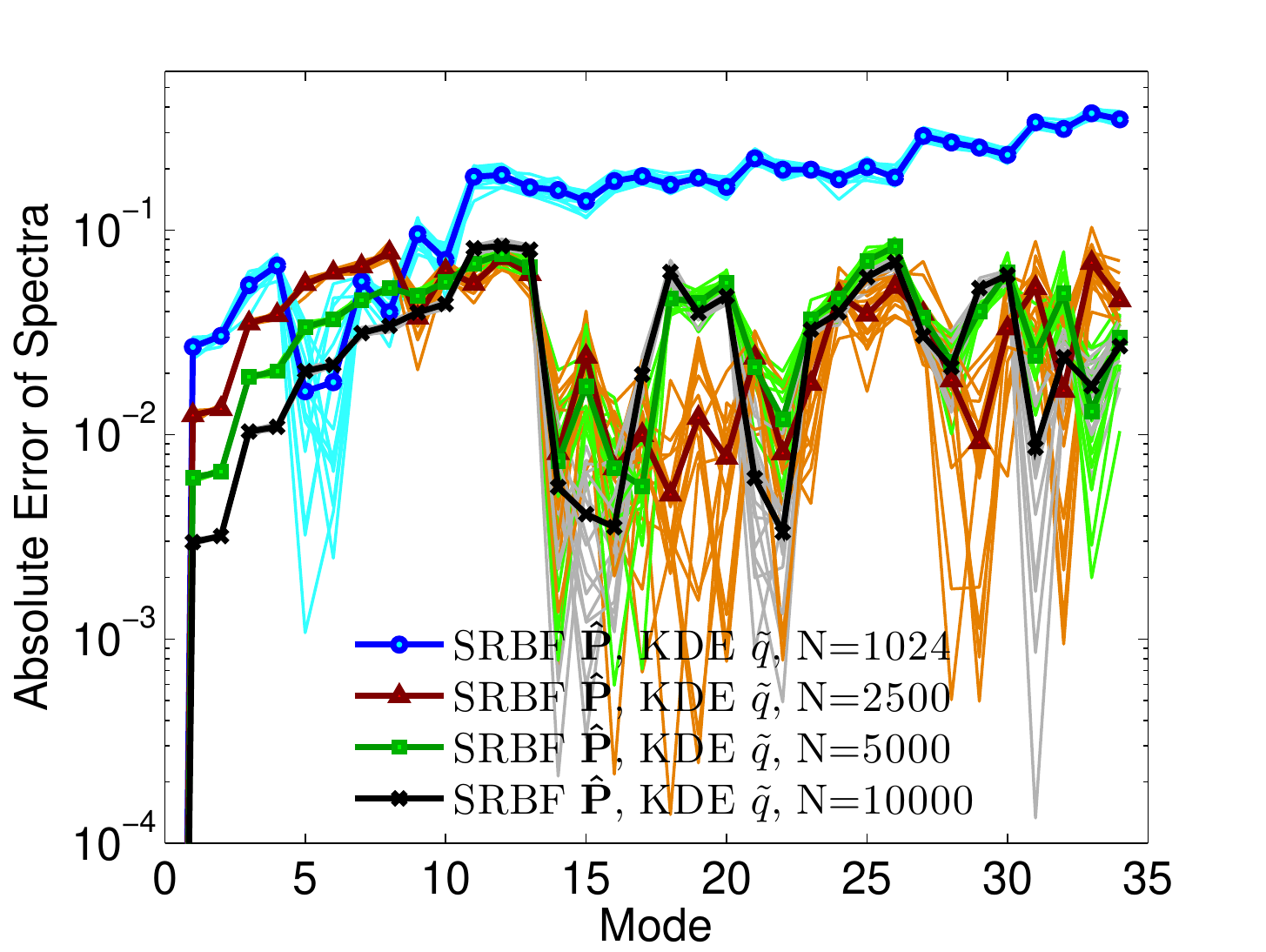}
&
\includegraphics[width=1.92
in, height=1.5 in]{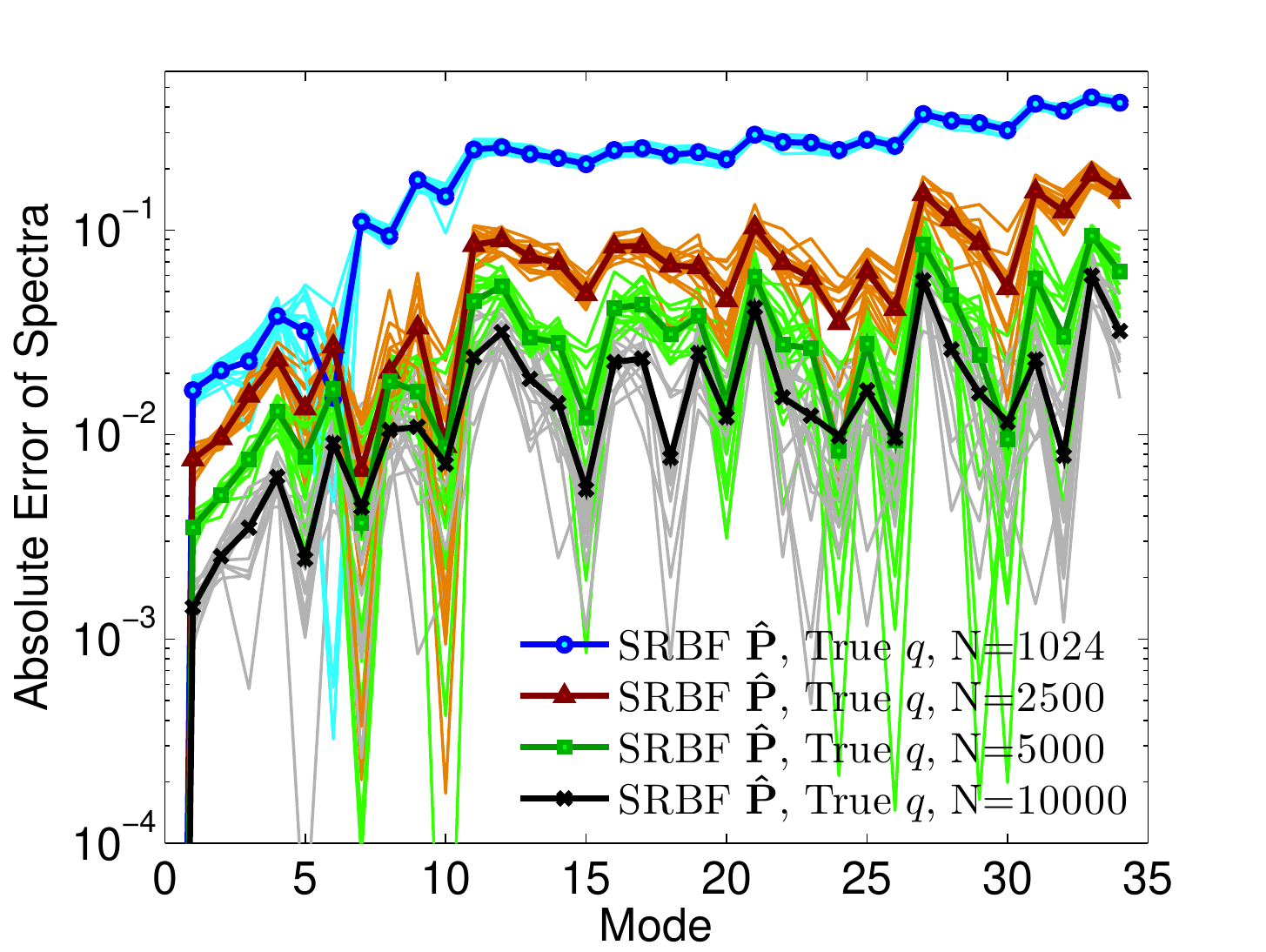}
\\
{\small (c2) DM Eigenfuncs.} & {\small (d2) SRBF $\mathbf{\hat{P}}$, KDE $%
\tilde{q}$, Eigenfuncs.} & {\small (e2) SRBF $\mathbf{\hat{P}}$, True ${q}$,
Eigenfuncs.} \\
\includegraphics[width=2.0
in, height=1.5 in]{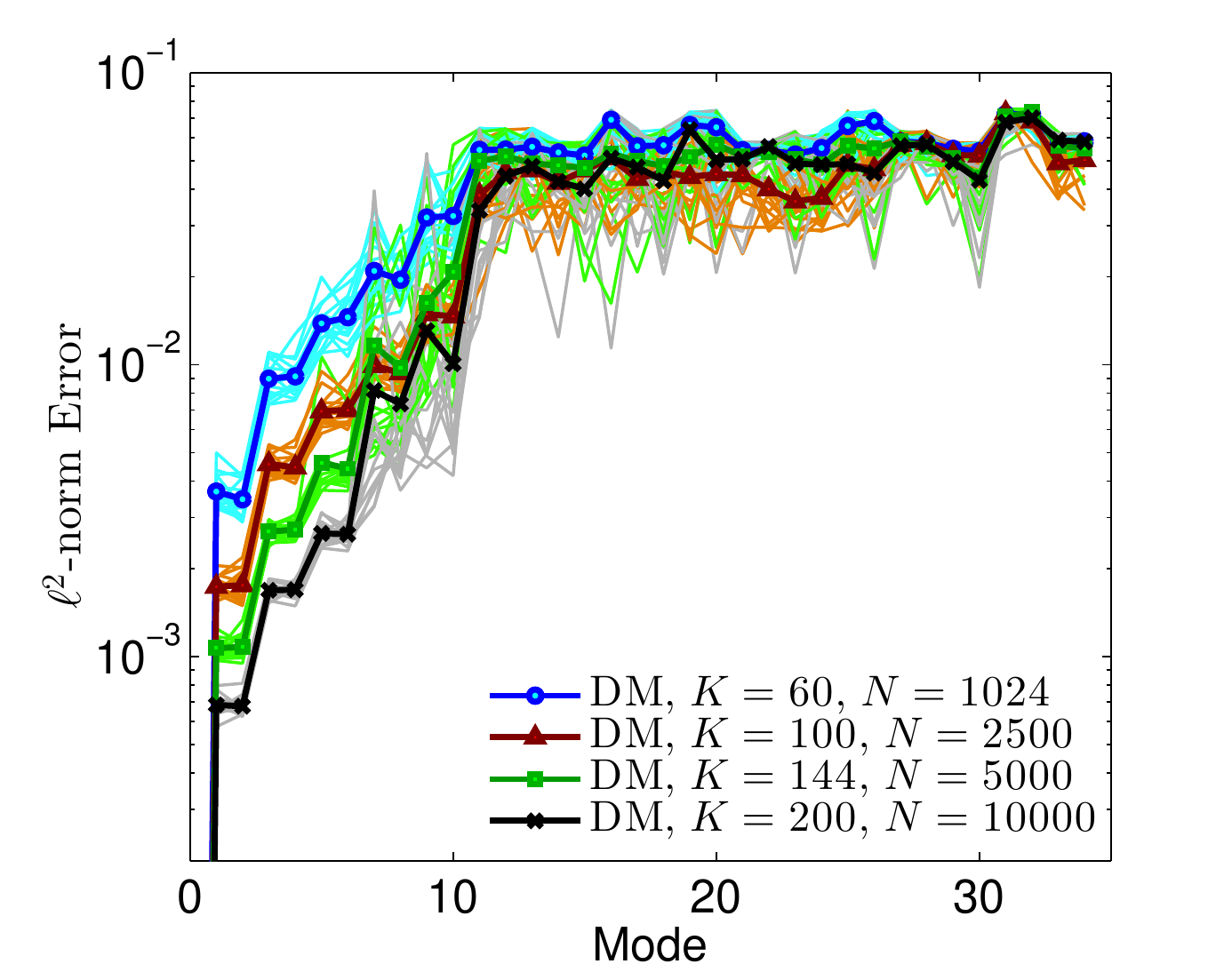}
&
\includegraphics[width=2.0
in, height=1.5 in]{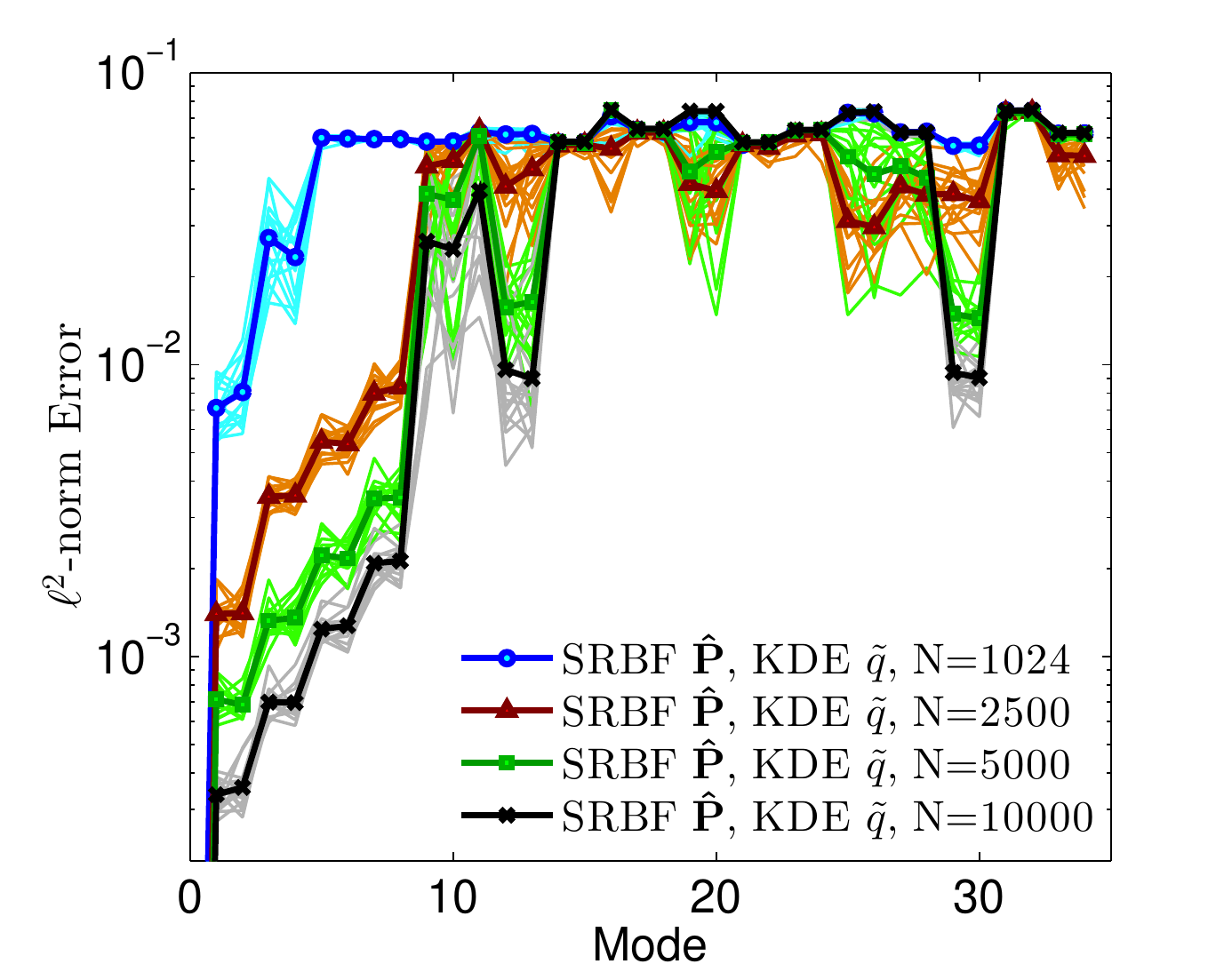}
&
\includegraphics[width=2.0
in, height=1.5 in]{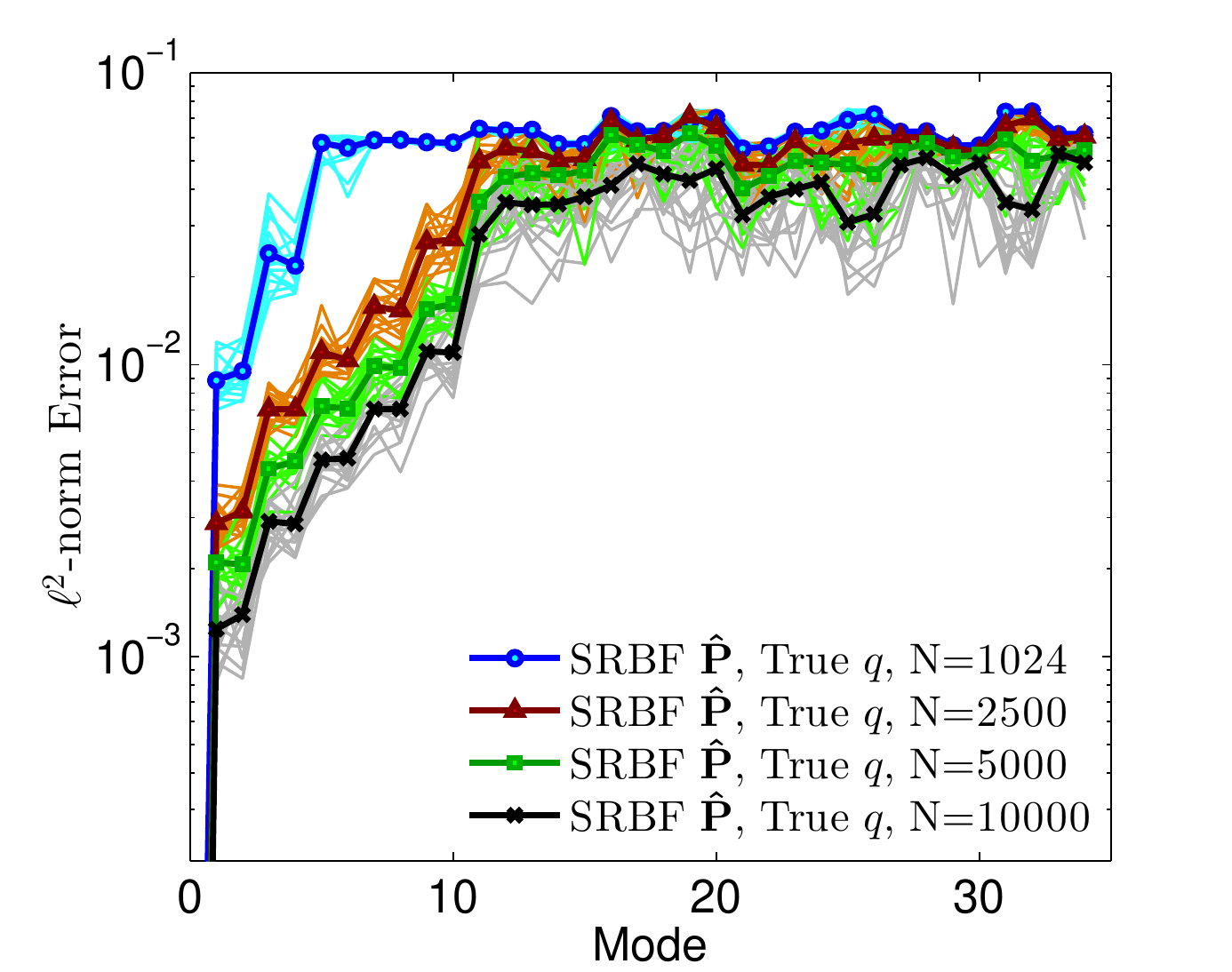}%
\end{tabular}
}
\caption{{\bf 2D general torus in $\mathbb{R}^{21}$.}
Convergence of (a) eigenvalues and
(b) eigenfunctions for DM and SRBF using $\mathbf{\hat{P}}$. For each $N$, the average
error of eigenvalues or eigenvectors over the leading four modes (2nd-5th modes) are plotted.
For SRBF, IQ kernel with $s=0.1$ was fixed for each $N$. Comparison of errors of eigenvalues for (c1) DM, (d1) SRBF
using estimated sampling density $\tilde{q}$, and (e1) SRBF using the true
sampling density $q$ are shown. Plotted in (c2)-(e2) are the corresponding comparison
of errors of eigenvectors. For each $N$, $16$ independent trials are run and
depicted by light color. For each $N$, the average of
all $16$ trials are depicted by dark color. For each trial, randomly distributed data points on the manifold are used for computation.}
\label{fig_gentor_3}
\end{figure*}

In Fig. \ref{fig_gentor_3}, we examine the convergence of eigenvalues and eigenvectors for DM and SRBF with $\mathbf{\hat{P}}$ in the case of the unknown manifold. Previously in Figs.~\ref{fig_gentor_1}-\ref{fig_gentor_5}, we showed result with $N=2500$ fixed, now we examine the convergence rate as $N$ increases. For DM and SRBF with $\mathbf{\hat{P}}$, the Laplacian matrix is always symmetric positive definite, so that their eigenvalues and eigenvectors must be real-valued and their eigenvalues must be positive. Figures \ref{fig_gentor_3}(c)-(e) display the errors of eigenvalues and eigenvectors for DM and SRBF for $N=1024,2500,5000,10000$. For robustness, we report estimates from 16 experiments, where each estimate corresponds to independent randomly drawn data (see thin lines). The thicker lines in each panel correspond to the average of these experiments. One can observe from Figs.~\ref{fig_gentor_3}(a) and (b) that the errors of the leading four modes for both methods are comparable and decay on the order of $N^{-1/2}$. This rate is consistent with Theorems~\ref{eigvalconv} and \ref{conveigvec} with smooth kernels for the SRBF. On the other hand, this rate is faster than the theoretical convergence rate predicted in \cite{calder2019improved}, $N^{-\frac{1}{d+4}}$. One can also see that SRBF using KDE $\tilde{q}$ (red dash-dotted curve) provides a slightly faster convergence rate for the error of eigenfunction compare to those of SRBF with analytic $q$, which is counterintuitive. In the next example, we will show the opposite result, which is more intuitive. Lastly, more detailed errors of the eigenvalues and eigenvectors estimates for each leading mode are reported in Figs. \ref{fig_gentor_3}(c)-(e), {from which one can  see the convergence for each leading mode for both SRBF and DM. One can also see the slight advantage of SRBF over DM in the estimation of non-leading eigenvalues while the slight advantage of DM over SRBF in the estimation of leading eigenvalues (consistent to the result with a fixed $N=2500$ shown in Figs.~\ref{fig_gentor_2} and \ref{fig_gentor_5}).}

\subsection{4D Flat Torus}\label{4dtorus}

We consider a $d-$dimensional manifold embedded in $\mathbb{R}^{2md}$ with
the following parameterization,
\begin{equation*}
{x}=\frac{1}{\sqrt{1+\dots +m^{2}}}\left(
\begin{array}{ccccc}
\cos (t_{1}), & \sin (t_{1}), & \cdots & \cos (mt_{1}), & \sin (mt_{1}), \\
\vdots & \cdots & \cdots & \cdots & \vdots \\
\cos (t_{d}), & \sin (t_{d}), & \cdots & \cos (mt_{d}), & \sin (mt_{d})%
\end{array}%
\right) ,
\end{equation*}%
with $0\leq t_{1}\leq 2\pi ,\cdots ,0\leq t_{d}\leq 2\pi $. The Riemannian
metric is given by a $d\times d$ identity matrix $\mathbf{I}_{d}$. The
Laplace-Beltrami operator can be computed as $\Delta _{g}u=-\sum_{i=1}^{d}%
\frac{\partial ^{2}u}{\partial t_{i}^{2}}$. For each dimension $t_{i}$, the
eigenvalues of operator $-\frac{\partial ^{2}}{\partial t_{i}^{2}}$\ are $%
\{0,1,1,4,4,\ldots ,k^{2},k^{2},\ldots \}$. The exact spectrum and multiplicities of the Laplace-Beltrami operator on the general flat torus depends on the intrinsic dimension $d$\ of the manifold. In this section, we
study the eigenmodes of Laplace-Beltrami operator for a flat torus with
dimension $d=4$ and ambient space $\mathbb{R}^{2md}=\mathbb{R}^{16}$. In this
case, the spectra of the flat torus can be calculated as

\begin{equation}
\begin{array}{cccccc}
\text{spectra} & 0 & 1 & 2 & 3 & 4 \\
\text{mode }k & 1 & 2\sim 9 & 10\sim 33 & 34\sim 65 & 66\sim 89%
\end{array}%
,  \label{eqn:flatlead}
\end{equation}%
where the eigenvalues $0,1,2,3,4$ have multiplicities of $1,8,24,32,24$,
respectively. Our motivation here is to investigate if RBF methods suffer from
curse of dimensionality when solving eigenvalue problems.


\begin{figure*}[htbp]
{\scriptsize \centering
\begin{tabular}{ccc}
{\normalsize (a) Eigenvalues} & {\normalsize (b) Error of eigenvalues} &
{\normalsize (c) Error of eigenvectors } \\
\includegraphics[width=1.9
in, height=1.4 in]{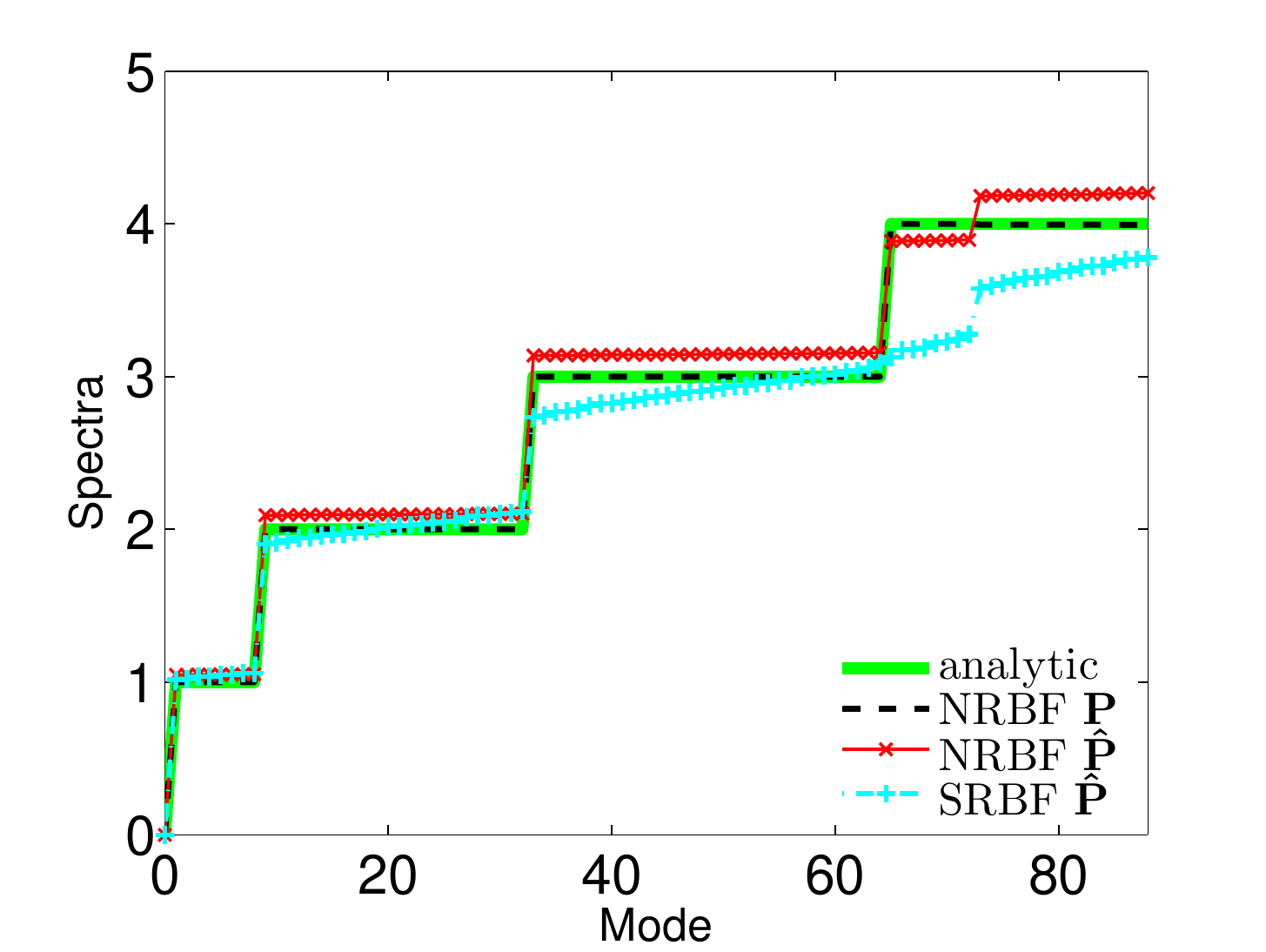}
&
\includegraphics[width=1.9
in, height=1.4 in]{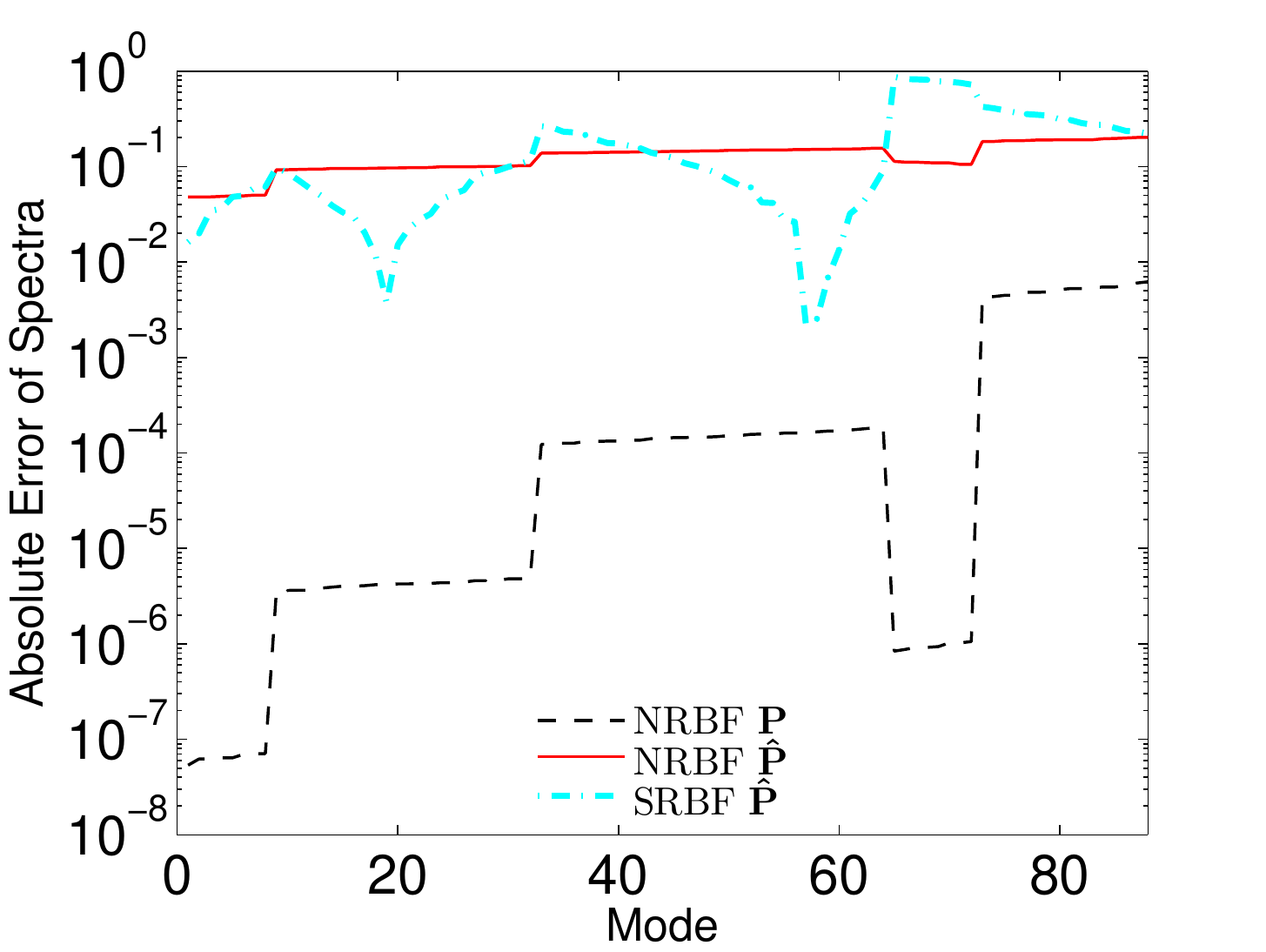}
&
\includegraphics[width=1.9
in, height=1.4 in]{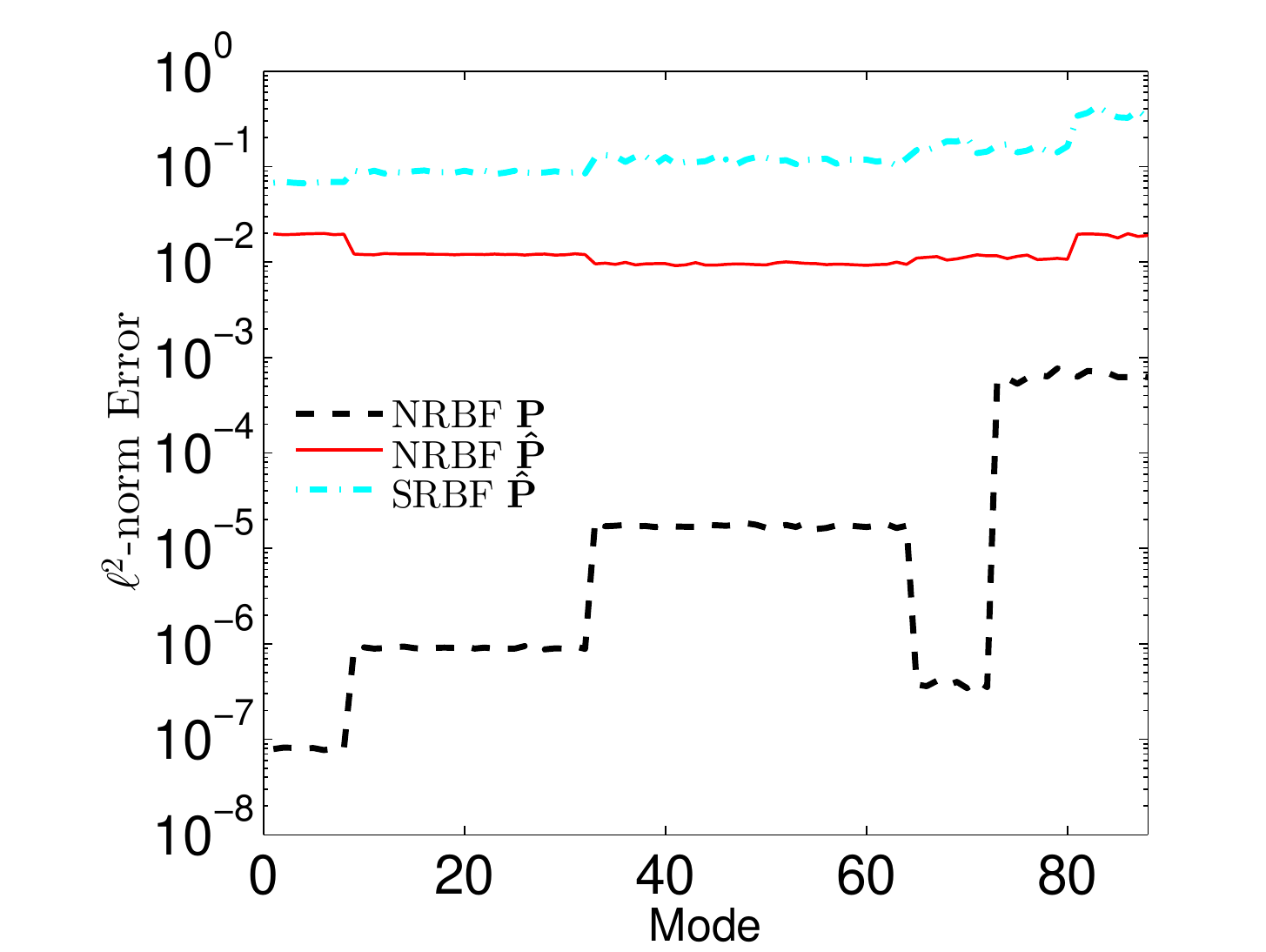}%
\end{tabular}
}
\caption{{\bf 4D flat torus in $\mathbb{R}^{16}$.} Comparison of
(a) eigenvalues, (b) error of eigenvalues, and (c) error of eigenvectors,
among NRBF and SRBF. For NRBF, GA kernel with $s=0.5$ is used, and for
SRBF, IQ kernel with $s=0.3$ is used. The $N=30,000$ data points are
randomly distributed on the flat torus.}
\label{fig_erreigflattori_1}
\end{figure*}

Numerically, data points are randomly distributed on the flat torus with uniform distribution.
To apply NRBF, we use GA kernel with $s=0.5$. To apply SRBF, we use IQ kernel with $s=0.3$. Figure \ref{fig_erreigflattori_1} shows the results of eigenvalues and eigenfunctions for NRBF with $\mathbf{P}$ and $\mathbf{\hat{P}}$, and SRBF with $\mathbf{\hat{P}}$. One can see from Figs. \ref{fig_erreigflattori_1}(b) and (c) that when $N=30000$ is large enough, NRBF with $\mathbf{P}$ (black dashed curve) performs much better than the other methods. One can also see that NRBF with $\mathbf{\hat{P}}$ (red curve) and SRBF with $\mathbf{\hat{P}}$ (cyan curve) are comparable when the manifold is assumed to be unknown.

\begin{figure*}[tbp]
{\scriptsize \centering
\begin{tabular}{cc}
{\normalsize (a) Conv. of NRBF Eigenvals.} & {\small (b) Conv. of NRBF wrt $\mathbf{\hat{P}}$} \\
\includegraphics[width=3.0
in, height=2.2 in]{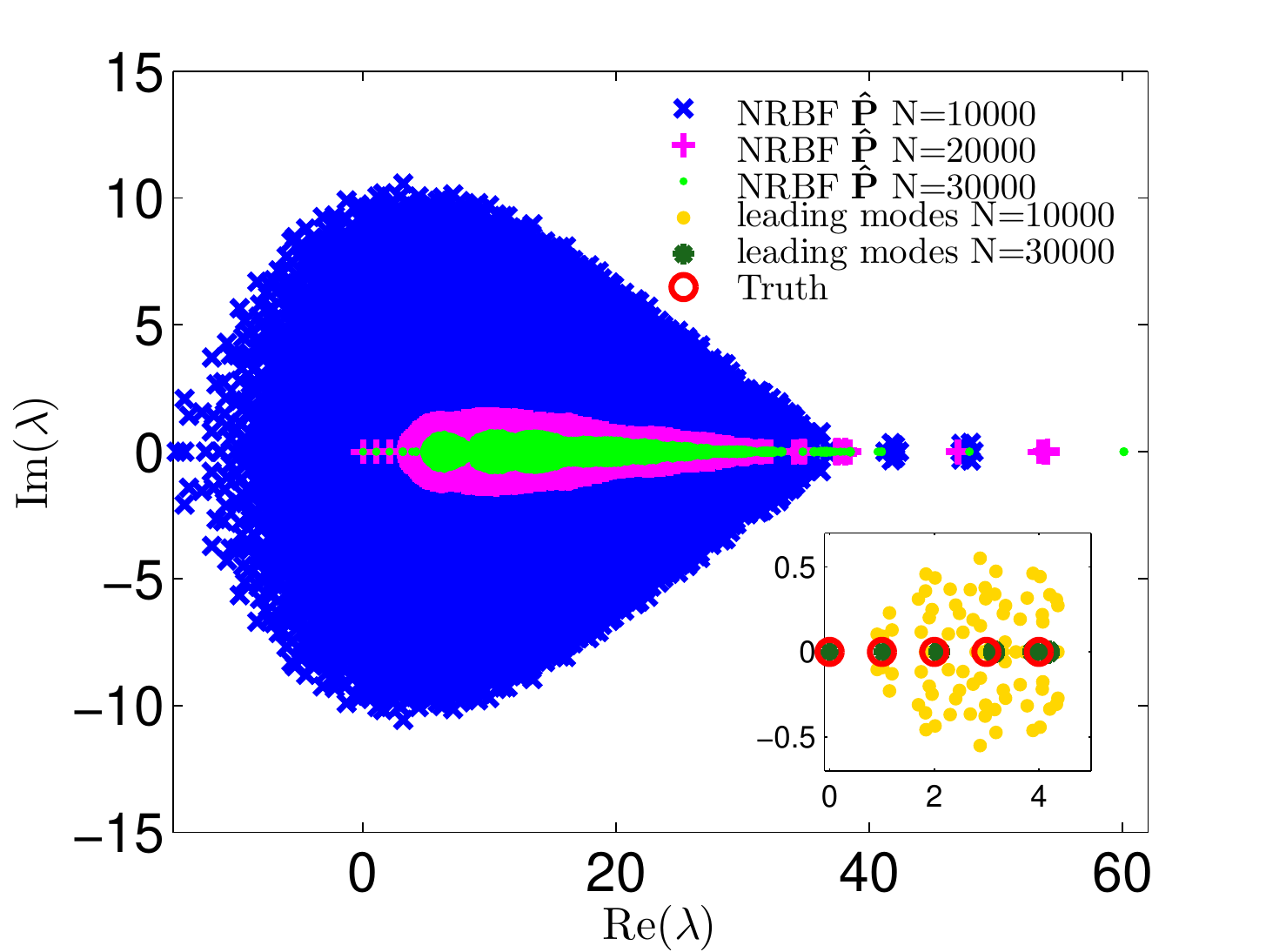}
&
\includegraphics[width=3.0
in, height=2.2 in]{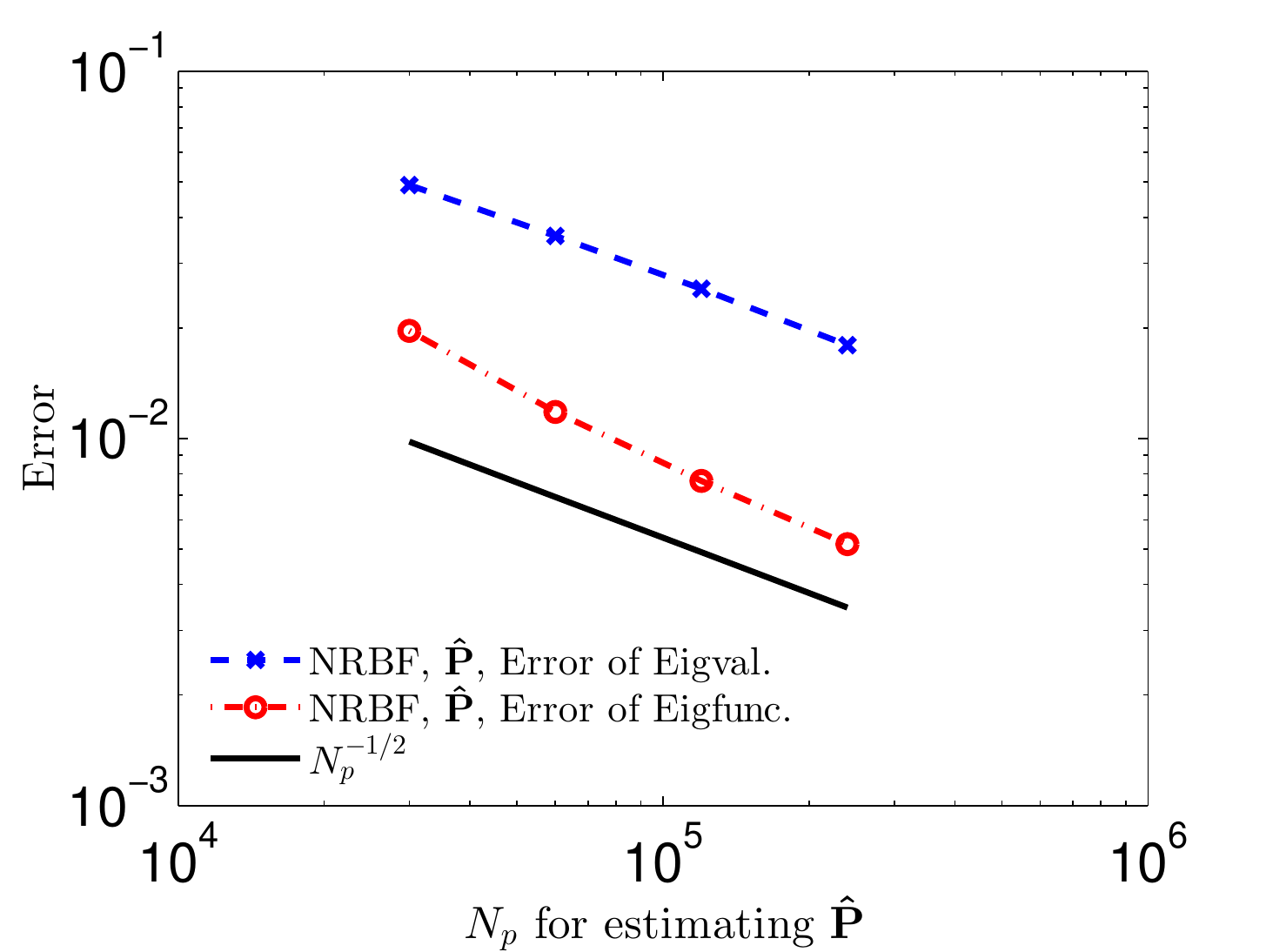}\\
{\normalsize (c) Conv. of
Spectra wrt $\mathbf{\hat{P}}$} & {\normalsize (d) Conv. of Eigenfuncs. wrt $%
\mathbf{\hat{P}}$}\\
\includegraphics[width=3.0
in, height=2.2 in]{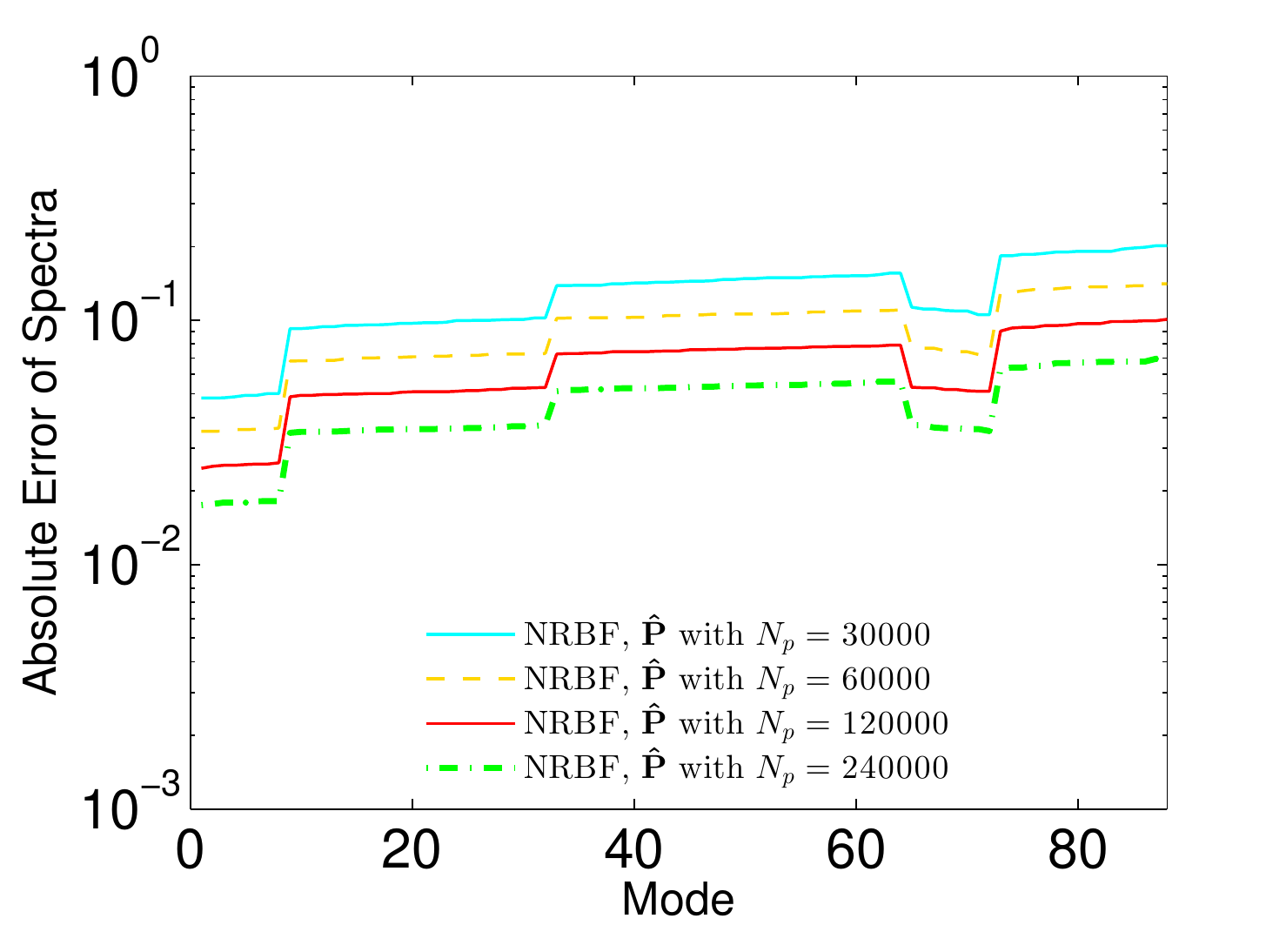}
&
\includegraphics[width=3.0
in, height=2.2 in]{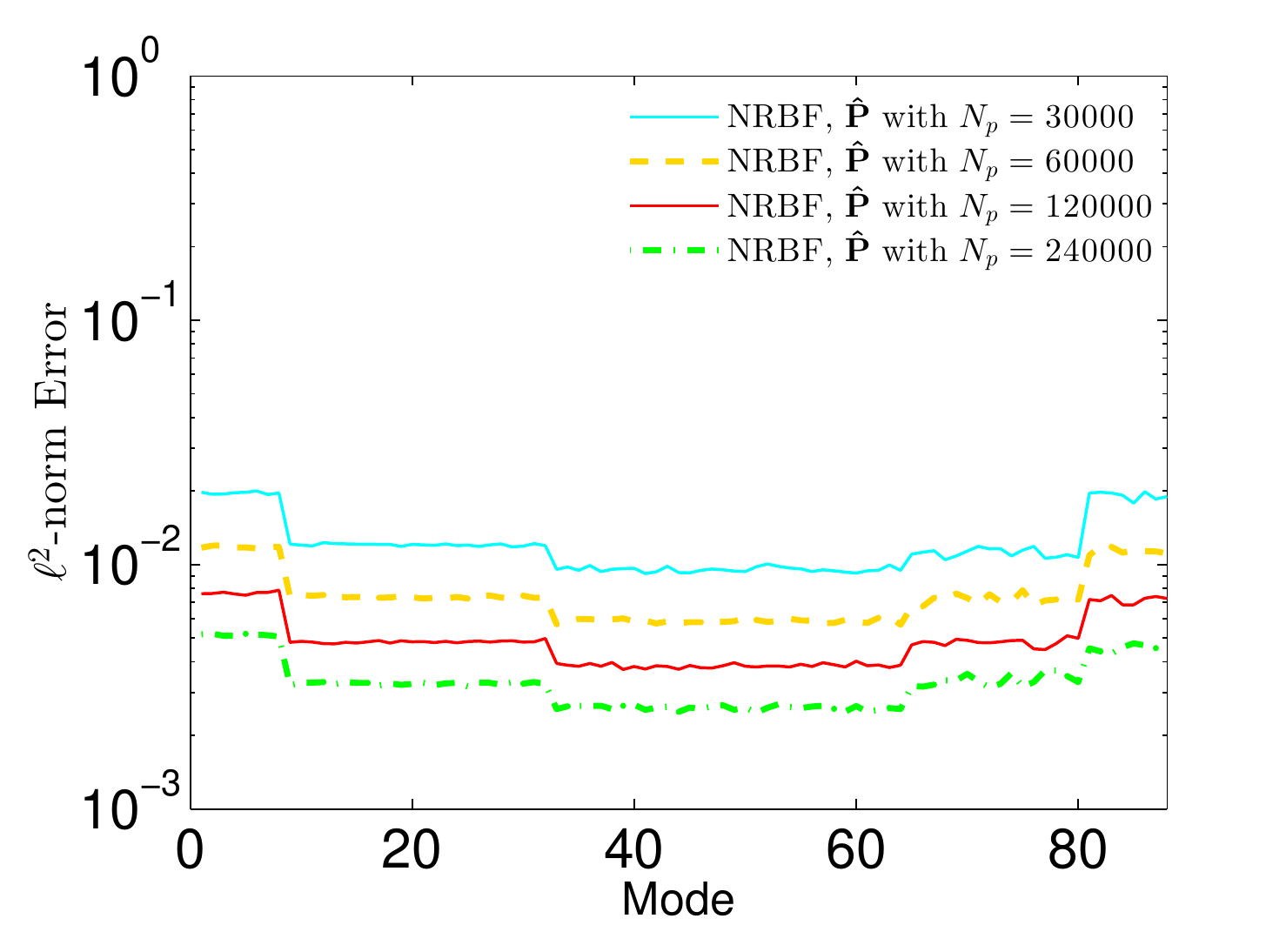}%
\end{tabular}
}
\caption{{\bf 4D flat torus in $\mathbb{R}^{16}$.} (a) Convergence
of eigenvalues with respect to $N$ for NRBF using $\mathbf{\hat{P}}$. For
panel (a), $N$ data points are used for both approximating $\mathbf{\hat{P}}$
and evaluating NRBF matrices. The leading modes in legend are
referred to as the $89$ numerical NRBF eigenvalues closest to the truth
listed in equation (\protect\ref{eqn:flatlead}). Convergence of (c) NRBF
eigenvalues and (d) NRBF eigenfunctions with respect to $\mathbf{\hat{P}}$
estimated from different $N_p$ number of points. For panels (c) and (d), the same
$N=30,000$ data points are used for evaluating NRBF matrices but different $N_p$
data points are used for approximating $\mathbf{\hat{P}}$ at those $30,000$
data points. {In panel (b), plotted is the convergence rate of $O(N_p^{-1/2})$
for the leading 8 modes (2nd-9th modes). } GA kernel with $s=0.5$ was fixed for all $N$. The data points
are randomly distributed on the flat torus.}
\label{fig_erreigflattori_1p5}
\end{figure*}



\begin{figure}[htbp]
{\centering
\begin{tabular}{ccc}
{\small (a) Conv. of Eigenvalues} & {\small (c) KDE $\tilde{q}$, Eigenvalues} & {\small (e) KDE $\tilde{q}$, Eigenfuncs.}  \\
\includegraphics[width=.32\textwidth,height=0.25\textwidth]{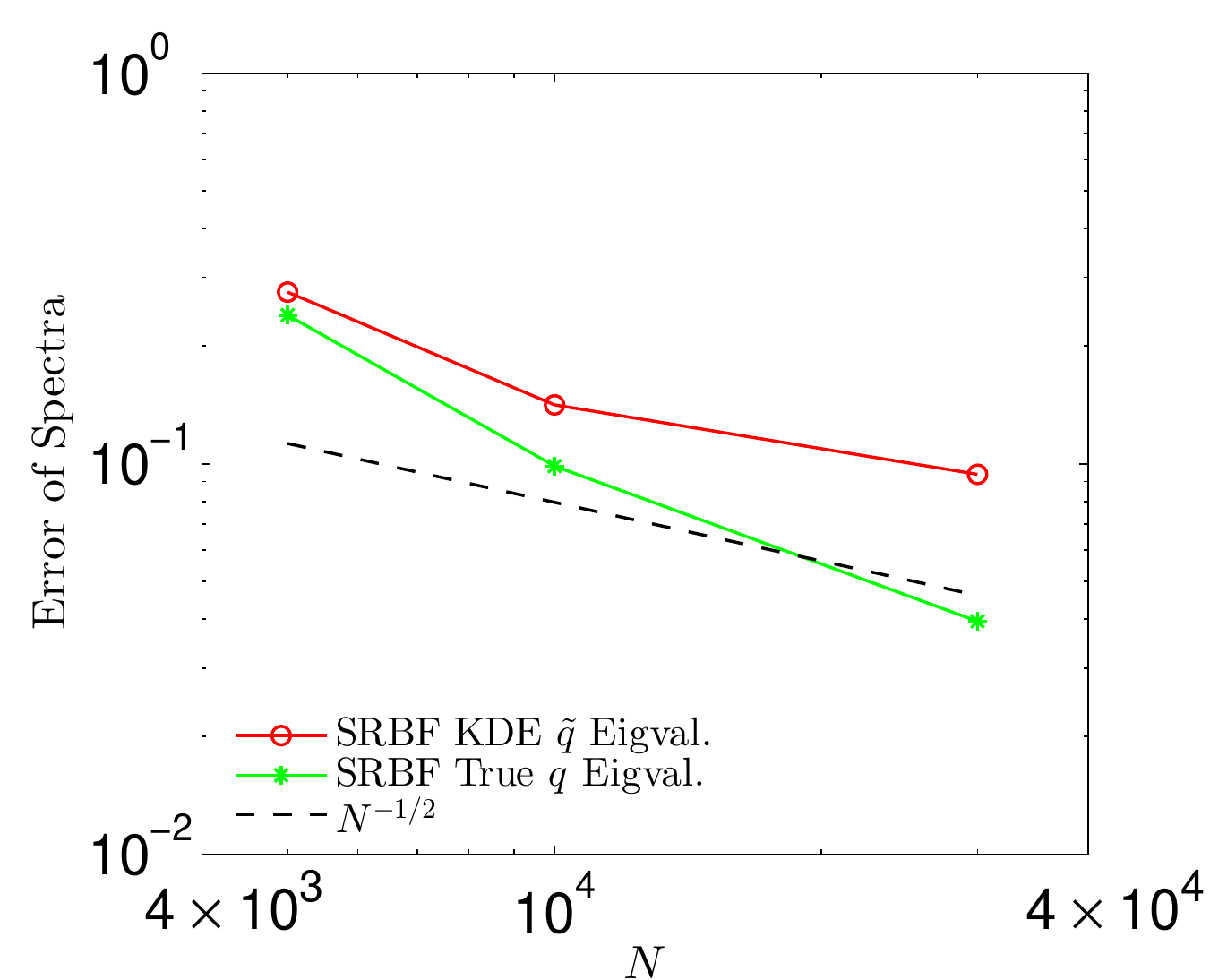}
&
\includegraphics[width=.32\textwidth,height=0.25\textwidth]{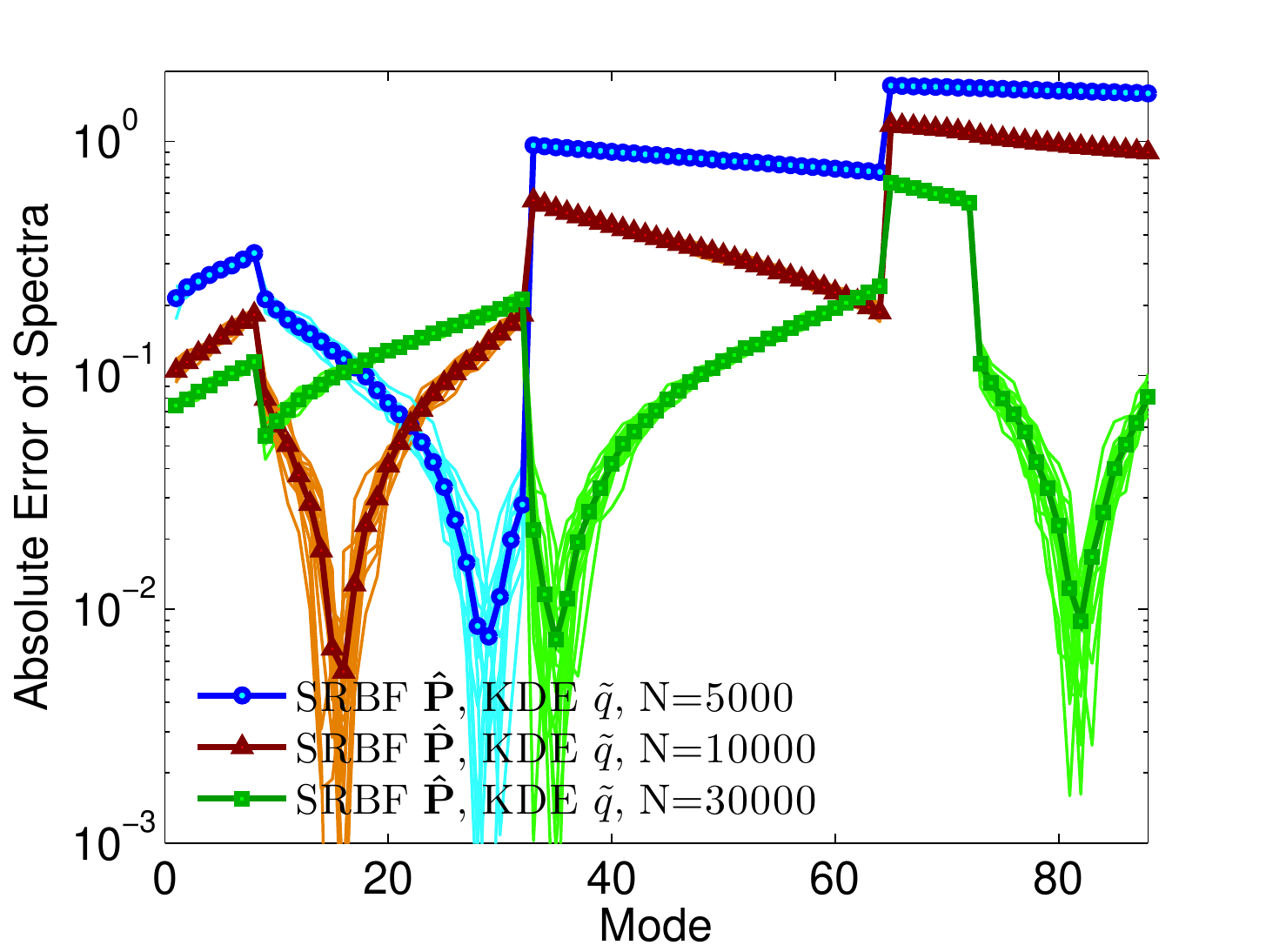}
&
\includegraphics[width=.32\textwidth,height=0.25\textwidth]{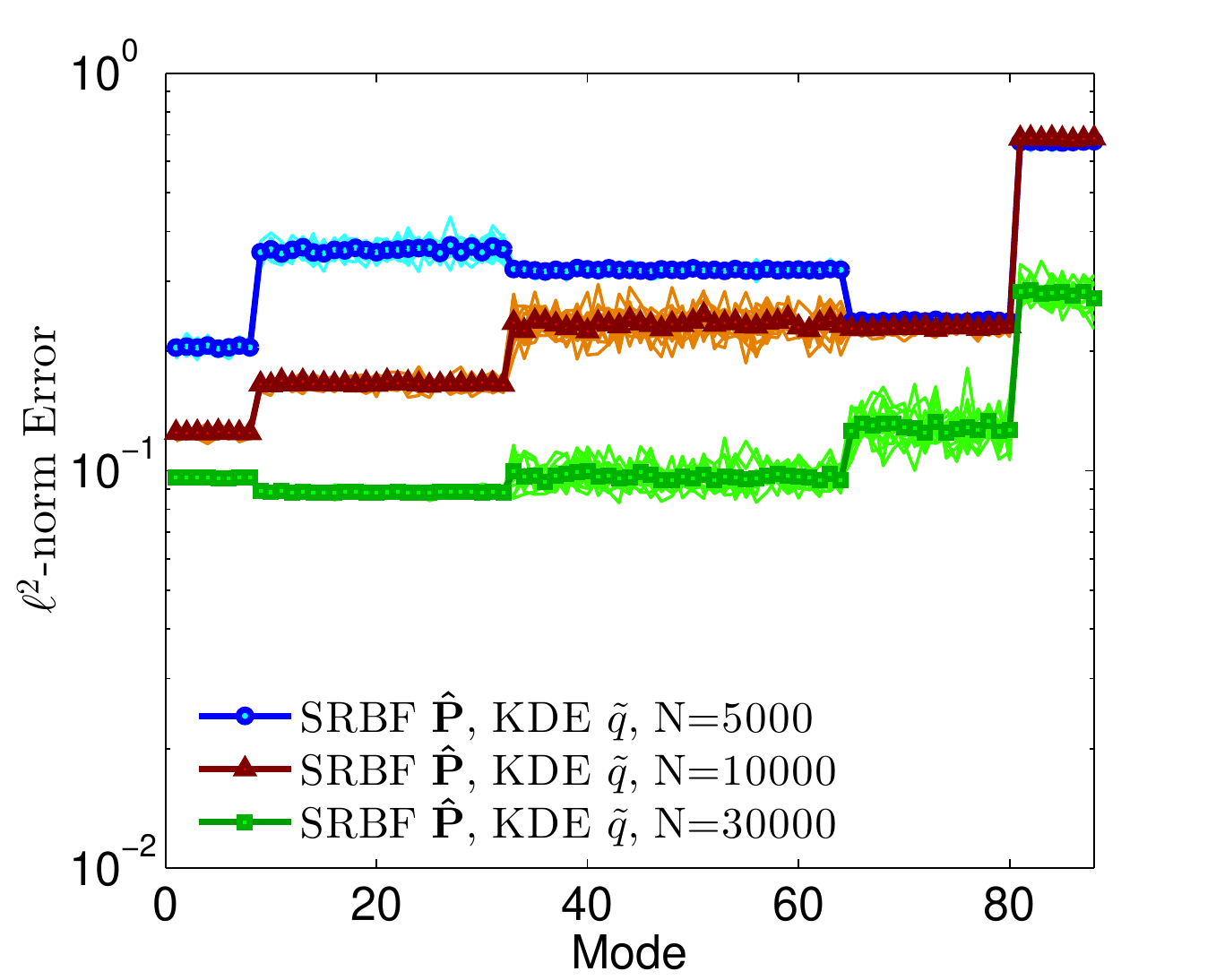} \\
{\small (b) Conv. of Eigenfunctions} & {\small (d) True ${q}$, Eigenvalues } & {\small (f) True ${q}$, Eigenfuncs.} \\
\includegraphics[width=.33\textwidth,height=0.25\textwidth]{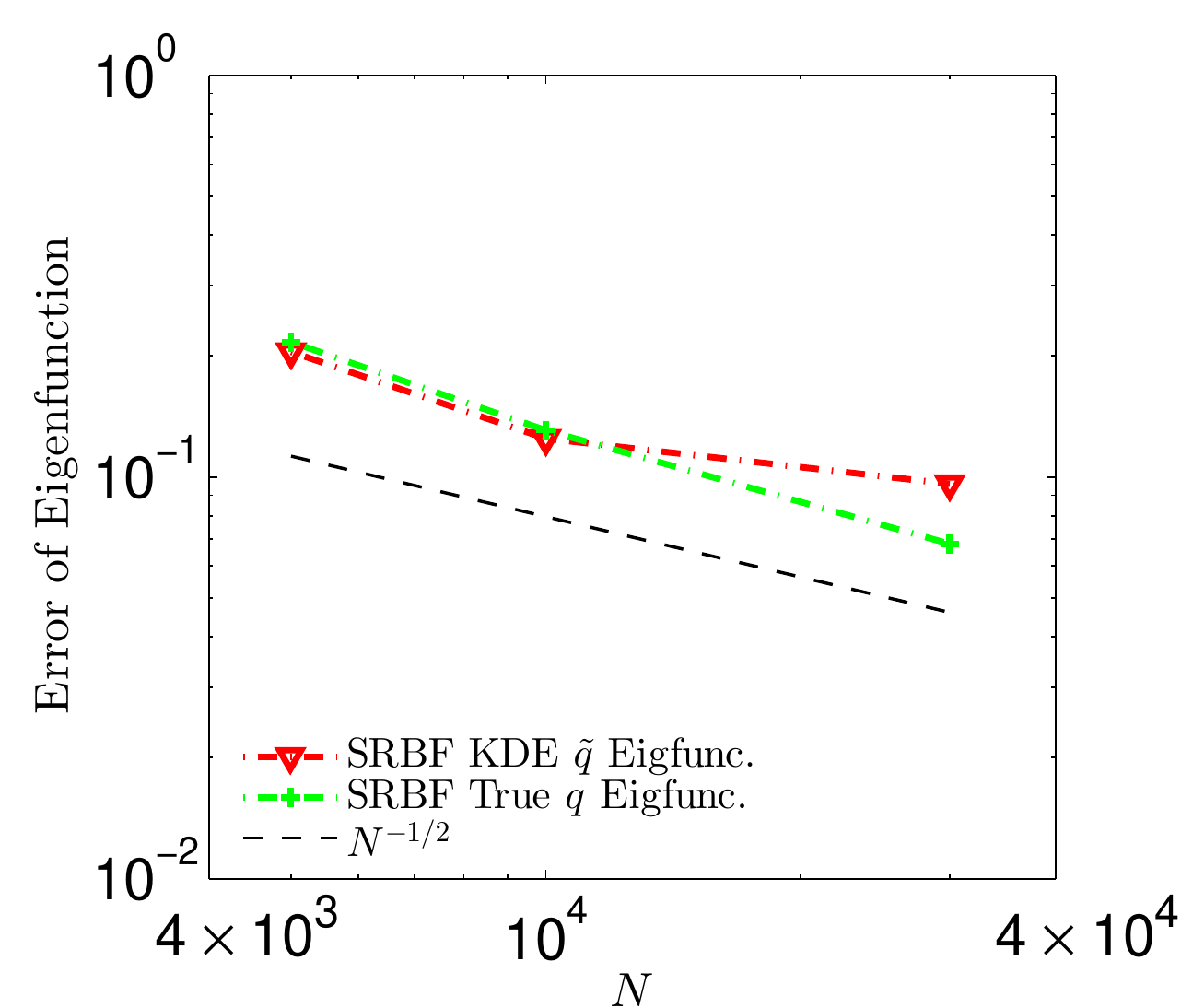}
&
\includegraphics[width=.32\textwidth,height=0.25\textwidth]{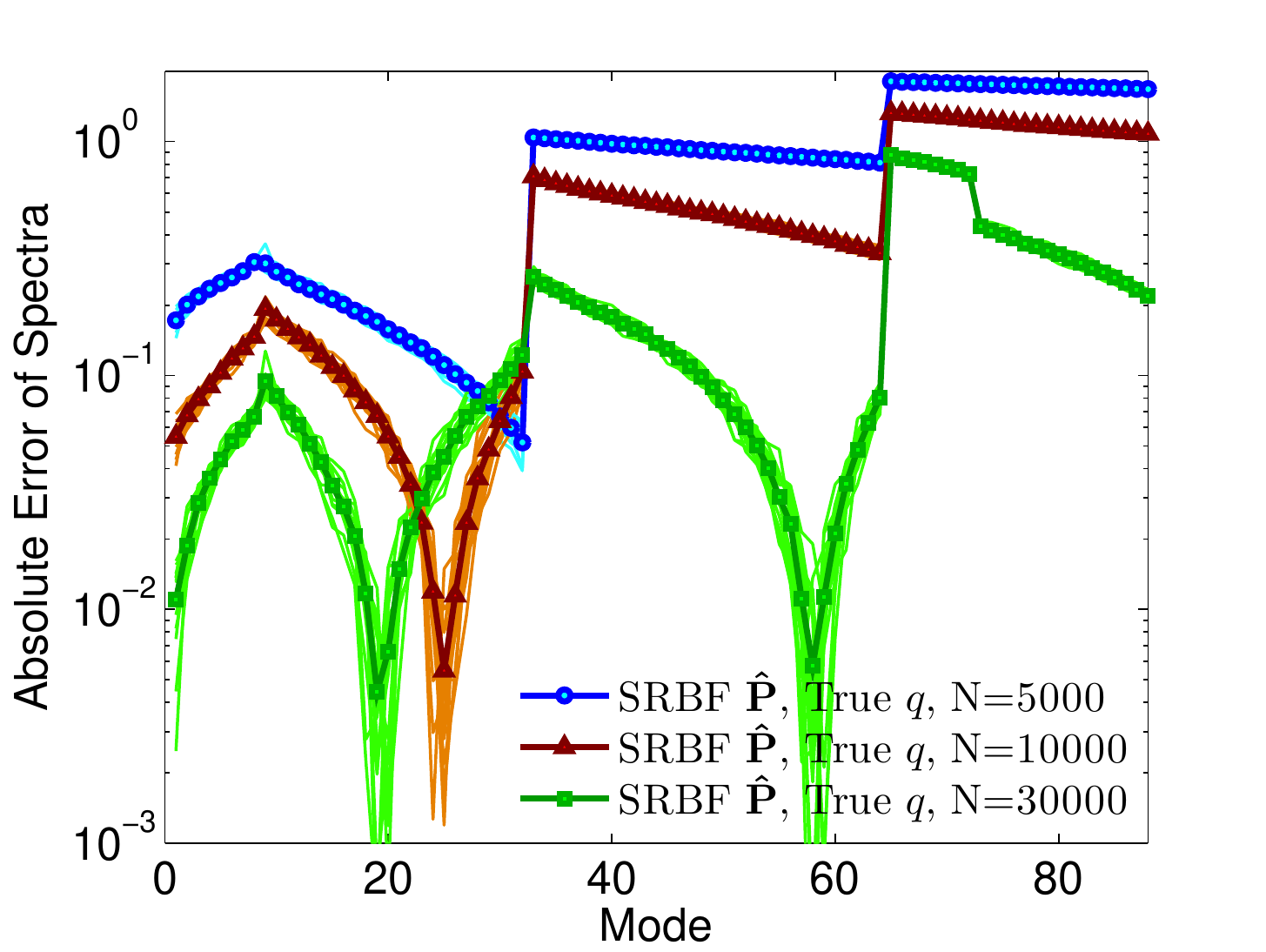}
&
\includegraphics[width=.32\textwidth,height=0.25\textwidth]{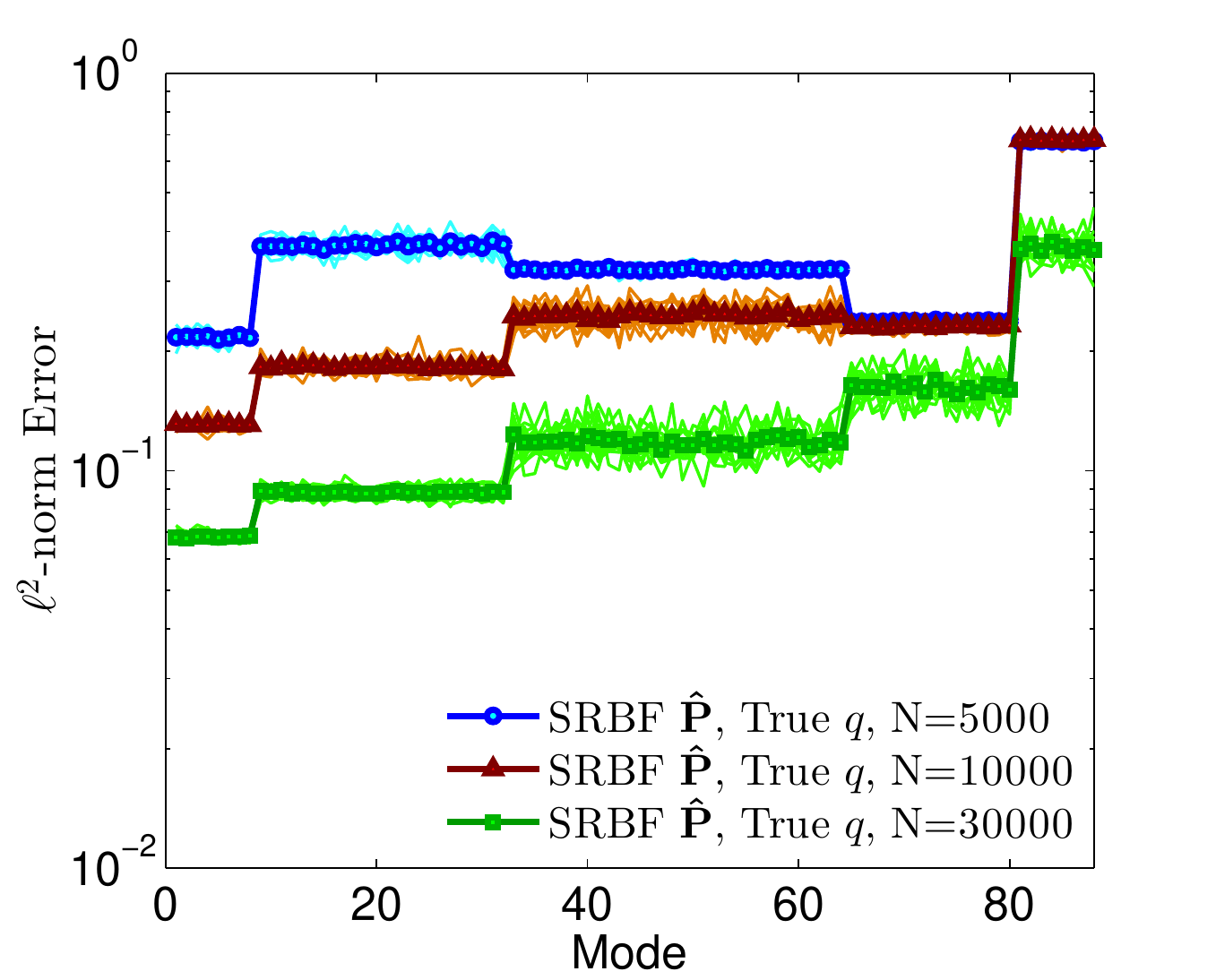}%
\end{tabular}
}
\caption{{\bf 4D flat torus in $\mathbb{R}^{16}$.} Convergence of
(a) eigenvalues and (b) eigenvectors for SRBF using $\mathbf{\hat{P}}$ averaged over the first eight modes (2nd-9th modes). In (c), and (d)
we plot the error of eigenvalues for each mode, while in (e), and (f) we plot the error of
eigenvectors. For each panel, IQ kernel with $s=0.3$ was used. Panels (c) and (d) show the errors of eigenvalues using
the estimated sampling density $\tilde{q}$ and true sampling density $q$, respectively.
Panels (e) and (f) show the corresponding errors of eigenvectors. In each of these panels, $16$
independent trials are run and depicted by light color. The
average of all $16$ trials are depicted by dark color. The data points used for
computation are randomly distributed on the manifold. }
\label{fig_erreigflattori_2}
\end{figure}


For the NRBF method using $\mathbf{\hat{P}}$, all the four observations in Example \ref{sec:gentorus}\ persist. However, we should point out that for the last observation (for all numerical eigenvalues in the right half plane with positive real parts) to hold, the number of points $N$ has to be around $20000$ as can be seen from Fig. \ref{fig_erreigflattori_1p5}(a). When $N=10000$ is not large enough, there are many irrelevant eigenvalues with negative real parts (blue crosses), which is a sign of spectral pollution. Moreover, the leading eigenvalues (yellow dots) are not close to the truth (red circles) [the inset of Fig. \ref{fig_erreigflattori_1p5}(a)]. When $N$ increases to $20000$ or $30000$, the irrelevant eigenvalues do not completely disappear (they appear near eigenvalues with larger real components) even though NRBF eigenvalues (magenta and green dots) lie in the right half plane. Here, the leading NRBF eigenvalues (dark green dots) approximate the truth (red circles) more accurately [the inset of Fig. \ref{fig_erreigflattori_1p5}(a)].


For NRBF, we also repeat the experiment with $N=10000$ data points using the analytic $\mathbf{P}$. We found that the profile of blue crosses in Fig. \ref{fig_erreigflattori_1p5}(a) still persists even if the analytic $\mathbf{P}$ was used to replace the approximated $\mathbf{\hat{P}}$\ [not shown]. The only slight difference was that the errors of the leading modes became smaller if $\mathbf{\hat{P}}$\ was replaced with $\mathbf{P}$ [red and black curves in Fig. \ref{fig_erreigflattori_1}(b)]. This numerical result suggests that these irrelevant eigenvalues (blue crosses in Fig. \ref{fig_erreigflattori_1p5}(a)) are due to not enough data points rather than the inaccuracy of $\mathbf{\hat{P}}$.

In contrast, when ($N=30000$) data points are used, we already observed in Fig.~\ref{fig_erreigflattori_1}(b) and (c) that NRBF with $\mathbf{P}$ (black dashed curve) performs much better than NRBF with $\mathbf{\hat{P}}$ (red solid curve). One can expect the errors of the NRBF with $\mathbf{\hat{P}}$\ (red curve in Fig. \ref{fig_erreigflattori_1}(b)) to decay to the errors of the NRBF with $\mathbf{P}$ (black curve in Fig. \ref{fig_erreigflattori_1}(b)) as the data points used to learn the tangential projection matrices $\mathbf{\hat{P}}$ is increased (beyond $30,000$ points) while the same fixed $N=30000$ data points are used to construct the Laplacian matrix for the eigenvalue problem. This scenario is practically useful since the approximation of $\mathbf{\hat{P}}$ is computationally cheap even with a large number of data, whereas solving the eigenvalue problem of a dense, non-symmetric RBF matrix is very expensive when $N$ is large. In Figs. \ref{fig_erreigflattori_1p5}(b) and (c) for NRBF, we demonstrate the improvement with this estimation scenario, especially with respect to the number of points $N_p$ used to construct $\mathbf{\hat{P}}$. We found that the rate is $N_p^{-1/2}$ which is consistent with our expectation as noted in Remark~\ref{rem5} for 4-dimensional examples.

%

We also inspect the convergence of NRBF when the dimension of the manifold varies. In particular, we compare the numerical result with a 3D flat torus example in $\mathbb{R}^{12}$ as well as a 5D flat torus example in $\mathbb{R}^{20}$. For the 3D flat torus, we found that irrelevant eigenvalues appear in the left half plane when $N=3,000$, and those eigenvalues disappear when $N $ increases to $10,000$ [not shown]. For the 5D flat torus, we found that irrelevant eigenvalues always exist in the left half plane even when $N$ increases to $30,000$ [not shown]. Based on these experiments, we believe that more data points are needed for accurate estimations of the leading order spectra of NRBF in higher dimensional cases. This implies that NRBF methods suffer from curse of dimensionality.


We now analyze the results of SRBF. Figure \ref{fig_erreigflattori_2} displays the errors of eigenmodes for SRBF for $N=5000$, $10000$, and $30000$. One can see from Figs. \ref{fig_erreigflattori_2}(a) and (b) that the errors of eigenvalues and eigenvectors of the leading eight modes
decrease on order of $N^{-1/2}$ for SRBF with $\mathbf{\hat{P}}$\ and true $q$. The convergence can also be examined for each leading mode for SRBF using $\mathbf{\hat{P}}$ and true $q$ in Figs. \ref{fig_erreigflattori_2}(d) and (f). However, the convergence is not clear for SRBF using $\mathbf{\hat{P}}$ and KDE $\tilde{q}$ as shown in Figs. \ref%
{fig_erreigflattori_2}(a)-(c)(e). This implies that $\mathbf{\hat{P}}$ is accurate enough, whereas the sampling density $\tilde{q}$\ is not. We suspect that the use of KDE for the density estimation in 4D may not be optimal. This leaves room for future investigations with more accurate density estimation methods.

\subsection{2D Sphere}\label{sphereexample}

In this section, we study the eigenvector field problem for the Hodge, Bochner and Lichnerowicz Laplacians, on a 2D
unit sphere. The parameterization of the unit sphere is given by
\begin{equation}
\boldsymbol{x}=\left[
\begin{array}{c}
x \\
y \\
z%
\end{array}%
\right] =\left[
\begin{array}{c}
\sin \theta \cos \phi \\
\sin \theta \sin \phi \\
\cos \theta%
\end{array}%
\right] ,\text{ \ \ for }%
\begin{array}{c}
\theta \in \lbrack 0,\pi ] \\
\phi \in \lbrack 0,2\pi )%
\end{array}%
,  \label{eqn:spher}
\end{equation}%
with Riemannian metric given as,
\begin{equation}
g=\left[
\begin{array}{cc}
1 & 0 \\
0 & \sin ^{2}\theta%
\end{array}%
\right] .  \label{eqn:gsphe}
\end{equation}%
The analytical solution of the eigenvalue problem for the Laplace-Beltrami operator
on the unit sphere consists of the set of Laplace's spherical harmonics with
corresponding eigenvalue $\lambda =l(l+1)$ for $l\in \mathbb{N}^{+}$. In
particular, the leading spherical harmonics can be written in Cartesian
coordinate as%
\begin{equation}
\begin{array}{ll}
\psi =x_{i}, & \text{for }\lambda =2, \\
\psi =3x_{i}x_{k}-\delta _{ik}r^{2}, & \text{for }\lambda =6, \\
\psi =15x_{i}x_{k}x_{n}-3\delta _{ik}r^{2}x_{n}-3\delta
_{kn}r^{2}x_{i}-3\delta _{ni}r^{2}x_{k}, & \text{for }\lambda =12,%
\end{array}
\label{eqn:eigf}
\end{equation}%
for indices $i,k,n=1,2,3$, where $(x_{1},x_{2},x_{3})=(x,y,z)$ and $%
r^{2}=x^{2}+y^{2}+z^{2}$.

We start with the analysis of spectra and corresponding eigenvector fields
of Hodge Laplacian on a unit sphere. The Hodge Laplacian defined on a $k$%
-form is

\begin{equation*}
\Delta _{H}=\mathrm{d}_{k-1}\mathrm{d}_{k-1}^{\ast }+\mathrm{d}^{\ast }_k\mathrm{d}_k,
\end{equation*}%
which is self-adjoint and positive definite. Here, $\mathrm{d}$\ is the
exterior derivative and $\mathrm{d}^{\ast }$ is the adjoint of $\mathrm{d}$
given by%
\begin{equation*}
\mathrm{d}^{\ast }_{k-1}=(-1)^{d(k-1)+1}\ast  \mathrm{d}_{d-k} \ast :\Omega
^{k}(M)\rightarrow \Omega ^{k-1}(M),
\end{equation*}%
where $d$ denotes the dimension of manifold $M$ and $\ast: \Omega
^{k}(M)\rightarrow \Omega ^{d-k}(M)$\ is the standard Hodge
star operator. Obviously, one has $\mathrm{dd}=0\ $and $\mathrm{d}^{\ast }%
\mathrm{d}^{\ast }=0$ and the following diagram:%
\begin{equation*}
\Omega ^{0}(M){{\autorightleftharpoons{$\mathrm{d}_0$}{$\mathrm{d}^\ast_0$}}}%
\Omega ^{1}(M)\autorightleftharpoons{$\mathrm{d}_1$}{$\mathrm{d}^\ast_1$}%
\Omega ^{2}(M)\autorightleftharpoons{$\mathrm{d}_2$}{$\mathrm{d}^\ast_2$}%
\cdots ,
\end{equation*}%
with $\Omega ^{0}(M)=C^{\infty }(M)$. We note that Hodge Laplacian $\Delta
_{H}$ reduces to Laplace-Beltrami $\Delta _{M}$ when acting on functions. To
obtain the solution of the eigenvalue problem for Hodge of $1$-form, we can
apply the results from \cite{Folland1989Harmonic} or we can provide one
derivation based on the following proposition:

\begin{prop}
\label{hodgeex1} Let $M$ be a $d-$dimensional manifold and $\Delta_H$ be the
Hodge Laplacian on k-form. Then,

\noindent (1) $\Omega ^{k}(M)=\mathrm{im}\,\mathrm{d}_{k-1}\oplus \mathrm{im}\,\mathrm{d}_{k}^{\ast }\oplus \ker \Delta _{H}$.

\noindent (2) $\lambda \left( \Delta _{H}\right) \backslash \{0\}=\lambda
\left( \mathrm{d}_{k-1}\mathrm{d}_{k-1}^{\ast }\right) \backslash \{0\}\cup
\lambda \left( \mathrm{d}_{k}^{\ast }\mathrm{d}_{k}\right) \backslash \{0\}$.

\noindent (3) $\lambda \left( \mathrm{d}_{k}\mathrm{d}_{k}^{\ast }\right)
\backslash \{0\}=\lambda \left( \mathrm{d}_{k}^{\ast}\mathrm{d}_{k}\right)
\backslash \{0\}$ for $k=0,1,\ldots ,d$.

\noindent (4) $\lambda \left( \mathrm{d}_{k}\mathrm{d}_{k}^{\ast }\right)
=\lambda \left( \mathrm{d}_{d-k-1}^{\ast }\mathrm{d}_{d-k-1}\right) $ for $%
k=0,1,\ldots ,d-1$.
\end{prop}

\begin{proof}
(1) is the Hodge Decomposition Theorem. The proof of (2) and (3) can be
referred to as \textit{pp} 138 of \cite{jost2008riemannian}. For (4), we
first show that $\lambda \left( \mathrm{d}_{k}\mathrm{d}_{k}^{\ast }\right)
\subseteq \lambda \left( \mathrm{d}_{d-k-1}^{\ast }\mathrm{d}_{d-k-1}\right)
$. Assume that $\mathrm{d}_{k}\mathrm{d}_{k}^{\ast }\alpha =\lambda \alpha $
for $\alpha \in \Omega ^{k+1}(M)$. Taking Hodge star $\ast $ at both sides,
one arrives at the result, $(-1)^{dk+1}\ast \mathrm{d}_k\ast \mathrm{d}_{d-k-1}%
\ast \alpha =\mathrm{d}^{\ast }_{d-k-1}\mathrm{d}_{d-k-1}(\ast \alpha )=\lambda \ast \alpha $%
. Vice versa one can prove $\lambda \left( \mathrm{d}_{k}\mathrm{d}%
_{k}^{\ast }\right) \supseteq \lambda \left( \mathrm{d}_{d-k-1}^{\ast }%
\mathrm{d}_{d-k-1}\right) $.
\end{proof}

\begin{coro}\label{cor:Hoddoub}
Let $M$ be a 2D surface embedded in $\mathbb{R}^{3}$\ and $\Delta
_{H}=\sharp \left( \mathrm{dd}^{\ast }+\mathrm{d}^{\ast }\mathrm{d}\right)
\flat $ be the Hodge Laplacian on vector fields $\mathfrak{X}(M)$. Then, the non-trivial eigenvalues of Hodge Laplacian, $\lambda \left( \Delta _{H}\right)$, are identical with the non-trivial eigenvalues of Laplace-Beltrami operator, $\lambda \left( \Delta_{M}\right) \backslash \{0\}$, where the number of eigenvalues of the Hodge Laplacian doubles those of the Laplace-Beltrami operator.
%
Specifically, the corresponding eigenvector fields to nonzero eigenvalues are $%
\mathbf{P}\nabla _{\mathbb{R}^{3}}\psi $ and $\boldsymbol{n}\times \nabla _{%
\mathbb{R}^{3}}\psi $, where $\psi $ is the eigenfunction of
Laplace-Beltrami $\Delta_{M}$ and $\boldsymbol{n}$ is the normal vector to
surface $M$.
\end{coro}

\begin{proof}
First, we have $\lambda \left( \mathrm{d}_{0}\mathrm{d}_{0}^{\ast }\right)
\backslash \{0\}=\lambda \left( \mathrm{d}_{0}^{\ast }\mathrm{d}_{0}\right)
\backslash \{0\}=\lambda \left( \Delta_{ M}\right) \backslash \{0\}$ using
Proposition \ref{hodgeex1} (3). Second, we have $\lambda \left( \mathrm{d}%
_{1}^{\ast }\mathrm{d}_{1}\right) \backslash \{0\}=\lambda \left( \mathrm{d}%
_{0}\mathrm{d}_{0}^{\ast }\right) \backslash \{0\}=\lambda \left( \Delta_{ M}\right) \backslash \{0\}$ using Proposition \ref{hodgeex1} (3) and (4).
Using Proposition \ref{hodgeex1} (2), we obtain the first part in Corollary~\ref%
{cor:Hoddoub}. That is, the eigenvalues of Hodge Laplacian for 1-forms or
vector fields are exactly of eigenvalues of Laplace-Beltrami, and the number of each eigenvalue of Hodge Laplacian doubles the multiplicity of the corresponding eigenspace of the Laplace-Beltrami operator.

To simplify the notation, we use $\mathrm{d}$ to denote $\mathrm{d}_k$ for arbitrary $k$, which is implicitly identified by whichever $k-$form it acts on.  We now examine the eigenforms. Assume that $\psi $ is an eigenfunction for the Laplace-Beltrami operator $\Delta_{ M}$ associated with the eigenvalue $%
\lambda $, that is, $\Delta_{ M}\psi =\mathrm{d}_{0}^{\ast }\mathrm{d}%
_{0}\psi =\lambda \psi $. Then, one can show that $\mathrm{d}\psi \in \Omega
^{1}(M)$ is an eigenform of $\Delta _{H}$,
\begin{equation}
\Delta _{H}\mathrm{d}\psi =\left( \mathrm{dd}^{\ast }+\mathrm{d}^{\ast }%
\mathrm{d}\right) \mathrm{d}\psi =\mathrm{dd}^{\ast }\mathrm{d}\psi =\mathrm{%
d}\lambda \psi =\lambda \mathrm{d}\psi .  \label{eqn:lmbdf}
\end{equation}%
One can also show that $\ast \mathrm{d}\psi \in \Omega ^{1}(M)$ is an
eigenform,%
\begin{eqnarray}
\Delta _{H}\ast \mathrm{d}\psi  &=&\left( \mathrm{dd}^{\ast }+\mathrm{d}%
^{\ast }\mathrm{d}\right) \ast \mathrm{d}\psi =\mathrm{d}^{\ast }\mathrm{d}%
\ast \mathrm{d}\psi   \notag \\
&=&-\ast \mathrm{d}\ast \mathrm{d}\ast \mathrm{d}\psi =\ast \mathrm{dd}%
^{\ast }\mathrm{d}\psi =\ast \mathrm{d}\lambda \psi =\lambda \ast \mathrm{d}%
\psi ,  \label{eqn:stdf}
\end{eqnarray}%
where we have used $\mathrm{d}_{1}^{\ast }=(-1)^{2(1)+1}\ast \mathrm{d}_0%
\ast $ and $\mathrm{d}_{0}^{\ast }=(-1)^{1}\ast \mathrm{d}_1\ast $.
Thus, based on Hodge Decomposition Theorem \ref{hodgeex1} (1), harmonic
forms and above eigenforms $\mathrm{d}\psi $\ and $\ast \mathrm{d}\psi $\
form a complete space of $\Omega ^{1}(M)$.

Last, we compute the corresponding eigenvector fields. The corresponding
eigenvector field for $\mathrm{d}\psi $\ is
\begin{equation}
\sharp \mathrm{d}\psi =g^{ij}\psi _{i}\frac{\partial }{\partial \theta ^{j}}=%
\mathbf{P}\nabla _{\mathbb{R}^{3}}\psi .  \label{eqn:fflat}
\end{equation}%
Assume that $\left\{ \theta ^{1},\theta ^{2}\right\} $\ is the local normal
coordinate so that the Riemannian metric is locally identity $g_{ij}=\delta
_{ij}$ at point $\boldsymbol{x}$. The corresponding eigenvector field for $%
\ast \mathrm{d}\psi $\ is
\begin{eqnarray}
\sharp \ast \mathrm{d}\psi  &=&\sharp \ast \psi _{i}\mathrm{d}\theta ^{i}=\sharp \delta
_{ij}^{12}\psi ^{i}\mathrm{d}\theta ^{j}=\sharp \left( \psi ^{1}\mathrm{d}\theta ^{2}-\psi
^{2}\mathrm{d}\theta ^{1}\right)   \notag \\
&=&-\frac{\partial \psi }{\partial \theta ^{2}}\frac{\partial }{\partial
\theta ^{1}}+\frac{\partial \psi }{\partial \theta ^{1}}\frac{\partial }{%
\partial \theta ^{2}}=-(\nabla _{\mathbb{R}^{3}}\psi \cdot \frac{\partial
\boldsymbol{x}}{\partial \theta ^{2}})\frac{\partial }{\partial \theta ^{1}}%
+(\nabla _{\mathbb{R}^{3}}\psi \cdot \frac{\partial \boldsymbol{x}}{\partial
\theta ^{1}})\frac{\partial }{\partial \theta ^{2}}  \notag \\
&=&\nabla _{\mathbb{R}^{3}}\psi \times (\frac{\partial \boldsymbol{x}}{%
\partial \theta ^{2}}\times \frac{\partial \boldsymbol{x}}{\partial \theta
^{1}})=\boldsymbol{n}\times \nabla _{\mathbb{R}^{3}}\psi ,  \label{eqn:fcul}
\end{eqnarray}%
where the equality in last line holds true for 2D surface in $\mathbb{R}^{3}$%
.
\end{proof}

Based on (\ref{eqn:fflat}) and (\ref{eqn:fcul}), one can immediately
obtain the leading eigenvalues and corresponding eigenvector fields for
Hodge Laplacian. When $\lambda ^{H}=2$, the $6$ corresponding Hodge eigen
vector fields are
\begin{small}
\begin{equation*}
U_{1,2,3}=\boldsymbol{n}\times \nabla _{\mathbb{R}^{3}}\psi =\left[
\begin{array}{c|c|c}
y & -z & 0 \\
-x & 0 & z \\
0 & x & -y%
\end{array}%
\right] ,U_{4,5,6}=\mathbf{P}\nabla _{\mathbb{R}^{3}}\psi =\left[
\begin{array}{c|c|c}
xz & xy & -y^{2}-z^{2} \\
yz & -x^{2}-z^{2} & xy \\
-x^{2}-y^{2} & zy & xz%
\end{array}%
\right] ,
\end{equation*}%
\end{small}
where $U_{1,2,3}$ are computed from the curl formula (\ref{eqn:fcul})\ and $%
U_{4,5,6}$ are computed from the projection formula (\ref{eqn:fflat}) by
taking $\psi =x_{i}$ $(i=1,2,3)$ in (\ref{eqn:eigf}). When $\lambda ^{H}=6$,
the $10$ Hodge eigenvector fields are%
\begin{eqnarray*}
U_{7\sim 11} &=&\boldsymbol{n}\times \nabla _{\mathbb{R}^{3}}\psi =\left[
\begin{array}{c|c|c|c|c}
-xz & y^{2}-z^{2} & xy & 0 & -yz \\
yz & -xy & z^{2}-x^{2} & xz & 0 \\
x^{2}-y^{2} & xz & -yz & -xy & xy%
\end{array}%
\right] , \\
U_{12\sim 16} &=&\mathbf{P}\nabla _{\mathbb{R}^{3}}\psi =\left[
\begin{array}{c|c|c|c|c}
y-2x^{2}y & z-2x^{2}z & -2xyz & x-x^{3} & -xy^{2} \\
x-2xy^{2} & -2xyz & z-2y^{2}z & -x^{2}y & y-y^{3} \\
-2xyz & x-2xz^{2} & y-2yz^{2} & -x^{2}z & -y^{2}z%
\end{array}%
\right] ,
\end{eqnarray*}%
where $U_{7\sim 11}$ are computed from the curl (\ref{eqn:fcul}) and $%
U_{12\sim 16}$ are computed from the projection (\ref{eqn:fflat}) by taking $%
\psi =3x_{i}x_{k}-\delta _{ik}r^{2}$ in (\ref{eqn:eigf}).

For the Bochner Laplacian, we notice that it is different from the Hodge Laplacian by a Ricci tensor and Ricci curvature is constant on the sphere. Therefore, the Bochner and Hodge Laplacians share the same eigenvector fields but have different eigenvalues. We examined that $U_{1\sim 6}$ are eigenvector fields for $\lambda ^{B}=1$\ and $U_{7\sim 16}$\ are eigenvector fields for $\lambda ^{B}=5$. In general, the spectrum are $\lambda ^{B}=l(l+1)-1$ for the Bochner Laplacian and $%
\lambda ^{H}=l(l+1)$ for the Hodge Laplacian for $l\in \mathbb{N}^{+}$.

For the Lichnerowicz Laplacian, we can verify that $U_{1,2,3}$ are in the null space $\ker \Delta _{L}$, which is often referred to as the Killing field.
We can further verify that $U_{4,5,6}$ correspond to $\lambda ^{L}=2$, $U_{7\sim 11}$ correspond to $\lambda ^{L}=4$, and $U_{12\sim 16}$ correspond to $\lambda ^{L}=10$. Moreover, we verify that $U_{17\sim 23}$ corresponds to $\lambda ^{L}=10$,$\ $where $U_{17\sim 23}$ are computed from the curl equation (\ref{eqn:fcul}) by taking $\psi =15x_{i}x_{k}x_{n}-3\delta_{ik}r^{2}x_{n}-3\delta _{kn}r^{2}x_{i}-3\delta _{ni}r^{2}x_{k}$. We refer to the eigenvalues $\lambda $ associated with the eigenvector fields $U$ obtained above as the analytic solutions to the eigenvalue problem of vector Laplacians.

\begin{figure*}[htbp]
{\scriptsize \centering
\begin{tabular}{c}
\includegraphics[width=6.3
in, height=1.1 in]{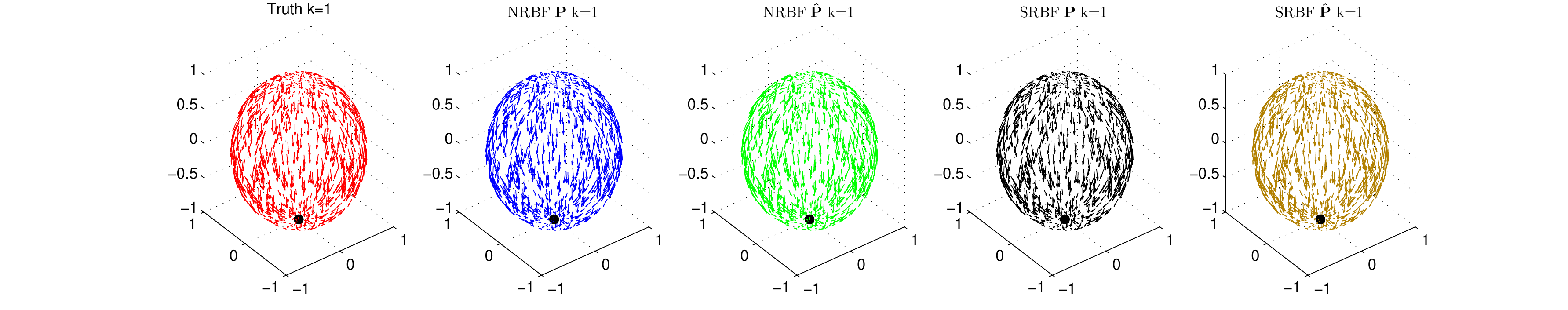}
\\
\includegraphics[width=6.3
in, height=1.1 in]{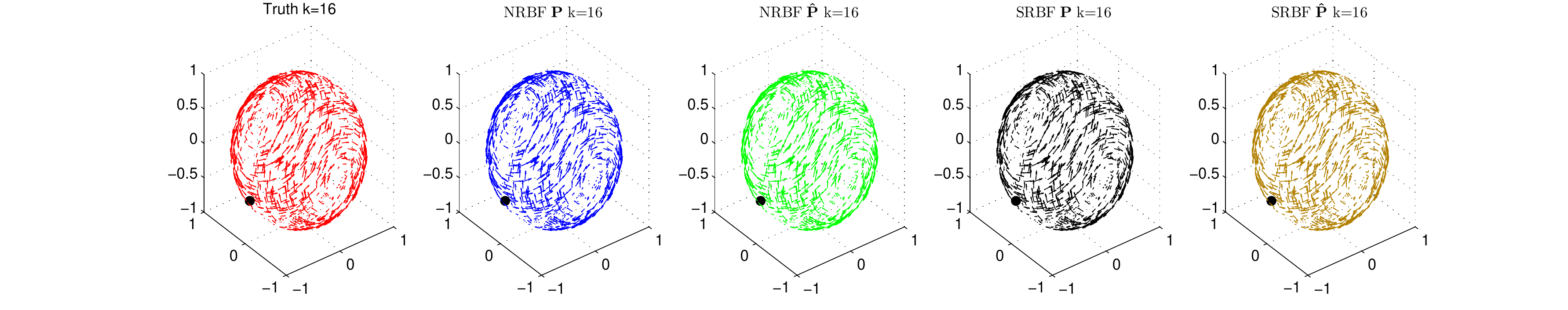}%
\end{tabular}
}
\caption{{\bf 2D Sphere in $\mathbb{R}^{3}$.} Comparison of
eigen-vector fields of Bochner Laplacian for $k=1, 16$ among NRBF using $%
\mathbf{P}$ and $\mathbf{\hat{P}}$, and SRBF using $\mathbf{P}$ and $\mathbf{%
\hat{P}}$. Black dots correspond the poles where the vector fields vanish.
For NRBF, GA kernel with $s=1.0$ is used, and for SRBF, IQ kernel with $%
s=0.5 $ is used. The $N=1024$ data points are randomly distributed on the
manifold.}
\label{fig_vectorsphere_1}
\end{figure*}

In the following, we compute eigenvalues and associated eigenvector fields of the Hodge, Bochner, and Lichnerowicz Laplacians on a unit sphere. We use uniformly random sample data points on the sphere for one trial comparison in most of the following figures, except for Fig. \ref{fig_vectorsphere_4} in which we use well-sampled data points for verifying the convergence of the SRBF method. Figure~\ref{fig_vectorsphere_1} displays the comparison of eigenvector fields of Bochner Laplacian for $k=1$ and $16$. When $\mathbf{P}$ is used, we have examined that these vector fields at each point are orthogonal to the normal directions $\boldsymbol{n=x}=(x,y,z)$ of the sphere. When $\mathbf{\hat{P}}$ is used (i.e. in an unknown manifold scenario), the vector fields lie in the approximated tangent space which is orthogonal to the normal $\boldsymbol{x}$\ within numerical accuracy. Based on Hairy ball theorem, a vector field vanishes at one or more points on the sphere. Here, we plot the poles where the vector field vanishes with black dots.

\begin{figure*}[htbp]
{\scriptsize \centering
\begin{tabular}{ccc}
{\small (a) NRBF, Eigenvalues} & {\small (b) NRBF, Error of Eigenvalues} &
{\small (c) NRBF, Error of Eigenvectors } \\
\includegraphics[width=2.0
in, height=1.5 in]{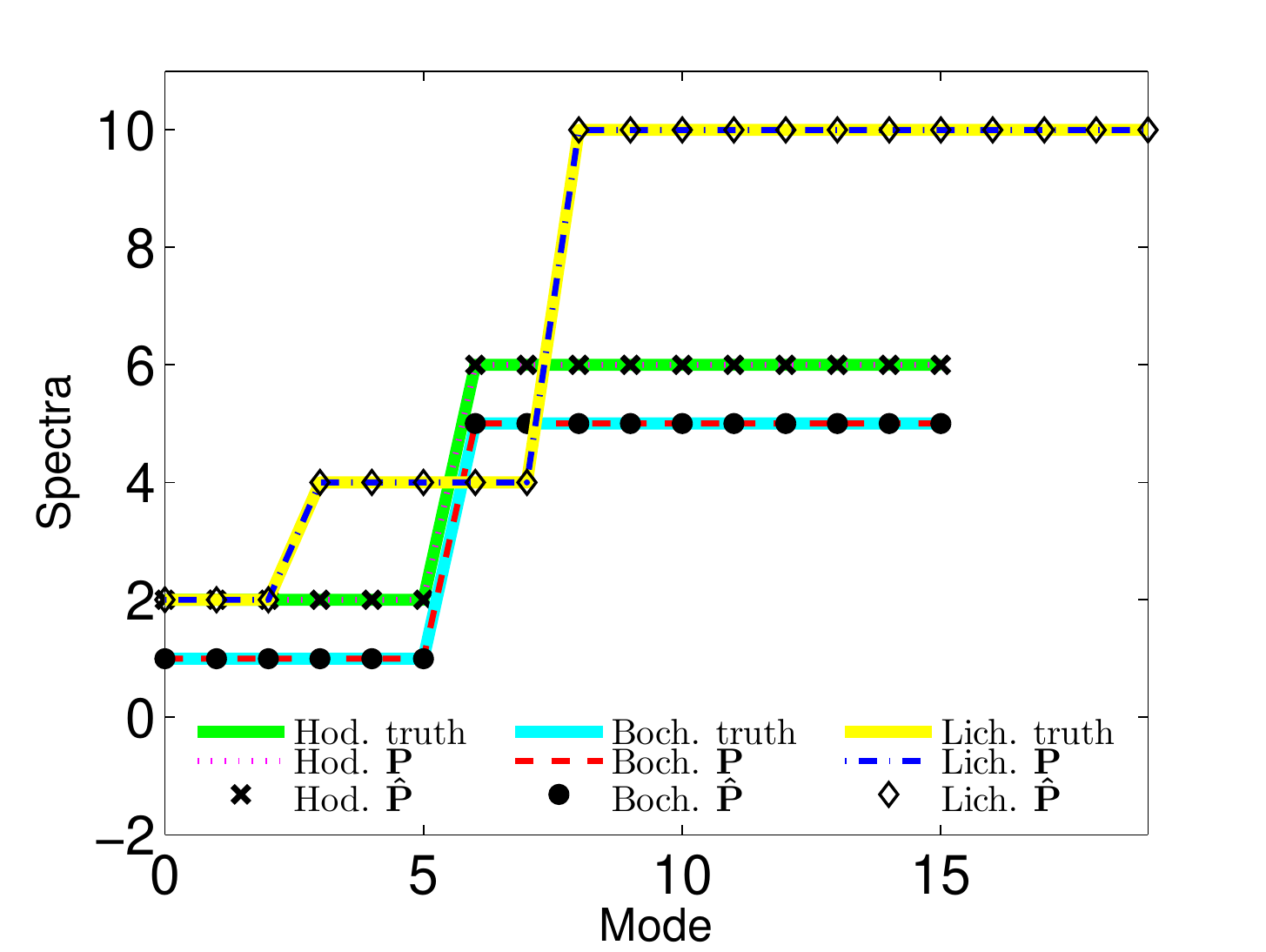}
&
\includegraphics[width=2.0
in, height=1.5 in]{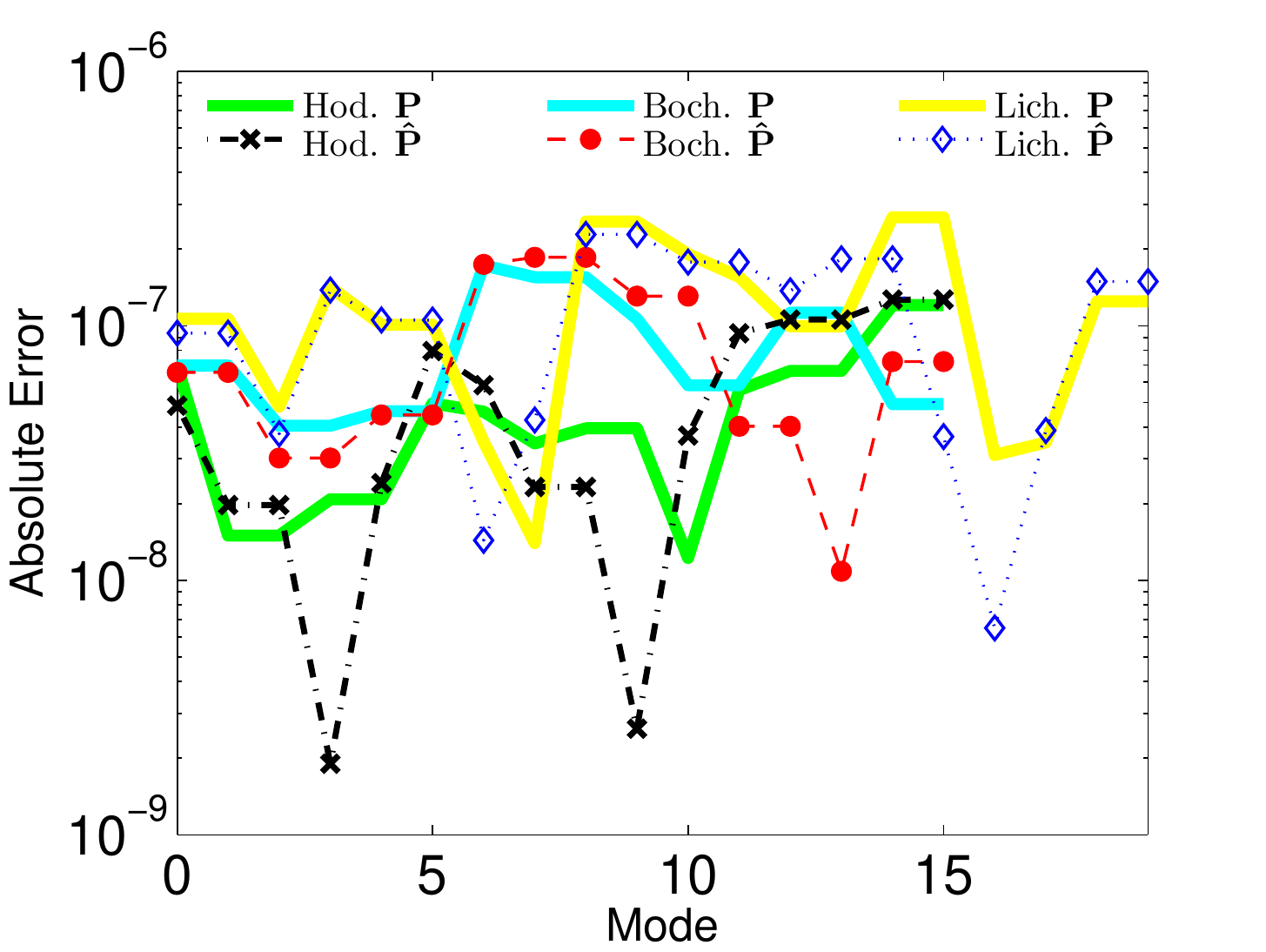}
&
\includegraphics[width=2.0
in, height=1.5 in]{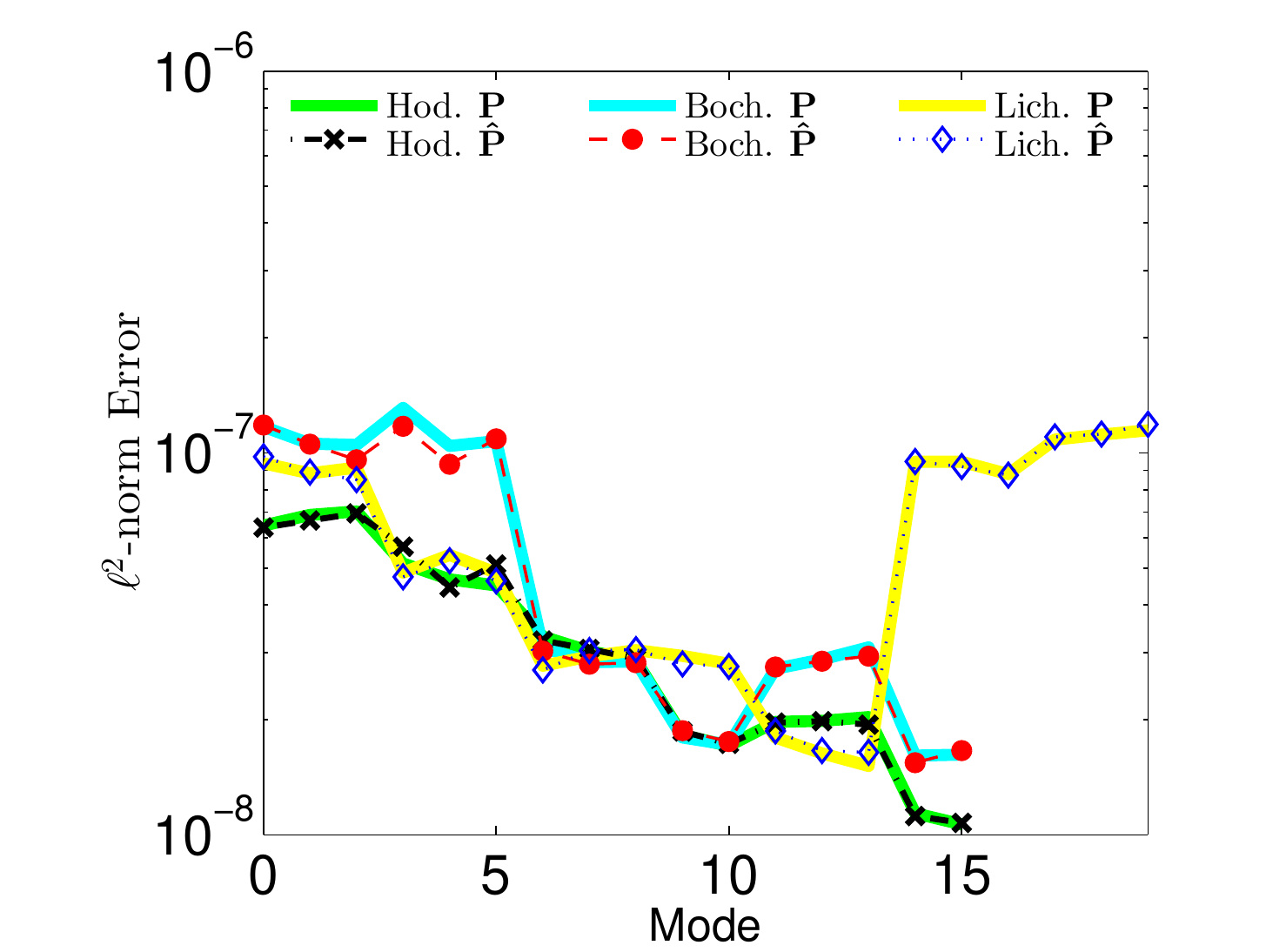}
\\
{\small (d) SRBF, Eigenvalues} & {\small (e) SRBF, Error of Eigenvalues} &
{\small (f) SRBF, Error of Eigenvectors } \\
\includegraphics[width=2.0
in, height=1.5 in]{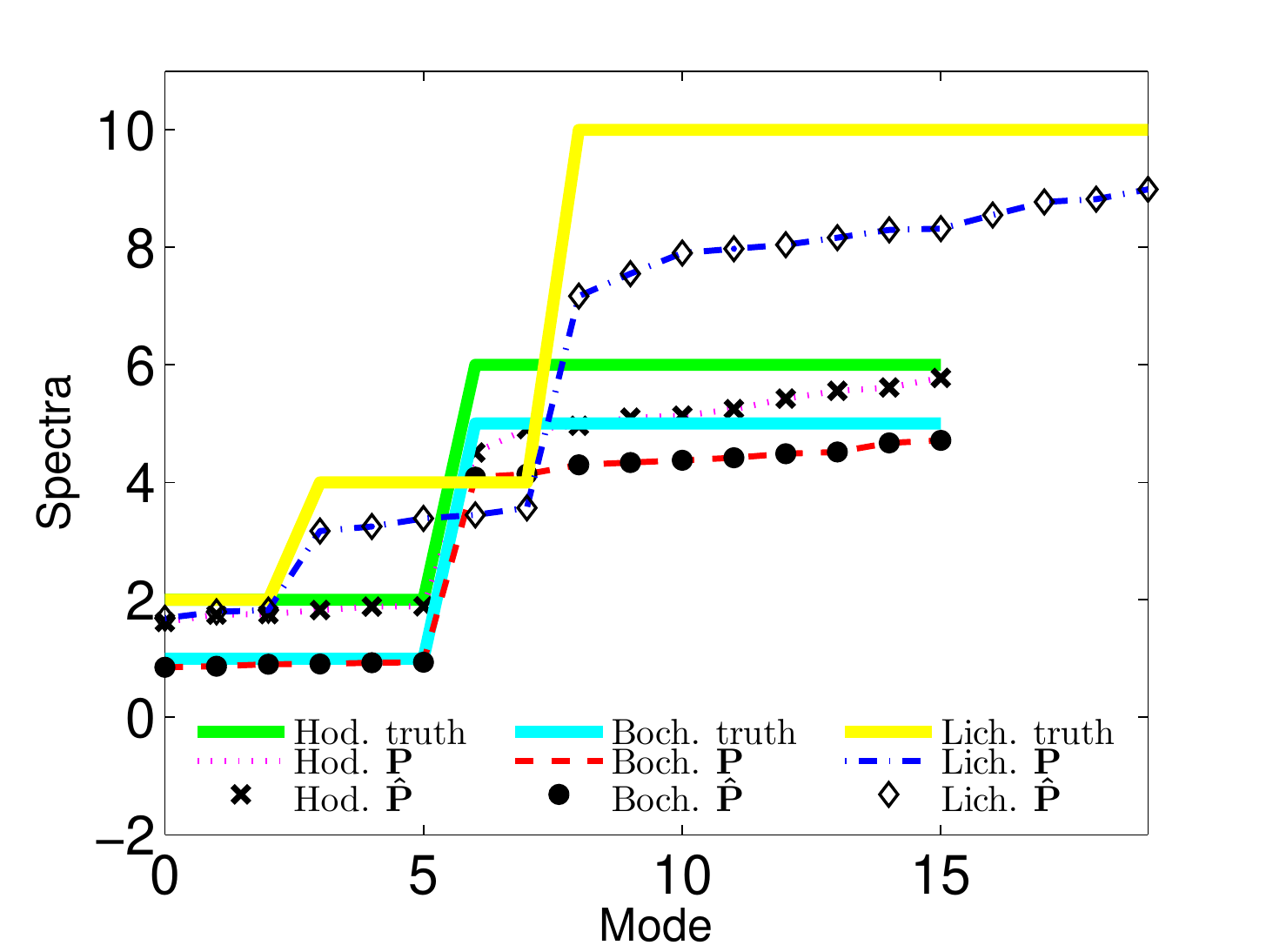}
&
\includegraphics[width=2.0
in, height=1.5 in]{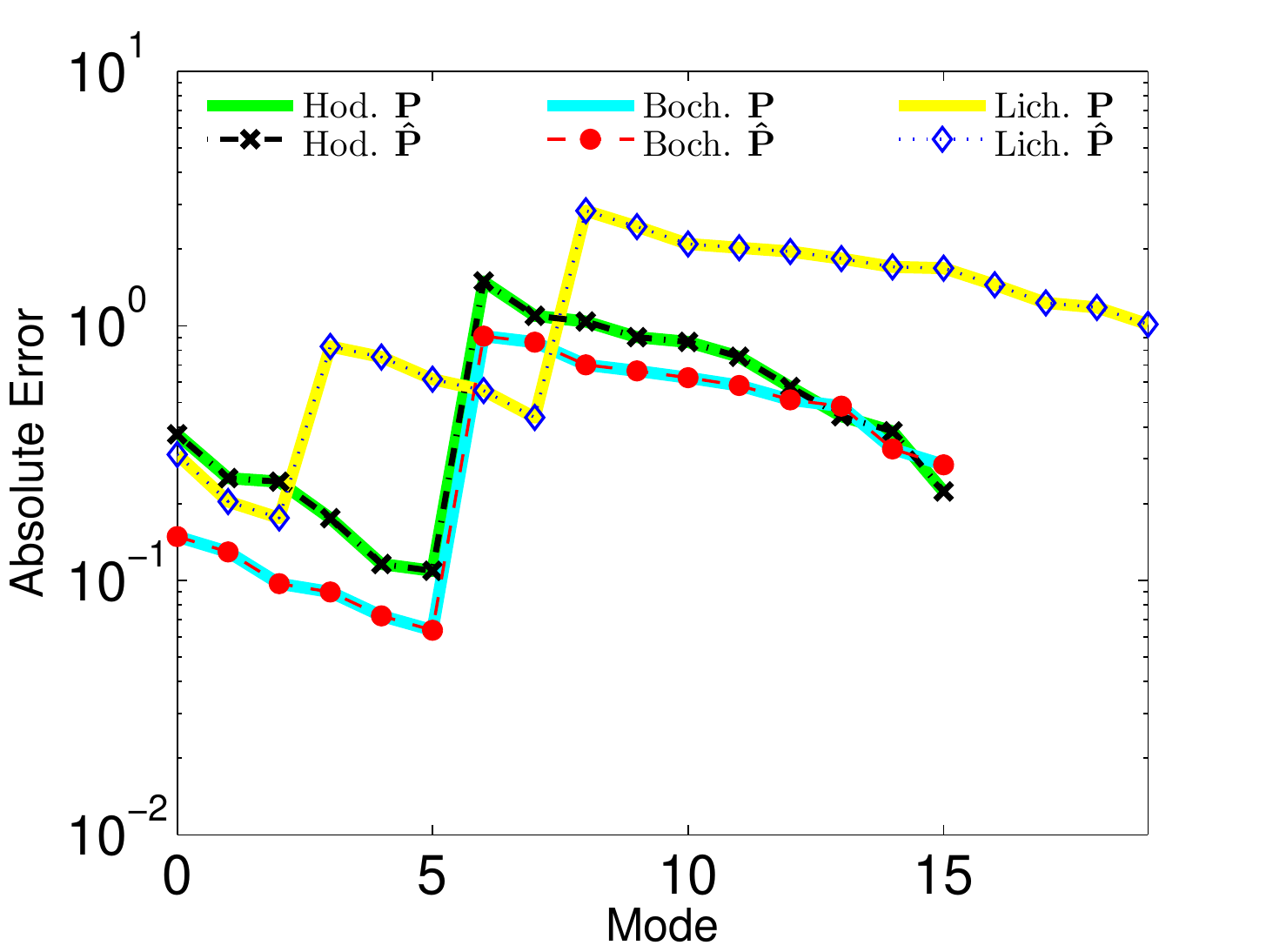}
&
\includegraphics[width=2.0
in, height=1.5 in]{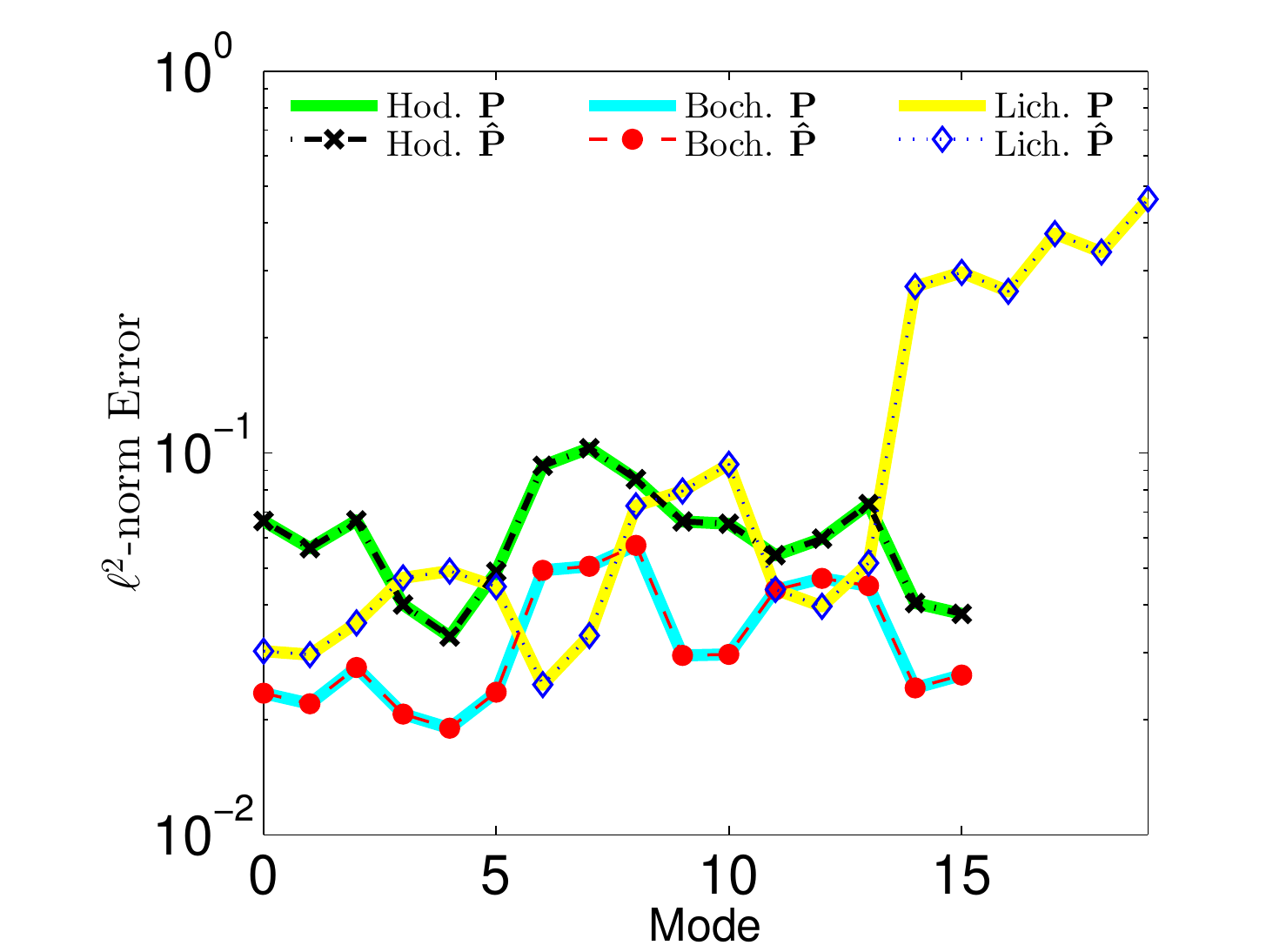}%
\end{tabular}
}
\caption{(Color online) 2D Sphere in $\mathbb{R}^{3}$. Comparison of (a)(d)
eigenvalues, (b)(e) error of eigenvalues, and (c)(f) error of
eigenvector-fields for Hodge, Bochner, Lichnerowicz Laplacians. (a)(b)(c)
correspond to NRBF and (d)(e)(f) correspond to SRBF results. For NRBF, GA
kernel with $s=1.0$ is used, and for SRBF, IQ kernel with $s=0.5$ is used.
The data points are randomly distributed with $N=1024$.}
\label{fig_vectorsphere_2}
\end{figure*}

Figure \ref{fig_vectorsphere_2} displays the comparison of the leading
eigenvalues and eigenvector fields for the Hodge, Bochner, and Lichnerowicz Laplacians. For NRBF, one can see from Fig. \ref{fig_vectorsphere_2}(a)-(c) that both eigenvalues and eigenvector fields can be approximated very accurately using either the analytic $\mathbf{P}$ or estimated $\mathbf{\hat{P}}$. This
small error result can be expected using analytic $\mathbf{P}$\ for the known manifold. It is a little unexpected that the estimation produces such a small error result using estimated $\mathbf{\hat{P}}$\ when the manifold is unknown. After further inspection, we found that our 2nd order SVD provides a super-convergence for $\mathbf{\hat{P}}$ on this particular 2D sphere example. For SRBF, one can see from Fig. \ref{fig_vectorsphere_2}(d)-(f) that eigenvalues and eigenvector fields can be approximated within certain errors for both $\mathbf{P}$
and $\mathbf{\hat{P}}$. This means that the Monte-Carlo error dominates the error for $\mathbf{\hat{P}}$\ in this example.





\begin{table}[tbp]
\caption{2D sphere in $\mathbb{R}^3$. Comparison of eigenvalues of Bochner,
Hodge, and Lichnerowicz Laplacians from NRBF and SRBF using approximated $%
\mathbf{\hat{P}}$. The $N=1024$ data points are randomly distributed on the
sphere. The eigenvalues of NRBF are shown with their  absolute
values when they are complex. For NRBF, GA kernel with $s=1.0$ is used, and
for SRBF, IQ kernel with $s=0.5 $ is used.}
\label{tab:eigformsphere_rand}
\par
\begin{center}
\scalebox{0.6}[0.6]{
\begin{tabular}{c|ccc|ccc|ccc||c|ccc|ccc|ccc}
\hline\hline
\rule{0pt}{13pt} $k$ & \text{Boch} & \text{Boch} & \text{Boch} & \text{Hodg} & \text{Hodg} & \text{Hodg} & \text{Lich} & \text{Lich} & \text{Lich} & $k$ & \text{Boch} & \text{Boch} & \text{Boch}  & \text{Hodg} & \text{Hodg} & \text{Hodg} & \text{Lich} & \text{Lich} & \text{Lich} \\
& \text{True} & \text{\footnotesize NRBF} & \text{\footnotesize SRBF} & \text{True} & \text{\footnotesize NRBF} & \text{\footnotesize SRBF} & \text{True} & \text{\footnotesize NRBF} & \text{\footnotesize SRBF} &  & \text{True} & \text{\footnotesize NRBF} & \text{\footnotesize SRBF}  & \text{True} & \text{\footnotesize NRBF} & \text{\footnotesize SRBF} & \text{True} & \text{\footnotesize NRBF} & \text{\footnotesize SRBF} \\ \hline
1&  1&  1.0& 0.85&  2&  2.0& 1.63&     0&     --&     --& 41& 19& 19.0& 15.7& 20& 20.0& 15.3& 28& 28.0& 19.0\\
2&  1&  1.0& 0.87&  2&  2.0& 1.75&     0&     --&     --& 42& 19& 19.0& 15.8& 20& 20.0& 15.5& 28& 28.0& 19.8\\
3&  1&  1.0& 0.90&  2&  2.0& 1.76&     0&     --&     --& 43& 19& 19.0& 16.1& 20& 20.0& 15.8& 28& 28.0& 20.4\\
4&  1&  1.0&  0.91&  2&  2.0& 1.82&  2&  2.0& 1.69& 44& 19& 19.0& 16.1& 20& 20.0& 15.8& 28& 28.0& 20.7\\
5&  1&  1.0& 0.93&  2&  2.0& 1.88&  2&  2.0&  1.80& 45& 19& 19.0& 16.4& 20& 20.0& 16.2& 28& 28.0& 21.4\\
6&  1&  1.0& 0.94&  2&  2.0& 1.89&  2&  2.0& 1.82& 46& 19& 19.0& 16.6& 20& 20.0& 16.4& 28& 28.0& 21.6\\
7&  5&  5.0&  4.09&  6&  6.0& 4.51&  4&  4.0& 3.17& 47& 19& 19.0& 17.0& 20& 20.0& 16.5& 28& 28.0& 22.1\\
8&  5&  5.0&  4.14&  6&  6.0& 4.91&  4&  4.0& 3.25& 48& 19& 19.0& 17.1& 20& 20.0& 16.7& 28& 28.0& 22.7\\
9&  5&  5.0&   4.30&  6&  6.0& 4.96&  4&  4.0& 3.38& 49& 29& 19.0& 18.9& 30& 30.0& 16.8& 28& 28.0& 23.3\\
10&  5&  5.0&  4.34&  6&  6.0&  5.10&  4&  4.0& 3.44& 50& 29& 19.0& 19.2& 30& 30.0& 17.3& 28& 28.0& 23.9\\
11&  5&  5.0&  4.38&  6&  6.0& 5.14&  4&  4.0& 3.56& 51& 29& 19.0& 19.7& 30& 30.0& 17.4& 38& 33.1& 24.2\\
12&  5&  5.0&  4.42&  6&  6.0& 5.24& 10& 10.0& 7.17& 52& 29& 19.0& 20.1& 30& 30.0& 17.6& 38& 33.4& 24.9\\
13&  5&  5.0&  4.49&  6&  6.0& 5.43& 10& 10.0& 7.55& 53& 29& 19.0& 20.9& 30& 30.0& 17.7& 38& 33.9& 25.6\\
14&  5&  5.0&  4.52&  6&  6.0& 5.56& 10& 10.0&  7.90& 54& 29& 19.0& 21.1& 30& 30.0& 17.8& 38& 34.5& 26.2\\
15&  5&  5.0&  4.67&  6&  6.0& 5.62& 10& 10.0& 7.97& 55& 29& 19.0& 21.4& 30& 30.0& 17.9& 38& 35.0& 26.4\\
16&  5&  5.0&  4.72&  6&  6.0& 5.78& 10& 10.0& 8.05& 56& 29& 19.0& 21.5& 30& 30.0& 18.2& 38& 35.2& 26.8\\
17& 11& 11.0&  8.27& 12& 12.0& 7.92& 10& 10.0& 8.17& 57& 29& 19.0& 22.1& 30& 30.0& 18.4& 38& 36.5& 27.0\\
18& 11& 11.0&  8.46& 12& 12.0& 8.38& 10& 10.0&  8.30& 58& 29& 19.0& 22.3& 30& 30.0& 18.6& 38& 36.9& 27.3\\
19& 11& 11.0&  8.65& 12& 12.0& 8.64& 10& 10.0& 8.32& 59& 29& 19.1& 22.6& 30& 30.0& 18.7& 38& 36.9& 27.8\\
20& 11& 11.0&   8.80& 12& 12.0& 9.26& 10& 10.0& 8.55& 60& 29& 19.1& 22.9& 30& 30.0& 18.8& 40& 37.7& 28.0\\
21& 11& 11.0&   8.90& 12& 12.0& 9.44& 10& 10.0& 8.77& 61& 29& 19.5& 23.6& 30& 30.0& 19.0& 40& 37.9& 28.3\\
22& 11& 11.0&  9.03& 12& 12.0& 9.61& 10& 10.0& 8.82& 62& 29& 20.1& 23.9& 30& 30.0& 19.1& 40& 37.9& 28.7\\
23& 11& 11.0&  9.13& 12& 12.0& 9.74& 10& 10.0& 8.99& 63& 29& 20.1& 24.2& 30& 30.0& 19.2& 40& 38.0& 29.3\\
24& 11& 11.0&  9.27& 12& 12.0& 9.99& 18& 18.0& 12.2& 64& 29& 20.6& 24.4& 30& 30.0& 19.5& 40& 38.0& 29.4\\
25& 11& 11.0&  9.47& 12& 12.0& 10.2& 18& 18.0& 12.7& 65& 29& 20.9& 24.8& 30& 30.0& 19.6& 40& 38.0& 30.2\\
26& 11& 11.0&  9.56& 12& 12.0& 10.5& 18& 18.0& 13.2& 66& 29& 21.1& 25.2& 30& 30.0& 19.8& 40& 38.0& 30.7\\
27& 11& 11.0&  9.87& 12& 12.0& 10.6& 18& 18.0& 13.4& 67& 29& 21.6& 25.7& 30& 30.0& 20.0& 40& 38.0& 31.3\\
28& 11& 11.0&  9.98& 12& 12.0& 10.9& 18& 18.0& 13.6& 68& 29& 21.6& 25.7& 30& 30.0& 20.1& 40& 38.0& 31.5\\
29& 11& 11.0&  10.3& 12& 12.0& 11.4& 18& 18.0& 13.8& 69& 29& 21.8& 26.1& 30& 30.0& 20.5& 40& 38.0& 32.4\\
30& 11& 11.0&  10.5& 12& 12.0& 11.6& 18& 18.0& 14.1& 70& 29& 22.0& 26.2& 30& 30.0& 20.7& 40& 38.0& 32.6\\
31& 19& 16.7&  13.3& 20& 20.0& 11.7& 18& 18.0& 14.7& 71& 41& 22.4& 26.9& 42& 42.0& 20.7& 40& 38.0& 33.7\\
32& 19& 16.8&  13.5& 20& 20.0& 12.6& 18& 18.0& 15.6& 72& 41& 22.8& 27.0& 42& 42.0& 21.0& 40& 38.5& 33.7\\
33& 19& 17.1&  13.7& 20& 20.0& 12.8& 22& 22.0& 15.9& 73& 41& 23.3& 27.4& 42& 42.0& 21.1& 54& 40.0& 34.8\\
34& 19& 17.4&  13.8& 20& 20.0& 13.1& 22& 22.0& 16.3& 74& 41& 23.7& 27.6& 42& 42.0& 21.2& 54& 40.0& 34.9\\
35& 19& 17.7&  14.1& 20& 20.0& 13.8& 22& 22.0& 16.6& 75& 41& 24.3& 28.0& 42& 42.0& 21.6& 54& 40.0& 35.7\\
36& 19& 18.0&  14.4& 20& 20.0& 13.9& 22& 22.0& 17.2& 76& 41& 25.4& 28.2& 42& 42.0& 21.7& 54& 40.0& 35.9\\
37& 19& 18.5&  14.6& 20& 20.0& 14.2& 22& 22.0& 17.3& 77& 41& 25.8& 28.6& 42& 42.0& 21.9& 54& 40.0& 36.4\\
38& 19& 18.6&  14.8& 20& 20.0& 14.6& 22& 22.0& 17.7& 78& 41& 29.0& 28.8& 42& 42.0& 22.0& 54& 40.0& 36.8\\
39& 19& 18.6&  15.2& 20& 20.0& 14.7& 22& 22.0& 18.3& 79& 41& 29.0& 29.2& 42& 42.0& 22.3& 54& 40.0& 37.2\\
40& 19& 18.9&  15.3& 20& 20.0& 15.1& 28& 28.0& 18.7& 80& 41& 29.0& 29.5& 42& 42.0& 22.5& 54& 40.0& 38.7\\
 \hline\hline
\end{tabular}
}
\end{center}
\end{table}

To inspect more of these spectra, we compare the leading eigenvalues of the truth, NRBF, and SRBF with $\mathbf{\hat{P}}$ up to $k=80$ in Table \ref{tab:eigformsphere_rand}. For NRBF, one can see from Table \ref{tab:eigformsphere_rand} that the leading NRBF eigenvalues are often in excellent agreement with the truth. Then, irrelevant eigenvalues may appear but are followed by the true eigenvalues again. For instance, the first $30$ Bochner eigenvalues are accurately approximated, the $31^{\mathrm{st}}\sim 40^{\mathrm{th}}$ eigenvalues are irrelevant (spectral pollution), and $41^{\mathrm{st}}\sim 58^{\mathrm{th}}$ eigenvalues become very accurate again. To further understand the irrelevant spectra phenomenon, we also solve the problem with different kernels and/or tune different shape parameters. Interestingly, we found that the leading $\sqrt{N} \sim 30$ true eigenvalues can always be accurately approximated even under different kernels or shape parameters. Unfortunately, the irrelevant eigenvalues also appear with different values starting from at least $k\geq\sqrt{N}$ for different kernels or shape parameters [not shown here]. In these experiments, the data points are always fixed. 
In contrast, SRBFs do not exhibit such an issue despite the fact that the eigenvalue estimates are not as accurate as NRBFs. Several advantages of SRBF are that their eigenvalues are real-valued, well-ordered, and do not produce irrelevant estimates.

\begin{figure*}[tbp]
{\scriptsize \centering
\begin{tabular}{ccc}
{\small (a) NRBF, Boch. Eigenvalues} & {\small (b) NRBF, Hodge Eigenvalues}
& {\small (c) NRBF, Lich. Eigenvalues} \\
\includegraphics[width=2.0
in, height=1.5 in]{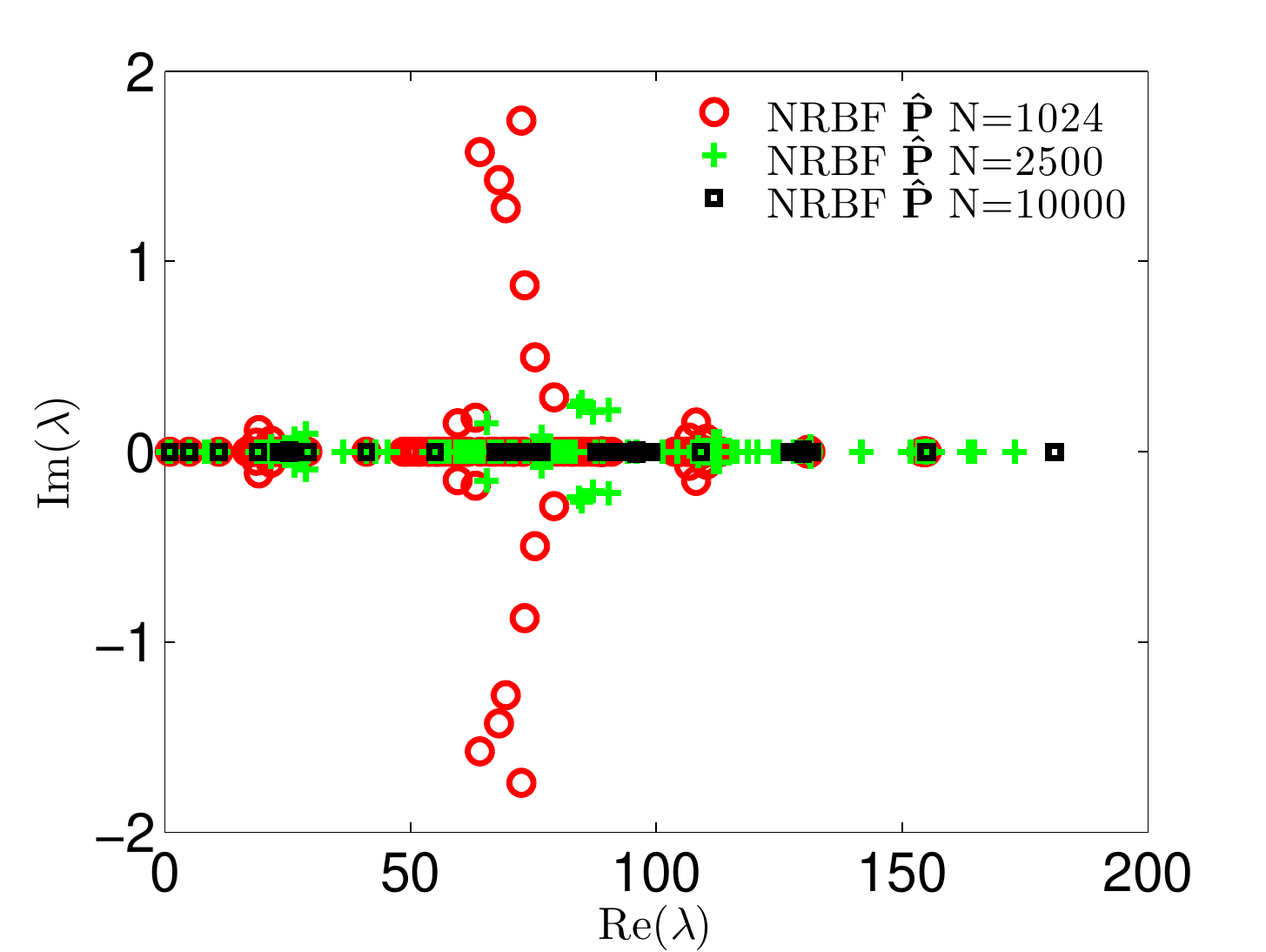}
&
\includegraphics[width=2.0
in, height=1.5 in]{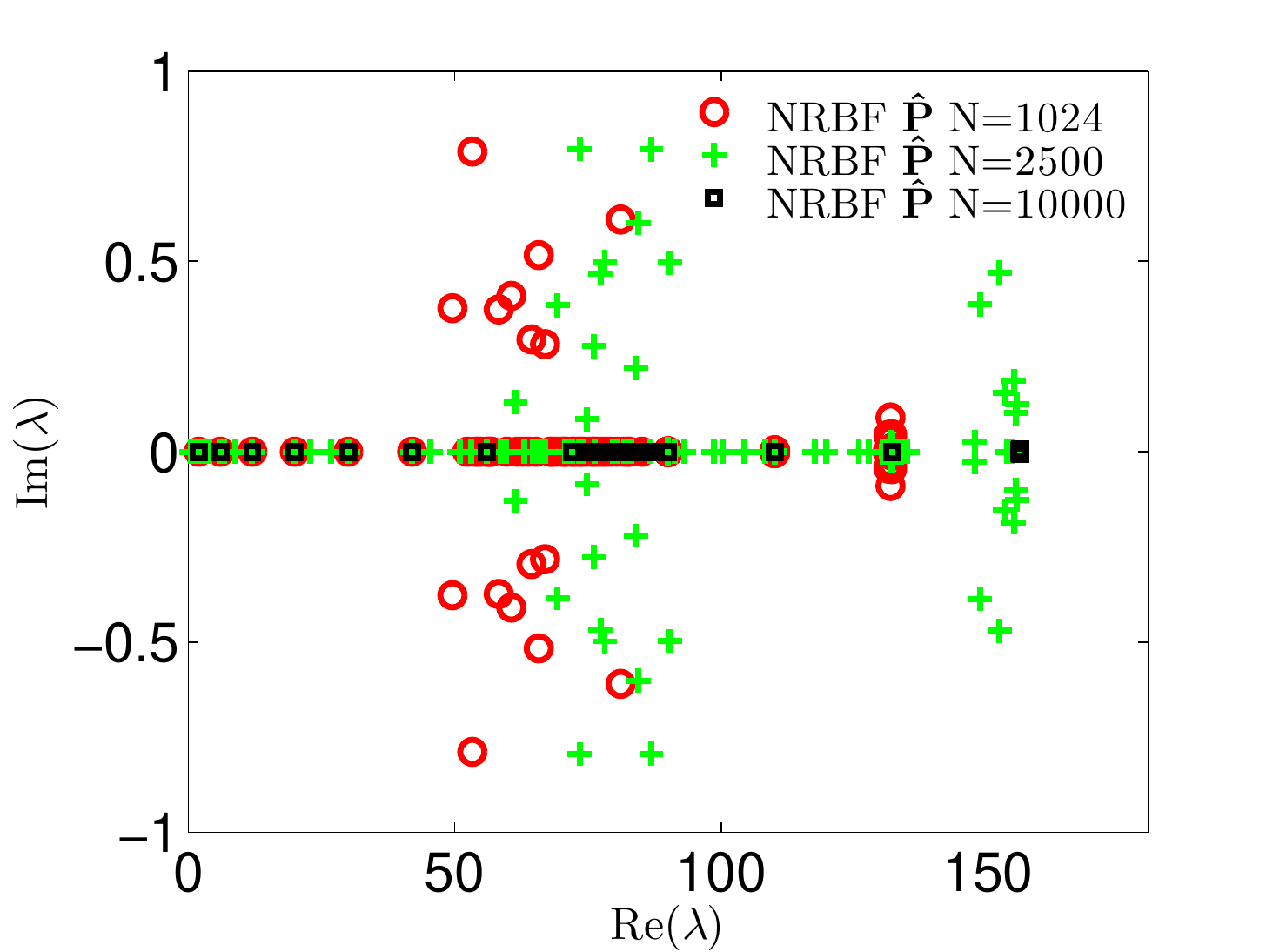}
&
\includegraphics[width=2.0
in, height=1.5 in]{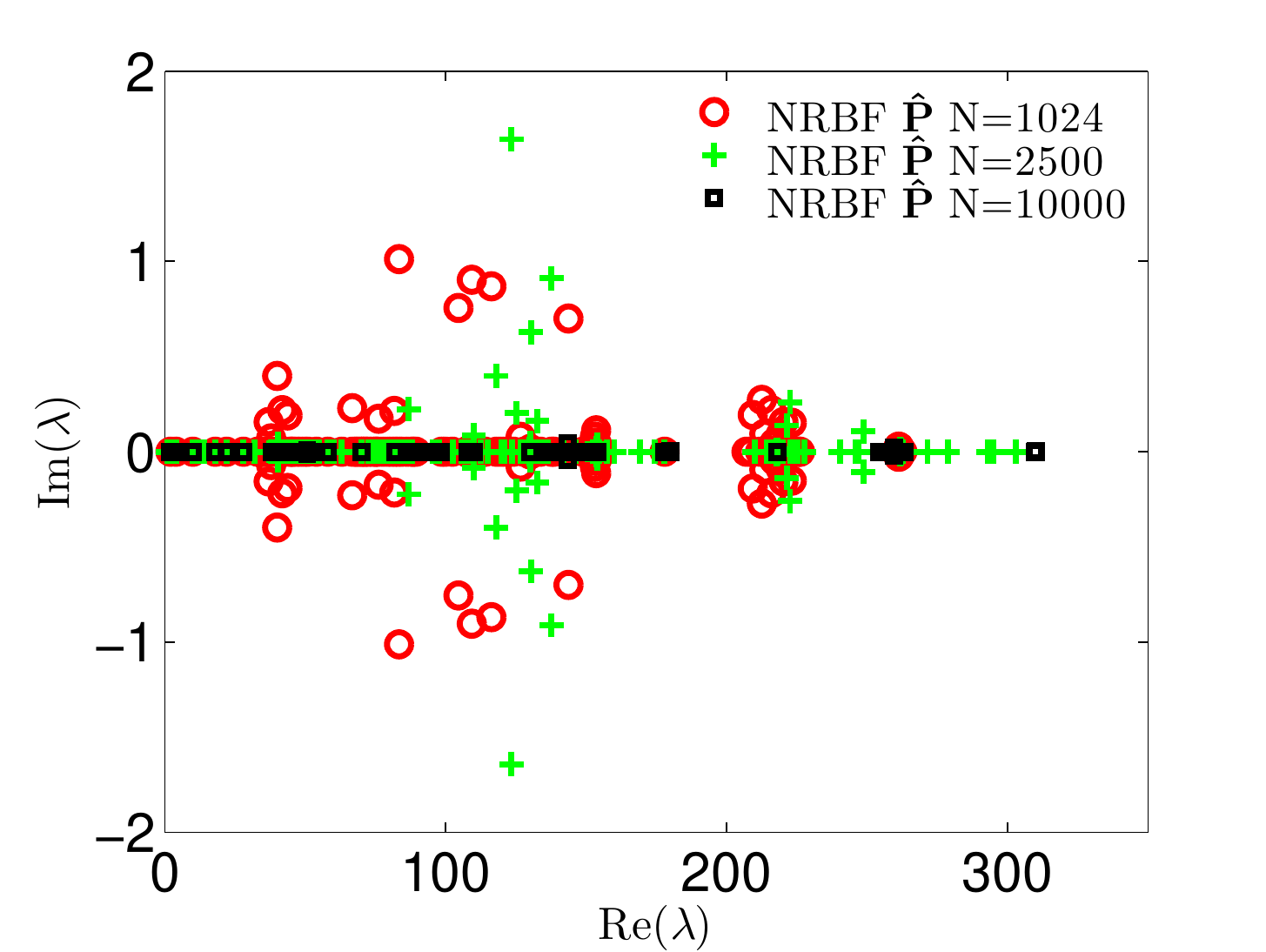}%
\end{tabular}
}
\caption{{\bf 2D Sphere in $\mathbb{R}^{3}$.} Convergence of
eigenvalues for NRBF using $\mathbf{\hat{P}}$ for (a) Bochner, (b) Hodge,
and (c) Lichnerowicz Laplacians. GA kernel with $s=1.0$ was fixed for
different $N$. The data points are randomly distributed.}
\label{fig_vectorsphere_3}
\end{figure*}

\begin{figure*}[htbp]
{\scriptsize \centering
\begin{tabular}{cc}
{\small (a) SRBF, Conv. of Eigenvalues } & {\small (b) SRBF, Conv. of Eigenvectors} \\
\includegraphics[width=3.0
in]{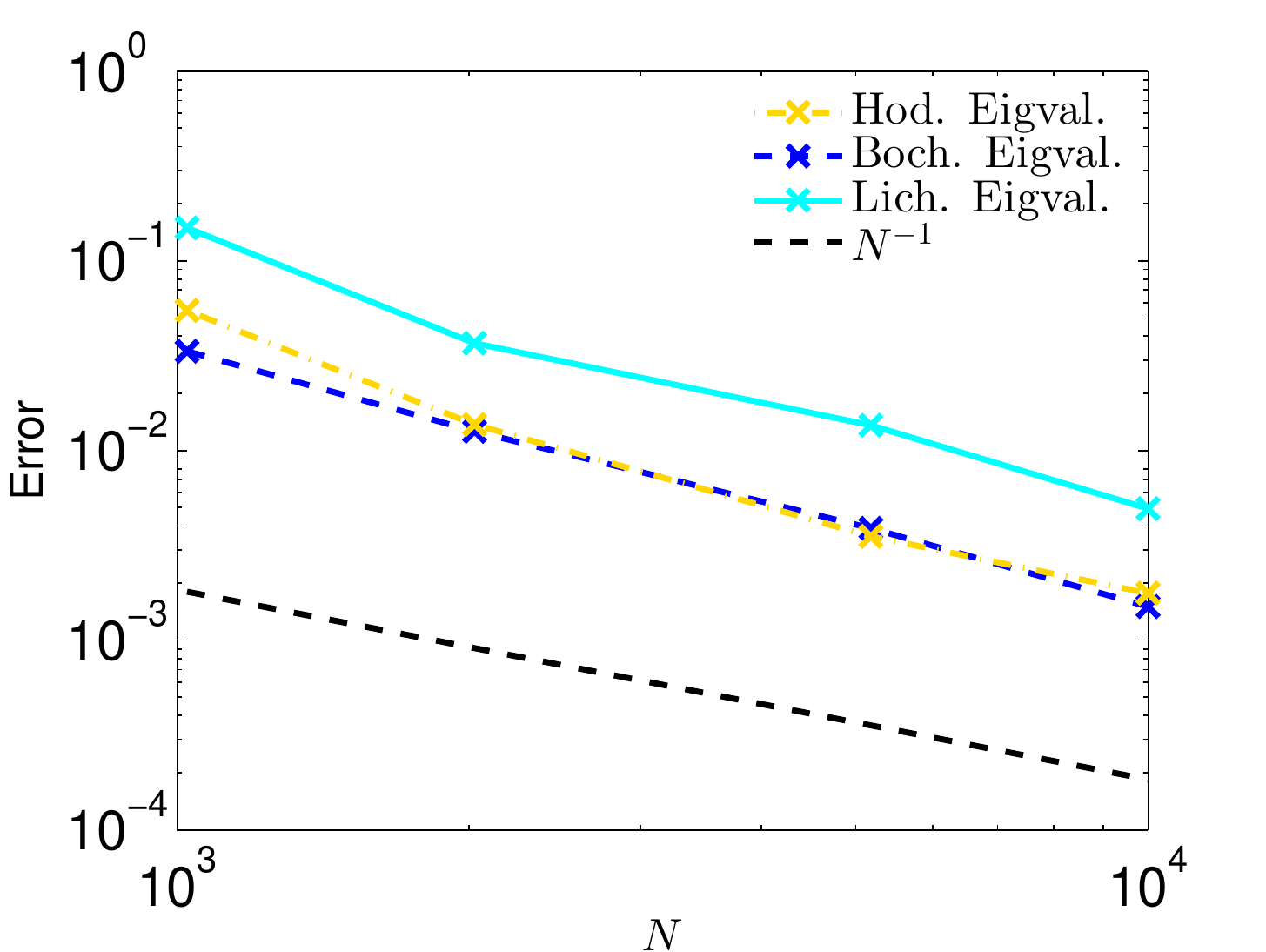} &
\includegraphics[width=3.0
in]{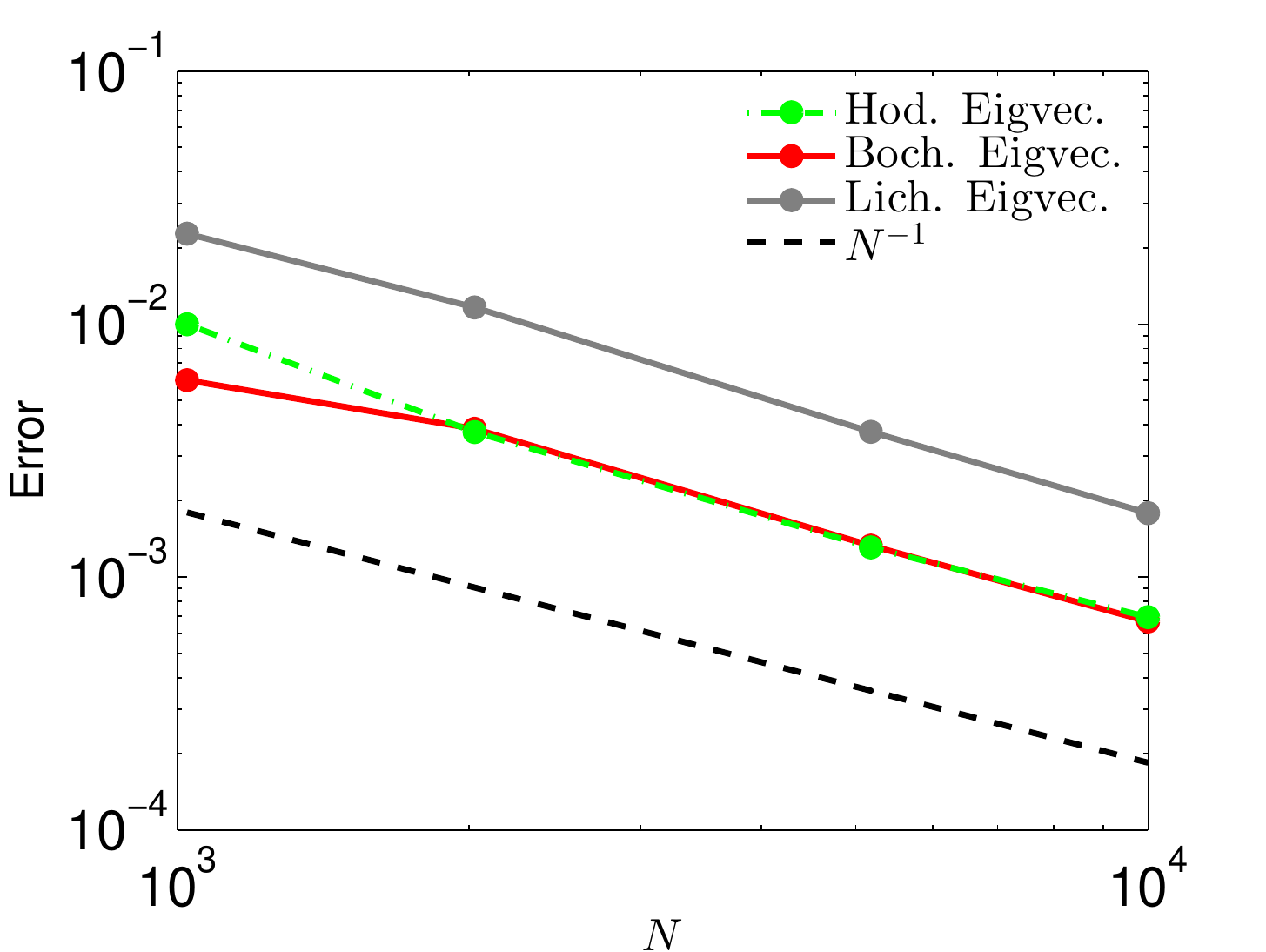} \\
\end{tabular}
\newline
\begin{tabular}{ccc}
{\small (c1) Hodge, Error of Eigenvalues} & {\small (d1) Bochner, Error of Eigenvalues} & {\small (e1) Lich., Error of Eigenvalues} \\
\includegraphics[width=2.0
in]{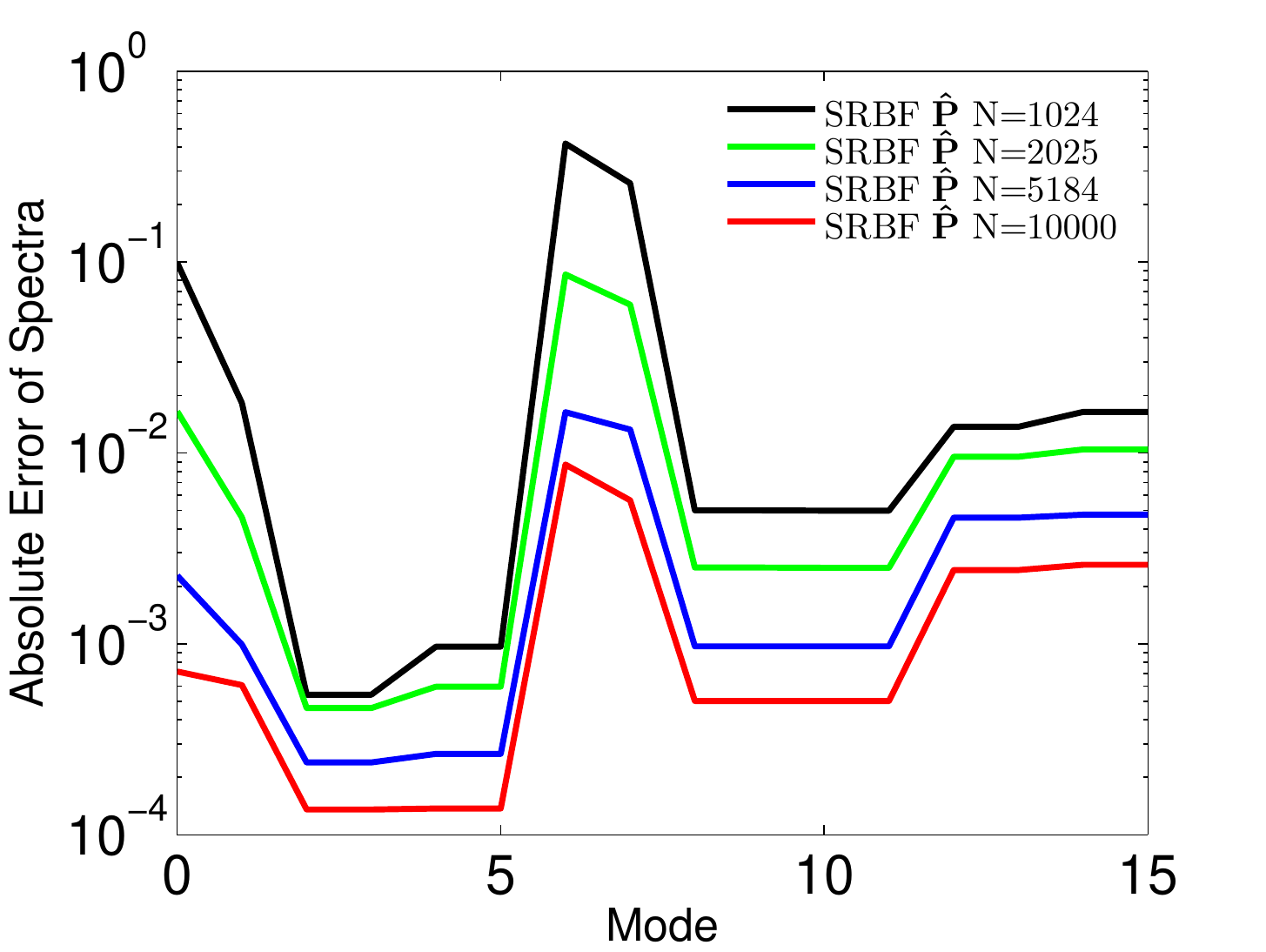} &
\includegraphics[width=2.0
in]{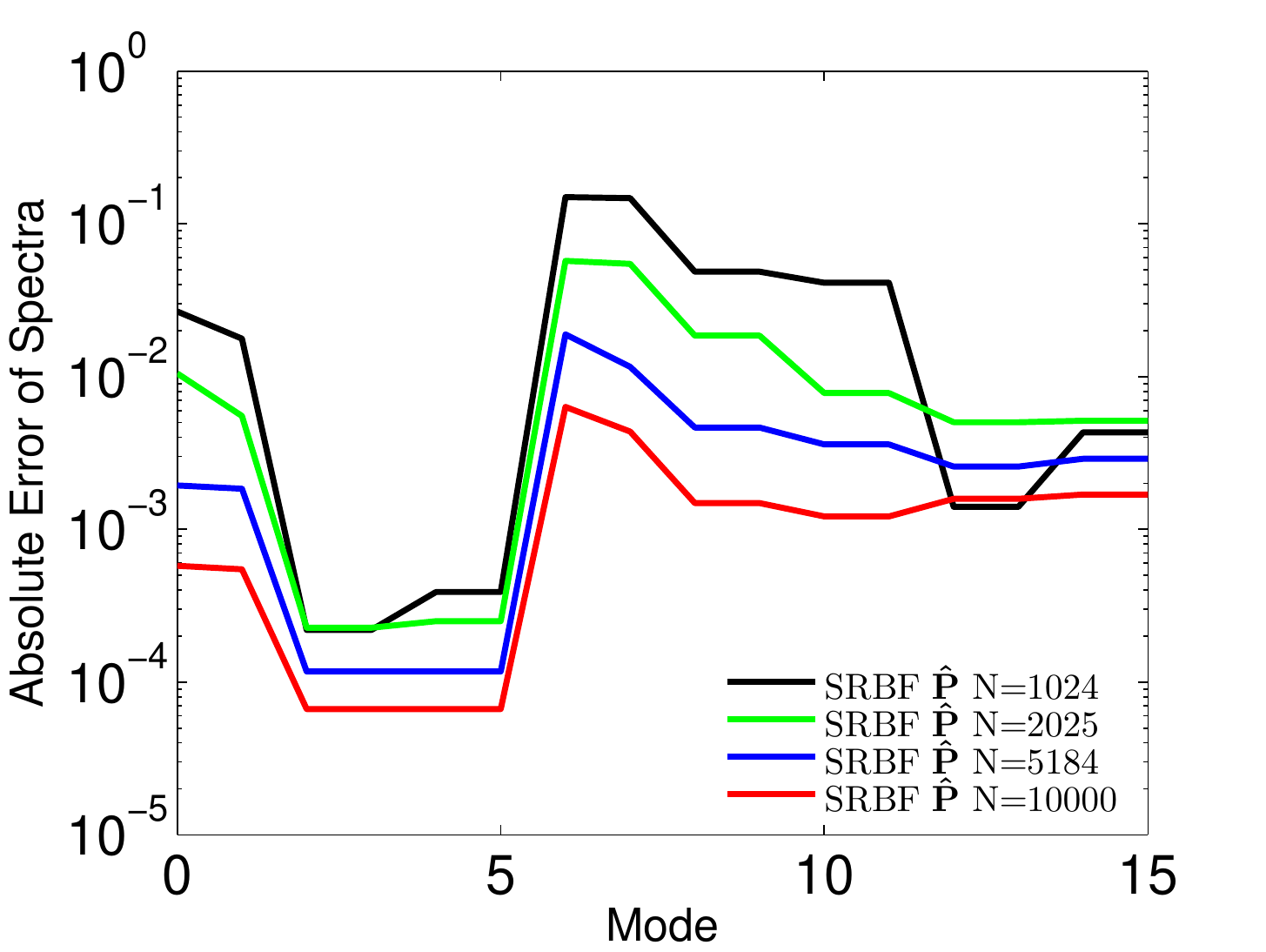} &
\includegraphics[width=2.0
in]{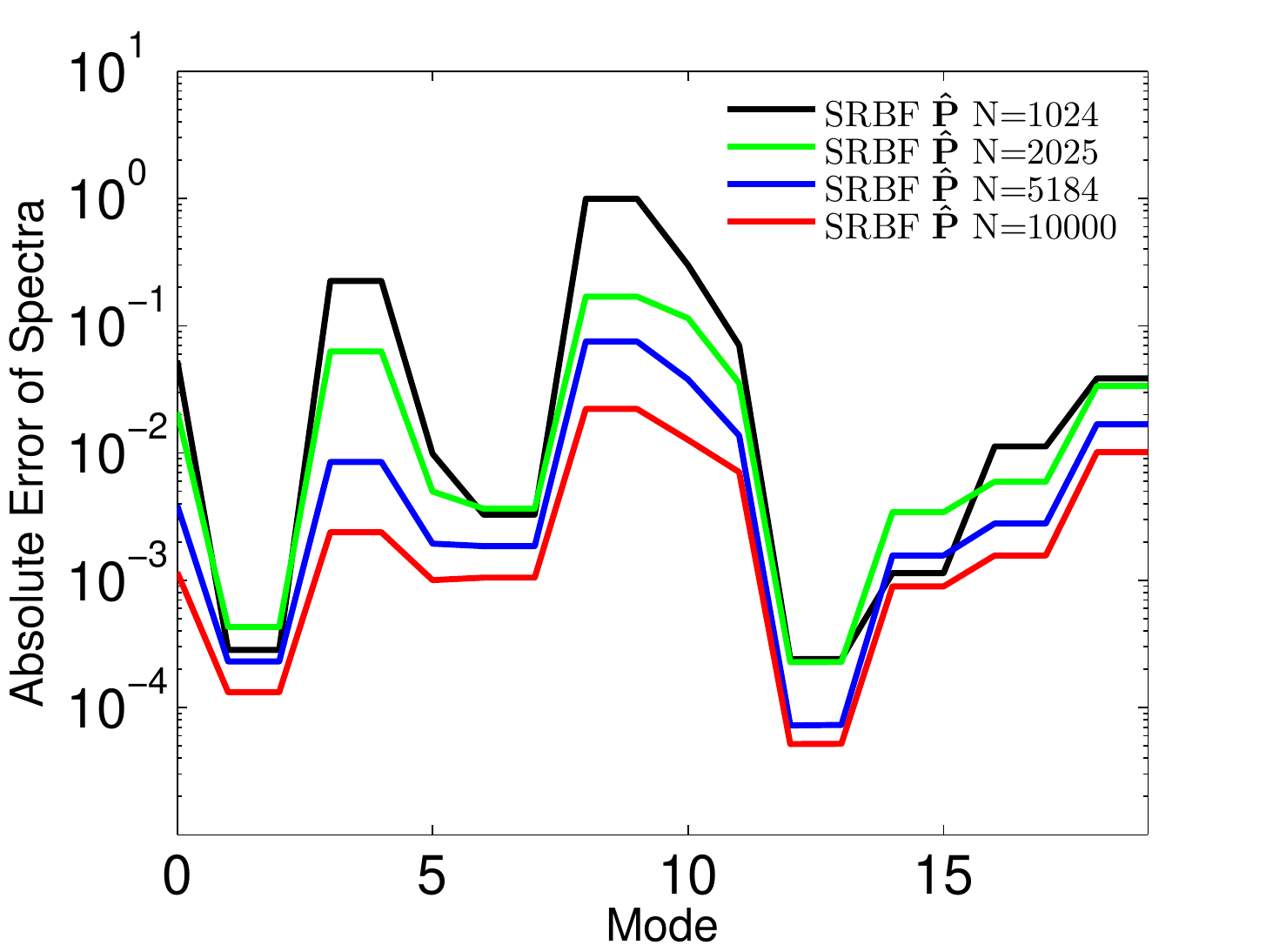}\\
{\small (c2) Hodge, Error of Eigenvecs.} & {\small (d2) Bochner,
Error of Eigenvecs.} & {\small (e2) Lich., Error of Eigenvecs.} \\
\includegraphics[width=2.08
in]{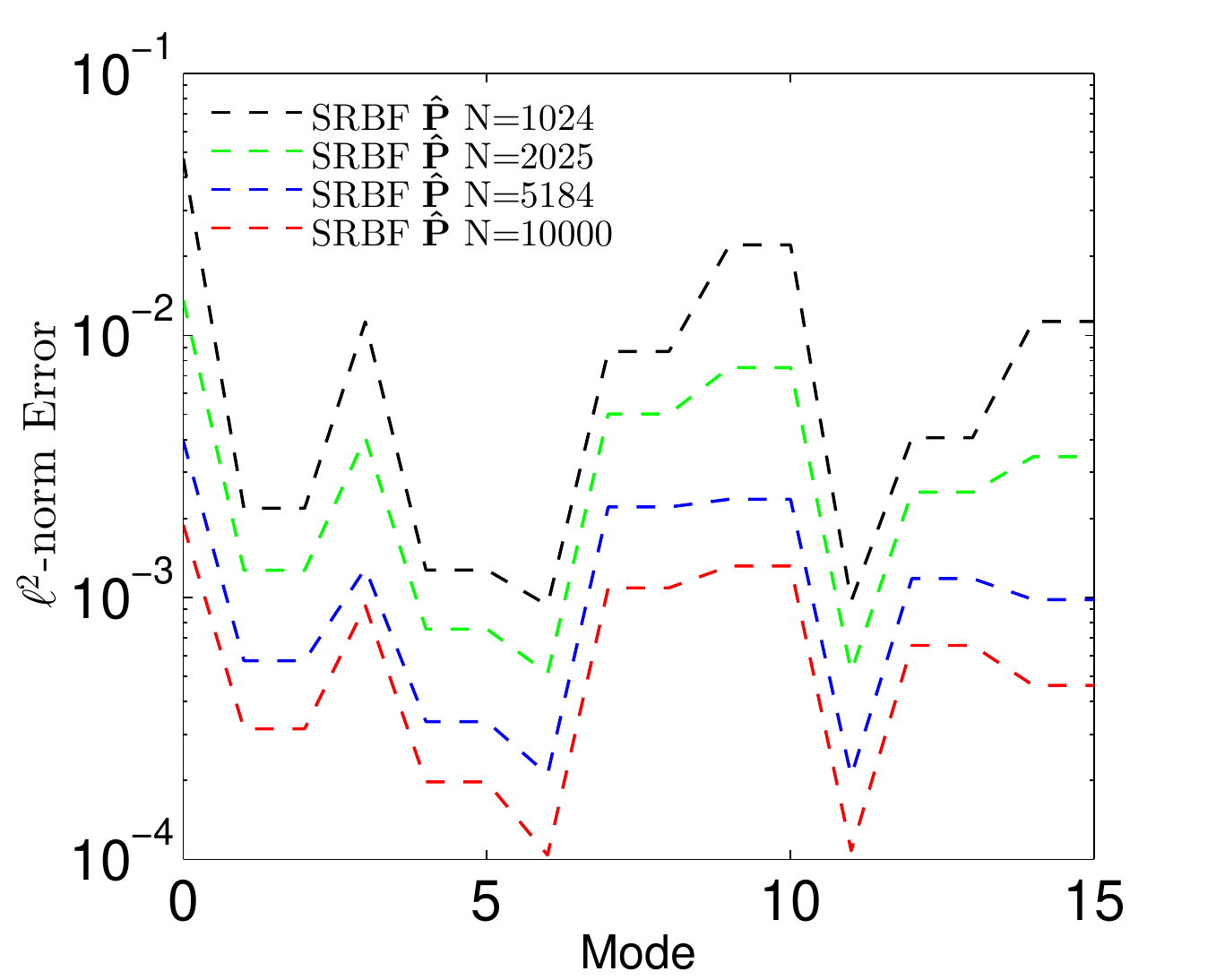} &
\includegraphics[width=2.08
in]{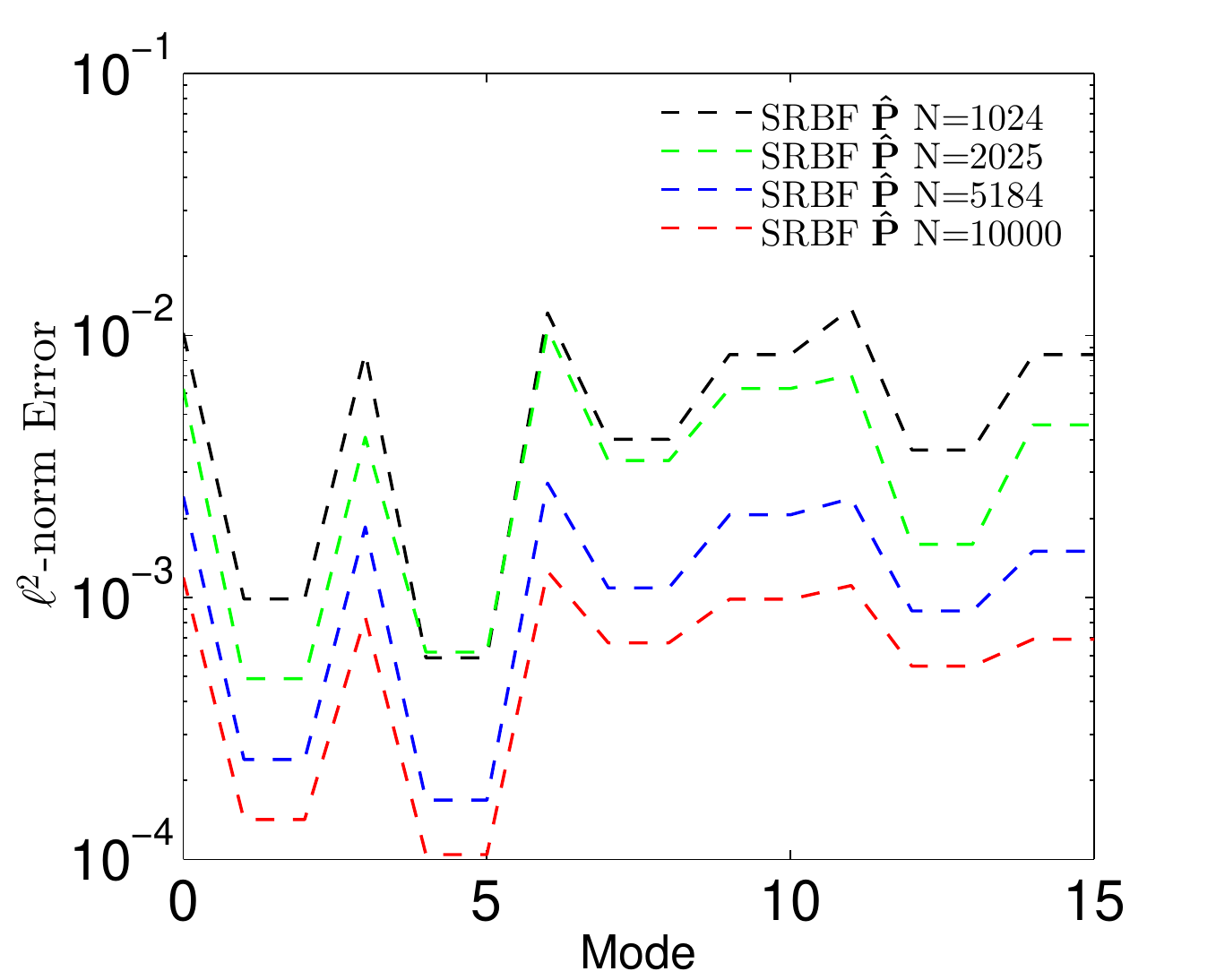} &
\includegraphics[width=2.08
in]{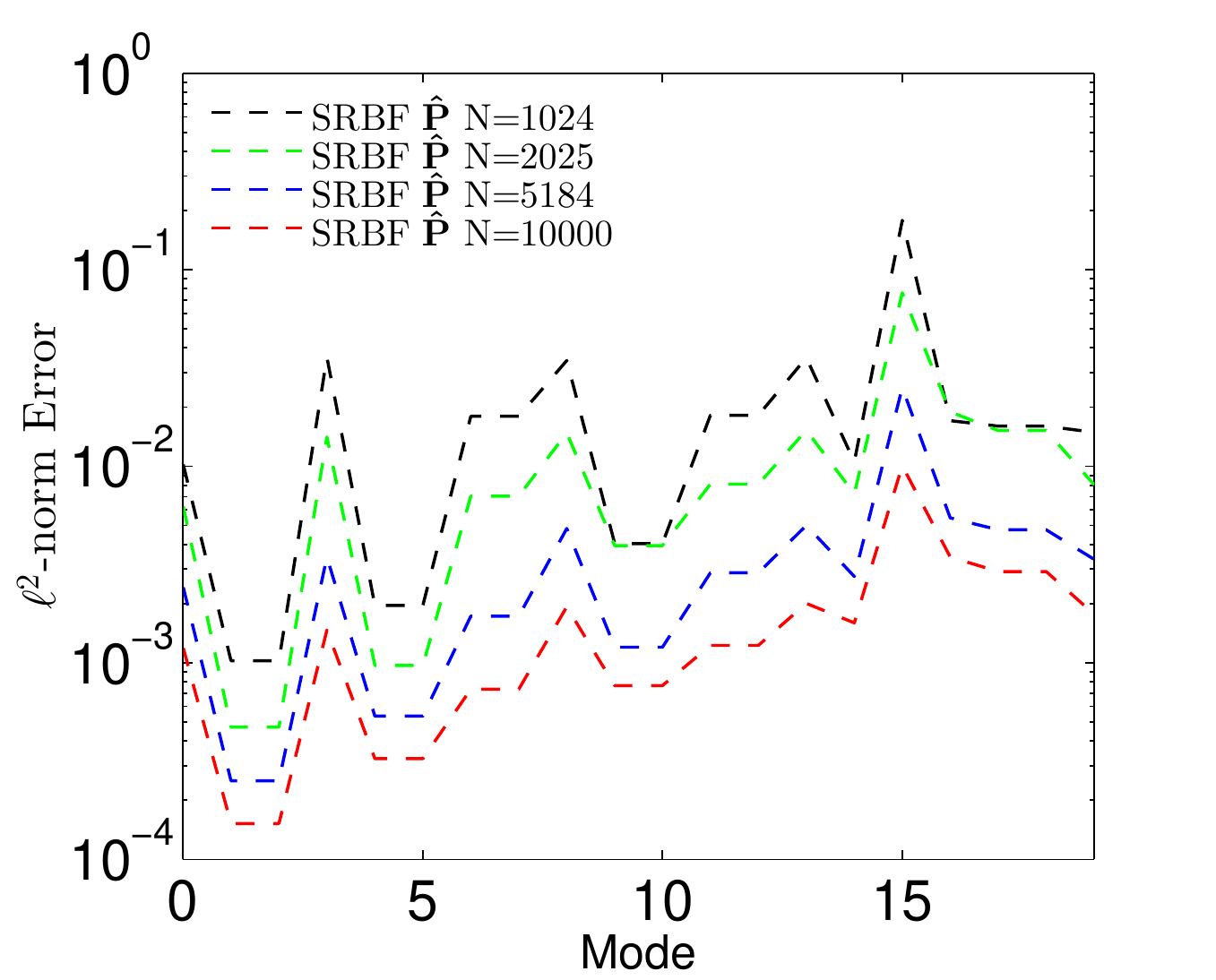}\\
\end{tabular}
}
\caption{{\bf 2D Sphere in $\mathbb{R}^{3}$.} The data points are
uniformly well-sampled on the manifold and the true sampling density $q$ was
used.
Convergence of (a) eigenvalues and (b) eigenvectors for the average of the
leading 16 modes for Bochner and Hodge Laplacians, and the average of the leading 20
modes for the Lichnerowicz Laplacian for SRBF using $\mathbf{\hat{P}}$ and true $q$ are shown.
Plotted in the second row are the errors of eigenvalues for (c1) Hodge, (d1) Bochner, and (e1) Lichnerowicz Laplacians for each leading mode for SRBF.  Plotted in (c2)-(e2) are the corresponding errors of eigenvectors. IQ kernel with $s=0.5$ was fixed for SRBF with different $N$.}
\label{fig_vectorsphere_4}
\end{figure*}

Figure \ref{fig_vectorsphere_3} shows the convergence of eigenvalues of NRBF methods for vector Laplacians. The data in this figure are randomly distributed. One can see that the four observations in Example \ref{sec:gentorus} are still valid for the behavior of NRBF
eigenvalues. Figure \ref{fig_vectorsphere_4} shows the convergence of
eigenvalues and eigenvectors of the SRBF method for vector Laplacians. For an efficient computational labor, we show results with well-sampled data points $\left\{ \theta _{i},\phi _{j}\right\}$, defined as equally spaced points in both direction of the intrinsic coordinates $[0,\pi ]\times \lbrack 0,2\pi )$ such that the total number of points are $N=32^2,  45^2, 71^2, 100^2$. We use the approximated $\mathbf{\hat{P}}$\ and the true sampling density $q$ to construct the Laplacians. One can see from Figs. \ref{fig_vectorsphere_4}(a) and (b) that the errors of eigenvalues and eigenvectors for all vector Laplacians decay on order of $N^{-1}$. More details can be found in Figs. \ref{fig_vectorsphere_4}(c)-(e) for the leading eigenmodes. Notice that for the well-sampled data, the errors here are relatively small compared to those in the case of random data shown in Figs. \ref{fig_vectorsphere_2}(d)-(f). Also, notice that the error rates here decay on order of $N^{-1}$ for well-sampled data, which is faster than Monte Carlo error rates of $N^{-1/2}$ for random data.

\section{Summary and Discussion}\label{sec7}

In this paper, we studied the Radial Basis Function (RBF) approximation to differential operators on smooth, closed Riemannian submanifolds of Euclidean space, identified by randomly sampled point cloud data. We summarize the findings in this paper and conclude with some open issues that arise from this study:
\begin{itemize}
\item The classical pointwise RBF formulation is an effective and accurate method for approximating general differential operators on closed manifolds. For unknown manifolds, i.e., identified only by point clouds, the approximation error is dominated by the accuracy of the estimation of local tangent space. With sufficient regularity on the functions and manifolds, we improve the local SVD technique to achieve a second-order accuracy, by appropriately accounting for the curvatures of the manifolds.
\item The pointwise non-symmetric RBF formulation can produce very accurate estimations of the leading eigensolutions of Laplacians when the size of data used to construct the RBF matrix $N$ is large enough and the projection operator $\mathbf{P}$ is accurately estimated. {We empirically found that the accuracy is significantly better than those produced by the graph-based method, which is supported by the results in Figure~\ref{fig_gentor_2}.
    In Figure~\ref{fig_gentor_4}, for sufficiently large $N$, one can further improve the estimates by increasing $N_p \gg N$, the number of points used to estimate $\mathbf{P}$.} While the second-order local SVD method is numerically efficient even for large $N_p$, the ultimate computational cost depends on the size of $N$ as we are solving the eigenvalue problem of dense, non-symmetric matrices of size $N\times N$ for Laplace-Beltrami or $nN\times nN$ for vector Laplacians. When both $N$ and $N_p$ are not large enough, this approximation produces eigensolutions that are not consistent with the spectral properties of the Laplacians in the sense that the eigenvalue estimates can be complex-valued and do not correspond to any of the underlying real-valued spectra. Since the spectral accuracy of NRBF relies on a very accurate estimation of the projection operator $\mathbf{P}$, we believe that this method may not be robust when the data is corrupted by noise, and thus, may not be reliable for manifold learning.
\item The proposed symmetric RBF approximation for Laplacians overcomes the issues in the pointwise formulation. The only caveat with the symmetric formulations is that this approximation may not be as accurate as the pointwise approximation whenever the latter works. Based on our analysis, the error of SRBF is dominated by the error of discretizing integrals in the weak formulation of appropriate Hilbert spaces, provided sufficiently smooth kernels are used. See Theorems~\ref{eigvalconv} and \ref{conveigvec} for the approximation of eigenvalues and eigenfunctions of the Laplace-Beltrami operator, respectively. See Theorems~\ref{eigvalconv Bochner} and \ref{conveigvec Bochner} for the approximation of eigenvalues and eigenvector fields of Bochner Laplacian, respectively. {Empirically, we found that SRBF consistently produces more accurate estimations of the non-leading eigenvalues and less accurate estimations of the leading eigenvalues compared to those obtained from diffusion maps algorithm for Laplacians acting on functions. These encouraging results suggest that the symmetric RBF discrete Laplacian approximation is potentially useful to complement graph-based methods when $d\ll n \ll N$, which arises in many PDE models on unknown sub-manifolds of moderately high dimensions, e.g., $n=O(10)$. Unlike graph-based approaches, which are restricted to approximating Laplacians, the general formulation in this paper supports PDE applications involving other differential operators. In a companion paper \cite{yan2023spectral}, we have introduced a spectral Galerkin framework to solve elliptic PDEs based on SRBF. }
{ \item The main limitation of this formulation is that it is not scalable for manifold learning problems where $n=O(N)$. One of the key computational issues with this formulation is that it requires $n$ number of dense $N\times N$ matrices $\mathbf{G}_i$ to be stored, which is memory intensive when $n=O(N)$. Further memory issues are encountered in the approximation of operators on vector fields and higher-order tensor fields. For instance, in our formulation, a Laplacian on vector fields is numerically resolved with an $nN\times nN$ matrix, which again requires significant memory whenever $n=O(N)$. An alternative approach to avoid this significant memory issue is to resort to an intrinsic approach \cite{liang2013solving}, where the authors use approximate tangent vectors to directly construct the metric tensor with a local polynomial regression and subsequently resolve the derivatives on manifolds using local tangent coordinates.
While they have demonstrated fast convergence of their method for solving PDEs on relatively low dimensional manifolds with low co-dimensions and their approach has been implemented to approximate differential operators on vector fields on 2-dimensional surfaces embedded in $\mathbb{R}^3$ \cite{gross2020meshfree}, in all of their examples the manifolds are well-sampled in the sense that the data points are visually regularly spaced. {It remains unclear how accurate and stable this method is on randomly sampled data as in our examples. In addition, the intrinsic formulation may be difficult to implement since the local representation of a specific differential operator requires a separate hardcode that depends on the inverse and/or derivatives of the metric tensor components that are being approximated.}
}
\item Parameter tuning: There are two key parameters to be tuned whenever RBF is used. First is the tolerance parameter in the pseudo-inverse, which determines the rank of the discrete approximation. A smaller tolerance parameter value increases the condition number, while too large of a tolerance parameter value reduces the number of non-trivial eigensolutions that can be estimated. Second, is the shape parameter of the RBF kernel. More accurate approximations seem to occur with smaller values of shape parameters, which yields a denser matrix such that the corresponding eigenvalue problem becomes expensive, especially for large data. This is the opposite of the graph-based approach which tends to be more accurate with smaller bandwidth parameters (inversely proportional to shape parameters) which induce sparser matrices. {In a forthcoming paper, we will avoid parameter tuning by replacing the radial basis interpolant with a local polynomial interpolant defined on the estimated tangent space.}
\item For the discrete approximation corresponding to an unweighted $L^2$-space, one needs to de-bias the sampling distribution induced by the random data with appropriate Monte-Carlo sampling weights (or density function). When the sampling density of the data is unknown, one needs to approximate it as in the graph-based approximation. Hence, the accuracy of the estimation depends on the accuracy of the density estimation, which we did not analyze in this paper as we only implement the classical kernel density estimation technique. In the case when the operator to be estimated is defined weakly over an $L^2$-space that is weighted with respect to the sampling distribution, then one does not need to estimate the sampling density function. This is in contrast to the graph Laplacian-based approach. Particularly, graph Laplacian approximation based on the diffusion maps asymptotic expansion imposes a ``right normalization'' over the square root of the approximate density.
{\item For the non-symmetric approximation of the Laplacians, the spectral consistency of the approximation in the limit of large data still remains open. While this approximation is subject to spectral pollution, it is worthwhile to consider or develop methods to identify spectra without pollution, following ideas in \cite{colbrook2021rigorous}.}
\item While we believe the symmetric formulation should be robust to {small amplitude noises in the ambient space as numerically studied in our companion paper \cite{yan2023spectral},} this claim remains to be theoretically justified. Particularly, further investigation is needed to extend Theorem~\ref{local svd result} for noisy data.

\end{itemize}

\section*{Acknowledgment}
The research of J.H. was partially supported by the NSF grants DMS-1854299, DMS-2207328, and the ONR grant N00014-22-1-2193. S. J. was supported by the NSFC Grant No. 12101408 and the HPC Platform of ShanghaiTech University. S. J. also thanks Chengjian Yao for useful discussion.


\newpage

\appendix
\section{More Operators on Manifolds}\label{app:A}

\subsection{Proof for Proposition \ref{prop2p2}}

We first provide the proof for equation (\ref{ambient cristoffel}) in Proposition \ref{prop2p2}.
\begin{proof}
From a direct calculation, we have
\BEA
\sum_r g^{ij} \frac{\partial X^r}{\partial \theta^j} \frac{\partial U^r }{\partial \theta^k} =  \sum_r g^{ij} \frac{\partial X^r}{\partial \theta^j} \frac{\partial  }{\partial \theta^k} \left( u^p \frac{\partial X^r}{\partial \theta^p} \right) =  \sum_r g^{ij} \frac{\partial X^r}{\partial \theta^j} \left( \frac{\partial u^p}{\partial \theta^k} \frac{\partial X^r}{\partial \theta^p} + u^p\frac{\partial^2 X^r}{ \partial \theta^k \partial \theta^p} \right),   \notag
\EEA
where we have used \eqref{sec2.2:eq1} in the first equality above. Using the fact that $\sum_r \frac{\partial X^r}{\partial \theta^i} \frac{\partial X^r}{\partial \theta^k}= g_{ik}$,
\begin{small}
\BEA
 \sum_r g^{ij} \frac{\partial X^r}{\partial \theta^j} \frac{\partial U^r }{\partial \theta^k} = g^{ij} g_{jp} \frac{\partial u^p}{\partial \theta^k} + \sum_{r} g^{ij} u^p \left( \frac{\partial X^r}{\partial \theta^j}\frac{\partial^2 X^r}{ \partial \theta^k \partial \theta^p} \right) = \frac{\partial u^i}{\partial \theta^k} + \sum_{r} g^{ij} u^p \left( \frac{\partial X^r}{\partial \theta^j}\frac{\partial^2 X^r}{ \partial \theta^k \partial \theta^p} \right). \label{prop2.2:eq1}
\EEA
\end{small}
It remains to observe that
\begin{footnotesize}
\BEA
\sum_{r}\left( \frac{\partial X^r}{\partial \theta^j}\frac{\partial^2 X^r}{ \partial \theta^k \partial \theta^p} \right) = \left\langle \frac{\partial \mathbf{X}}{\partial \theta^j}, \frac{\partial^2 \mathbf{X}}{ \partial \theta^k \partial \theta^p}  \right\rangle = \left\langle \frac{\partial \mathbf{X}}{\partial \theta^j},   \sum_r \Gamma^r_{pk} \frac{\partial \mathbf{X}}{ \partial \theta^r}  + \mathbf{n} \right\rangle  = \sum_r \Gamma^r_{pk}  \left\langle \frac{\partial \mathbf{X}}{\partial \theta^j}, \frac{\partial \mathbf{X}}{\partial \theta^k}\right\rangle = \sum_r g_{jr} \Gamma^r_{pk},  \notag
\EEA
\end{footnotesize}
where we have used the fact that $\frac{\partial^2 \mathbf{X}}{ \partial \theta^k \partial \theta^p}$ is a linear combination of partials of $\mathbf{X}$ with Christoffel symbols as coefficients, plus a vector $\mathbf{n}$ orthogonal to the tangent space. Plugging this to \eqref{prop2.2:eq1} yields Equation \eqref{ambient cristoffel}.
\end{proof}

In the remainder of this appendix, we now show explicitly how other differential operators
can be approximated using the tangential projection formulation. In particular, we derive formulas for approximations of  the Lichnerowicz and Hodge Laplacians, as well as the covariant derivative. Before deriving such results, we explicitly compute the matrix form of divergence of a $(2,0)$ tensor field, which will be needed in the following subsections. Throughout the following calculations, we will use the following shorthand notation for simplification:
 $$
 v_{,s}^{jk}:=v^{mk}\Gamma _{sm}^{j}+%
\frac{\partial v^{jk}}{\partial \theta ^{s}}+v^{jm}\Gamma _{sm}^{k},
$$
where $v = v^{jk} \frac{\partial}{\partial \theta^j} \otimes \frac{\partial}{\partial \theta^k}$ is a $(2,0)$ tensor field.
While we do not prove convergence in the probabilistic sense, the authors suspect that techniques similar to those used in Section \ref{section4} can be used to obtain the convergence results for the other differential operators.   In Section \ref{sec6}, convergence is demonstrated numerically for the approximations derived in this appendix.
\subsection{Derivation of Divergence of a (2,0) Tensor Field}
Let $v$ be a $(2,0)$ tensor field of $v=v^{jk}%
\frac{\partial }{\partial \theta ^{j}}\otimes \frac{\partial }{\partial
\theta ^{k}}$. The divergence of $v$ is defined as
\BEA
\mathrm{div}_{1}^{r}\left( v\right)
=C_{1}^{r}(\nabla v) \notag
\EEA
for $r=1$ or $r=2$, where $C^r_1$ denotes the contraction. More explicitly, the divergence $%
\mathrm{div}_{1}^{r}\left( v\right) $ can be calculated as,
\begin{equation}
\mathrm{div}_{1}^{r}\left( v\right) =C_{1}^{r}\left(
v_{,s}^{jk}\frac{\partial }{\partial \theta ^{j}}\otimes \frac{\partial }{%
\partial \theta ^{k}}\otimes \mathrm{d}\theta ^{s}\right) =\left\{
\begin{array}{cc}
v_{,j}^{jk}\frac{\partial }{\partial \theta ^{k}}, & r=1 \\
v_{,k}^{jk}\frac{\partial }{\partial \theta ^{j}} & r=2%
\end{array}%
\right. ,  \label{eqn:divgv}
\end{equation}%
where  $v_{,s}^{jk}=v^{mk}\Gamma _{sm}^{j}+%
\frac{\partial v^{jk}}{\partial \theta ^{s}}+v^{jm}\Gamma _{sm}^{k}$. For any $p\in M$, we can extend $v$ to $V$, defined on a neighborhood of $p$ with ambient space representation
\begin{equation*}
V=\frac{\partial X^{s}}{\partial \theta ^{j}}v^{jk}\frac{\partial X^{t}}{%
\partial \theta ^{k}}\frac{\partial }{\partial X^{s}}\otimes \frac{\partial
}{\partial X^{t}}\equiv V^{st}\frac{\partial }{\partial X^{s}}\otimes \frac{%
\partial }{\partial X^{t}},
\end{equation*}%
so that $V|_{p}=v|_{p}$ agrees at $p\in M$. We show the
following formula for the divergence for a $(2,0)$ tensor field:%
\begin{equation}
\mathrm{div}_{1}^{r}\left( v\right) = {\mathcal{P}}C_{1}^{r}\left( {\mathcal{P}_3}\bar{\nabla}_{\mathbb{R%
}^{n}}V\right) ,  \label{eqn:divrv}
\end{equation}%
where the right hand side is defined as%
\begin{eqnarray*}
{\mathcal{P}}C_{1}^{r}\left( {\mathcal{P}_3}\bar{\nabla}_{\mathbb{R}^{n}}V\right)
&\equiv &{\mathcal{P}}C_{1}^{r}\left( \left[ \delta _{cq}\frac{\partial X^{c}}{%
\partial \theta ^{a}}g^{ab}\frac{\partial X^{d}}{\partial \theta ^{b}}%
\mathrm{d}X^{q}\otimes \frac{\partial }{\partial X^{d}}\right] \left[ \frac{%
\partial V^{st}}{\partial X^{m}}\frac{\partial }{\partial X^{s}}\otimes
\frac{\partial }{\partial X^{t}}\otimes \mathrm{d}X^{m}\right] \right) \\
&=&{\mathcal{P}}C_{1}^{r}\left( \delta _{cq}\frac{\partial X^{c}}{\partial
\theta ^{a}}g^{ab}\frac{\partial X^{m}}{\partial \theta ^{b}}\frac{\partial
V^{st}}{\partial X^{m}}\frac{\partial }{\partial X^{s}}\otimes \frac{%
\partial }{\partial X^{t}}\otimes \mathrm{d}X^{q}\right).
\end{eqnarray*}
Here, $\mathcal{P}_3$ is as defined in \eqref{eqn:Pmatr} and $\mathcal{P}_3$ acts on the last 1-form component $\mathrm{d}X^m$. For $r=1$, RHS of (\ref{eqn:divrv}) can be calculated as,
\begin{eqnarray*}
&&{\color{black}\mathcal{P}}C_{1}^{r}\left( {\color{black}\mathcal{P}_3}\bar{\nabla}_{\mathbb{R}^{n}}V\right)\\
&=& {\color{black}\mathcal{P}}\left( \delta _{cs}\frac{\partial X^{c}}{\partial \theta ^{a}}%
g^{ab}\frac{\partial X^{m}}{\partial \theta ^{b}}\frac{\partial V^{st}}{%
\partial X^{m}}\frac{\partial }{\partial X^{t}}\right) = {\color{black}\mathcal{P}}\left(
\sum_{s=1}^{n}\frac{\partial X^{s}}{\partial \theta ^{a}}g^{ab}\frac{%
\partial }{\partial \theta ^{b}}\left( \frac{\partial X^{s}}{\partial \theta
^{j}}v^{jk}\frac{\partial X^{t}}{\partial \theta ^{k}}\right) \frac{\partial
}{\partial X^{t}}\right) \\
&=& {\color{black}\mathcal{P}}\left( \sum_{s=1}^{n}\frac{\partial X^{s}}{\partial \theta ^{a}%
}g^{ab}\left( \frac{\partial ^{2}X^{s}}{\partial \theta ^{b}\partial \theta
^{j}}v^{jk}\frac{\partial X^{t}}{\partial \theta ^{k}}+\frac{\partial X^{s}}{%
\partial \theta ^{j}}\frac{\partial v^{jk}}{\partial \theta ^{b}}\frac{%
\partial X^{t}}{\partial \theta ^{k}}+\frac{\partial X^{s}}{\partial \theta
^{j}}v^{jk}\frac{\partial ^{2}X^{t}}{\partial \theta ^{b}\partial \theta ^{k}%
}\right) \frac{\partial }{\partial X^{t}}\right) ,
\end{eqnarray*}%
where the second line follows from Leibniz rule. Unraveling definitions, one obtains%
\begin{eqnarray*}
{\mathcal{P}}C_{1}^{r}\left( {\color{black}\mathcal{P}_3}\bar{\nabla}_{\mathbb{R}^{n}}V\right)
&=& {\color{black}\mathcal{P}}\left( \left( g^{ab}v^{jk}\frac{\partial X^{t}}{\partial
\theta ^{k}}g_{am}\Gamma _{bj}^{m}+g_{aj}g^{ab}\frac{\partial v^{jk}}{%
\partial \theta ^{b}}\frac{\partial X^{t}}{\partial \theta ^{k}}%
+g_{aj}g^{ab}v^{jk}\frac{\partial ^{2}X^{t}}{\partial \theta ^{b}\partial
\theta ^{k}}\right) \frac{\partial }{\partial X^{t}}\right) \\
&=&\left( v^{jk}\frac{\partial X^{t}}{\partial \theta ^{k}}\Gamma _{mj}^{m}+%
\frac{\partial v^{jk}}{\partial \theta ^{j}}\frac{\partial X^{t}}{\partial
\theta ^{k}}\right) \frac{\partial }{\partial X^{t}}+ {\color{black}\mathcal{P}}\left( v^{jk}%
\frac{\partial ^{2}X^{t}}{\partial \theta ^{j}\partial \theta ^{k}}\frac{%
\partial }{\partial X^{t}}\right) ,
\end{eqnarray*}%
where we have used that ${\color{black}\mathcal{P}}v=v$ for $v\in
\mathfrak{X}(M)$, which holds for the first and second terms above. Using the definition of projection tensor and the ambient space formulation of the connection for the third term above, one obtains
\begin{small}
\begin{eqnarray*}
&& {\color{black}\mathcal{P}} C_{1}^{r}\left( {\color{black}\mathcal{P}_3} \bar{\nabla}_{\mathbb{R}^{n}}V\right) \\
&=&\left( v^{jk}\frac{\partial X^{t}}{\partial \theta ^{k}}\Gamma _{mj}^{m}+%
\frac{\partial v^{jk}}{\partial \theta ^{j}}\frac{\partial X^{t}}{\partial
\theta ^{k}}\right) \frac{\partial }{\partial X^{t}}+\left( \delta _{il}%
\frac{\partial X^{i}}{\partial \theta ^{r}}g^{rp}\frac{\partial X^{h}}{%
\partial \theta ^{p}}\mathrm{d}X^{l}\otimes \frac{\partial }{\partial X^{h}}%
\right) \left( v^{jk}\frac{\partial ^{2}X^{t}}{\partial \theta ^{j}\partial
\theta ^{k}}\frac{\partial }{\partial X^{t}}\right) \\
&=&\left( v^{jk}\frac{\partial X^{t}}{\partial \theta ^{k}}\Gamma _{mj}^{m}+%
\frac{\partial v^{jk}}{\partial \theta ^{j}}\frac{\partial X^{t}}{\partial
\theta ^{k}}\right) \frac{\partial }{\partial X^{t}}+\sum_{t=1}^{n}g^{rp}%
\frac{\partial X^{h}}{\partial \theta ^{p}}v^{jk}\frac{\partial ^{2}X^{t}}{%
\partial \theta ^{j}\partial \theta ^{k}}\frac{\partial X^{t}}{\partial
\theta ^{r}}\frac{\partial }{\partial X^{h}} \\
&=&\left( v^{jk}\frac{\partial X^{t}}{\partial \theta ^{k}}\Gamma _{mj}^{m}+%
\frac{\partial v^{jk}}{\partial \theta ^{j}}\frac{\partial X^{t}}{\partial
\theta ^{k}}\right) \frac{\partial }{\partial X^{t}}+g^{rp}v^{jk}g_{rm}%
\Gamma _{jk}^{m}\frac{\partial X^{h}}{\partial \theta ^{p}}\frac{\partial }{%
\partial X^{h}} \\
&=&\left( v^{jk}\Gamma _{mj}^{m}+\frac{\partial v^{jk}}{\partial \theta ^{j}}%
+v^{jm}\Gamma _{jm}^{k}\right) \frac{\partial X^{t}}{\partial \theta ^{k}}%
\frac{\partial }{\partial X^{t}}=v_{,j}^{jk}\frac{\partial X^{t}}{\partial
\theta ^{k}}\frac{\partial }{\partial X^{t}}=v_{,j}^{jk}\frac{\partial }{%
\partial \theta ^{k}}=\mathrm{div}_{1}^{1}\left( v\right) .
\end{eqnarray*}%
\end{small}
Following the same steps, we can obtain $\mathrm{div}%
_{1}^{r}\left( v\right) ={\color{black}\mathcal{P}} C_{1}^{r}\left( {\color{black}\mathcal{P}_3}\bar{\nabla}_{%
\mathbb{R}^{n}}V\right) $ for $r=2$. In matrix form, the divergence of a $%
(2,0)$ tensor field can be written as
\begin{equation*}
\mathrm{div}_{1}^{r}\left( v\right) =\mathbf{P}\mathrm{tr}_{1}^{r}\left(
\mathbf{P}\bar{\nabla}_{\mathbb{R}^{n}}\left( V\right) \right).  \notag
\end{equation*}
The above formula is needed when deriving approximations for the Lichnerowicz and Hodge Laplacian.
\subsection{RBF Approximation for Lichnerowicz Laplacian}
The Lichnerowicz Laplacian can be computed as%
\begin{eqnarray}
-\Delta _{L}u &\equiv &2\mathrm{div}_{1}^{1}\left( Su\right)
=2C_{1}^{1}\left( \nabla Su\right) =C_{1}^{1}\left[ \left(
g^{ij}u_{,is}^{k}+g^{ki}u_{,is}^{j}\right) \frac{\partial }{\partial \theta
^{j}}\otimes \frac{\partial }{\partial \theta ^{k}}\otimes \mathrm{d}\theta
^{s}\right]  \notag \\
&=&\left( g^{ij}u_{,ij}^{k}+g^{ki}u_{,ij}^{j}\right) \frac{\partial }{%
\partial \theta ^{k}}=\mathrm{div}_{1}^{1}({\mathrm{grad}}_{g}{u)}+\mathrm{%
div}_{1}^{2}({\mathrm{grad}}_{g}{u)},  \label{eqn:deL}
\end{eqnarray}%
where $S$ denotes the symmetric tensor, defined as
\begin{equation*}
Su=\frac{1}{2}\left( g^{ki}u_{,i}^{j}+g^{ij}u_{,i}^{k}\right) \frac{\partial
}{\partial \theta ^{j}}\otimes \frac{\partial }{\partial \theta ^{k}}.
\end{equation*}%
In matrix form, (\ref{eqn:deL}) can be written as,
\begin{footnotesize}
\begin{equation}
-\Delta _{L}u=\mathrm{div}_{1}^{1}({\mathrm{grad}}_{g}{u)}+\mathrm{div}%
_{1}^{2}({\mathrm{grad}}_{g}{u)}=\mathbf{P}\mathrm{tr}_{1}^{1}\left( \mathbf{%
P}\bar{\nabla}_{\mathbb{R}^{n}} \left(  \mathbf{P}\overline{{
\mathrm{grad}}}_{\mathbb{R}^{n}}U  \mathbf{P}\right)\right) +\mathbf{P}\mathrm{tr}%
_{1}^{2}\left( \mathbf{P}\bar{\nabla} _{\mathbb{R}^{n}}\left( \mathbf{P}\overline{{%
\mathrm{grad}}}_{\mathbb{R}^{n}}U\mathbf{P}\right) \right),
\label{eqn:Lbochu}
\end{equation}%
\end{footnotesize}
at each $x \in M$. One also has the following formula for the Lichnerowicz Laplacian:
\begin{equation*}
\Delta _{L}u=-\mathrm{div}_{1}^{1}\left({{\mathrm{grad}}_{g}u+(%
{\mathrm{grad}}_{g}u)}^{\top }\right)
\end{equation*}%
for a vector field $u\in \mathfrak{X}(M)$. Hence, after extension to Euclidean
space, the Lichnerowicz Laplacian can be written in a matrix form in the following way:%
\begin{eqnarray}
\bar{\Delta}_{L}U &=& -\mathbf{P}\mathrm{tr}_{1}^{1}\left( \mathbf{P}\bar{\nabla} _{%
\mathbb{R}^{n}}(\mathbf{P}\overline{{
\mathrm{grad}}}_{\mathbb{R}^{n}}U%
\mathbf{P}+(\mathbf{P}\overline{{
\mathrm{grad}}}_{\mathbb{R}^{n}}U\mathbf{P%
})^{\top })\right)  \notag \\
&=&-\mathbf{P}\mathrm{tr}_{1}^{1}\left( \mathbf{P}\bar{\nabla} _{\mathbb{R}%
^{n}}(\mathrm{Sym}(\mathbf{P}\overline{{
\mathrm{grad}}}_{\mathbb{R}^{n}}U%
\mathbf{P})\mathbf{)}\right) ,  \label{eqn:LlapSy}
\end{eqnarray}%
where $\mathrm{Sym}(\mathbf{P}\overline{{
\mathrm{grad}}}_{\mathbb{R}^{n}}U%
\mathbf{P})=\mathbf{P}\overline{{
\mathrm{grad}}}_{\mathbb{R}^{n}}U\mathbf{%
P}+(\mathbf{P}\overline{{
\mathrm{grad}}}_{\mathbb{R}^{n}}U\mathbf{P}%
)^{\top }$\ is a $(2,0)$ tensor field.

We now provide the non-symmetric approximation to
the Lichnerowicz Laplacian acting on a vector field $U$.  Following the notation and methodology used for the Bochner Laplacian, one can approximate Lichnerowicz Laplacian as,
\begin{equation*}
\mathbf{U} \mapsto -\left[
\begin{array}{c}
\mathbf{H}_{1} \\
\vdots \\
\mathbf{H}_{n}%
\end{array}%
\right] \cdot \mathrm{Sym}\left[
\begin{array}{c}
\mathbf{H}_{1}\mathbf{U} \\
\vdots \\
\mathbf{H}_{n}\mathbf{U}%
\end{array}%
\right] := -\left[
\begin{array}{c}
\mathbf{H}_{1} \\
\vdots \\
\mathbf{H}_{n}%
\end{array}%
\right] \cdot \mathrm{Sym}\left[
\begin{array}{c}
\mathbf{\bar{V}}_{1} \\
\vdots \\
\mathbf{\bar{V}}_{n}%
\end{array}%
\right] \mathbf{,}
\end{equation*}%
where $\mathbf{\bar{V}}_{i} := \mathbf{H}_i \mathbf{U} = \left[ \mathbf{V}_{i1},\ldots ,\mathbf{V}_{in}%
\right] ^{^{\top }}$ is the $i$th row of the tensor $\mathbf{V}=\left[
\mathbf{V}_{ij}\right] _{i,j=1}^{n}$. The symmetric operator can be written
in detail as%
\begin{small}
\begin{equation*}
\mathrm{Sym}\left[
\begin{array}{c}
\mathbf{\bar{V}}_{1} \\
\vdots \\
\mathbf{\bar{V}}_{n}%
\end{array}%
\right] =\mathrm{Sym}\left[
\begin{array}{cccc}
\mathbf{V}_{11} & \mathbf{V}_{12} & \cdots & \mathbf{V}_{1n} \\
\mathbf{V}_{21} & \mathbf{V}_{22} & \cdots & \vdots \\
\vdots & \vdots & \ddots & \vdots \\
\mathbf{V}_{n1} & \cdots & \cdots & \mathbf{V}_{nn}%
\end{array}%
\right] =\left[
\begin{array}{cccc}
2\mathbf{V}_{11} & \mathbf{V}_{12}+\mathbf{V}_{21} & \cdots & \mathbf{V}%
_{1n}+\mathbf{V}_{n1} \\
\mathbf{V}_{21}+\mathbf{V}_{12} & 2\mathbf{V}_{22} & \cdots & \vdots \\
\vdots & \vdots & \ddots & \vdots \\
\mathbf{V}_{n1}+\mathbf{V}_{1n} & \cdots & \cdots & 2\mathbf{V}_{nn}%
\end{array}%
\right] .
\end{equation*}%
\end{small}
Then Lichnerowicz Laplacian can be written as%
\begin{equation}
\mathbf{U} \mapsto -[\mathbf{H}_{1}\left( \mathbf{\bar{V}}_{1}+\mathbf{%
V|}_{1}\right) +\cdots +\mathbf{H}_{n}\left( \mathbf{\bar{V}}_{n}+\mathbf{V|}%
_{n}\right)] ,  \label{eqn:Llaprpc}
\end{equation}%
where $\mathbf{V|}_{i}=\left[ \mathbf{V}_{1i},\ldots ,\mathbf{V}_{ni}\right]
^{^{\top }}$ is the $i$th column of the tensor $\mathbf{V}=\left[ \mathbf{V}%
_{ij}\right] _{i,j=1}^{n}$.

We now write the approximate Lichnerowicz Laplacian in (\ref{eqn:Llaprpc}) in terms of $\mathbf{U}$.
Since $\mathbf{\bar{V}}_{i} := \mathbf{H}_i \mathbf{U}$, we only need to focus on $\mathbf{{V}}|_{i}$
for $i=1,\ldots,n$. Since the $(2,0)$ tensor $V_{ij}=[{\mathrm{grad}}_{g}{u]}_{ij}$, we see that the $i$th
column of $\left[ V_{ij}\right] _{i,j=1}^{n}$ can be calculated as,%
\begin{eqnarray}
V\mathbf{|}_{i} &=&\left[
\begin{array}{c}
V_{1i} \\
V_{2i} \\
\vdots \\
V_{ni}%
\end{array}%
\right] _{n\times 1}=\left[
\begin{array}{cccc}
\mathcal{G}_{1}U^{1} & \mathcal{G}_{1}U^{2} & \cdots & \mathcal{G}_{1}U^{n}
\\
\mathcal{G}_{2}U^{1} & \mathcal{G}_{2}U^{2} & \ddots & \vdots \\
\vdots & \ddots & \ddots & \vdots \\
\mathcal{G}_{n}U^{1} & \cdots & \cdots & \mathcal{G}_{n}U^{n}%
\end{array}%
\right] _{n\times n}\left[
\begin{array}{c}
{\color{black}P}_{1i} \\
{\color{black}P}_{2i} \\
\vdots \\
{\color{black}P}_{ni}%
\end{array}%
\right] _{n\times 1}  \notag \\
&=&\left[
\begin{array}{c}
{\color{black}P}_{1i}\mathcal{G}_{1}U^{1}+\cdots +{\color{black}P}_{ni}\mathcal{G}_{1}U^{n} \\
{\color{black}P}_{1i}\mathcal{G}_{2}U^{1}+\cdots +{\color{black}P}_{ni}\mathcal{G}_{2}U^{n} \\
\vdots \\
{\color{black}P}_{1i}\mathcal{G}_{n}U^{1}+\cdots +{\color{black}P}_{ni}\mathcal{G}_{n}U^{n}%
\end{array}%
\right] _{n\times 1} \notag \\
&=&\left[
\begin{array}{cccc}
{\color{black}P}_{1i}\mathcal{G}_{1} & {\color{black}P}_{2i}\mathcal{G}_{1} & \cdots & {\color{black}P}_{ni}\mathcal{G}%
_{1} \\
{\color{black}P}_{1i}\mathcal{G}_{2} & {\color{black}P}_{2i}\mathcal{G}_{2} & \ddots & \vdots \\
\vdots & \ddots & \ddots & \vdots \\
{\color{black}P}_{1i}\mathcal{G}_{n} & \cdots & \cdots & {\color{black}P}_{ni}\mathcal{G}_{n}%
\end{array}%
\right] _{n\times n}\left[
\begin{array}{c}
U^{1} \\
U^{2} \\
\vdots \\
U^{n}%
\end{array}%
\right] _{n\times 1} := \mathcal{S}_i U.  \label{eqn:Vcol}
\end{eqnarray}%
After discretization using RBF approximation, above (\ref{eqn:Vcol}) can be
approximated by%
\begin{equation*}
\mathbf{V|}_{i}=\left[
\begin{array}{cccc}
\mathrm{diag}\left( \mathbf{p}_{1i}\right) \mathbf{G}_{1} & \mathrm{diag}%
\left( \mathbf{p}_{2i}\right) \mathbf{G}_{1} & \cdots & \mathrm{diag}\left(
\mathbf{p}_{ni}\right) \mathbf{G}_{1} \\
\mathrm{diag}\left( \mathbf{p}_{1i}\right) \mathbf{G}_{2} & \mathrm{diag}%
\left( \mathbf{p}_{2i}\right) \mathbf{G}_{2} & \ddots & \vdots \\
\vdots & \ddots & \ddots & \color{black}\vdots \\
\mathrm{diag}\left( \mathbf{p}_{1i}\right) \mathbf{G}_{n} & \cdots & \cdots
& \mathrm{diag}\left( \mathbf{p}_{ni}\right) \mathbf{G}_{n}%
\end{array}%
\right] _{Nn\times Nn}\left[
\begin{array}{c}
\mathbf{U}^{1} \\
\mathbf{U}^{2} \\
\vdots \\
\mathbf{U}^{n}%
\end{array}%
\right] _{Nn\times 1} := \mathbf{S}_{i}\mathbf{U},
\end{equation*}%
where $\mathbf{p}_{ij}=\left( {\color{black}P}_{ij}(\boldsymbol{x}_{1}),\ldots ,{\color{black}P}_{ij}(%
\boldsymbol{x}_{N})\right) \in \mathbb{R}^{N\times 1}$. Then, Lichnerowicz
Laplacian in (\ref{eqn:Llaprpc}) can be derived as,%
\begin{eqnarray*}
\mathbf{U} &\mapsto&-[\mathbf{H}_{1}\left( \mathbf{\bar{V}}_{1}+%
\mathbf{V|}_{1}\right) +\cdots +\mathbf{H}_{n}\left( \mathbf{\bar{V}}_{n}+%
\mathbf{V|}_{n}\right)] \\
&=&-\left[ \mathbf{H}_{1}\left( \mathbf{H}_{1}+\mathbf{S}_{1}\right) +\cdots +%
\mathbf{H}_{n}\left( \mathbf{H}_{n}+\mathbf{S}_{n}\right) \right] \mathbf{U.}
\end{eqnarray*}%
Thus, the symmetric operator can be expressed as
\begin{equation}
\mathrm{Sym}\left[
\begin{array}{c}
\mathbf{H}_{1}\mathbf{U} \\
\vdots \\
\mathbf{H}_{n}\mathbf{U}%
\end{array}%
\right] =\left[
\begin{array}{c}
\left( \mathbf{H}_{1}+\mathbf{S}_{1}\right) \mathbf{U} \\
\vdots \\
\left( \mathbf{H}_{n}+\mathbf{S}_{n}\right) \mathbf{U}%
\end{array}%
\right] ,  \label{eqn:symL}
\end{equation}%
and Lichnerowicz Laplacian matrix given by  $-\mathbf{H}_{1}\left(
\mathbf{H}_{1}+\mathbf{S}_{1}\right) -\cdots -\mathbf{H}_{n}\left( \mathbf{H}%
_{n}+\mathbf{S}_{n}\right) $ can be used to study the spectral properties of the Lichnerowicz Laplacian numerically.

\subsection{RBF Approximation for Hodge Laplacian}
The Hodge Laplacian is initially defined on $k$ forms as

\begin{equation*}
\Delta _{H}=\left( \mathrm{dd}^{\ast }+\mathrm{d}^{\ast }\mathrm{d}\right) ,
\end{equation*}%
where the Hodge Laplacian is positive definite. Here, $\mathrm{d}$\ is
the exterior derivative and $\mathrm{d}^{\ast }$ is the adjoint of $\mathrm{d%
}$ given by%
\begin{equation*}
\mathrm{d}^{\ast }=(-1)^{dr+d+1}\ast \circ \ \mathrm{d}\circ \ast
:{\Omega}^{r}(M)\rightarrow {\Omega}^{r-1}(M),
\end{equation*}%
where $\ast $\ is Hodge star operator. See \cite{chow2007ricci} for details. The Hodge Laplacian acting on a vector field $u=u^{k}%
\frac{\partial }{\partial \theta ^{k}}$ is defined as%
\begin{equation*}
\Delta _{H}u=\sharp \left( \mathrm{dd}^{\ast }+\mathrm{d}^{\ast }\mathrm{d}%
\right) \flat u,
\end{equation*}%
where $\sharp $ and $\flat $ are the standard musical isomorphisms. First, we can compute
\begin{equation}
\sharp \mathrm{dd}^{\ast }\flat u=\sharp \mathrm{d}%
(-g^{jk}u_{j,k})=-g^{jk}u_{,ij}^{i}\frac{\partial }{\partial \theta ^{k}}=-{%
\mathrm{grad}}_{g}\left( \mathrm{div}_{g}\left( {u}\right) \right) ,
\label{eqn:dds}
\end{equation}%
where $\flat u=u_{i}\mathrm{d}\theta ^{i}=g_{ij}u^{j}\mathrm{d}\theta ^{i}$.
Next, we can compute
\begin{eqnarray}
\sharp \mathrm{d}^{\ast }\mathrm{d}\flat u &=&\sharp \mathrm{d}^{\ast }%
\mathrm{d}\left( u_{i}\mathrm{d}\theta ^{i}\right) =\sharp \mathrm{d}^{\ast
}(\frac{1}{2!}\left( u_{j,i}-u_{i,j}\right) \mathrm{d}\theta ^{i}\wedge
\mathrm{d}\theta ^{j})=\sharp \left( -g^{ij}\left( u_{k,ij}-u_{i,kj}\right)
\mathrm{d}\theta ^{k}\right)  \notag \\
&=&-g^{ij}u_{,ij}^{k}\frac{\partial }{\partial \theta ^{k}}+g^{kj}u_{,ji}^{i}%
\frac{\partial }{\partial \theta ^{k}}=-\mathrm{div}_{1}^{1}({\mathrm{grad}}%
_{g}{u)}+\mathrm{div}_{1}^{2}({\mathrm{grad}}_{g}{u)},  \label{eqn:dsd}
\end{eqnarray}%
where in last equality we have used definitions of gradient and divergence. Lastly, from (\ref%
{eqn:dds}) and (\ref{eqn:dsd}), we can compute the Hodge Laplacian as%
\begin{eqnarray}
\Delta _{H}u &=&\sharp \left( \mathrm{d}^{\ast }\mathrm{d}+\mathrm{dd}%
^{\ast }\right) \flat u=-\left(
g^{ij}u_{,ij}^{k}-g^{kj}u_{,ji}^{i}+g^{jk}u_{,ij}^{i}\right) \frac{\partial
}{\partial \theta ^{k}}  \notag \\
&=&\Delta _{B}u+g^{jk}\left( u_{,ji}^{i}-u_{,ij}^{i}\right) \frac{\partial }{%
\partial \theta ^{k}}=\Delta _{B}u+g^{jk}u^{m}R_{mj}\frac{\partial }{%
\partial \theta ^{k}}\equiv \Delta _{B}u+{\mathrm{Ri}}(u),  \label{eqn:LBHR}
\end{eqnarray}%
where the Ricci identity was used in the second line. One can see from (\ref%
{eqn:LBHR}) that the Hodge Laplacian $\Delta _{H}$ is different from the
Bochner Laplacian $\Delta _{B}$ by a Ricci tensor ${\mathrm{Ri}}(u)\equiv
g^{jk}u^{m}R_{mj}\frac{\partial }{\partial \theta ^{k}}$\ (see \textit{pp}
478 of \cite{chow2007ricci} for $k$-form for example).

Now we present the Hodge Laplacian written in matrix form. First, using previous formulas for the divergence and gradient, we have
\begin{equation*}
\sharp \mathrm{dd}^{\ast }\flat u=-{\mathrm{grad}}_{g}\left( \mathrm{div}%
_{g}\left( {u}\right) \right) =-\mathbf{P}\overline{{\mathrm{grad}}}_{%
\mathbb{R}^{n}}\left( \mathrm{tr}_{1}^{1}\left[ \mathbf{P}\bar{\nabla}_{%
\mathbb{R}^{n}}\mathbf{U}\right] \right) .
\end{equation*}%
Secondly, we have
\begin{eqnarray*}
\sharp \mathrm{d}^{\ast }\mathrm{d}\flat u &=&-g^{ij}u_{,ij}^{k}\frac{%
\partial }{\partial \theta ^{k}}+g^{kj}u_{,ji}^{i}\frac{\partial }{\partial
\theta ^{k}}=-\mathrm{div}_{1}^{1}({\mathrm{grad}}_{g}{u)}+\mathrm{div}%
_{1}^{2}({\mathrm{grad}}_{g}{u)} \\
&=&-\mathbf{P}\mathrm{tr}_{1}^{1}\left( \mathbf{P}\bar{\nabla}_{\mathbb{R}%
^{n}}(\mathbf{P}\overline{{
\mathrm{grad}}}_{\mathbb{R}^{n}}U\mathbf{P)}%
\right) +\mathbf{P}\mathrm{tr}_{1}^{2}\left( \mathbf{P}\bar{\nabla}_{\mathbb{%
R}^{n}}(\mathbf{P}\overline{{
\mathrm{grad}}}_{\mathbb{R}^{n}}U\mathbf{P)}%
\right)
\end{eqnarray*}%
Then, we can calculate Hodge Laplacian as
\begin{eqnarray}
\Delta _{H}u &=& \sharp \left( \mathrm{d}^{\ast }\mathrm{d}+\mathrm{dd}%
^{\ast }\right) \flat u= -\left(
g^{ij}u_{,ij}^{k}-(g^{kj}u_{,ji}^{i}-g^{jk}u_{,ij}^{i})\right) \frac{%
\partial }{\partial \theta ^{k}}  \notag \\
&=&-\mathrm{div}_{1}^{1}({\mathrm{grad}}_{g}u)+\left[ \mathrm{div}%
_{1}^{2}\left( {\mathrm{grad}}_{g}u\right) -{\mathrm{grad}}_{g}\left(
\mathrm{div}_g\left( {u}\right) \right) \right]  \notag \\
&=&-\mathbf{P}\mathrm{tr}_{1}^{1}\left( \mathbf{P}\bar{\nabla}_{\mathbb{R}%
^{n}}(\mathbf{P}\overline{{
\mathrm{grad}}}_{\mathbb{R}^{n}}U\mathbf{P}%
)\right) \notag \\
&+&\left[ \mathbf{P}\mathrm{tr}_{1}^{2}\left( \mathbf{P}\bar{\nabla}_{%
\mathbb{R}^{n}}(\mathbf{P}\overline{{
\mathrm{grad}}}_{\mathbb{R}^{n}}U%
\mathbf{P})\right) -\mathbf{P}\overline{{\mathrm{grad}}}_{\mathbb{R}^{n}}\left( \mathrm{%
tr}_{1}^{1}\left[ \mathbf{P}\bar{\nabla}_{\mathbb{R}^{n}}U\right]
\right) \right] , \label{eqn:Hod2}
\end{eqnarray}%
where the last line holds at each $x \in M$. The first term of (\ref{eqn:Hod2}) indeed is the Bochner Laplacian $%
\Delta _{B}u$. Hence, it only remains to compute the second and third terms. At each $x\in M$, we can write the second term as
\begin{eqnarray}
\mathbf{P}\mathrm{tr}_{1}^{2}\left( \mathbf{P}\bar{\nabla} _{\mathbb{%
R}^{n}}\mathbf{V}\right) &=& \left[
\begin{array}{ccc}
{P}_{11} & \cdots & {P}_{1n} \\
\vdots & \ddots & \vdots \\
{P}_{n1} & \cdots & {P}_{nn}%
\end{array}%
\right] _{n\times n}\left[
\begin{array}{c}
\sum_{k=1}^{n}\mathcal{G}_{k}V_{1k} \\
\vdots \\
\sum_{k=1}^{n}\mathcal{G}_{k}V_{nk}%
\end{array}%
\right] _{n\times 1}  \notag \\
&=&\left[
\begin{array}{ccc}
{P}_{11} & \cdots & {P}_{1n} \\
\vdots & \ddots & \vdots \\
{P}_{n1} & \cdots & {P}_{nn}%
\end{array}%
\right] _{n\times n}\left[
\begin{array}{c}
\mathcal{G}_{1}V_{11}+\mathcal{G}_{2}V_{12}+\cdots +\mathcal{G}_{n}V_{1n} \\
\vdots \\
\mathcal{G}_{1}V_{n1}+\mathcal{G}_{2}V_{n2}+\cdots +\mathcal{G}_{n}V_{nn}%
\end{array}%
\right] _{n\times 1},  \label{eqn:Pt21}
\end{eqnarray}%
and write the third term as%
\begin{equation}
\mathbf{P}\overline{{\mathrm{grad}}}_{\mathbb{R}^{n}}\left( \mathrm{tr}_{1}^{1}\left[
\mathbf{P}\bar{\nabla}_{\mathbb{R}^{n}}U\right] \right) =\mathbf{P}{%
\overline{\mathrm{grad}}}_{\mathbb{R}^{n}}\left( \sum_{k=1}^{n}\mathcal{G}%
_{k}U_{k}\right) =\left[
\begin{array}{c}
\mathcal{G}_{1} \\
\vdots \\
\mathcal{G}_{n}%
\end{array}%
\right] \left( \sum_{k=1}^{n}\mathcal{G}_{k}U_{k}\right) .  \label{eqn:Pd11}
\end{equation}%
Substituting \eqref{eqn:Pt21},\eqref{eqn:Pd11}, as well as the formula for the Bochner Laplacian, into \eqref{eqn:Hod2}, we obtain the formula for Hodge Laplacian.

Another way to write the Hodge Laplacian is given by
\BEA
-\Delta _{H}u&=&\mathrm{div}_{1}^{1}({\mathrm{grad}}_{g}{u)}-\mathrm{div}%
_{1}^{2}({\mathrm{grad}}_{g}{u)+\mathrm{grad}}_{g}\left( \mathrm{div}%
_{g}\left( {u}\right) \right) \notag \\
&=&\mathrm{div}_{1}^{1}({{\mathrm{grad}}_{g}u-({%
\mathrm{grad}}_{g}u)}^{\top })+{\mathrm{grad}}_{g}\left( \mathrm{div}%
_{g}\left( {u}\right) \right).  \label{eqn:Lhog}
\EEA
The first term, involving the anti-symmetric part of ${{\mathrm{grad}}%
_{g}u}$, after pre-composing with the interpolating operator, can be approximated by%
\begin{equation*}
\mathrm{Ant}\left[
\begin{array}{c}
\mathbf{H}_{1}\mathbf{U} \\
\vdots \\
\mathbf{H}_{n}\mathbf{U}%
\end{array}%
\right] =\left[
\begin{array}{c}
\left( \mathbf{H}_{1}-\mathbf{S}_{1}\right) \mathbf{U} \\
\vdots \\
\left( \mathbf{H}_{n}-\mathbf{S}_{n}\right) \mathbf{U}%
\end{array}%
\right] .
\end{equation*}%
Thus, we only need to consider the RBF approximation for the last term ${{%
\mathrm{grad}}_{g}\left( \mathrm{div}_{g}\left( {u}\right) \right) }$ in (%
\ref{eqn:Lhog}). Using the formula for gradient of a function and divergence of a vector field in \eqref{eqn:Pd11} directly, we can obtain that
\begin{eqnarray*}
&& {{\mathrm{grad}}_{g}\left( \mathrm{div}_{g}\left( {u}\right) \right) } \\
&=&%
\left[
\begin{array}{c}
\mathcal{G}_{1} \\
\vdots \\
\mathcal{G}_{n}%
\end{array}%
\right] \sum_{i=1}^{n}\mathcal{G}_{i}U^{i}=\left[
\begin{array}{c}
\mathcal{G}_{1}\sum_{i=1}^{n}\mathcal{G}_{i}U^{i} \\
\vdots \\
\mathcal{G}_{n}\sum_{i=1}^{n}\mathcal{G}_{i}U^{i}%
\end{array}%
\right]
= \left[
\begin{array}{c}
\mathcal{G}_{1} \\
\vdots \\
\mathcal{G}_{n}%
\end{array}%
\right] \big[\mathcal{G}_{1} \cdots \mathcal{G}_{n}\big] \left[
\begin{array}{c}
U^{1} \\
\vdots \\
U^{n}%
\end{array}%
\right]  \\
&=&
\left[
\begin{array}{cccc}
\mathcal{G}_{1}\mathcal{G}_{1} & \mathcal{G}_{1}\mathcal{G}_{2} & \cdots & \mathcal{G}_{1}\mathcal{G}_{n} \\
\mathcal{G}_{2}\mathcal{G}_{1} & \mathcal{G}_{2}\mathcal{G}_{2} & \ddots & \vdots \\
\vdots & \ddots & \ddots & \vdots \\
\mathcal{G}_{n}\mathcal{G}_{1} & \cdots & \cdots & \mathcal{G}_{n}\mathcal{G}_{n}%
\end{array}%
\right] _{n\times n}\left[
\begin{array}{c}
U^{1} \\
U^{2} \\
\vdots \\
U^{n}%
\end{array}%
\right] _{n\times 1}.
\end{eqnarray*}%
After discretization, we obtain that this term can be approximated by%
\begin{equation*}
\mathbf{U} \mapsto
\left[
\begin{array}{cccc}
\mathbf{G}_{1}\mathbf{G}_{1} & \mathbf{G}_{1}\mathbf{G}_{2} & \cdots & \mathbf{G}_{1}\mathbf{G}_{n} \\
\mathbf{G}_{2}\mathbf{G}_{1} & \mathbf{G}_{2}\mathbf{G}_{2} & \ddots & \vdots \\
\vdots & \ddots & \ddots & \vdots \\
\mathbf{G}_{n}\mathbf{G}_{1} & \cdots & \cdots & \mathbf{G}_{n}\mathbf{G}_{n}%
\end{array}%
\right] _{Nn\times Nn}\left[
\begin{array}{c}
\mathbf{U}^{1} \\
\mathbf{U}^{2} \\
\vdots \\
\mathbf{U}^{n}%
\end{array}%
\right] _{Nn\times 1} := \mathbf{TU.}
\end{equation*}%
It follows that the Hodge Laplacian matrix can be approximated by
\begin{equation}
\bar{\Delta}_{H}=- \left(\mathbf{H}_{1}\left( \mathbf{H}_{1}-\mathbf{S}_{1}\right)
+\cdots +\mathbf{H}_{n}\left( \mathbf{H}_{n}-\mathbf{S}_{n}\right) \right) - \mathbf{%
T.}  \label{eqn:HoLap}
\end{equation}

\subsection{RBF Approximation for Covariant Derivative}
In the following, we examine $\nabla _{u}y= \mathcal{P}\bar{\nabla}_{U}Y$ via direct calculation, where the RHS will be
defined below. Let $u=u^{i}\frac{\partial }{\partial \theta ^{i}}\in
\mathfrak{X}(M)$ and $y=y^{i}\frac{\partial }{\partial \theta ^{i}}\in
\mathfrak{X}(M)$ be two vector fields, where $u^{i},y^{i}\in C^{\infty }(M)$
are smooth functions. Then, the covariant derivative is defined as%
\begin{equation*}
\nabla _{u}y=u^{k}y_{,k}^{i}\frac{\partial }{\partial \theta ^{i}}%
=u^{k}\left( \frac{\partial y^{i}}{\partial \theta ^{k}}+y^{j}\Gamma
_{jk}^{i}\right) \frac{\partial }{\partial \theta ^{i}},
\end{equation*}%
where the covariant derivative operator $\nabla :\mathfrak{X}(M)\times
\mathfrak{X}(M)\rightarrow \mathfrak{X}(M)$ takes vector fields $u$ and $y$
in $\mathfrak{X}(M)$ to a vector field $\nabla _{u}y$ in $\mathfrak{X}(M)$.
We can rewrite covariant derivative as
\begin{eqnarray}
\nabla _{u}y &=&u^{k}\left( \delta _{rs}g^{ij}\frac{\partial X^{r}}{\partial
\theta ^{j}}\frac{\partial Y^{s}}{\partial \theta ^{k}}\right) \frac{%
\partial }{\partial \theta ^{i}}=u^{k}\delta _{rs}g^{ij}\frac{\partial X^{r}%
}{\partial \theta ^{j}}\left( \frac{\partial Y^{s}}{\partial X^{m}}\frac{%
\partial X^{m}}{\partial \theta ^{k}}\right) \left( \frac{\partial X^{p}}{%
\partial \theta ^{i}}\frac{\partial }{\partial X^{p}}\right)  \notag \\
&=&\delta _{rs}\frac{\partial X^{p}}{\partial \theta ^{i}}g^{ij}\frac{%
\partial X^{r}}{\partial \theta ^{j}}\frac{\partial Y^{s}}{\partial X^{m}}%
U^{m}\frac{\partial }{\partial X^{p}}  \notag \\
&=&\left( \delta _{rs}\frac{\partial X^{p}}{\partial \theta ^{i}}g^{ij}\frac{%
\partial X^{r}}{\partial \theta ^{j}}\mathrm{d}X^{s}\otimes \frac{\partial }{%
\partial X^{p}}\right) \left( U^{m}\frac{\partial Y^{k}}{\partial X^{m}}%
\frac{\partial }{\partial X^{k}}\right) := {\color{black}\mathcal{P}}\bar{\nabla}_{U}Y,
\label{eqn:duyP}
\end{eqnarray}%
where the first line follows from the chain rule, the second line follows
from $U^{m}=u^{k}\frac{\partial X^{m}}{\partial \theta ^{k}}$, and the last
line defines $\bar{\nabla}_{U}Y\equiv U^{m}\frac{\partial Y^{k}}{\partial
X^{m}}\frac{\partial }{\partial X^{k}}$ for the covariant derivative in
Euclidean space. In matrix-vector form, \eqref{eqn:duyP} is
straightforward to write as
\begin{equation}
\nabla _{u}y=\mathbf{P}\bar{\nabla}_{U}Y=\left[
\begin{array}{ccc}
{\color{black}P}_{11} & \cdots & {\color{black}P}_{1n} \\
\vdots & \ddots & \vdots \\
{\color{black}P}_{n1} & \cdots & {\color{black}P}_{nn}%
\end{array}%
\right] \left[
\begin{array}{ccc}
\frac{\partial Y^{1}}{\partial X^{1}} & \cdots & \frac{\partial Y^{1}}{%
\partial X^{n}} \\
\vdots & \ddots & \vdots \\
\frac{\partial Y^{n}}{\partial X^{1}} & \cdots & \frac{\partial Y^{n}}{%
\partial X^{n}}%
\end{array}%
\right] \left[
\begin{array}{c}
U^{1} \\
\vdots \\
U^{n}%
\end{array}%
\right] .  \label{eqn:covd}
\end{equation}

The RBF approximation for covariant derivative is slightly different from above linear Laplacian operators
since it is nonlinear involving two vector fields $U$ and $Y$ as in (\ref%
{eqn:covd}). We first compute the covariant derivative $\bar{\nabla}_{U}Y$
in Euclidean space. For each $Y^{r}$, its RBF interpolant we can use the RBF interpolant to evaluate the $r$th row of $\bar{%
\nabla}_{U}Y$ in (\ref{eqn:covd}) at all node locations. This yields%
\begin{eqnarray*}
\bar{\nabla}_{\mathbf{U}}\mathbf{Y}^{r} &:= &\sum_{k=1}^{n}\frac{%
\partial I_{\phi_s }Y^{r}}{\partial X^{k}}\left.U^{k}\right|_X \in \mathbb{R}^N.
\end{eqnarray*}%
Concatenating all the $\bar{\nabla}_{\mathbf{U}}\mathbf{Y}^{r}$ for $r=1,\ldots,n$ to form an augmented vector $\bar{\nabla}_{%
\mathbf{U}}\mathbf{Y}=(\bar{\nabla}_{\mathbf{U}}\mathbf{Y}^{1},\ldots ,\bar{%
\nabla}_{\mathbf{U}}\mathbf{Y}^{n})\in \mathbb{R}^{Nn\times 1}$, we can
obtain the RBF formula for covariant derivative as
\begin{equation}
\mathbf{P}\bar{\nabla}_{U}I_{\phi_s }Y=\mathbf{P}^{\otimes }\bar{%
\nabla}_{\mathbf{U}}\mathbf{Y}.  \label{eqn:cPa}
\end{equation}

\section{Interpolation Error}\label{app:B}
In this appendix, we develop results regarding interpolation error in the probabilistic setting. Namely, we formally state and prove interpolation results for functions, vector fields, and $(1,1)$ tensor fields. The setting for this section is self-contained, and can be summarized as follows. Let $X = \{x_1 , \dots , x_N\}$ be finitely many uniformly sampled data points of a closed, smooth Riemannian manifold $M$ of dimension $d$, embedded in a higher dimensional Euclidean space $M \subseteq \mathbb{R}^n$. We assume additionally that the injectivity radius $\iota(M)$ is bounded away from zero from below by a constant $r>0$.
\subsection{Probabilistic Mesh Size Result }
In this setting, we have the following result from \cite{croke1980some}:
\begin{lem}
(Proposition 14 in \cite{croke1980some}) Let $B_\delta(x)$ denote a geodesic ball of radius $\delta$ around a point $x \in M$. For $\delta < \iota(M)/2$, we have
$$
\textrm{Vol}(B_\delta(x)) \geq C(d) \delta^d,
$$
where $C(d)$ is a constant depending only on the dimension $d$ of the manifold.
\end{lem}
Consider the quantity
$$
h_{X, M} = \sup_{x \in M} \min_{x_i \in X} \|x-x_i\|_g,
$$
often referred to as mesh-size. We will show that this quantity converges to $0$, in high probability, after $N \to \infty$. Here, $\| \cdot \|_g$ denotes the geodesic distance.

\begin{lem} \label{mesh size lemma} We have the following result regarding the mesh size $h_{X,M}:$
\[
\mathbb{P}_{X \sim \mathcal{U}}\left( h_{X,M} > \delta \right) \leq \exp(-C N \delta^d).
\]
Here, $C = C(d)/\textrm{Vol}(M)$, and $\mathcal{U}$ denotes the uniform distribution on $M$.
\end{lem}
\begin{proof}
Suppose $h_{X, M} > \delta$, so there is $x \in M$ such that $\min_{x_i \in X} \|x - x_i\|_g \geq \delta$. So $B_\delta(x) \cap X = \emptyset$. In other words, each $x_i \in X$ is in $M \setminus B_\delta(x) $. This has measure $1 - \frac{\textrm{Vol}(B_\delta(x)}{\textrm{Vol}(M)} \leq 1 - C\delta^d$. Hence, $\min_{x_i \in X} \|x - x_i\|_g \geq \delta$ occurs with probability less than $(1-C\delta^d)^N \leq \exp(-CN\delta^d)$. This completes the proof.
\end{proof}
\subsection{Interpolation of Functions, Vector Fields, and (2,0) Tensors}\label{sec:B2}
Given a radial basis function $\phi_s : M \times M \to \BR$, consider the interpolation map
\BEA
I_{\phi_s} f = \sum_{i=1}^N c_i \phi_s(\cdot, x_i),
\EEA
where $c_i$'s are chosen so that $I_{\phi_s}f|_X = f|_X$. In the following discussion, we assume that the kernel $\phi_s$ is a Mercer kernel that is at least $C^2$. Our probabilistic interpolation result relies heavily on Theorem $10$ from \cite{fuselier2012scattered}. We adapt this theorem to our notation and state it here, for convenience.
\begin{theo}
\label{Theorem 10 FW}
(adapted from {Corollary 13} in \cite{fuselier2012scattered}) Let $M$ be a $d$-dimensional submanifold of $\mathbb{R}^n$, and let $\phi_s$ be a kernel with RKHS norm equivalent to a Sobolev space of order $\alpha > n/2$. Let also $2 \leq q \leq \infty$, and $0 \leq \mu \leq \alpha -n/2 {+d/q}- 1$. Then there exists a constant $h_M$ such that whenever a finite node set $X$ satisfies $h_{X,M}<h_M$, then for all $f \in H^{\alpha - \frac{(n-d)}{2}}(M)$, we have
$$
\| f - I_{\phi_s} f \|_{W^{\mu,q}(M)} \leq C h_{X,M}^{\alpha - \frac{(n-d)}{2} - \mu - n(1/2 - 1/q)} \|f \|_{H^{\alpha - \frac{(n-d)}{2}}(M)}.
$$
\end{theo}
The above result is stated to proof Lemma~\ref{function interpolation}.
\comment{
\begin{lem}
\label{function interpolation}
Let $\phi_s$ be a kernel whose RKHS norm equivalent to Sobolev space of order $\alpha > n/2$. Then there is sufficiently large $N = |X|$ such that with probability higher than $1 - \frac{1}{N} $, for all $f \in H^{\alpha - \frac{(n-d)}{2}}(M)$, we have
$$
\| I_{\phi_s} f - f\|_{L^2(M)} = O\left( N^{\frac{-2\alpha + (n-d)}{2d}} \right).
$$
\end{lem}
}

{
\paragraph{\bf Proof of Lemma~\ref{function interpolation}:}
  By Lemma \ref{mesh size lemma}, choosing $\delta = \left(\frac{\log N}{CN}\right)^{1/d}$, we have that $h_{X,M} = O\left( N^{-1/d} \right)$ with probability $1-\frac{1}{N}$. For $N$ sufficiently large, it follows that $h_{X,M} \leq h_M$, where $h_M$ is as in Theorem \ref{Theorem 10 FW}. Hence, the hypotheses of Theorem \ref{Theorem 10 FW} are satisfied (with $q = 2$ and $\mu = 0$) and we have that
  $$
  \| I_{\phi_s} f - f \|_{L^2(M)} = O(h_{X,M}^{\alpha - (n-d)/2}).
  $$
  Using the scaling of the mesh size with $N$, we obtain the desired result. \hfill
  $\blacksquare$

When the interpolator is sufficiently regular, the same result from \cite{fuselier2012scattered} can be used to obtain convergence in certain Sobolev spaces. While this is not necessary for our results, we do need Sobolev norms of interpolated functions $I_{\phi_s}f$ to be bounded independent of $N$. We state and proof this result in the following lemma.
\begin{lem}
    \label{bounded derivatives}
    Let $\phi_s$ be a kernel with RKHS norm equivalent to a Sobolev space of order $\alpha \geq n/2 + 3$. Then for any $f \in H^{\alpha - \frac{(n-d)}{2}}(M)$, we have
    $$
    \| I_{\phi_s} f \|_{W^{2, \infty} (M)} = O(1),
    $$
    for all sufficiently large $N$ with probability higher than $1 - \frac{1}{N}$.
\end{lem}
\begin{proof}
    Let $0 \leq \mu \leq \alpha - n/2 -1$. It follows from Theorem \ref{Theorem 10 FW} that there is a constant $h_M$ such that if $h_{X,M} \leq h_M$, then for all $f \in H^{\alpha - (n-d)/2}(M)$, we have
   $$
   \|f - I_{\phi_s} f \|_{W^{\mu,\infty}(M)} \leq h^{\alpha - n/2 - \mu}_{X,M} \|f\|_{H^{\alpha - (n-d)/2}(M)}.
   $$
   By the same argument as Lemma \ref{function interpolation}, this occurs with probability $1 - \frac{1}{N}$ for sufficiently large $N$. Since $\alpha \geq n/2 + 3$ {suggests that $ \alpha - n/2 -1 \geq 2$, so $\mu=2$ lies} in the valid parameter regime, we have that
   $$
   \|f - I_{\phi_s} f \|_{W^{2,\infty}(M)} = O(1),
   $$
   where the constant depends on the manifold and $\|f\|_{H^{\alpha - (n-d)/2}(M)}$. Hence
   $$
    \| I_{\phi_s} f \|_{W^{2, \infty}(M)} \leq  \|f - I_{\phi_s} f \|_{W^{2,\infty}(M)} + \|f\|_{W^{2,\infty}(M)} = O(1),
   $$
   where the above constant depends on the manifold, $\|f\|_{H^{\alpha - (n-d)/2}(M)}$, and $\|f\|_{W^{2, \infty}(M)}$. This completes the proof.
\end{proof}}
For the interpolation of vector fields, recall that given a vector field $u = u^i \frac{\partial}{\partial \theta^i}$ on $M$, we can extend it to a smooth vector field $U = U^i \frac{\partial }{\partial X^i}$ defined on an open $\mathbb{R}^n$ neighborhood of $M$. Moreover, if $U$ and $V$ are extensions of $u$ and $v$ respectively, then we have that
$$
\langle u,v \rangle_x = \begin{bmatrix} U^1(x) \\
U^2(x) \\
\vdots \\
U^n(x)
\end{bmatrix} \cdot \begin{bmatrix} V^1(x) \\
V^2(x) \\
\vdots \\
V^n(x)
\end{bmatrix},
$$
at each $x \in M$, where $\langle \cdot , \cdot \rangle_x$ denotes the Riemannian inner product, and $\cdot$ denotes the standard Euclidean inner product. Recall that the interpolation of a vector field is defined component-wise in the ambient space coordinates:
$$
I_{\phi_s}u = I_{\phi_s}U^i \frac{\partial}{\partial X^i},
$$
where again $U^i$ denotes the ambient space components of the extension of $u$ to a neighborhood of $M$. We are now ready to prove the following lemma.
\begin{lem}
\label{vector field interpolation appendix}
For any $u = u^i \frac{\partial}{\partial \theta^i} \in \mathfrak{X}(M)$, we have that with probability $1 - {\frac{n}{N}}$,
$$
\| u - I_{\phi_s}u \|_{L^2(\mathfrak{X}(M))}= {O\left( N^{\frac{-2\alpha + (n-d)}{2d}} \right).}
$$
\end{lem}
\begin{proof}
    Notice that
    $$
    \| u - I_{\phi_s}u \|^2_{L^2(\mathfrak{X}(M))} = \int_M \langle u - I_{\phi_s}u,u - I_{\phi_s}u \rangle_x d\textup{Vol}(x) = \sum_{i=1}^n \int_M (U^i(x) - I_{\phi_s}U^i(x))^2 d \textup{Vol}(x).
    $$
    Using Lemma \ref{function interpolation} $n$ times, we see that with probability higher than $1 - {\frac{n}{N}}$,
    $$
   \left( \int_M (U^i(x) - I_{\phi_s}U^i(x))^2 d \textup{Vol}(x) \right)^{1/2} = {O\left( N^{\frac{-2\alpha + (n-d)}{2d}} \right).} \qquad i=1,2,\dots,n.
    $$
    Hence,
    $$
    \| u - I_{\phi_s}u \|_{L^2(\mathfrak{X}(M))} = {O\left( N^{\frac{-2\alpha + (n-d)}{2d}} \right).}
    $$
    which complete the proof.
\end{proof}
{ Similar to before, consider the norm
$$
\|u\|^2_{W^{2, \infty}(\mathfrak{X}(M))} := \sum_{i=1}^n \|U^i \|^2_{W^{2, \infty}(M)},
$$
where $u$ is a vector field with ambient space coefficients $U^i$ as before. The exact same reasoning as Lemma \ref{bounded derivatives} applied to each coefficient yields a result analogous to Lemma \ref{bounded derivatives} but for the interpolation of vector fields. We now state the result.
\begin{lem}
    \label{bounded derivatives vector field}
Let $\phi_s$ be a kernel with RKHS norm equivalent to a Sobolev space of order $\alpha \geq n/2 + 3$. For any vector field $u$ with ambient space coefficients $U^i$ satisfying $U^i \in H^{\alpha - \frac{(n-d)}{2}}(M)$, we have
$$
\|I_{\phi_s} u\|^2_{W^{2, \infty}(\mathfrak{X}(M))} = O(1)
$$
for sufficiently large $N$, with probability higher than $1 - \frac{n}{N}$.
\end{lem}}
Similarly, if $a = a_{ij} \frac{\partial}{\partial \theta^i} \otimes \frac{\partial}{ \partial \theta^j}$ is a $(2,0)$ tensor field, we can extend $a$ to $A = A_{ij} \frac{\partial}{\partial X^i} \otimes \frac{\partial}{\partial X^j}$, defined on an $\mathbb{R}^n$ neighborhood of $M$. Moreover, if $a$ and $b$ are $(2,0)$ tensor fields on $M$ with extensions $A$ and $B$ respectively, we have that
$$
\langle a,b \rangle_x = \textrm{tr}(A(x)^\top B(x)),
$$
for each $x \in M$, where $A(x),B(x)$ are thought of $n \times n$ matrices with components $A_{ij}(x), B_{ij}(x)$ respectively. We can now prove the following lemma.
\begin{lem}
\label{tensor field interpolation appendix}
For any $a = \sum_{ij} a_{ij} \frac{\partial}{\partial \theta^i} \otimes \frac{\partial}{ \partial \theta^j} \in T^{(2,0)}TM$, we have that with probability $1 - {\frac{n^2}{N}}$,
$$
\|a - I_{\phi_s} a \|_{L^2(T^{(2,0)}TM)} =  {O\left( N^{\frac{-2\alpha + (n-d)}{2d}} \right).}
$$
\end{lem}
\begin{proof}
    Again, notice that
    $$
    \langle a - I_{\phi_s}a , a - I_{\phi_s} a \rangle_x = \sum_{j=1}^n \sum_{i=1}^n (A_{ij}(x) - I_{\phi_s} A_{ij}(x))^2,
    $$
    for every $x \in M$. Integrating the above over $M$ and using Lemma \ref{function interpolation} $n^2$ times, we see that
    $$
    \left( \int_M \langle a - I_{\phi_s}a , a - I_{\phi_s} a \rangle_x d\textup{Vol} \right)^{1/2} = {O\left( N^{\frac{-2\alpha + (n-d)}{2d}} \right),}
    $$
    with probability higher than $1 - {\frac{n^2}{N}}$.  This completes the proof.
\end{proof}

\section{Proof of Spectral Convergence: the Laplace-Beltrami Operator }\label{app:C}
In this appendix, we study the consistency of the symmetric matrix,
$$
\mathbf{G}^\top\mathbf{G} = \sum_{i=1}^n \mathbf{G}^\top_i \mathbf{G}_i,
$$
as an approximation of the Laplace-Beltrami operator. We begin by focusing on the continuous (unrestricted) operator acting on a fixed, smooth function $f$, $\sum_{i} (\mathcal{G}_i I_{\phi_s})^*(\mathcal{G}_i I_{\phi_s})f$ and prove its convergence to $\Delta_Mf$ with high probability. Such results depend on the accuracy with which $I_{\phi_s}$ can approximate $f$ and its derivatives. We then quantify the error obtained when restricting to the data set and constructing a matrix. Finally, we prove convergence of eigenvalues and eigenvectors. Since the estimator in this case is symmetric, convergence of eigenvalues requires only a weak convergence result. To prove convergence of eigenvectors, however, we need convergence of our estimator in $L^2$ sense.

For the discrete approximation, we consider an inner product over restricted functions which is consistent with the inner product of $L^2(M)$ as the number of data points $N$ approaches infinity. The notation for this discrete inner product is given in Definition \ref{discrete L2}.
In the remainder of this section, we will use the notation based on the following definition to account for all of the errors.
\begin{definition}\label{epsilon_f}
Denote the $L^2$ norm error between $I_{\phi_s}f$ and $f$ by $\epsilon_{f}$, i.e.,
\BEA
\epsilon_f:=\| I_{\phi_s} f - f \|_{L^2(M)}. \notag 
\EEA
For concentration bound, we will also define a parameter $0\leq\delta_f\leq 1$ to probabilistically characterize an upper bound for $\epsilon_f$. For example, Lemma \ref{function interpolation} states that $\epsilon_f  = {O\left( N^{\frac{-2\alpha + (n-d)}{2d}} \right),}$ with probability higher than $1-\delta_f$, where $\delta_f = {1/N}$.
\end{definition}
For the next spectral result, we define the formal adjoint of $\textup{grad}_g I_{\phi_s}$ to be the unique linear operator $(\textup{grad}_g I_{\phi_s})^*: \mathfrak{X}(M) \to C^\infty(M)$ satisfying
$$
\langle (\textup{grad}_g I_{\phi_s})f, u \rangle_{L^2(\mathfrak{X}(M))} = \langle f , (\textup{grad}_g I_{\phi_s})^*u \rangle_{L^2(M)}
$$
for any $f \in C^\infty(M), u \in \mathfrak{X}(M)$.
Before proving the { spectral convergence results,} we prove a number of necessary pointwise and weak convergence results.

\subsection{Pointwise and Weak Convergence Results: Interpolation Error}\label{app:C1}
Using Cauchy-Schwarz, paried with the fact that the formal adjoint of $\textup{grad}_g$ is $-\textup{div}_g$, we immediately have the following Lemma.
\begin{lem}
\label{continuous gradient estimation}
Let $f \in C^\infty(M)$, and let $u \in \mathfrak{X}(M)$. Then with probability higher than $1 - \delta_f$,
$$
 \left| \langle \textup{grad}_g f - \textup{grad}_g I_{\phi_s} f, u \rangle_{L^2(\mathfrak{X}(M))} \right| \leq \epsilon_f \| \textup{div}_g(u) \|_{L^2(M)}.
$$
\end{lem}
A little more work also yields the following.
\begin{lem}
Let $f \in C^\infty(M)$. Then with probability higher than $1 - \delta_f$,
$$
\left| \| \textup{grad}_g f \|^2_{L^2(\mathfrak{X}(M))}- \| \textup{grad}_g I_{\phi_s} f\|^2_{L^2(\mathfrak{X}(M))}  \right| \leq \epsilon_f \left( \| \Delta_M f \|_{L^2(M)} + \| \Delta_M I_{\phi_s}f\|_{L^2(M)} \right).
$$
\end{lem}
\begin{proof}
Without ambiguity, we use the notation $\langle \cdot,\cdot\rangle$ for both inner products with respect to $L^2(M)$ and $L^2(\mathfrak{X}(M))$ in the derivations below to simplify the notation.  We begin by adding and subtracting the mixed term $\langle \textup{grad}_g f, \textup{grad}_g I_{\phi_s} f \rangle:$
\BEA
&&\left| \| \textup{grad}_g f \|^2 - \langle \textup{grad}_g f, \textup{grad}_g I_{\phi_s} f \rangle + \langle \textup{grad}_g f, \textup{grad}_g I_{\phi_s} f \rangle - \| \textup{grad}_g I_{\phi_s} f\|^2  \right|   \notag \\
&\leq& \left| \| \textup{grad}_g f \|^2 - \langle \textup{grad}_g f, \textup{grad}_g I_{\phi_s} f \rangle \right| + \left| \langle \textup{grad}_g f, \textup{grad}_g I_{\phi_s} f \rangle - \| \textup{grad}_g I_{\phi_s} f\|^2  \right|.  \notag \EEA
By the previous lemma, the first term is $\epsilon_f \| \Delta_M f \|_{L^2(M)}$, while the second term is \newline
$\epsilon_f \|\Delta_M I_{\phi_s} f \|_{L^2(M)}$.
\end{proof}

We also need the following result for studying the convergence of eigenvectors.
\begin{coro}
Let $f \in C^\infty(M)$. Then with probability higher than $1 - \delta_f$,
$$
\| \textup{grad}_g f - \textup{grad}_g I_{\phi_s} f \|^2_{L^2(\mathfrak{X}(M))} \leq \epsilon_f \left( \| \Delta_M f \|_{L^2(M)} + \| \Delta_M I_{\phi_s}f\|_{L^2(M)} \right).
$$
\end{coro}
\begin{proof}
Expanding yields
\begin{eqnarray*}
\| \textup{grad}_g f - \textup{grad}_g I_{\phi_s} f \|^2_{L^2(\mathfrak{X}(M))} &=& \langle \textup{grad}_g f, \textup{grad}_g f \rangle - \langle \textup{grad}_g f, \textup{grad}_g I_{\phi_s} f \rangle \\
&-& \langle \textup{grad}_g I_{\phi_s} f, \textup{grad}_g  f \rangle  + \langle \textup{grad}_g I_{\phi_s} f, \textup{grad}_g I_{\phi_s} f \rangle.   \notag
\end{eqnarray*}
Grouping the first two and last two terms, the desired result is immediate.
\end{proof}
{ We now prove Lemma~\ref{weak G convergence continuous} which states that $\Delta_M f $ can be weakly approximated.

\comment{\begin{lem}
\label{weak G convergence continuous}
Let $f,h \in C^\infty(M)$. Then with probability higher than $1 - \delta_f - \delta_h,$
$$
\left| \langle \Delta_M f , h \rangle_{L^2(M)} -\langle \textup{grad}_g I_{\phi_s} f , \textup{grad}_g I_{\phi_s} h \rangle_{L^2(M)}  \right| \leq \epsilon_h \|\Delta_M f\|_{L^2(M)} + \epsilon_f \| \Delta_M I_{\phi_s} h \|_{L^2(M)}.
$$
Here $0<\delta_h<1$ and $\epsilon_h>0$ are defined exactly as in Definition~\ref{epsilon_f} with $f$ replaced by $h$.
\end{lem}
}

\paragraph{\bf Proof of Lemma~\ref{weak G convergence continuous}}
We again add and subtract a mixed term.
\BEA
&&\langle \Delta_M f, h \rangle - \left\langle \textup{grad}_g I_{\phi_s}f, \textup{grad}_g I_{\phi_s} h \right\rangle \notag\\
&=& \langle \Delta_M f, h \rangle - \langle \textup{grad}_g f, \textup{grad}_g I_{\phi_s} h \rangle \notag
+  \langle \textup{grad}_g f, \textup{grad}_g I_{\phi_s} h \rangle - \left\langle \textup{grad}_g I_{\phi_s} f, \textup{grad}_g I_{\phi_s} h \right\rangle. \notag
\EEA
The first two terms are bounded by $\epsilon_h \|\Delta_M f\|_{L^2(M)}$ while the second two terms are bounded by $\epsilon_f \| \Delta_M I_{\phi_s} h \|_{L^2(M)}$ with a total probability higher than $1 - \delta_f - \delta_h.$ Two repeated uses of Lemma \ref{function interpolation} yield the final result. \hfill $\blacksquare$

}

Using similar arguments as above, we now have a convergence in $L^2$ norm result.
\begin{lem}\label{norm G convergence continuous}
Let $f \in C^\infty(M)$, and denote $(\textup{grad}_g I_{\phi_s})^*(\textup{grad}_g I_{\phi_s}) f$ by $h$. With probability higher than $1 - \delta_f - \delta_{\Delta_M f} - \delta_h,$
\BEA
\| \Delta_M f - h \|^2_{L^2(M)} &\leq& {
\epsilon_f \|\Delta^2_M f \|_{L^2(M)} + \epsilon_{\Delta_M f} \|\Delta_M f\|_{L^2(M)}\notag }\\
&+& {\epsilon_{h} \| \Delta_M f \|_{L^2(M)} + \epsilon_f \| \Delta_M I_{\phi_s} h  \|_{L^2(M)} }.\notag
\EEA
\end{lem}
\begin{proof}
Expanding $\| \Delta_M f - (\textup{grad}_g I_{\phi_s})^*(\textup{grad}_g I_{\phi_s}) f \|^2_{L^2(M)}$ yields
\BEA
&&\| \Delta_M f - (\textup{grad}_g I_{\phi_s})^*(\textup{grad}_g I_{\phi_s}) f  \|^2_{L^2(M)} \notag\\
&=&
\langle \Delta_M f, \Delta_M f \rangle - \langle \Delta_M f, (\textup{grad}_g I_{\phi_s})^*(\textup{grad}_g I_{\phi_s}) f \rangle
- \langle (\textup{grad}_g I_{\phi_s})^*(\textup{grad}_g I_{\phi_s}) f , \Delta_M f \rangle \notag \\
&+& \langle (\textup{grad}_g I_{\phi_s})^*(\textup{grad}_g I_{\phi_s}) f ,(\textup{grad}_g I_{\phi_s})^*(\textup{grad}_g I_{\phi_s}) f  \rangle.   \notag
\EEA
We use the previous Lemma twice, once on the first two terms (which gives an error of $\epsilon_{\Delta_M f} \| \Delta_Mf\|_{L^2(M)} + \epsilon_f \| \Delta_M I_{\phi_s} \Delta_M f \|_{L^2(M)}$)
and once on the last two terms (which gives an error of $\epsilon_{h} \| \Delta_M f \|_{L^2(M)} + \epsilon_f \| \Delta_M I_{\phi_s} h  \|_{L^2(M)}$).
This completes the proof.
\end{proof}
{
\begin{remark}
    \label{bounded interpolation remark}
Note that since each estimated function in this section lies within the RKHS space, it follows from Lemma \ref{function interpolation} that the error (denoted by $\epsilon$ with a subscript) converges with specified rate as $N \to \infty$. Moreover, Lemma \ref{bounded derivatives} ensures that no norms on the right-hand-side of the above estimates blow-up as $N \to \infty$.
\end{remark}
}

\subsection{Pointwise and Weak Convergence Results: Empirical Error}
\label{app:C2}
We now quantify the error obtained when discretizing our estimators on the data set $X$. The results of this section are primarily based on the law of large numbers. {First, we prove Lemma~\ref{weak G convergence empirical}.}

\comment{
The following Lemma is a direct consequence of a standard concentration result.
\begin{lem}
\label{weak G convergence empirical}
Let $f,h \in C^\infty(M)$. Let the sufficiently regular interpolator Assumption~\ref{weaklyunstable} be valid. Then
$$
\mathbb{P}_X\left( \left| \langle \mathbf{G}^\top \mathbf{G} \mathbf{f}, \mathbf{h} \rangle_{L^2(\mu_N)} - \langle \textup{grad}_g I_{\phi_s} f, \textup{grad}_g I_{\phi_s} h \rangle_{L^2(\mathfrak{X}(M))} \right| \geq \epsilon \right) \leq 2 \textup{exp}\left( \frac{-2\epsilon^2N}{C} \right),
$$
for some constant $C>0$.
\end{lem}
}

{
\paragraph{\bf Proof of Lemma~\ref{weak G convergence empirical}}
Using the fact that $\textup{grad}_g I_{\phi_s} f = (
\mathcal{G}_1I_{\phi_s} f,
\mathcal{G}_2I_{\phi_s} f,
 \ldots,
\mathcal{G}_n I_{\phi_s}f )^\top$ as defined in \eqref{sec2.1:gradg}, it is clear that,

\BEA
\langle \textup{grad}_g I_{\phi_s} f, \textup{grad}_g I_{\phi_s} f \rangle_{L^2(\mathfrak{X}(M))} =
\int_M \sum_{i=1}^n \left( \mathcal{G}_i I_{\phi_s} f (x) \right)\left( \mathcal{G}_i I_{\phi_s} h (x) \right) d\textup{Vol}(x), \notag
\EEA
and we can see immediately that the result follows from a concentration inequality on the random variable $\sum_{i=1}^n \left( \mathcal{G}_i I_{\phi_s} f (x) \right)\left( \mathcal{G}_i I_{\phi_s} h (x) \right)$. The range of $I_{\phi_s}$ is in $C^{\alpha - \frac{(n-d)}{2}}(M)$, { and by the Assumption~\ref{weaklyunstable} along with Lemma \ref{bounded derivatives}}, since $\mathcal{G}_i$ are simply differential operators, we see that the random variable is bounded by a constant $C$ (depending on the kernel $\phi_s$, as well as $f$). By Hoeffding's inequality, we obtain,
$$
\mathbb{P}_X\left( \left| \langle \mathbf{G}^\top \mathbf{G} \mathbf{f}, \mathbf{h} \rangle_{L^2(\mu_N)} - \langle \textup{grad}_g I_{\phi_s} f, \textup{grad}_g I_{\phi_s} h \rangle_{L^2(\mathfrak{X}(M))} \right| \geq \epsilon \right) \leq 2 \textup{exp}\left( \frac{-2\epsilon^2N}{c} \right),
$$
for some constant $c>0$. Take $N^{-1} = \textup{exp}\left( \frac{-2\epsilon^2N}{c} \right)$, solve for $\epsilon$, the proof is complete. \hfill$\blacksquare$}

Since $\mathbf{G}$ is simply the restricted version of $\textup{grad}_g I_{\phi_s}$, the same reasoning as above gives the following result, which will be needed to prove convergence of eigenvectors.
\begin{lem}
\label{norm G convergence empirical}
Let $f \in C^\infty (M)$.  Let  Assumption~\ref{weaklyunstable} be valid. Then
\begin{footnotesize}
$$
\mathbb{P}_X\left( \left| \Vert \mathbf{G}^\top \mathbf{G} \mathbf{f}  - R_N \Delta_M f \Vert^2_{L^2(\mu_N)} - \|(\textup{grad}_g I_{\phi_s})^*(\textup{grad}_g I_{\phi_s}) f - \Delta_M f \|^2_{L^2(M)} \right| \geq \epsilon \right) \leq 2 \textup{exp}\left( \frac{-2\epsilon^2 N}{C} \right),
$$
\end{footnotesize}
for some constant $C>0$.
\end{lem}

\subsection{Proof of Eigenvector Convergence}\label{app:C3}
We now prove the convergence of eigenvectors. The outline of this proof follows the arguments in the convergence analysis found in  \cite{calder2019improved} and \cite{peoples2021spectral}. It is important to note that since the matrix $\mathbf{G}^\top\mathbf{G}$ is symmetric, there exists an orthonormal basis of eigenvectors of $\mathbf{G}^\top\mathbf{G}$, which is key in the following proof of Theorem \ref{conveigvec}.
\begin{proof}
Fix some $\ell \in \mathbb{N}$. For convenience, we let $\epsilon_{\lambda_\ell}$ denote the error in approximating the $\ell$-th eigenvalue, from the previous section. Similarly, we let $\delta_{\lambda_\ell}$ denote the quantity such that eigenvalue approximation occurs with probability higher than $1 - \delta_{\lambda_\ell}$.  Let $m$ be the geometric multiplicity of the eigenvalue $\lambda_\ell$, i.e., there is an $i$ such that $ \lambda_{i+1} =  \lambda_{i+2} = \dots = \lambda_\ell = \dots = \lambda_{i+m}$. Let
\begin{equation*}
    c_\ell = \frac{1}{2} \textup{ min} \left\{ |\lambda_{\ell} - \lambda_{i}|, |\lambda_{\ell} - \lambda_{i+m+1}| \right\}.
\end{equation*}
By Theorem~\ref{eigvalconv}, if $\epsilon_{\lambda_i},\epsilon_{\lambda_{i+m+1}} < c_\ell$, then with probability $1 - \delta_{\lambda_i} - \delta_{\lambda_{i+m+1}}$,
\begin{equation*}
    |\hat{\lambda}_{i} - \lambda_{i}| < c_\ell , \quad |\hat{\lambda}_{i+m+1} - \lambda_{i+m+1}| < c_\ell.
\end{equation*}
Let $\mathbf{\hat{u}}_{1}, \dots \mathbf{\hat{u}}_{N}$ be an orthonormal basis of $L^2(\mu_N)$ consisting of eigenvectors of $\mathbf{G}^\top\mathbf{G}$, where $\mathbf{\hat{u}}_j$ has eigenvalue $\hat{\lambda}_j$.  Let $S$ be the $m$ dimensional subspace of $L^2(\mu_N)$ corresponding to the span of $\{ \mathbf{\hat{u}}_{j} \}^{i+m}_{j=i+1}$, and let $P_S$ (resp. $P^\perp_S$) denote the projection onto $S$ (resp. orthogonal complement of $S$). Let $f$ be a norm $1$ eigenfunction of $\Delta_M$ corresponding to eigenvalue $\lambda_\ell$. Notice that
\begin{equation*}
    P^\perp_S R_N \Delta_M f = \lambda_\ell P^\perp_S R_N f = \lambda_{\ell} \sum_{j \neq i+1, \dots , i+m} \langle R_N f, \mathbf{\hat{u}}_j \rangle_{L^2(\mu_N)} \mathbf{\hat{u}}_j.
\end{equation*}
Similarly,
\begin{equation*}
     P^\perp_S \mathbf{G}^\top\mathbf{G} R_N f = \sum_{j \neq i+1, \dots , i+m} \hat{\lambda}_{j}\langle  R_N f, \mathbf{\hat{u}}_j \rangle_{L^2(\mu_N)} \mathbf{\hat{u}}_j.
\end{equation*}
Hence,
\BEA
\Big\Vert P^\perp_S R_N \Delta_M f &-& P^\perp_S \mathbf{G}^\top\mathbf{G} R_N f \Big\Vert_{L^2(\mu_N)} =  \Big\Vert \sum_{j \neq i+1, \dots , i+m} (\lambda_\ell - \hat{\lambda}_{j} )\langle  R_N f, \mathbf{\hat{u}}_j \rangle_{L^2(\mu_N)} \mathbf{\hat{u}}_j \Big\Vert_{L^2(\mu_N)} \notag\\
&\geq &\min \Big\{ |\lambda_\ell - \hat{\lambda}_{i}|, |\lambda_\ell - \hat{\lambda}_{i+m+1}|\Big\} \Big\Vert \sum_{j \neq i+1, \dots , i+m}  \langle  R_N f, \mathbf{\hat{u}}_j \rangle_{L^2(\mu_N)} \mathbf{\hat{u}}_j \Big\Vert_{L^2(\mu_N)} \notag\\
&\geq &\min \Big\{ |\lambda_\ell - \hat{\lambda}_{i}|, |\lambda_\ell - \hat{\lambda}_{i+m+1}| \Big\} \Big\Vert P_{S}^\perp  R_N f \Big\Vert_{L^2(\mu_N)}.\notag
\EEA
But $P^\perp_S$ is an orthogonal projection, so
\BEA
\textup{ min} \big\{ |\lambda_\ell - \hat{\lambda}_{i}|, |\lambda_\ell - \hat{\lambda}_{i+m+1}|  \big\} \big\| P_S^\perp R_N f \big\|_{L^2(\mu_N)} &\leq& \big\|P^\perp_S R_N \Delta_M f - P^\perp_S \mathbf{G}^\top\mathbf{G} R_N f \big\|_{L^2(\mu_N)} \notag \\
&\leq& \big\| R_N \Delta_M f - \mathbf{G}^\top\mathbf{G} R_N f \big\|_{L^2(\mu_N)}.\notag
\EEA
Without loss of generality, assume that $\textup{ min} \{ |\lambda_\ell - \hat{\lambda}_{i}|, |\lambda_\ell - \hat{\lambda}_{i+m+1}|  \} = |\lambda_\ell - \hat{\lambda}_{i}|$. Notice that
\begin{equation*}
    |\lambda_\ell - \hat{\lambda}_{i}| \geq \left|  |\lambda_\ell - \lambda_{i}| - |\lambda_{i} - \hat{\lambda}_{i}| \right| > c_\ell,
\end{equation*}
by the hypothesis. Hence,
\begin{equation*}
    \| P_S^\perp R_N f \|^2_{L^2(\mu_N)} \leq \frac{1}{c^2_\ell} \| R_N \Delta_M f - \mathbf{G}^\top\mathbf{G} R_N f \|^2_{L^2(\mu_N)}.
\end{equation*}
By Lemma \ref{norm G convergence empirical} paired with Lemma \ref{norm G convergence continuous}, this upper bound is smaller than
$$
{ \epsilon_f \|\Delta^2_M f \|_{L^2(M)} + \epsilon_{\Delta_M f} \|\Delta_M f\|_{L^2(M)} + \epsilon_{h} \| \Delta_M f \|_{L^2(M)} + \epsilon_f \| \Delta_M I_{\phi_s} h  \|_{L^2(M)} + O\big( N^{-\frac{1}{2}} \big), }
$$
with probability higher than {$1 - \frac{2}{N} - \delta_f - \delta_{\Delta_M f} - \delta_h$, where $h$ is defined as in Lemma \ref{norm G convergence continuous}.} Notice that $P_S^\perp R_N f = R_N f - P_S R_N f$. Hence, if $\{f_1, f_2, \dots , f_m \}$ are an orthonormal basis for the eigenspace corresponding to $\lambda_\ell$, applying the above reasoning  $m$ times, we see that
with a total probability of   $ 1 - \frac{2}{N} - \delta_{f_1} - \delta_{\Delta_M f_1} - \delta_{h_1} - \dots - \frac{2}{N} - \delta_{f_m} - \delta_{\Delta_M f_m} - \delta_{h_m}$,
\BEA
    \|R_N f_j - P_S R_N f_j \|^2_{L^2(\mu_N)} &\leq& \frac{1}{c^2_\ell} \Big(  \epsilon_{f_j} C_{f_j} \|\Delta^2_M f_j \|_{L^2(M)} + 2 \epsilon_{\Delta_M f_j} \|\Delta_M f_j\|_{L^2(M)} \notag \\
     &+& \epsilon_{h_j} \|\Delta_M f_j\|_{L^2(M)} + \epsilon_{f_j} \|\Delta_M I_{\phi_s} h_j\|_{L^2(M)} + O\big(N^{-\frac{1}{2}} \big) \Big) \notag \\
     &:=& \textup{Error}(j), \label{L2eigvecerror}
\EEA
for  $j = 1, 2, \dots m$. Let $C_\ell$ denote an upper bound on the essential supremum of the eigenvectors $\{f_1, f_2, \dots , f_m \}$. For any $i,j$,
\begin{equation*}
    \big|f_i(x)f_j(x) - \int_M f_j(y)f_i(y) d\mu(y)\big| \leq C_\ell^2(1+\mbox{Vol}(M)).
\end{equation*}
Hence, using Hoeffding's inequality with $\alpha = 2\sqrt{2}C_\ell \sqrt{\frac{\log(N)}{N}}$, with probability $1 - \frac{2}{N}$, \begin{equation*}
    \left|\frac{1}{N}\sum_{l=1}^N f_i(x_l)f_j(x_l) - \int_{M} f_i(y)f_j(y) d\mu(y)  \right| < \alpha.
\end{equation*}
Since $\{f_1,\ldots,f_m\}$ are orthonormal in $L^2(M)$, by Hoeffding's inequality used $m^2$ times, we see that probability higher than $1 - \frac{2m^2}{N}$,
\begin{equation*}
    \langle R_Nf_i, R_N f_j \rangle_{L^2(\mu_N)} = \delta_{ij} + O\left( \frac{1}{\sqrt{N}} \right).
\end{equation*}
Hence, with a total probability higher than $1 - \delta_{\lambda_i} - \delta_{\lambda_{i+m+1}} - \frac{2m^2}{N} - \frac{2}{N} - \delta_{f_1} - \delta_{\Delta_M f_1} - \delta_{h_1} - \dots - \frac{2}{N} - \delta_{f_m} - \delta_{\Delta_M f_m} - \delta_{h_m} $, we have that
\begin{small}
\BEA
\langle P_S R_N f_i , P_S R_N f_j \rangle_{L^2(\mu_N)} &=& \langle R_N f_i , R_N f_j \rangle_{L^2(\mu_N)} - \langle R_N f_i - P_S R_N f_i , R_N f_j - P_S R_N f_j \rangle_{L^2(\mu_N)}\notag
\\ &=& \delta_{ij} + O\left( \frac{1}{\sqrt{N}}\right) + \sqrt{\textup{Error}(i)} \sqrt{\textup{Error}(j)} ,\notag
\EEA
\end{small}
where $\textup{Error}(j)$ is as defined in \eqref{L2eigvecerror}.
Letting $\mathbf{v}_1 = \frac{P_SR_Nf_1}{\| P_S R_N f_1 \|_{L^2(\mu_N)}}$, we see that
\[\| P_SR_N f_1 - \mathbf{v}_1 \|^2_{L^2(\mu_N)} =  O\big( 1 / \sqrt{N}\big) + O\big( \textup{Error}(1) \big).\]
 Similarly, letting $\mathbf{\tilde{v}}_2 = P_SR_Nf_2 - \frac{\langle P_SR_Nf_1, P_SR_Nf_2 \rangle_{L^2(\mu_N)}}{\|P_SR_N f_1 \|^2_{L^2(\mu_N)}} P_SR_Nf_1$ and $\mathbf{v}_2 = \frac{\mathbf{\tilde{v}}_2}{\|\mathbf{\tilde{v}}_2\|_{L^2(\mu_N)}}$, it is easy to see that
\begin{equation*}
\|P_SR_Nf_2 - \mathbf{\tilde{v}}_2\|^2_{L^2(\mu_N)} = O\big( 1 / \sqrt{N}\big) + O\big( \textup{Error}(2)  \big),
\end{equation*}
and hence,
\begin{equation*}
   \|P_SR_Nf_2 - \mathbf{v}_2\|^2_{L^2(\mu_N)} =  O\big( 1 / \sqrt{N}\big) + O \big( \textup{Error}(2) \big).
\end{equation*}
Continuing in this way, we see that the Gram-Schmidt procedure on  $\{ P_SR_Nf_j \}^m_{j=1}$, yields an orthonormal set of $m$ vectors $\{ \mathbf{v}_j\}_{j=1}^m$ spanning $S$ such that
\begin{equation*}
    \|P_SR_Nf_j - \mathbf{v}_j\|^2_{L^2(\mu_N)} = O\big( 1 / \sqrt{N}\big) + O\big( \textup{Error}(j) \big),  \qquad j = 1,2,\dots, m,
\end{equation*}
and therefore
\BEA
    \|R_Nf_j - \mathbf{v}_j\|_{L^2(\mu_N)} &\leq& \| R_N f_j - P_S R_N f_j \|_{L^2(\mu_N)} + \| P_S R_N f_j - \mathbf{v}_j \|_{L^2(\mu_N)} \notag \\
    &=& 2 \sqrt{O\big( 1 / \sqrt{N}\big) + O\big( \textup{Error}(j) \big) },\qquad j = 1,2,\dots, m. \notag
\EEA
Therefore, for any eigenvector $\mathbf{u} = \sum_{j=1}^m b_j \mathbf{v}_j$ with $L^2(\mu_N)$ norm $1$, notice that $f = \sum_{j=1}^m b_j f_j$ is a $L^2(M)$ norm $1$ eigenfunction of $\Delta_M$ with eigenvalue $\lambda_{\ell}$. Indeed,
\BEA
\Delta_M f = \sum_{j=1}^m b_j \Delta_M f_j = \lambda_\ell f,\notag
\EEA
and
\BEA
\|f\|^2_{L^2(M)} = \langle f,f \rangle_{L^2(M)} = \sum_{i=1}^m\sum_{j=1}^m b_i b_j \langle f_i , f_j \rangle_{L^2(M)} = \sum_{j=1}^m b_j^2 = 1, \notag
\EEA where the last equality follows from the fact that
$$
\|\mathbf{u}\|_{L^2(\mu_N)}^2  = \|\sum_{i=1}^m b_i \mathbf{v}_i  \|_{L^2(\mu_N)}^2= \sum_{i=1}^m b_i^2 =1.
$$
Moreover, the function $f$ also satisfies
\begin{equation*}
    \| R_N f - \mathbf{u} \|^2_{L^2(\mu_N)} \leq \sum_{j=1}^m |b_j|^2 \| R_N f_j - \mathbf{v}_j \|^2_{L^2(\mu_N)} = O\big( 1 / \sqrt{N}\big) + \sum_{j=1}^m O \big( \textup{Error}(j) \big).
\end{equation*}
Using Lemma \ref{function interpolation} and collecting the probabilities, the above holds with probability higher than $1 - { \left(\frac{2m^2 + 5m + 24}{N}\right)}$. Moreover, each $\textup{Error}(j)$ is on the order of $O\left( N^{-\frac{1}{2}} \right) + {O\left( N^{\frac{-2\alpha + (n-d)}{2d}} \right)}$. Taking the square root yields the final result. This completes the proof.
\end{proof}

\section{Proof of Spectral Convergence: the Bochner Laplacian }\label{app:D}
In this appendix, we discuss theoretical results concerning the Bochner Laplacian approximated by the symmetric matrix
$$
\mathbf{P}^{\otimes}\mathbf{H}^\top\mathbf{H} \mathbf{P}^{\otimes} = \sum_{i=1}^n  \mathbf{P}^{\otimes}\mathbf{H}^\top_i \mathbf{H}_i\mathbf{P}^{\otimes}.
$$
This appendix is organized analogously to Appendix~\ref{app:C} which studies the spectral convergence of the symmetric approximation to the  Laplace-Beltrami operator.
{Since the discussion of spectral convergence involves interpolating $(2,0)$ tensor fields and approximating the corresponding continuous inner products with discrete ones, we begin by giving a brief discussion of these details.
We then} investigate the continuous counterpart of the above discrete estimator, and prove its convergence in terms of interpolation error to the Bochner Laplacian in the weak sense, as well as in $L^2(\mathfrak{X}(M))$ sense.
After applying law of large numbers results to quantify the error obtained by discretizing to the data set, we prove spectral convergence results.

{We note that Lemma \ref{vector field interpolation appendix}, which is the analogue of Lemma \ref{function interpolation} for vector fields, is especially useful for this section. Its proof is simply an application of Lemma \ref{function interpolation} $n$ times.
Similarly, an interpolation of $(2,0)$ tensor-fields result is needed. This can also be found in Lemma \ref{tensor field interpolation appendix}, and is essentially an application of Lemma \ref{function interpolation} $n^2$ times.
We also need the following definition, analogous to Definition \ref{epsilon_f}, but in the setting of vector fields.}
\begin{definition}\label{epsilon_v}
Denote the $L^2$ norm error between the vector fields $I_{\phi_s}v$ and $v$ by $\epsilon_{v}$, i.e.,
\BEA
\epsilon_v:=\| I_{\phi_s} v - v \|_{L^2(\mathfrak{X}(M)}. \notag 
\EEA
We will also define a parameter $0\leq\delta_v\leq 1$ to probabilistically characterize an upper bound for $\epsilon_v$. For example, Lemma \ref{vector field interpolation appendix} states that $\epsilon_v  = {O\left( N^{\frac{-2\alpha + (n-d)}{2d}} \right)}$ with probability higher than $1-\delta_v$, where $\delta_v = {n/N}$.
\end{definition}

\subsection{Interpolation of (2,0)-Tensor Fields and Approximation of Inner Products}
In the Bochner Laplacian discussion, we need to interpolate $(2,0)$ tensor fields, and approximate the corresponding continuous inner product with a discrete inner product. We outline the strategy for doing so presently. Given a $(2,0)$ tensor field $a =  a_{ij} \frac{\partial}{\partial \theta^i} \otimes \frac{\partial}{\partial \theta^j} $, we can extend $a$ to $A = A_{ij} \frac{\partial}{\partial X^i} \otimes \frac{\partial}{\partial X^j} $, defined on a neighborhood of $M$ in $\mathbb{R}^n$, and write it as an $n \times n$ matrix in the basis $\frac{\partial}{\partial X^i} \otimes \frac{\partial}{\partial X^j}$. We define $I_{\phi_s} A$ to be the $(2,0)$ tensor field with components $I_{\phi_s} A_{ij}$ in the above ambient space basis. Recall that if $a,b \in T^{(2,0)}TM$, then by definition to the Riemannian inner product at $x$ is given by
$$
\langle a , b \rangle_x = \sum_{i,j,k,l} a_{ij} b_{kl} \left\langle  \frac{\partial}{\partial \theta^i}  ,   \frac{\partial}{\partial \theta^k} \right\rangle_x  \left\langle  \frac{\partial}{\partial \theta^j},  \frac{\partial}{\partial \theta^l} \right\rangle_x.
$$
Performing a change of basis to the ambient space coordinates, a computation shows that
$$
\langle a , b \rangle_x = \textup{tr}(A(x)^\top B(x)),
$$
where $A,B$ are the extensions of $a,b$ and thought of as $n \times n$ matrices written w.r.t. the ambient space coordinates. Hence, defining the restriction of a $(2,0)$ tensor field $A = A_{ij} \frac{\partial}{\partial X^i} \otimes \frac{\partial}{\partial X^i} $ to be the tensor $R_N {A} \in \mathbb{R}^{n \times n \times N}$ with components,
$$
(R_NA)_{ijk} = A_{ij}(x_k),
$$
it follows that the inner product on $\mathbb{R}^{n \times n \times N}$ given by
$$
\langle R_NA, R_N B \rangle_{L^2(\mu_{N, n\times n})} := \frac{1}{N} \sum_{k=1}^N \textup{tr}\left( A(x_k)^\top B(x_k) \right),
$$
approximates the continuous inner product on $(2,0)$ tensor fields over $M$. In a previous notion of discrete $(2,0)$ tensor fields $[\mathbf{U}_1 , \dots , \mathbf{U}_n] \in \mathbb{R}^{{nN} \times n}$, each column $\mathbf{U}_i$ was thought of as a restricted vector field, equipped with the inner product,
$$
\langle  [\mathbf{U}_1 , \dots , \mathbf{U}_n], [\mathbf{V}_1 ,\dots , \mathbf{V}_n] \rangle_{L^2(\mu_{N, N \times n})} := \frac{1}{N} \sum_{i=1}^n \mathbf{U}_i \cdot \mathbf{V}_i.
$$
These two notions are equivalent in the following sense. Define $\Phi: \mathbb{R}^{nN \times n} \to \mathbb{R}^{n \times n \times N}$ by
$$
\Phi \mathbf{E}_{(i-1)N+k, j} = \mathbf{E}_{i,j,k} \qquad i,j=1,2, \dots , n \textup{ and } k=1, \dots , N,
$$
where $\mathbf{E}_{(i-1)N+k, j}$ denotes the $nN \times n$ matrix with a one in entry $((i-1)N + k,j)$ and zeros elsewhere. Similarly for $\mathbf{E}_{i,j,k}$. One can easily check that $\Phi$ is an isometric isomorphism between the two inner product spaces defined above. In what follows, we will use this identification to consider $\mathbf{H}: \mathbb{R}^{nN} \to \mathbb{R}^{nN \times n}$ as a map with range in $\mathbb{R}^{n \times n \times N}$. This will be useful for computing the adjoint. Denote by $L^2(\mu_{N,n})$ the inner product on $\mathbb{R}^{nN}$ which approximates $L^2(\mathfrak{X}(M))$:
$$
\langle \mathbf{U}, \mathbf{V} \rangle_{L^2(\mu_{N,n})} = \frac{1}{N} \sum_{j=1}^{n} \mathbf{U}^j \cdot \mathbf{V}^j,
$$
where $\mathbf{U}^j \in \mathbb{R}^N$, and $\mathbf{U} = ((\mathbf{U}^1)^{\top} , \dots , (\mathbf{U}^{n})^{\top})^{\top}$. Using the identification above, it is simple to check that the transpose of $ \Phi \mathbf{H}: \mathbb{R}^{nN} \to \mathbb{R}^{n \times n \times N}$ with respect to inner products defined above
is given by,
$$
[\mathbf{\tilde{U}}^1 , \dots , \mathbf{\tilde{U}}^n] \mapsto [\mathbf{U}^1 , \dots , \mathbf{U}^n] \mapsto \sum_i \mathbf{H}_i^\top \mathbf{U}^i,
$$
where $[\mathbf{\tilde{U}}^1 , \dots , \mathbf{\tilde{U}}^n] \in \mathbb{R}^{n \times n \times N}$. Since $\Phi$ is an isometric isomorphism, $\mathbf{H}^\top \Phi^\top \Phi \mathbf{H} = \mathbf{H}^\top \mathbf{H}$. This shows that $\mathbf{H}^\top \mathbf{H}:\mathbb{R}^{nN}\to \mathbb{R}^{nN}$ is indeed given by the formula,
$$
\mathbf{H}^\top \mathbf{H} = \sum_i \mathbf{H}^\top_i \mathbf{H}_i,
$$
which will be used extensively in the following calculations.

\subsection{Pointwise and Weak Convergence Results: Interpolation Error}
Using Cauchy-Schwarz, along with the fact that the formal adjoint of $\textup{grad}_g$ acting on vector fields is $-\textup{div}^1_1$ as defined in \eqref{div_1^1}, we immediately have the following result.
\begin{lem}
\label{continuous gradient estimation Bochner}
Let $u \in \mathfrak{X}(M)$, and let $a \in T^{(2,0)}TM$. Then with probability higher than $1 - \delta_u$,
$$
 \big| \langle \textup{grad}_g u - \textup{grad}_g I_{\phi_s} u, a \rangle_{L^2(T^{(2,0)}TM)} \big| \leq \epsilon_u \| \textup{div}^1_1(a) \|_{L^2(\mathfrak{X}(M))}.
$$
\end{lem}
We also have the following norm result.
\begin{coro}
Let $u \in \mathfrak{X}(M)$. Then with probability higher than $1 - \delta_u$,
$$
\| \textup{grad}_g u - \textup{grad}_g I_{\phi_s} u \|^2_{L^2(T^{(2,0)}TM)} \leq \epsilon_u \left( \| \Delta_B u \|_{L^2(\mathfrak{X}(M))} + \| \Delta_B I_{\phi_s}u\|_{L^2(\mathfrak{X}(M))} \right).
$$
\end{coro}
\begin{proof}
Note that,
\BEA
\| \textup{grad}_g u - \textup{grad}_g I_{\phi_s} u \|^2 &=& \langle \textup{grad}_g u, \textup{grad}_g u \rangle - \langle \textup{grad}_g u, \textup{grad}_g I_{\phi_s} u \rangle - \langle \textup{grad}_g I_{\phi_s} u, \textup{grad}_g  u \rangle  \notag\\
&+& \langle \textup{grad}_g I_{\phi_s} u, \textup{grad}_g I_{\phi_s} u \rangle.   \notag
\EEA
Grouping the first two and last two terms, the desired result is immediate.
\end{proof}
Using the same reasoning as before, we can deduce a weak convergence result.
\begin{lem}
\label{weak H convergence continuous}
Let $v,w \in \mathfrak{X}(M)$. Then with probability higher than $1 - \delta_v - \delta_w,$
\BEA
&&\left| \langle \Delta_B v , w \rangle_{L^2(\mathfrak{X}(M))} -\langle \textup{grad}_g I_{\phi_s} v , \textup{grad}_g I_{\phi_s}w \rangle_{L^2(T^{(2,0)}TM)}  \right|  \notag\\
&\leq& \epsilon_w \|\Delta_B v\|_{L^2(\mathfrak{X}(M))} + \epsilon_v \| \Delta_B I_{\phi_s} w \|_{L^2(\mathfrak{X}(M))}.\notag
\EEA
\end{lem}
\begin{proof}
We again add and subtract a mixed term,
$$
\langle \Delta_B v, w \rangle - \langle \textup{grad}_g v, \textup{grad}_g I_{\phi_s} w \rangle +  \langle \textup{grad}_g v, \textup{grad}_g I_{\phi_s} w \rangle - \left\langle \textup{grad}_g I_{\phi_s} v, \textup{grad}_g I_{\phi_s}w \right\rangle.
$$
The first two terms are bounded by $\epsilon_w \|\Delta_B v\|_{L^2(\mathfrak{X}(M))}$ while the second two terms are bounded by $\epsilon_v \| \Delta_B I_{\phi_s} w \|_{L^2(\mathfrak{X}(M))}$.
\end{proof}
Similarly, we can derive an $L^2$ convergence result. First, similar to the previous section, let us denote the formal adjoint of $\textup{grad}_g I_{\phi_s}:\mathfrak{X}(M) \to T^{(2,0)}TM$  by $(\textup{grad}_g I_{\phi_s})^*: T^{(2,0)}TM \to \mathfrak{X}(M)$. The exact same proof as before yields the following Lemma.

\begin{lem}
\label{norm H convergence continuous}
Let $v \in \mathfrak{X}(M)$, { and denote by $w$ the vector field $(\textup{grad}_g I_{\phi_s})^*(\textup{grad}_g I_{\phi_s})v$}. With probability higher than $1 - \delta_v - \delta_{\Delta_B v} - \delta_{w},$
\BEA
\| \Delta_B v - w \|^2 \leq
\epsilon_{\Delta_B v} \| \Delta_B v\| &+& \epsilon_v \| \Delta_B I_{\phi_s} \Delta_B v \|\notag \\
\epsilon_{w} \| \Delta_B v\| &+& \epsilon_v \| \Delta_B I_{\phi_s} \textup{grad}_g w  \|. \notag
\EEA
\end{lem}
\begin{proof}
Expanding $\| \Delta_B v - (\textup{grad}_g I_{\phi_s})^*(\textup{grad}_g I_{\phi_s}) v \|^2$ yields
\BEA
\| \Delta_B v - (\textup{grad}_g I_{\phi_s})^*(\textup{grad}_g I_{\phi_s}) v  \|^2 &=&
\langle \Delta_B v, \Delta_B v \rangle - \langle \Delta_B v, (\textup{grad}_g I_{\phi_s})^*(\textup{grad}_g I_{\phi_s}) v \rangle \notag \\
&-& \langle (\textup{grad}_g I_{\phi_s})^*(\textup{grad}_g I_{\phi_s}) v , \Delta_B v \rangle \notag \\
&+& \langle (\textup{grad}_g I_{\phi_s})^*(\textup{grad}_g I_{\phi_s}) v ,(\textup{grad}_g I_{\phi_s})^*(\textup{grad}_g I_{\phi_s}) v  \rangle.   \notag
\EEA
We use the previous Lemma twice, once on the first two terms (which gives an error of
$\epsilon_{\Delta_B v} \| \Delta_B v\| + \epsilon_v \| \Delta_B I_{\phi_s} \Delta_B v \|$) and once on the last two terms
(which gives an error of $\epsilon_{w} \| \Delta_B v\| + \epsilon_v \| \Delta_B I_{\phi_s} \textup{grad}_g w  \|$). This completes the proof.
\end{proof}
{
\begin{remark}
Similar to Remark \ref{bounded interpolation remark} for the Laplace-Beltrami operator, we note that Lemma \ref{vector field interpolation appendix} guarantees all estimation error terms (errors indicated by $\epsilon$ with corresponding subscript) converge as $N \to \infty$. Additionally, it follows from Lemma \ref{bounded derivatives vector field}, since the Bochner Laplacian is obtained from second order differential operators on the ambient space coefficients of each vector field, that no norm terms on the right-hand-side of each estimate above blow up as $N \to \infty$.
\end{remark}}

\subsection{Pointwise and Weak Convergence Results: Empirical Error}
The following is a simple concentration result.
\begin{lem}
\label{weak H convergence empirical}
Let $u,v \in \mathfrak{X}(M)$. Additionally, let Assumption \ref{weaklyunstable} be valid. Then
\begin{footnotesize}
$$
\mathbb{P}\left( \left| \langle \mathbf{H} \mathbf{P}^{\otimes} R_N u, \mathbf{H} \mathbf{P}^{\otimes} R_N v \rangle_{L^2(\mu_{N,n \times n})} - \langle \textup{grad}_g I_{\phi_s} u, \textup{grad}_g I_{\phi_s} v \rangle_{L^2(T^{(2,0)}TM)} \right| \geq \epsilon \right) \leq 2 \textup{exp}\left( \frac{-2\epsilon^2N}{C} \right),
$$
\end{footnotesize}
for some constant $C>0$.
\end{lem}
\begin{proof}
We note that
\BEA
\langle \textup{grad}_g I_{\phi_s} u, \textup{grad}_g I_{\phi_s} v \rangle &=& \int_M \textup{tr}\left( \left[\mathcal{H}_1 U , \dots \mathcal{H}_n U\right]^\top \left[\mathcal{H}_1 V , \dots \mathcal{H}_n V \right] \right) d\textup{Vol}(x)  \notag \\
&=& \int_M \sum_{i=1}^n \left\langle \mathcal{H}_i U , \mathcal{H}_i V \right\rangle_x d\textup{Vol}(x) = \int_M \sum_{i=1}^n \left\langle \mathcal{H}_i \mathbf{P} U , \mathcal{H}_i \mathbf{P} V \right\rangle_x d\textup{Vol}(x), \notag
\EEA
where $U$ and $V$ are extensions of $u,v$ onto an $\mathbb{R}^n$ neighborhood of $M$. The last equality comes from the fact that since $U,V$ extend $u,v$, then at each point $x \in M$, we have $\mathbf{P}U = U, \mathbf{P}V = v$. We can see immediately that the result follows by using a concentration inequality on the random variable $\sum_{i=1}^n \left\langle \mathcal{H}_i \mathbf{P} U , \mathcal{H}_i \mathbf{P} V \right\rangle_x$. Plugging in to Hoeffding's inequality and using smoothness assumptions, { along with Assumption \ref{weaklyunstable} and Lemma \ref{bounded derivatives vector field},} gives the desired result.
\end{proof}
Again, $\mathbf{H}$ is simply the restricted version of $\textup{grad}_g I_{\phi_s}$, as defined in \eqref{sec2.4:eq2}. Hence, the same reasoning as above yields the norm convergence result, which plays a part in proving the convergence of eigenvectors.
\begin{lem}
\label{norm H convergence empirical}
Let $v \in \mathfrak{X} (M)$. Additionally, let Assumption \ref{weaklyunstable} be valid. Then
\begin{footnotesize}
$$
\mathbb{P}\left( \left| \| \mathbf{P}^{\otimes}\mathbf{H}^\top \mathbf{H} \mathbf{P}^\otimes v  - R_N \Delta_B v \|^2_{L^2(\mu_N,n)} - \|(\textup{grad}_g I_{\phi_s})^*(\textup{grad}_g I_{\phi_s}) v - \Delta_B v \|^2_{L^2(\mathfrak{X}(M))} \right| \geq \epsilon \right) \leq 2 \textup{exp}\left( \frac{-2\epsilon^2 N}{C} \right),
$$
\end{footnotesize}
for some constant $C>0$.
\end{lem}

\subsection{Proof of Spectral Convergence}
Here, we prove Theorem \ref{eigvalconv Bochner}, the convergence of eigenvalue result for the Bochner Laplacian operator.
\begin{proof}
Enumerate the eigenvalues of $\mathbf{P}^\otimes\mathbf{H}^\top\mathbf{H}\mathbf{P}^\otimes$ and label them $\hat{\lambda}_1 \leq \hat{\lambda}_2 \leq \dots \leq \hat{\lambda}_N$. Let $\mathcal{S}_i' \subseteq \mathfrak{X}(M)$ denote an $i$-dimensional subspace of smooth functions on which the quantity $\textup{max}_{v \in \mathcal{S}_i} \frac{\langle \mathbf{P}^\otimes\mathbf{H}^\top\mathbf{H}\mathbf{P}^\otimes R_Nv , R_Nv \rangle_{L^2(\mu_{N,n})}}{\|R_N v \|_{L^2(\mu_{N,n})}}$ achieves its minimum. Let $\tilde{v} \in \mathcal{S}_i'$ be the smooth vector field on which the maximum $\textup{max}_{v \in \mathcal{S}'_i} \langle \Delta_B v , v \rangle$ occurs. WLOG, assume that $\|\tilde{v}\|^2_{L^2(\mathfrak{X}(M))} = 1.$ Assume that $N$ is sufficiently large so that by Hoeffding's inequality $\left| \|R_N \tilde{v} \|^2_{L^2(\mu_{N,n})} -  1 \right| \leq \frac{\textup{Const}}{\sqrt{N}} \leq 1/2$, with probability $1 - \frac{2}{N}$, and thus $\|R_N \tilde{v} \|^2_{L^2(\mu_{N,n})}$ is bounded away from zero. Hence, we can Taylor expand the denominator term of $ \frac{\langle  \mathbf{P}^\otimes \mathbf{H}^\top\mathbf{H} \mathbf{P}^\otimes R_N \tilde{v} , R_N \tilde{v} \rangle_{L^2(\mu_{N,n})}}{\|R_N \tilde{v} \|^2_{L^2(\mu_{N,n})}}$ to obtain
\begin{footnotesize}
\BEA
 \frac{\langle \mathbf{P}^\otimes \mathbf{H}^\top\mathbf{H} \mathbf{P}^\otimes R_N \tilde{v} , R_N \tilde{v} \rangle_{L^2(\mu_{N,n})}}{ \|R_N \tilde{v}\|^2_{L^2(\mu_{N,n})}} = \langle \mathbf{P}^\otimes \mathbf{H}^\top\mathbf{H} \mathbf{P}^\otimes R_N \tilde{v} , R_N \tilde{v} \rangle_{L^2(\mu_{N,n})}
 -  \frac{ \textup{Const} \langle \mathbf{P}^\otimes \mathbf{H}^\top\mathbf{H} \mathbf{P}^\otimes R_N \tilde{v} , R_N \tilde{v} \rangle_{L^2(\mu_{N,n})} }{\sqrt{N}}. \notag
\EEA
\end{footnotesize}
Note that with probability higher than $1 - \frac{2}{N}$, we have that
$$
\left| \langle \mathbf{P}^\otimes \mathbf{H}^\top\mathbf{H} \mathbf{P}^\otimes R_N \tilde{v} , R_N \tilde{v} \rangle_{L^2(\mu_{N,n})} - \langle \textup{grad}_g I_{\phi_s} \tilde{v}, \textup{grad}_g I_{\phi_s} \tilde{v} \rangle_{L^2(T^{(1,1)}(TM))} \right| = O\left( N^{- \frac{1}{2}} \right),
$$
by Lemma \ref{weak H convergence empirical}, choosing $\epsilon =  \sqrt{\frac{\log(N)}{N^{1/2}}}$ and ignoring log factors. Combining with Lemma \ref{weak H convergence continuous} and plugging in the result from Lemma \ref{vector field interpolation appendix}, we obtain that
$$
\langle \Delta_B \tilde{v} , \tilde{v} \rangle_{L^2(\mathfrak{X}(M))} \leq   \frac{\langle \mathbf{P}^\otimes \mathbf{H}^\top\mathbf{H}\mathbf{P}^\otimes R_N\tilde{v} , R_N \tilde{v} \rangle_{L^2(\mu_{N,n})}}{\|R_N \tilde{v} \|^2_{L^2(\mu_{N,n})}}  + O\left( N^{- \frac{1}{2}} \right) + {O\left( N^{\frac{-2\alpha + (n-d)}{2d}} \right)},
$$
with total probability higher than $1 - {\frac{2n + 2}{N}}$. Since $\tilde{v} = \arg\max_{v \in \mathcal{S}_i'}\langle \Delta_B v, v \rangle_{L^2(\mathfrak{X}(M))}$,  and since, \[
\frac{\langle \mathbf{P}^\otimes \mathbf{H}^\top\mathbf{H} \mathbf{P}^\otimes R_N \tilde{v} , R_N \tilde{v} \rangle_{L^2(\mu_{N,n})}}{\|R_N v \|_{L^2(\mu_{N,n})}} \leq \textup{max}_{v \in \mathcal{S}_i'} \frac{\langle  \mathbf{P}^\otimes \mathbf{H}^\top\mathbf{H} \mathbf{P}^\otimes R_Nv , R_Nv \rangle_{L^2(\mu_{N,n})}}{\|R_N v \|_{L^2(\mu_{N,n})}},\]
we have the following:
\begin{small}
$$
\textup{max}_{v \in \mathcal{S}_i'} \langle \Delta_B v , v \rangle_{L^2(\mathfrak{X}(M))} \leq  \textup{max}_{v \in \mathcal{S}_i'} \frac{\langle \mathbf{P}^\otimes \mathbf{H}^\top\mathbf{H} \mathbf{P}^\otimes R_Nv , R_Nv \rangle_{L^2(\mu_{N,n})}}{\|R_N v \|_{L^2(\mu_{N,n})}}  + O\left( N^{- \frac{1}{2}} \right) + O\left(N^{-\frac{2\alpha + (n-d)}{2d}}\right).
$$
\end{small}
Since $\mathcal{S}_i'$ is the exact subspace on which $\textup{max}_{v \in \mathcal{S}_i} \frac{\langle \mathbf{P}^\otimes \mathbf{H}^\top\mathbf{H}\mathbf{P}^\otimes R_Nv , R_Nv \rangle_{L^2(\mu_{N,n})}}{\|R_N v \|_{L^2(\mu_{N,n})}}$ achieves its minimum, then
$$
\textup{max}_{v \in \mathcal{S}_i'} \langle \Delta_B v , v \rangle_{L^2(\mathfrak{X}(M))} \leq \hat{\lambda}_i +  O\left( N^{- \frac{1}{2}} \right) + {O\left( N^{\frac{-2\alpha + (n-d)}{2d}} \right)}.
$$
But the left-hand-side certainly bounds above the minimum of $\textup{max}_{v \in \mathcal{S}_i} \langle \Delta_B v , v \rangle_{L^2(\mathfrak{X}(M))}$ over all $i$-dimensional smooth subspaces $\mathcal{S}_i$. Hence,
$$
\lambda_i \leq \hat{\lambda}_i +  O\left( N^{- \frac{1}{2}} \right) + {O\left( N^{\frac{-2\alpha + (n-d)}{2d}} \right)}.
$$
The same argument yields that $\hat{\lambda}_i \leq \lambda_i +  O\left( N^{- \frac{1}{2}} \right) + {O\left( N^{\frac{-2\alpha + (n-d)}{2d}} \right)}$, with probability higher than $1 - {\frac{2+2n}{N}}$. This completes the proof.
\end{proof}
\begin{remark}
    Since the exact same proof as the eigenvector convergence for the Laplace Beltrami operator is valid, we do not rewrite it in full detail. Instead, we simply note that while Theorem \ref{conveigvec} is proved using Theorem \ref{eigvalconv}, as well as  Lemmas \ref{function interpolation}, \ref{norm G convergence continuous}, and \ref{norm G convergence empirical}, similarly, we have that Theorem \ref{conveigvec Bochner} can be proved with the exact same argument, using instead Theorem \ref{eigvalconv Bochner}, along with Lemmas \ref{vector field interpolation appendix}, \ref{norm H convergence continuous}, and \ref{norm H convergence empirical}. Statements and proofs of the lemmas mentioned above can be found in the present section.
\end{remark}

{
\section{Additional Numerical Results with Diffusion Maps}\label{app:E}

In this section, we report additional numerical results with diffusion maps on the 2D torus example to support the conclusion in Section~\ref{sec:gentorus}. Particularly, we would like to demonstrate that other choices of nearest neighbor parameter, $K$, and kernel bandwidth parameter, $\epsilon$, do not improve the accuracy of the estimates, and thus, the particular choice of $K$ and $\epsilon$ in Section~\ref{sec:gentorus} is representative.

Figure~\ref{fig_gentor1} shows the eigenvector estimates for mode $k=13$. Each row in these figures corresponds to manually tuned $\epsilon$ for a fixed $K$. In the last row, we also include the auto-tuned $\epsilon$ (using the method that was originally proposed in \cite{coifman2008TuningEpsilon}). One can see that the auto-tuned $\epsilon$ seems to only work for smaller values of $K$. The denser the matrix, the estimates seem to be more sensitive to the choice of $\epsilon$. On the other hand, one can also see that qualitatively the eigenvector estimates for various choices of $K$ and $\epsilon$ (including $K=N$) are not very different from those with smaller values of $K$. In Figure~\ref{fig_gentor2}, we give further details on errors in the estimation of eigenvalues and eigenvectors for each $K$ and $\epsilon$. As a baseline, we plot the errors corresponding to SRBF $\hat{\mathbf{P}}$. Overall, one can see that when $K$ is small, the leading eigenvalues can be more accurately estimated by DM with an appropriate choice of $\epsilon$. For larger $K$ (e.g. $K=400, 2500$, see panel (d)), one can see $\epsilon=0.05$ leads to accurate estimation of only the leading eigenvalues (red curves) but for $\epsilon=0.2$ the non-leading spectra becomes more accurate in the expense of less accurate leading eigenvalue estimation. At the same time, the accuracy in the estimation of the corresponding eigenfunction is more or less similar for any choice of $K$, confirming the qualitative results shown in Figure~\ref{fig_gentor1}. Indeed if one chooses $K=400$ and $\epsilon=0.2$, while the errors of the estimation of eigenvalues are smaller than those of SRBF $\hat{\mathbf{P}}$ (compare grey and cyan curves in panel (d) in Figure~\ref{fig_gentor2}), the corresponding eigenvector estimate for mode $k=13$ does not qualitatively reflect on an accurate estimation at all (see row 4, column 5 in Figure~\ref{fig_gentor1}).

Based on the results in Figure~\ref{fig_gentor2} panel (f), Notice that the best estimate corresponds to the case of $K=100$ (red curves). Based on this empirical result, we present the case $K=100$ with auto-tuned $\epsilon$ in Figures~\ref{fig_gentor_1}-\ref{fig_gentor_5}.

\begin{figure*}[tbp]
{\scriptsize \centering
\begin{tabular}{c}
{\small {\bf{manually-tuned}} $\epsilon$, {\bf{mode}} $k={\mathbf{13}}$ }  \\
\includegraphics[height=1.15
in]{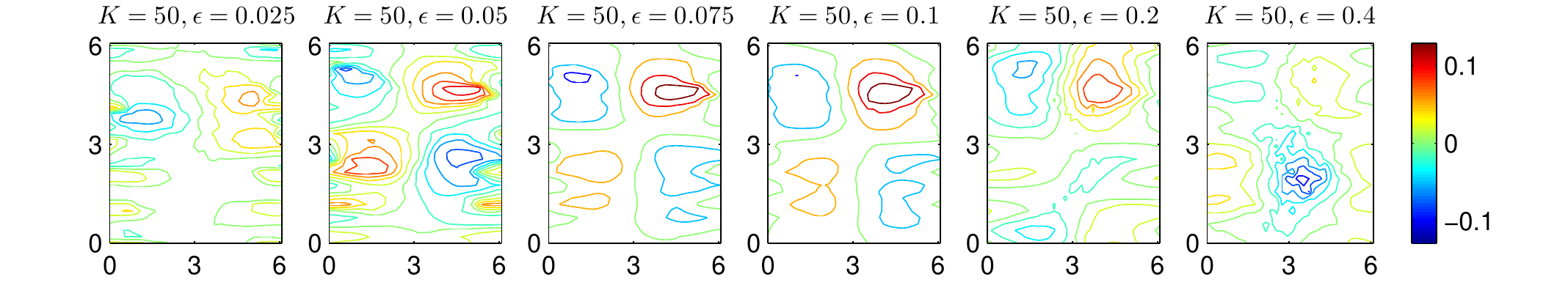}
\\
\includegraphics[height=1.15
in]{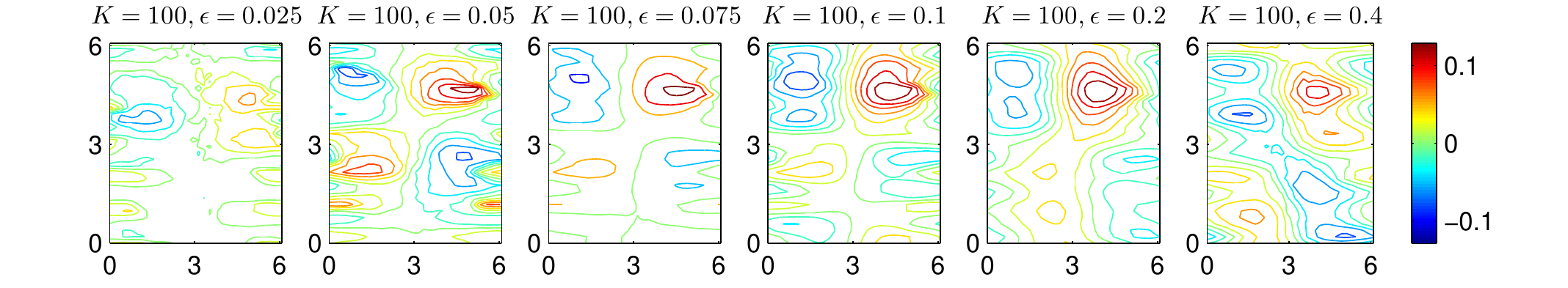}
\\
\includegraphics[height=1.15
in]{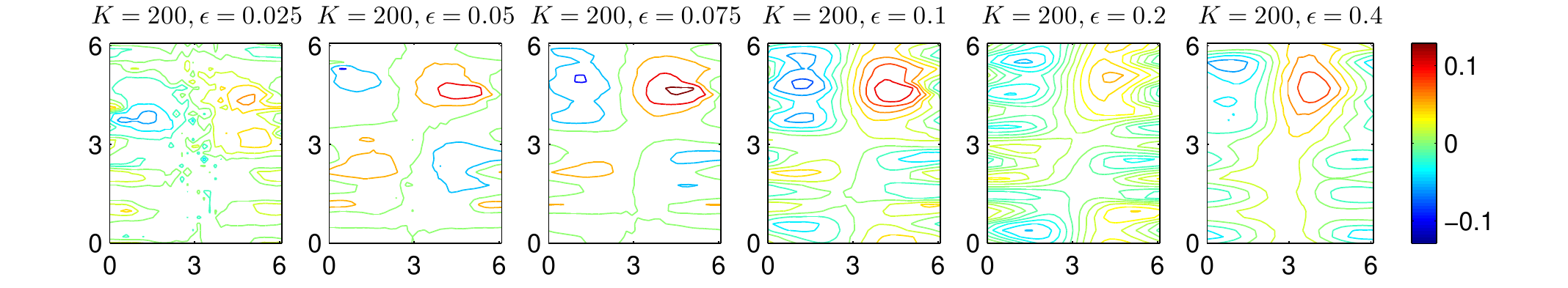}
\\
\includegraphics[height=1.15
in]{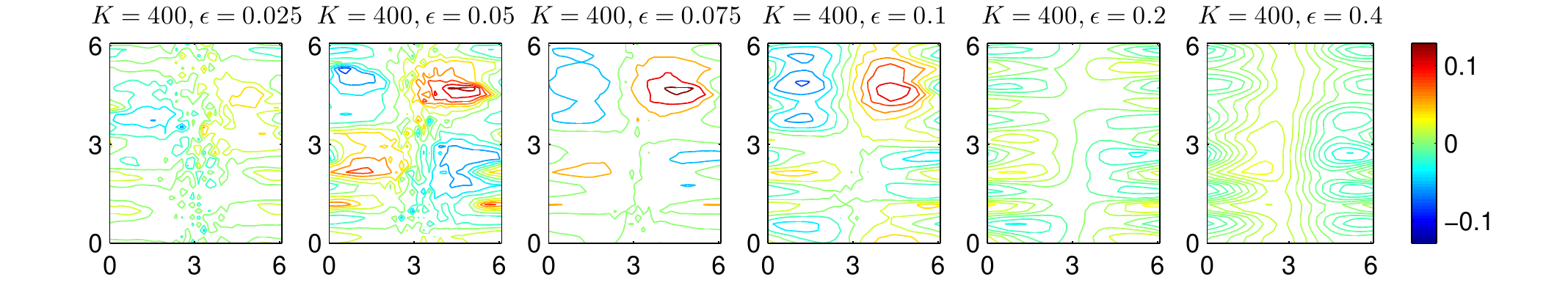}
\\
\includegraphics[height=1.15
in]{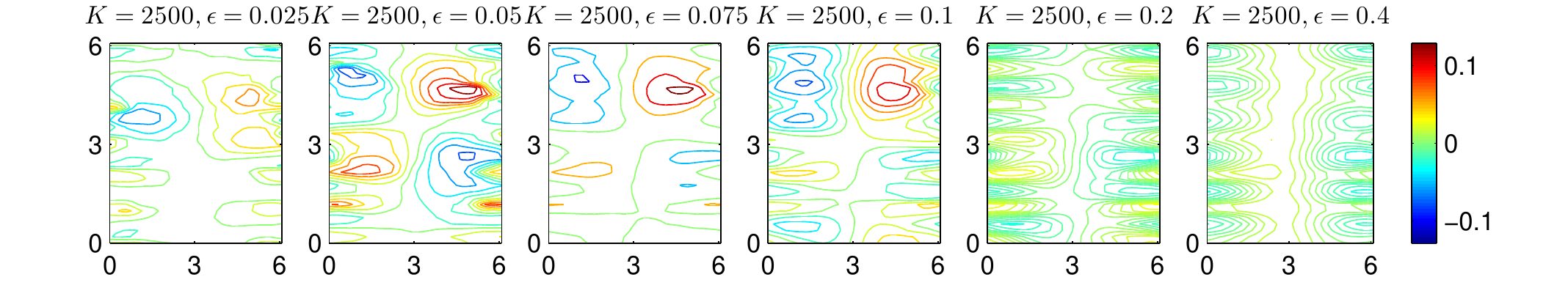}%
\\
\hline
{\small {\bf{auto-tuned}} $\epsilon$, {\bf{mode}} $k={\mathbf{13}}$ }  \\
\includegraphics[height=1.15
in]{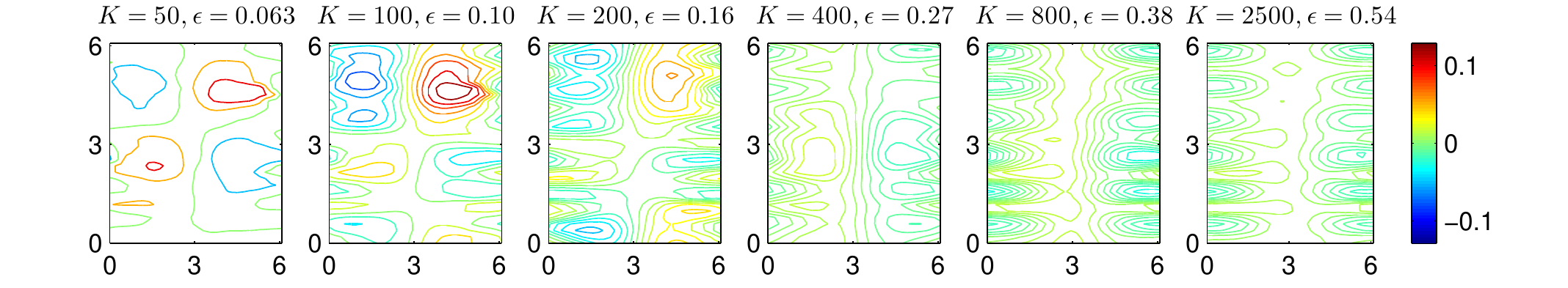}%
\end{tabular}
}
\caption{\textbf{2D general torus in $\mathbb{R}^{21}$.} Mode $k=13$. Comparison of
eigenfunctions of Laplace -Beltrami among various DMs with different $K$-nearest neighbors and bandwidth $\epsilon$.
Panels from the first row  to the fifth row  correspond to manually-tuned $\protect%
\epsilon$ for $K=50,100,200,400,2500$, respectively. Panels in the last row correspond
to auto-tuned $\epsilon$ for various $K$.
The
randomly distributed $N=2500$ data points on the manifold are used for computing
the eigenvalue problem.
}
\label{fig_gentor1}
\end{figure*}

\begin{figure*}[tbp]
{\scriptsize \centering
\begin{tabular}{cc|cc}
{\small (a) $K={\mathbf{50}}$} & {\small manually tuned $\epsilon$} & {\small (b) $K={\mathbf{100}}$} & {\small manually tuned $\epsilon$} \\
{\small Error of Eigenvalues} & {\small Error of Eigenvectors} & {\small
Error of Eigenvalues } & {\small Error of Eigenvectors } \\
\includegraphics[width=1.5
in, height=1.1 in]{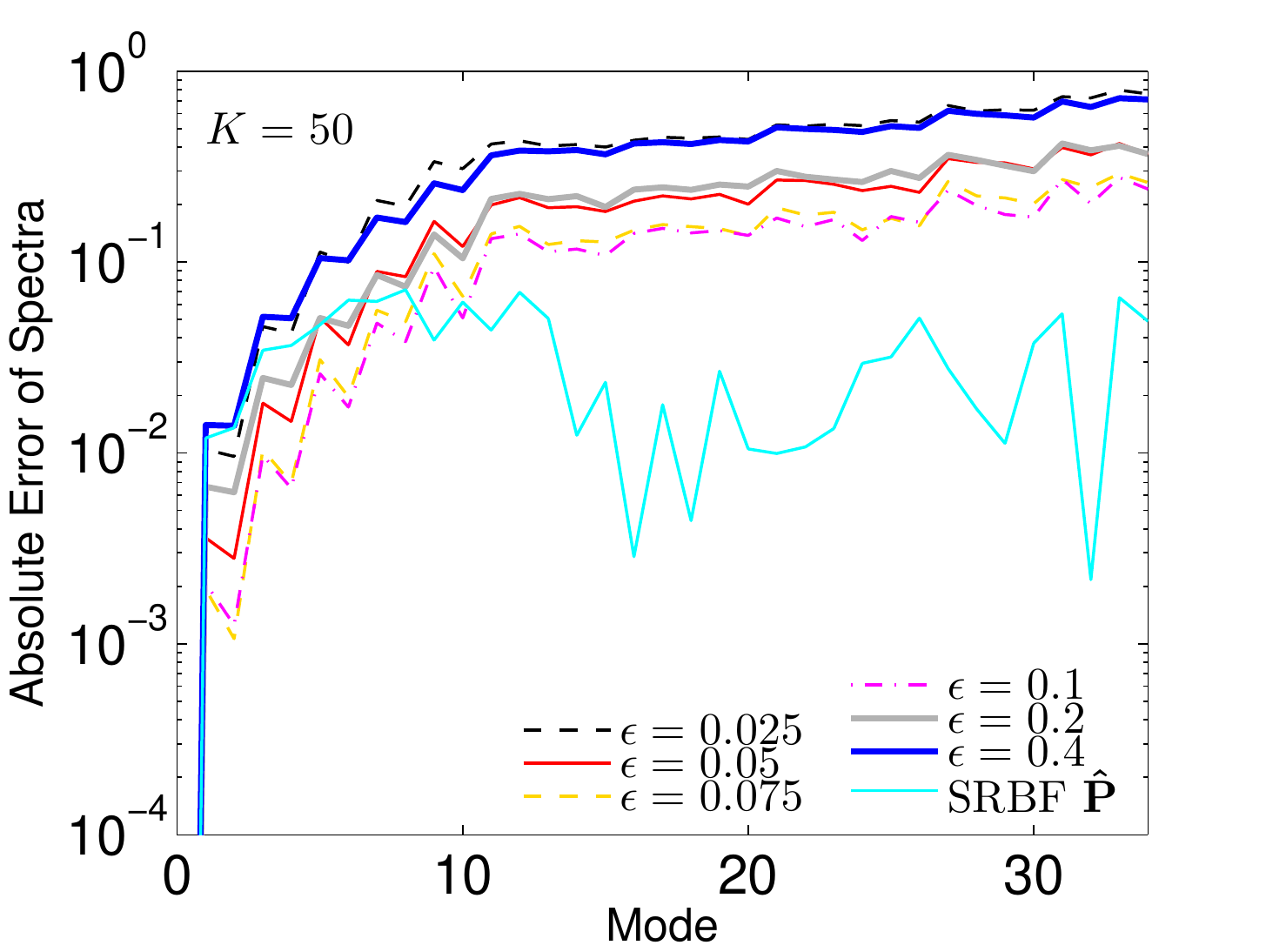}
&
\includegraphics[width=1.5
in, height=1.1 in]{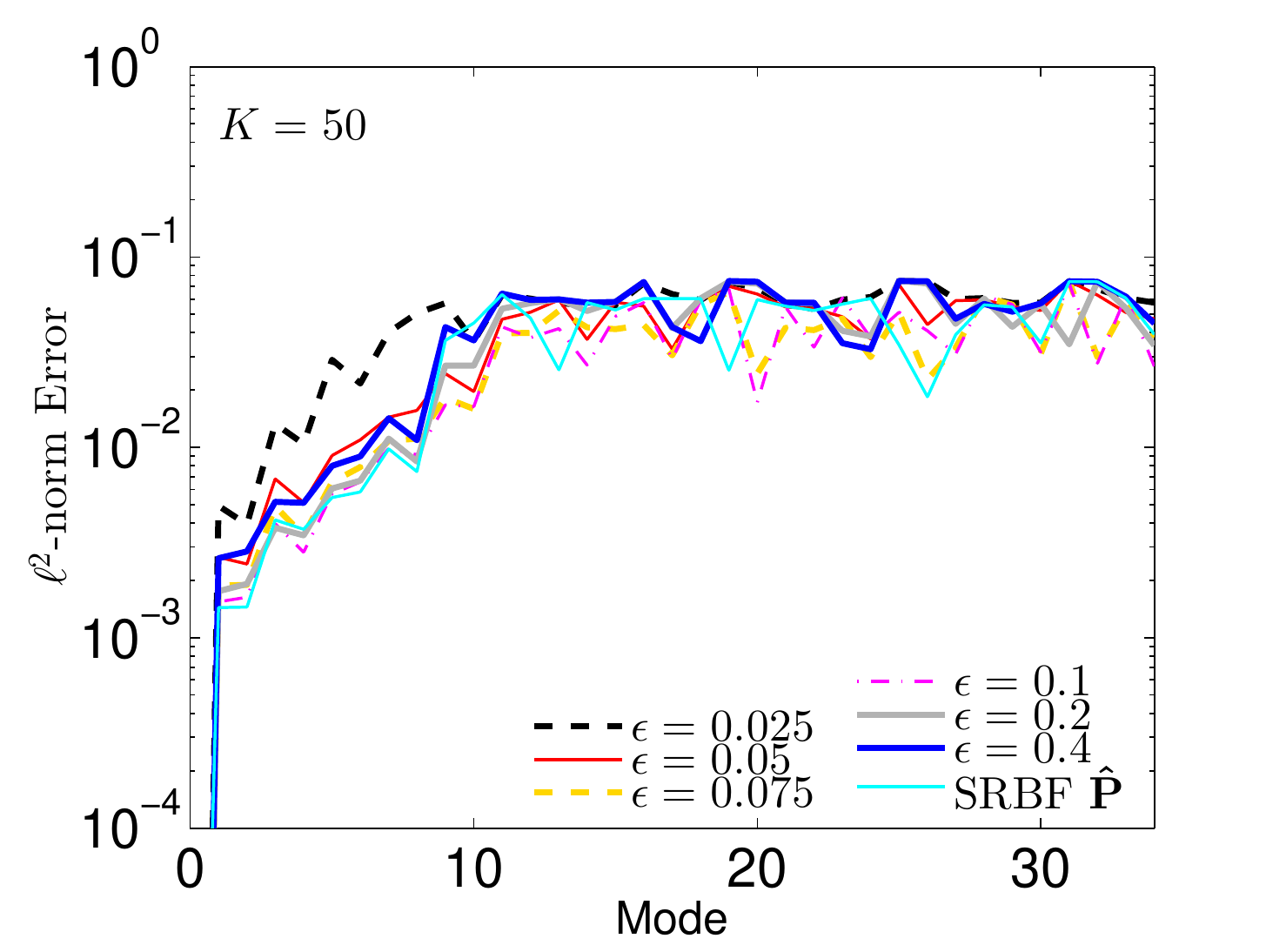}
&
\includegraphics[width=1.5
in, height=1.1 in]{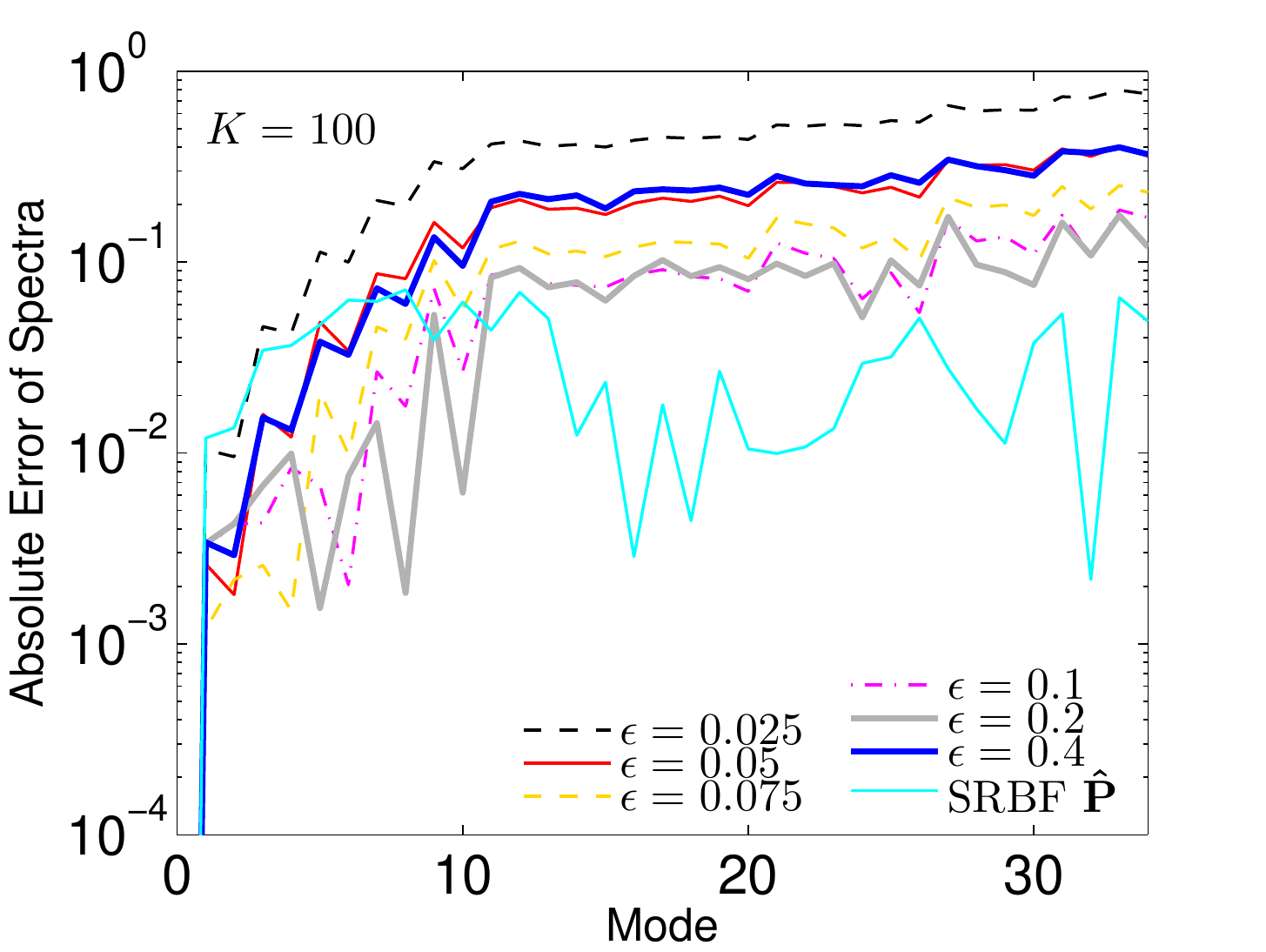}
&
\includegraphics[width=1.5
in, height=1.1 in]{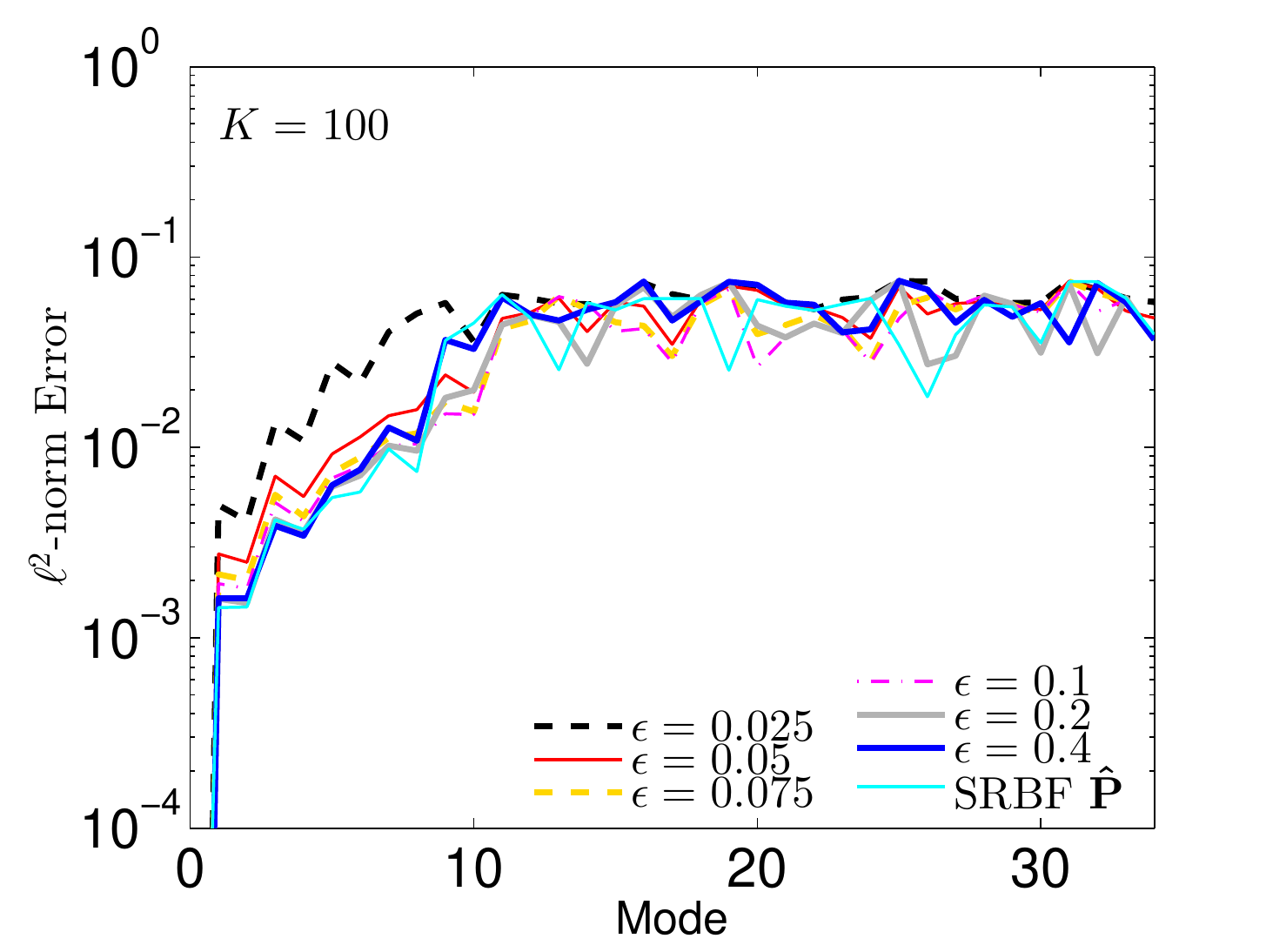}
\\
\hline
{\small (c) $K={\mathbf{200}}$} & {\small manually tuned $\epsilon$} & {\small (d) $K={\mathbf{400}}$} & {\small manually tuned $\epsilon$} \\
{\small Error of Eigenvalues} & {\small Error of Eigenvectors} & {\small
Error of Eigenvalues } & {\small Error of Eigenvectors } \\
\includegraphics[width=1.5
in, height=1.1 in]{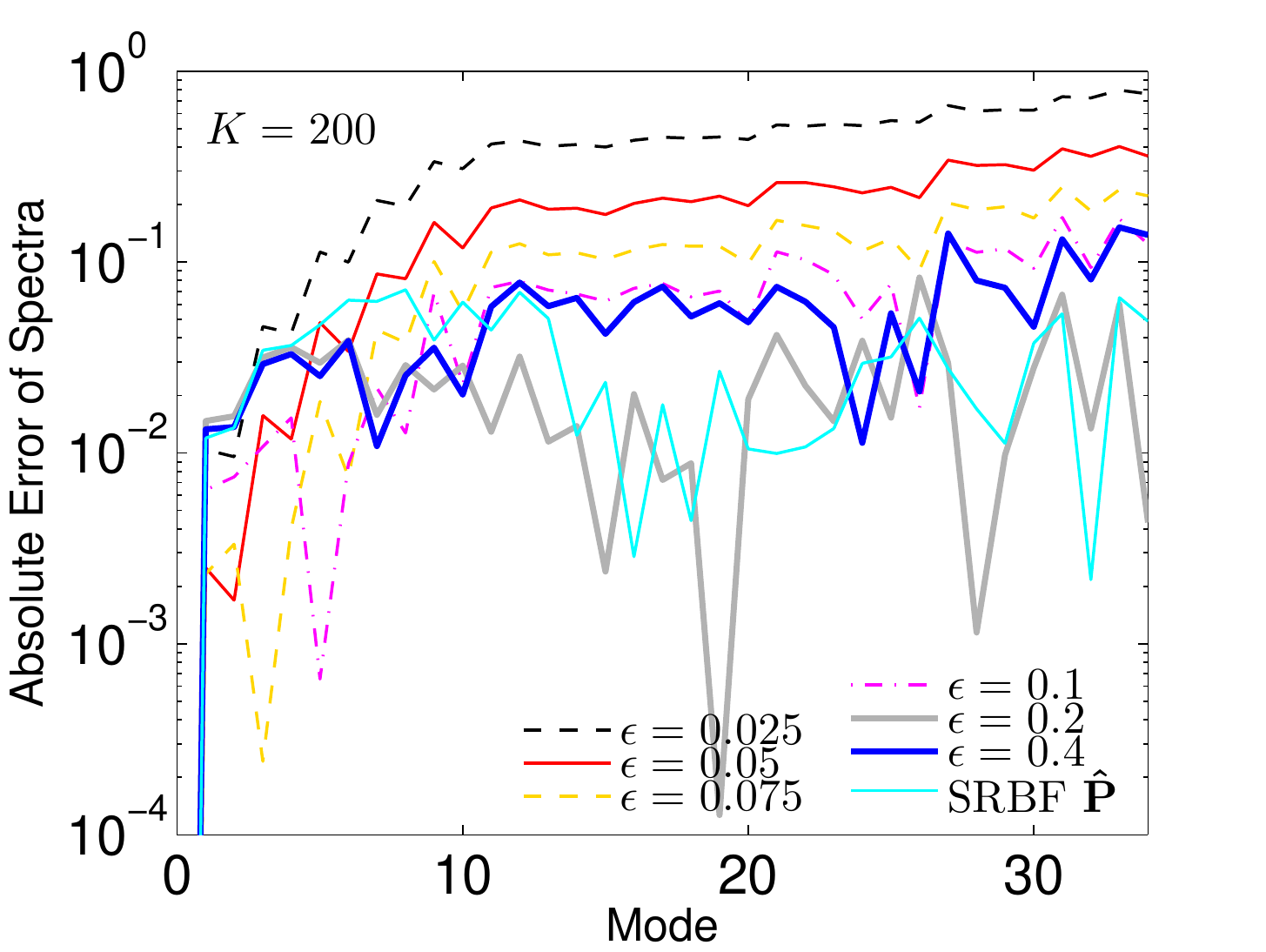}
&
\includegraphics[width=1.5
in, height=1.1 in]{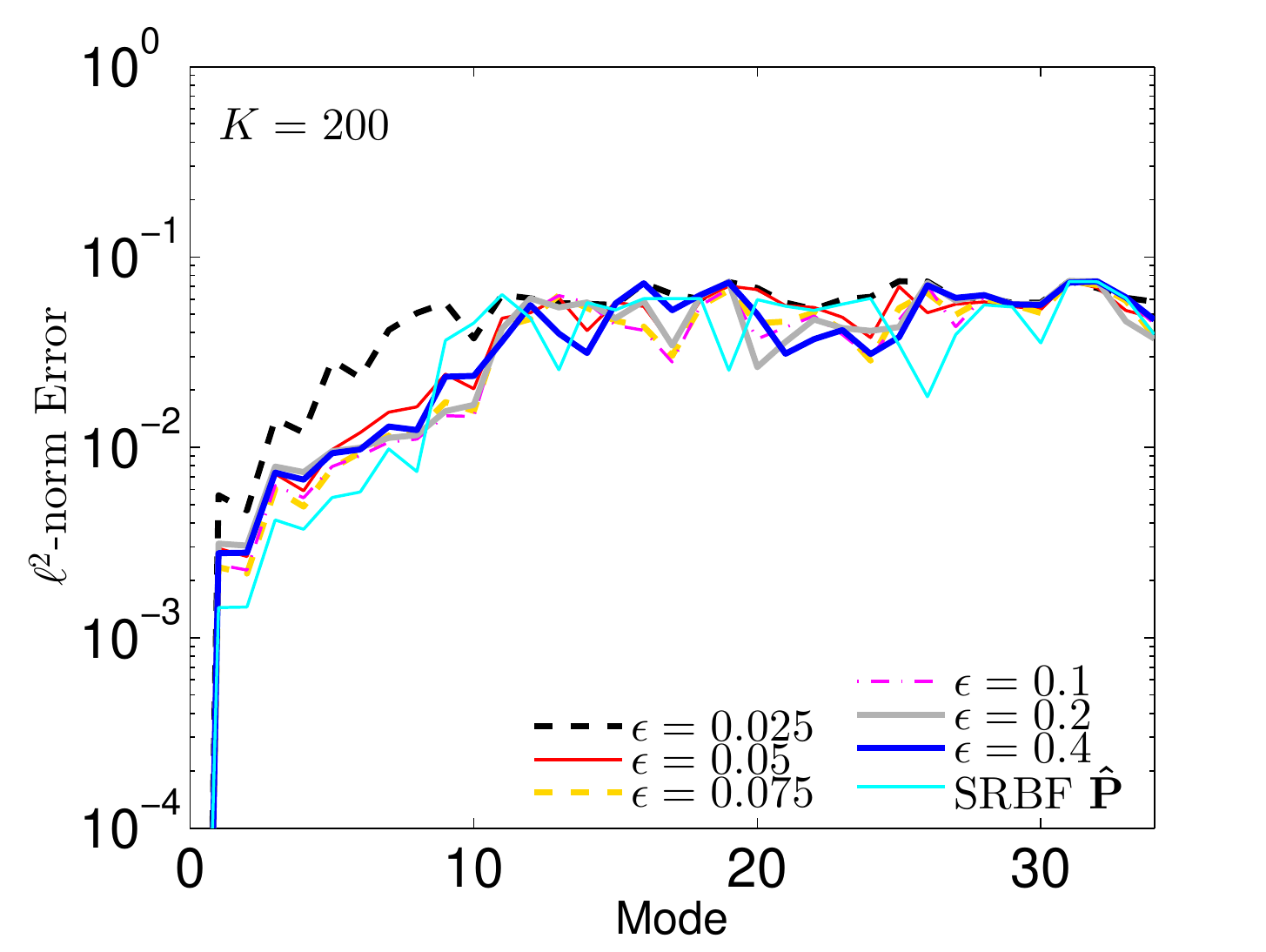}
&
\includegraphics[width=1.5
in, height=1.1 in]{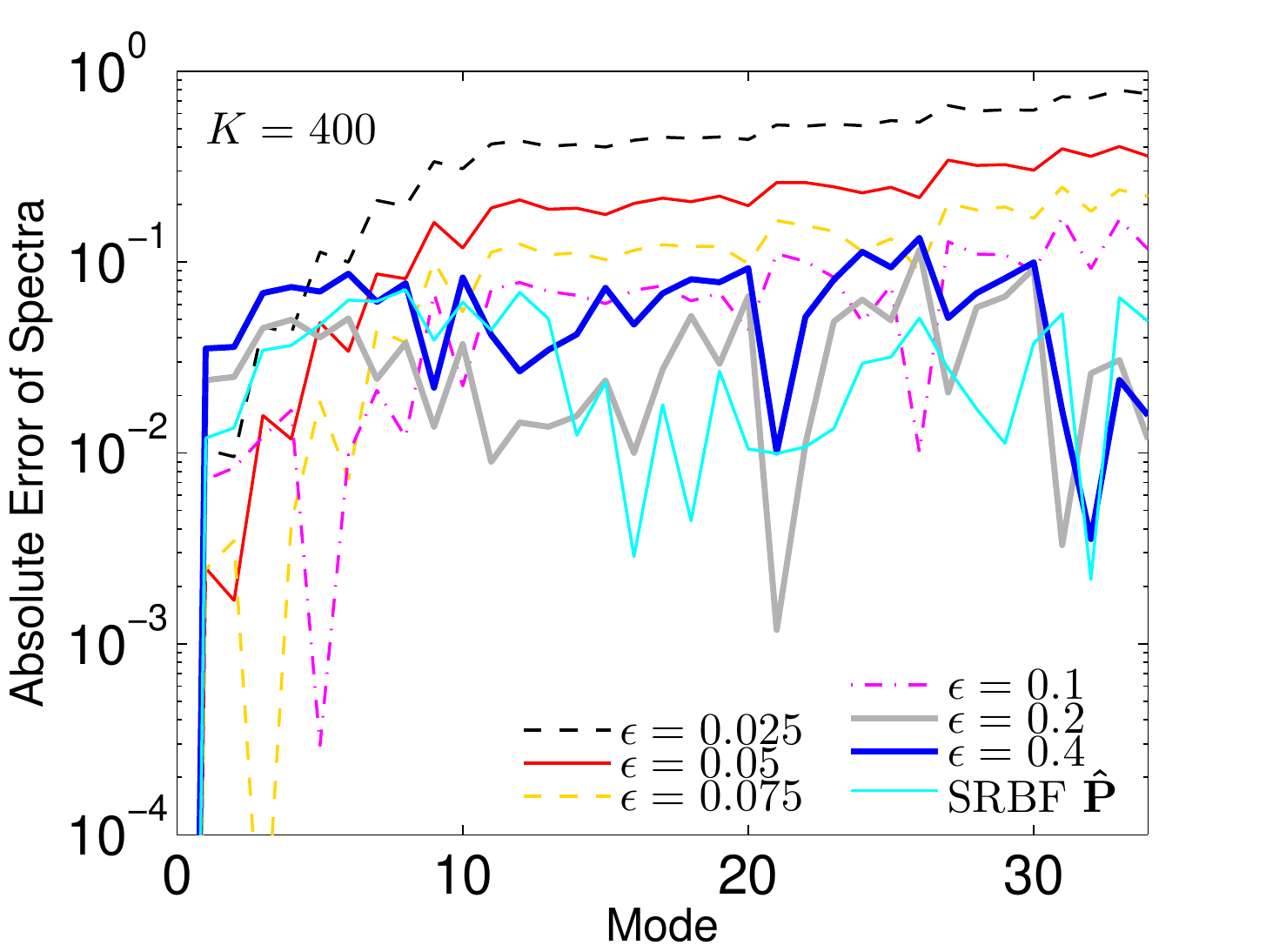}%
&
\includegraphics[width=1.5
in, height=1.1 in]{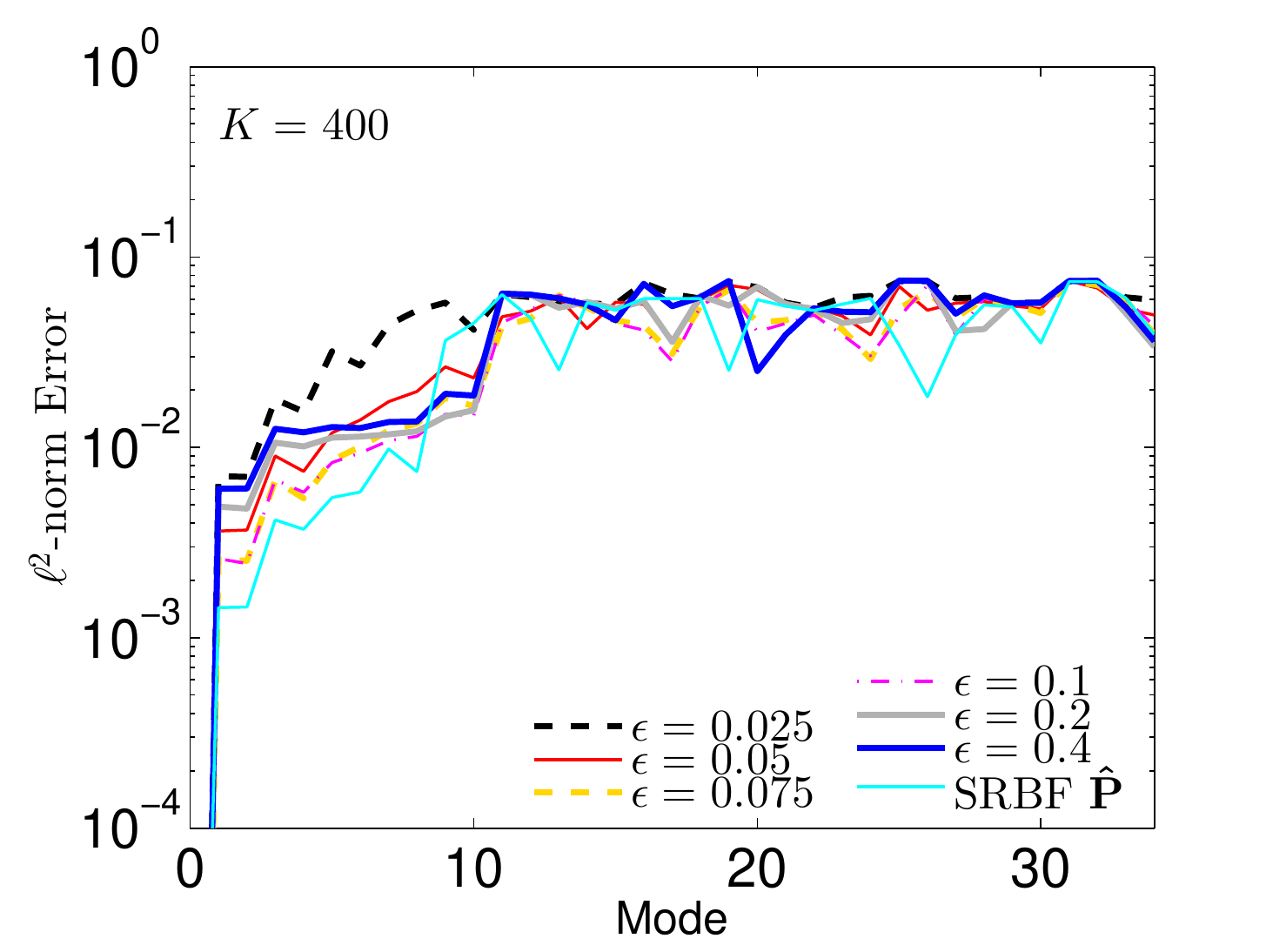}%
\\
\hline
{\small (e) $K={\mathbf{2500}}$} & {\small manually tuned $\epsilon$} & {\small (f) {\bf{auto-tuned}} $\epsilon$} & {\small {\bf{auto-tuned}} $\epsilon$} \\
{\small Error of Eigenvalues} & {\small Error of Eigenvectors} & {\small
Error of Eigenvalues } & {\small Error of Eigenvectors } \\
\includegraphics[width=1.5
in, height=1.1 in]{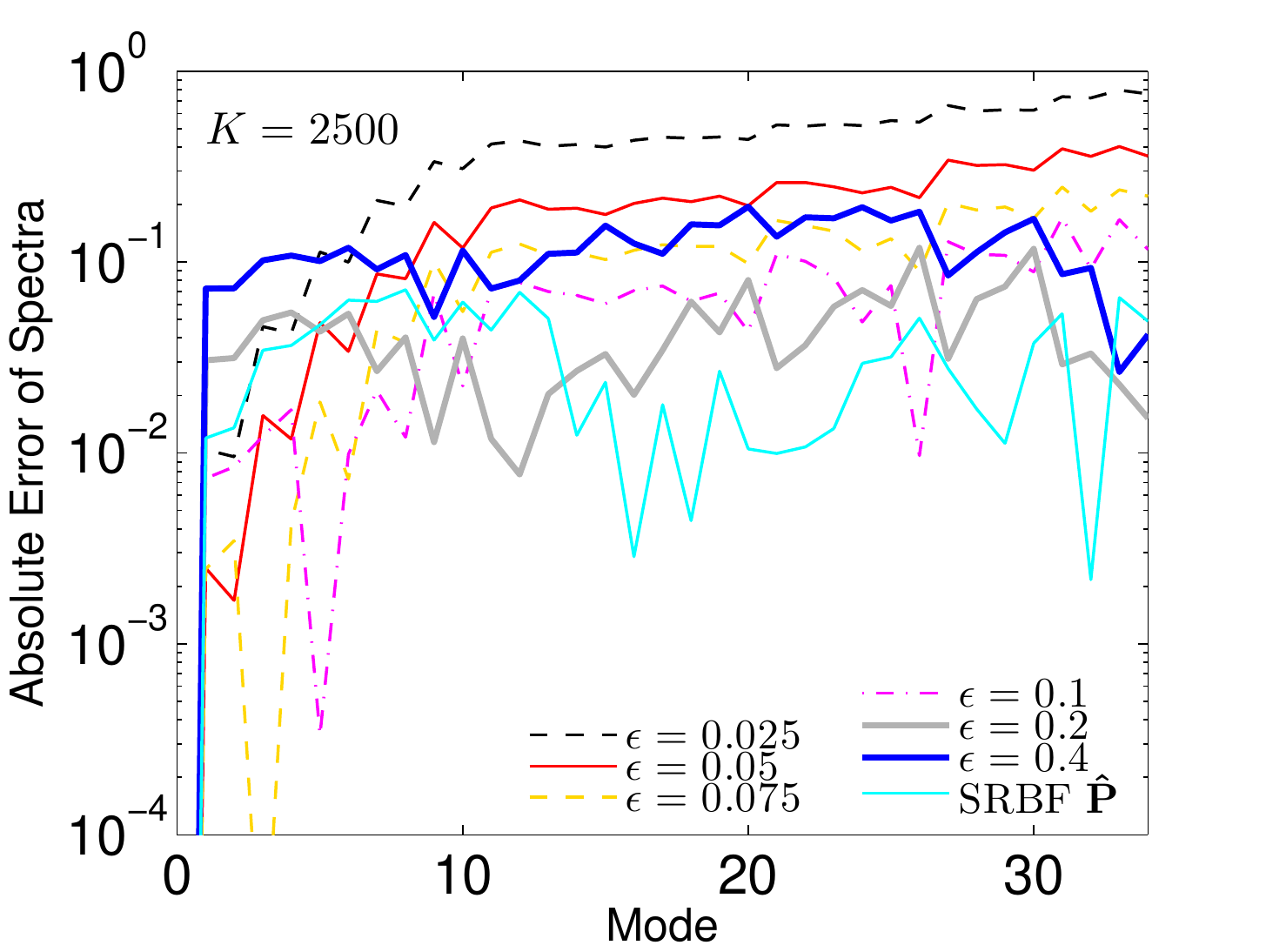}
&
\includegraphics[width=1.5
in, height=1.1 in]{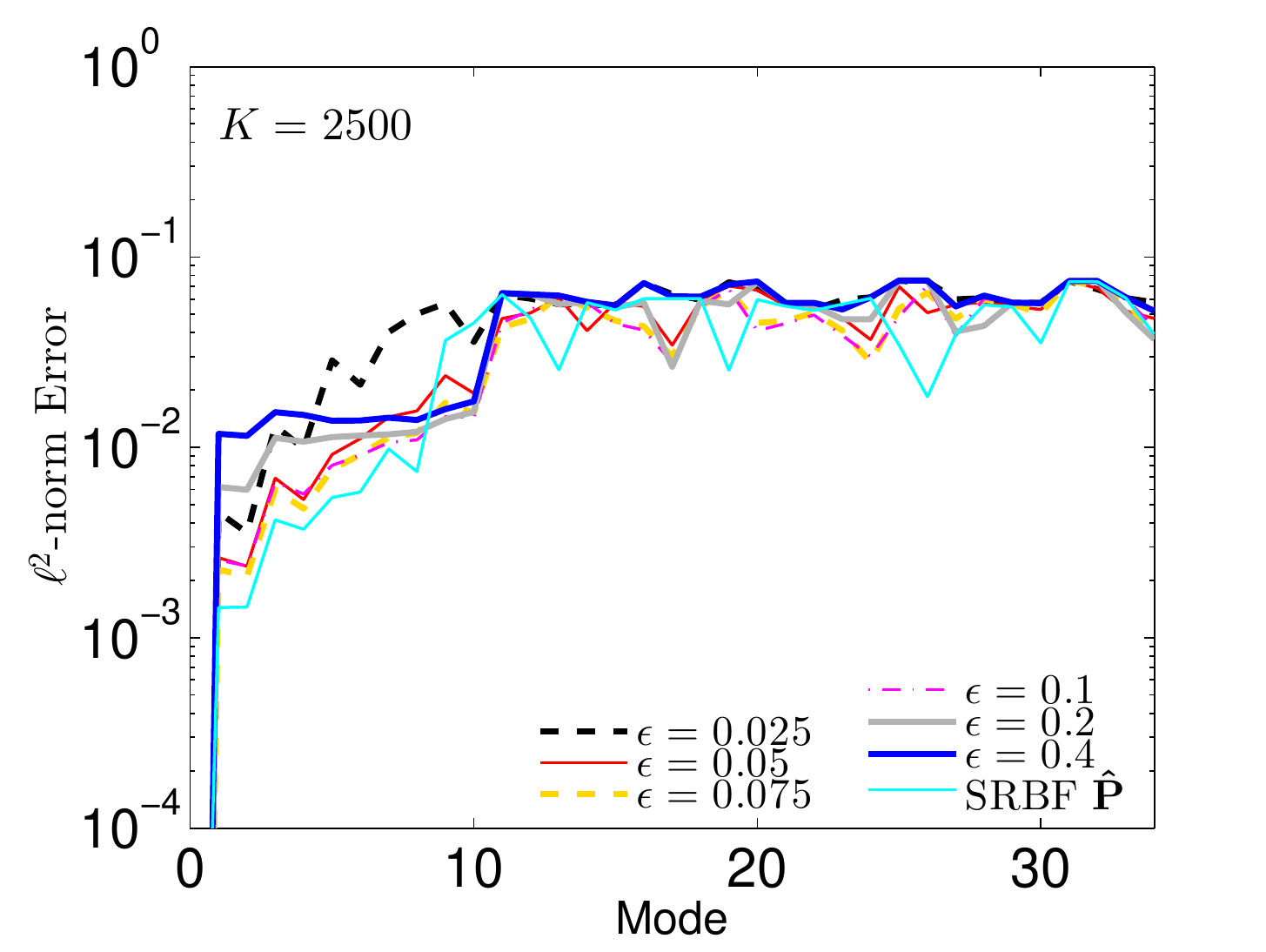}
&
\includegraphics[width=1.5
in, height=1.1 in]{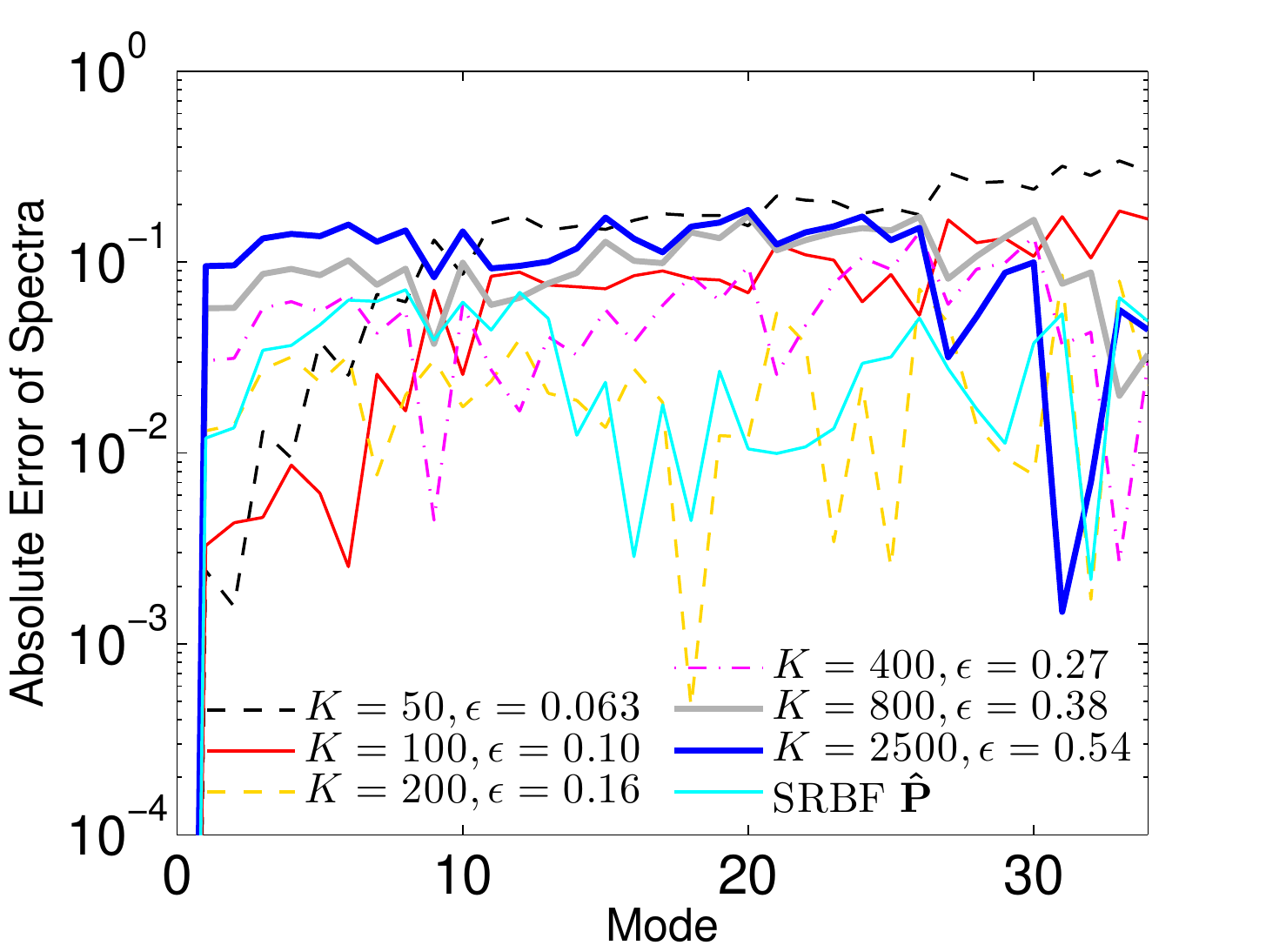}
&
\includegraphics[width=1.5
in, height=1.1 in]{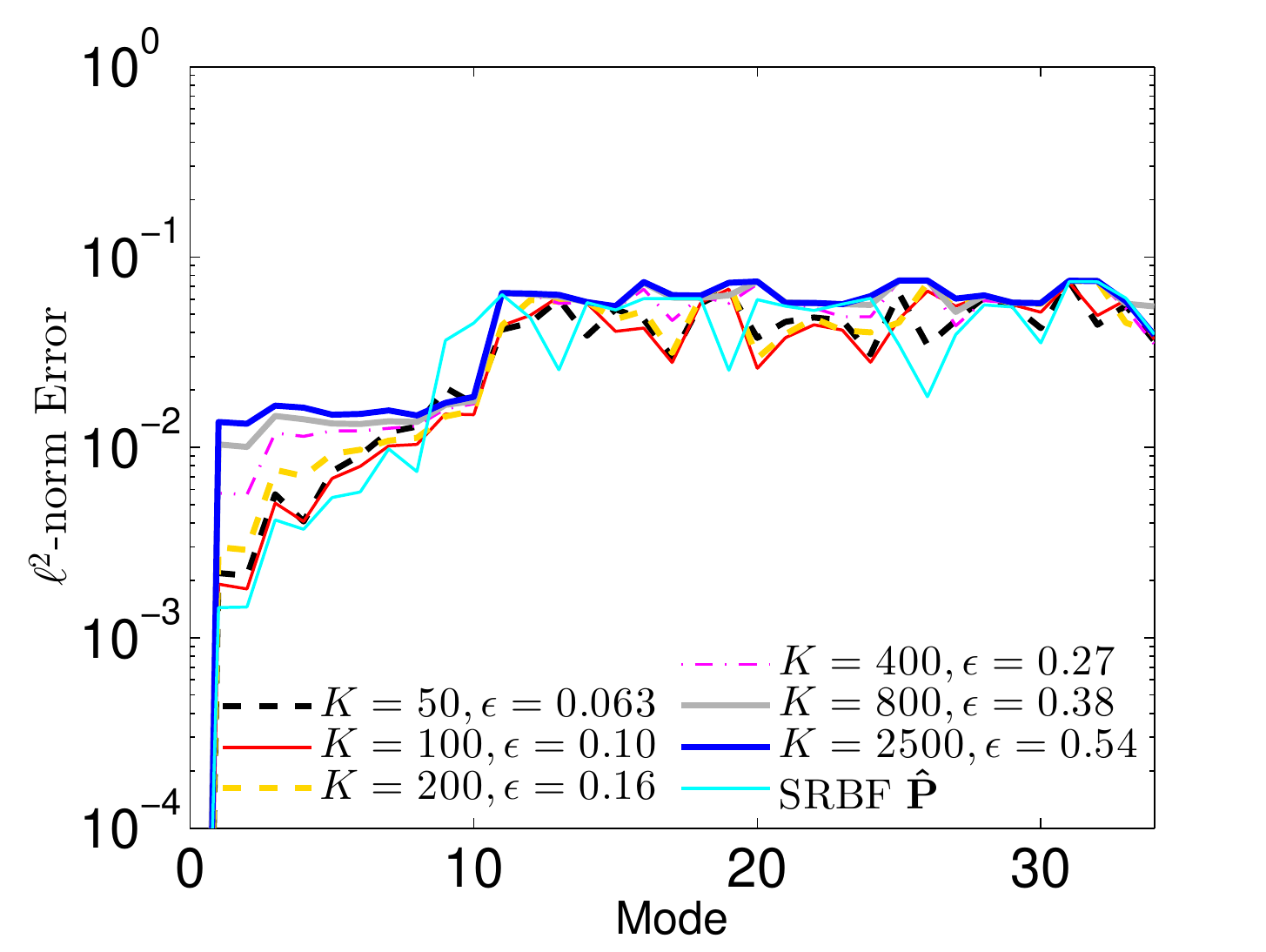}%
\end{tabular}
}
\caption{\textbf{2D general torus in $\mathbb{R}^{21}$.} Comparison of errors of eigenvalues and  errors of eigenvectors for DM among various parameters.
Panels (a)-(e)  correspond to manually-tuned $\protect%
\epsilon$ for $K=50,100,200,400,2500$, respectively. Panel (f) corresponds
to auto-tuned $\epsilon$ for various $K$.
The randomly
distributed $N=2500$ data points on the manifold are used for solving the
eigenvalue problem. }
\label{fig_gentor2}
\end{figure*}

}

\end{document}